\documentclass[12pt,english]{smfbook}

\usepackage[english,french]{babel}

\usepackage{bm,amsfonts,amsmath,amssymb,dsfont,mathrsfs,bbm,fancybox,multicol} 
\usepackage{times}
\usepackage{graphicx} 
\usepackage{enumerate} 
\usepackage{caption}
\usepackage{subcaption}

\makeatletter
\renewcommand\theequation{\thechapter.\arabic{equation}}
\renewcommand\thefigure{\thechapter.\arabic{figure}}
\@addtoreset{equation}{chapter}
\makeatother

\usepackage{diagrams}
\diagramstyle[labelstyle=\scriptstyle]

\usepackage{color}
\definecolor{webred}{rgb}{0.75,0,0}
\usepackage[citecolor=webred,colorlinks=true,linkcolor=webred]{hyperref}

\usepackage{etoolbox}
\patchcmd{\thebibliography}{*}{}{}{}
\pretocmd\thebibliography{\csname c@secnumdepth\endcsname=-2 }{}{}
\patchcmd{\theindex}{*}{}{}{}
\pretocmd\theindex{\csname c@secnumdepth\endcsname=-2 }{}{}

\newtheorem{theorem}{Theorem}[chapter]
\newtheorem{proposition}[theorem]{Proposition}
\newtheorem{lemma}[theorem]{Lemma}
\newtheorem{corollary}[theorem]{Corollary}

\theoremstyle{definition}
\newtheorem{definition}[theorem]{Definition}
\newtheorem{notation}[theorem]{Notation}
\newtheorem{example}[theorem]{Example}
\newtheorem{desc}[theorem]{Description}

\theoremstyle{remark}
\newtheorem{remark}[theorem]{Remark}

\newcommand{\N}{\mathbb{N}}
\newcommand{\R}{\mathbb{R}}

\newcommand{\cA}{\mathcal{A}}
\newcommand{\cB}{\mathcal{B}}
\newcommand{\cC}{\mathcal{C}}
\newcommand{\cD}{\mathcal{D}}

\newcommand{\cF}{\mathcal{F}}

\newcommand{\cI}{\mathcal{I}}
\newcommand{\cO}{\mathcal{O}}

\newcommand{\cS}{\mathcal{S}}

\newcommand{\cU}{\mathcal{U}}
\newcommand{\cV}{\mathcal{V}}
\newcommand{\cW}{\mathcal{W}}

\newcommand{\gA}{\mathfrak{A}}
\newcommand{\gC}{\mathfrak{C}}
\newcommand{\gD}{\mathfrak{D}}
\newcommand{\gE}{\mathfrak{E}}
\newcommand{\gF}{\mathfrak{F}}
\newcommand{\gN}{\mathfrak{N}}

\newcommand{\gO}{\mathfrak{O}}
\newcommand{\gP}{\mathfrak{P}}
\newcommand{\gS}{\mathfrak{S}}
\newcommand{\gT}{\mathfrak{T}}
\newcommand{\gV}{\mathfrak{V}}

\newcommand{\gVc}{\mathfrak{V}^{\circ}}
\newcommand{\gX}{\mathfrak{X}}

\newcommand{\ogD}{\overline\gD}
\newcommand{\ogP}{\overline\gP}

\newcommand{\dD}{\mathbb{D}}
\newcommand{\dy}{\mathbb{Y}}

\newcommand{\dx}{\mathbb{X}}
\newcommand{\dS}{\mathbb{S}}

\newcommand{\dec}{{\bv}}
\newcommand{\tbB}{\widetilde{\bB}}

\newcommand{\bA}{{\boldsymbol{\mathsf{A}}}}
\newcommand{\bB}{{\boldsymbol{\mathsf{B}}}}

\newcommand{\bc}{{\boldsymbol{\mathsf{c}}}}

\newcommand{\be}{{\boldsymbol{\mathsf{e}}}}
\newcommand{\bff}{{\boldsymbol{\mathsf{f}}}}
\newcommand{\bn}{{\boldsymbol{\mathsf{n}}}}
\newcommand{\bp}{{\boldsymbol{\mathsf{p}}}}

\newcommand{\bt}{{\boldsymbol{\mathsf{t}}}}
\newcommand{\bu}{{\boldsymbol{\mathsf{u}}}}
\newcommand{\bv}{{\boldsymbol{\mathsf{v}}}}

\newcommand{\bx}{{\boldsymbol{\mathsf{x}}}}
\newcommand{\by}{{\boldsymbol{\mathsf{y}}}}
\newcommand{\bz}{{\boldsymbol{\mathsf{z}}}}

\newcommand{\bZ}{{\boldsymbol{\mathsf{Z}}}}

\newcommand{\bfz}{{\boldsymbol{0}}}

\newcommand{\sA}{\mathscr{A}}
\newcommand{\sC}{\mathscr{C}}
\newcommand{\sE}{\mathscr{E}}

\newcommand{\sP}{\mathscr{P}}

\newcommand{\sZ}{\mathscr{Z}}

\newcommand{\QR}{\mathscr{Q}}

\newcommand{\rd}{\mathrm{d}}
\newcommand{\re}{\mathrm{e}}
\newcommand{\ri}{i}
\newcommand{\rD}{\mathrm{D}}
\newcommand{\rG}{\mathrm{G}}
\newcommand{\rJ}{\mathrm{J}}

\newcommand{\rK}{\mathrm{K}}

\newcommand{\rN}{\mathrm{N}}
\newcommand{\rT}{\mathrm{T}}

\newcommand{\uA}{\underline{\bA}}
\newcommand{\uB}{\underline{\bB}}

\newcommand{\udiffeo}{\ee\underline{\me\mathrm U\me}\ee} 
\newcommand{\uPsi}{\,\underline{\!\Psi\!}\,}
\newcommand{\upsi}{\underline{\psi\!}\,}

\newcommand{\ess}{\mathrm{ess}}
\newcommand{\loc}{\mathrm{loc}}
\newcommand{\ee}{\hskip 0.15ex}
\newcommand{\me}{\hskip -0.15ex}

\newcommand{\tronc}{\xi_{\bc}} 
\newcommand{\troncg}{\chi_{(\bx,r)}}
\newcommand{\troncp}{\xi_{(\bx,r)}}
\newcommand{\troncz}{\xi_{(\bx_0,r_0)}}
\newcommand\dist{\operatorname{dist}}
\newcommand\dom{\operatorname{Dom}}
\newcommand\supp{\operatorname{supp}}
\newcommand\curl{\operatorname{curl}}

\newcommand\Id{\operatorname{\mathbb{I}}}
\renewcommand\Re{\operatorname{Re}}

\newcommand{\OP}{H} 
\newcommand{\DG}[1]{\mathcal{H}(#1)} 
\newcommand{\En}{E} 
\newcommand{\seE}{\mathscr{E}^*} 
\newcommand{\pot}{\widetilde{\bA}}

\newcommand{\phihX}[1]{\varphi_{h}^{[#1]}}
\newcommand{\Lx}{\Lambda_{\dx}}

\newcommand{\diffeo}{\mathrm U}

\newcommand{\diffeoT}{\mathrm T}
\newcommand{\diffeoZ}{\mathrm Z}

\newcommand{\dir}{\widehat{\bx}}

\newcommand{\de}[1]{\delta_{#1}}

\newcommand{\chiboxh}{\chi^{\scriptscriptstyle\,\Box}_h\,}
\newcommand{\cVboxh}{\cV^{\scriptscriptstyle\,\square}_h}

\newcommand{\Ref}[1]{\textup{\ref{#1}}}

\newcommand{\pp}{p}

\title[Ground state energy of the magnetic Laplacian on corner domains]{Ground state energy\\ of the magnetic Laplacian\\ on corner domains}
\alttitle{Niveau fondamental du laplacien magn\'etique dans des domaines \`a coins}

\author{Virginie Bonnaillie-No\"el}
\address{Virginie Bonnaillie-No\"el, D\'epartement de Math\'ematiques et Applications (DMA UMR 8553), PSL, CNRS, ENS Paris, 45 rue d'Ulm, F-75230 Paris Cedex 05, France} 
\email{virginie.bonnaillie@ens.fr}

\author{Monique Dauge}
\address{Monique Dauge, IRMAR UMR 6625 - CNRS, Universit\'e de Rennes 1, Campus de Beaulieu, 35042 Rennes Cedex, France} 
\email{monique.dauge@univ-rennes1.fr}

\author{Nicolas Popoff}
\address{Nicolas Popoff, IMB UMR 5251 - CNRS, Universit\'e de Bordeaux, 351 cours de la lib\'eration, 33405 Talence Cedex, France}
\email{nicolas.popoff@u-bordeaux.fr}

\makeindex

\begin{document}
\frontmatter

\begin{abstract}
The asymptotic behavior of the first eigenvalue of a magnetic Laplacian in the strong field limit and with the Neumann realization in a smooth domain is characterized for dimensions 2 and 3 by model problems inside the domain or on its boundary. In dimension 2, for polygonal domains, a new set of model problems on sectors has to be taken into account. In this work, we consider the class of general corner domains. In dimension 3, they include as particular cases polyhedra and axisymmetric cones. We attach model problems not only to each point of the closure of the domain, but also to a hierarchy of ``tangent substructures'' associated with singular chains.  
We investigate spectral properties of these model problems, namely semicontinuity and existence of bounded generalized eigenfunctions. We prove estimates for the remainders of our asymptotic formula. Lower bounds are obtained with the help of an IMS type partition based on adequate two-scale coverings of the corner domain, whereas upper bounds are established by a novel construction of quasimodes, qualified as sitting or sliding according to spectral properties of local model problems. A part of our analysis extends to any dimension.
\end{abstract}

\begin{altabstract}
Le comportement asymptotique de la premi\`ere valeur propre du Laplacien magn\'e\-tique en pr\'esence d'un champ de forte intensit\'e et avec les conditions de Neumann sur un domaine r\'egulier, est caract\'eris\'e  en dimension 2 et 3 par des probl\`emes mod\`eles \`a l'int\'erieur du domaine et sur son bord. En dimension 2, quand il s'agit d'un domaine polygonal, on doit inclure dans l'analyse un nouvel ensemble de probl\`emes mod\`eles sur des secteurs plans. Dans ce travail, nous consid\'erons la classe g\'en\'erale des domaines \`a coins. En dimension 3, ceux-ci comprennent en particulier les poly\`edres et les c\^ones de r\'evolution. Nous associons des probl\`emes mod\`eles non seulement \`a chaque point de l'adh\'erence du domaine, mais \'egalement \`a une hi\'erarchie de structures tangentes associ\'ees \`a des cha\^{\i}nes singuli\`eres. Nous explorons des propri\'{e}t\'{e}s spectrales de ces probl\`emes mod\`eles, en particulier la semi-continuit\'e du niveau fondamental et l'existence de vecteurs propres g\'en\'eralis\'es. Nous d\'emontrons des estimations de reste pour nos formules asymptotiques. Les bornes inf\'erieures sont obtenues \`a l'aide de partitions de type IMS bas\'ees sur des recouvrements \`a deux \'echelles des domaines \`a coins. Les bornes sup\'erieures sont \'etablies gr\^ace \`a une construction originale de quasimodes, qualifi\'{e}s de fixes ou glissants selon les propri\'et\'es spectrales des probl\`emes mod\`eles locaux. Une partie de notre analyse s'\'etend \`a la dimension quelconque.
\end{altabstract}


\thanks{
This work was partially supported by the ANR (Agence Nationale de la Recherche), project {\sc Nosevol} ANR-11-BS01-0019. 
The third author was also supported by the {\sc ARCHIMEDE} Labex (ANR-11-LABX-0033) and the A*MIDEX project (ANR-11-IDEX-0001-02) funded by the "Investissements d'Avenir" French government program managed by the ANR.}

\maketitle

{\parskip 0.5pt
\tableofcontents
}

\mainmatter
%
%
%
\ifx\figforTeXisloaded\relax \else\global\let\figforTeXisloaded=\relax\fi
\message{version 1.9}
\catcode`\@=11
\ifx\ctr@ln@m\undefined\else%
    \immediate\write16{*** Fig4TeX WARNING : \string\ctr@ln@m\space already defined.}\fi
\def\ctr@ln@m#1{\ifx#1\undefined\else%
    \immediate\write16{*** Fig4TeX WARNING : \string#1 already defined.}\fi}
\ctr@ln@m\ctr@ld@f
\def\ctr@ld@f#1#2{\ctr@ln@m#2#1#2}
\ctr@ld@f\def\ctr@ln@w#1#2{\ctr@ln@m#2\csname#1\endcsname#2}
{\catcode`\/=0 \catcode`/\=12 /ctr@ld@f/gdef/BS@{\}}
\ctr@ld@f\def\ctr@lcsn@m#1{\expandafter\ifx\csname#1\endcsname\relax\else%
    \immediate\write16{*** Fig4TeX WARNING : \BS@\expandafter\string#1\space already defined.}\fi}
\ctr@ld@f\edef\colonc@tcode{\the\catcode`\:}
\ctr@ld@f\edef\semicolonc@tcode{\the\catcode`\;}
\ctr@ld@f\def\t@stc@tcodech@nge{{\let\c@tcodech@nged=\z@%
    \ifnum\colonc@tcode=\the\catcode`\:\else\let\c@tcodech@nged=\@ne\fi%
    \ifnum\semicolonc@tcode=\the\catcode`\;\else\let\c@tcodech@nged=\@ne\fi%
    \ifx\c@tcodech@nged\@ne%
    \immediate\write16{}
    \immediate\write16{!!!=============================================================!!!}
    \immediate\write16{ Fig4TeX WARNING:}
    \immediate\write16{ The category code of some characters has been changed, which will}
    \immediate\write16{ result in an error (message "Runaway argument?").}
    \immediate\write16{ This probably comes from another package that changed the category}
    \immediate\write16{ code after Fig4TeX was loaded. If that proves to be exact, the}
    \immediate\write16{ solution is to exchange the loading commands on top of your file}
    \immediate\write16{ so that Fig4TeX is loaded last. For example, in LaTeX, we should}
    \immediate\write16{ say :}
    \immediate\write16{\BS@ usepackage[french]{babel}}
    \immediate\write16{\BS@ usepackage{fig4tex}}
    \immediate\write16{!!!=============================================================!!!}
    \immediate\write16{}
    \fi}}
\ctr@ld@f\def\FigforTeX{F\kern-.05em i\kern-.05em g\kern-.1em\raise-.14em\hbox{4}\kern-.19em\TeX}
\ctr@ld@f\def\W@rnmesoldA#1{\W@rnmesold}
\ctr@ld@f\def\W@rnmesoldAB#1(#2){\W@rnmesold}
\ctr@ld@f\def\W@rnmesold{%
    \immediate\write16{}
    \immediate\write16{!!!=============================================================!!!}
    \immediate\write16{ Fig4TeX WARNING:}
    \immediate\write16{ The file to be compiled is not compatible with the current version}
    \immediate\write16{ of Fig4TeX. To fix that, upgrade the source file (mainly change \BS@ ps*}
    \immediate\write16{ macros by \BS@ fig* macros), or use fig4tex184.tex instead (\BS@ input fig4tex184}
    \immediate\write16{ or \BS@ usepackage{fig4tex184}).}
    \immediate\write16{!!!=============================================================!!!}
    \immediate\write16{}}
\ctr@ln@m\psbeginfig\let\psbeginfig\W@rnmesoldA
\ctr@ln@m\psset\let\psset\W@rnmesoldAB
\ctr@ln@m\pssetdefault\let\pssetdefault\W@rnmesoldAB
\ctr@ln@m\pssetupdate\let\pssetupdate\W@rnmesoldA
\ctr@ln@w{newdimen}\epsil@n\epsil@n=0.00005pt
\ctr@ln@w{newdimen}\Cepsil@n\Cepsil@n=0.005pt
\ctr@ln@w{newdimen}\dcq@\dcq@=254pt
\ctr@ln@w{newdimen}\PI@\PI@=3.141592pt
\ctr@ln@w{newdimen}\DemiPI@deg\DemiPI@deg=90pt
\ctr@ln@w{newdimen}\PI@deg\PI@deg=180pt
\ctr@ln@w{newdimen}\DePI@deg\DePI@deg=360pt
\ctr@ld@f\chardef\t@n=10
\ctr@ld@f\chardef\c@nt=100
\ctr@ld@f\chardef\@lxxiv=74
\ctr@ld@f\chardef\@xci=91
\ctr@ld@f\mathchardef\@nMnCQn=9949
\ctr@ld@f\chardef\@vi=6
\ctr@ld@f\chardef\@xxx=30
\ctr@ld@f\chardef\@lvi=56
\ctr@ld@f\chardef\@@lxxi=71
\ctr@ld@f\chardef\@lxxxv=85
\ctr@ld@f\mathchardef\@@mmmmlxviii=4068
\ctr@ld@f\mathchardef\@ccclx=360
\ctr@ld@f\mathchardef\@dccxx=720
\ctr@ln@w{newcount}\p@rtent \ctr@ln@w{newcount}\f@ctech \ctr@ln@w{newcount}\result@tent
\ctr@ln@w{newdimen}\v@lmin \ctr@ln@w{newdimen}\v@lmax \ctr@ln@w{newdimen}\v@leur
\ctr@ln@w{newdimen}\result@t\ctr@ln@w{newdimen}\result@@t
\ctr@ln@w{newdimen}\mili@u \ctr@ln@w{newdimen}\c@rre \ctr@ln@w{newdimen}\delt@
\ctr@ld@f\def\degT@rd{0.017453 }  
\ctr@ld@f\def\rdT@deg{57.295779 } 
\ctr@ln@m\v@leurseule
{\catcode`p=12 \catcode`t=12 \gdef\v@leurseule#1pt{#1}}
\ctr@ld@f\def\repdecn@mb#1{\expandafter\v@leurseule\the#1\space}
\ctr@ld@f\def\arct@n#1(#2,#3){{\v@lmin=#2\v@lmax=#3%
    \maxim@m{\mili@u}{-\v@lmin}{\v@lmin}\maxim@m{\c@rre}{-\v@lmax}{\v@lmax}%
    \delt@=\mili@u\m@ech\mili@u%
    \ifdim\c@rre>\@nMnCQn\mili@u\divide\v@lmax\tw@\c@lATAN\v@leur(\z@,\v@lmax)
    \else%
    \maxim@m{\mili@u}{-\v@lmin}{\v@lmin}\maxim@m{\c@rre}{-\v@lmax}{\v@lmax}%
    \m@ech\c@rre%
    \ifdim\mili@u>\@nMnCQn\c@rre\divide\v@lmin\tw@
    \maxim@m{\mili@u}{-\v@lmin}{\v@lmin}\c@lATAN\v@leur(\mili@u,\z@)%
    \else\c@lATAN\v@leur(\delt@,\v@lmax)\fi\fi%
    \ifdim\v@lmin<\z@\v@leur=-\v@leur\ifdim\v@lmax<\z@\advance\v@leur-\PI@%
    \else\advance\v@leur\PI@\fi\fi%
    \global\result@t=\v@leur}#1=\result@t}
\ctr@ld@f\def\m@ech#1{\ifdim#1>1.646pt\divide\mili@u\t@n\divide\c@rre\t@n\m@ech#1\fi}
\ctr@ld@f\def\c@lATAN#1(#2,#3){{\v@lmin=#2\v@lmax=#3\v@leur=\z@\delt@=\tw@ pt%
    \un@iter{0.785398}{\v@lmax<}%
    \un@iter{0.463648}{\v@lmax<}%
    \un@iter{0.244979}{\v@lmax<}%
    \un@iter{0.124355}{\v@lmax<}%
    \un@iter{0.062419}{\v@lmax<}%
    \un@iter{0.031240}{\v@lmax<}%
    \un@iter{0.015624}{\v@lmax<}%
    \un@iter{0.007812}{\v@lmax<}%
    \un@iter{0.003906}{\v@lmax<}%
    \un@iter{0.001953}{\v@lmax<}%
    \un@iter{0.000976}{\v@lmax<}%
    \un@iter{0.000488}{\v@lmax<}%
    \un@iter{0.000244}{\v@lmax<}%
    \un@iter{0.000122}{\v@lmax<}%
    \un@iter{0.000061}{\v@lmax<}%
    \un@iter{0.000030}{\v@lmax<}%
    \un@iter{0.000015}{\v@lmax<}%
    \global\result@t=\v@leur}#1=\result@t}
\ctr@ld@f\def\un@iter#1#2{%
    \divide\delt@\tw@\edef\dpmn@{\repdecn@mb{\delt@}}%
    \mili@u=\v@lmin%
    \ifdim#2\z@%
      \advance\v@lmin-\dpmn@\v@lmax\advance\v@lmax\dpmn@\mili@u%
      \advance\v@leur-#1pt%
    \else%
      \advance\v@lmin\dpmn@\v@lmax\advance\v@lmax-\dpmn@\mili@u%
      \advance\v@leur#1pt%
    \fi}
\ctr@ld@f\def\c@ssin#1#2#3{\expandafter\ifx\csname COS@\number#3\endcsname\relax\c@lCS{#3pt}%
    \expandafter\xdef\csname COS@\number#3\endcsname{\repdecn@mb\result@t}%
    \expandafter\xdef\csname SIN@\number#3\endcsname{\repdecn@mb\result@@t}\fi%
    \edef#1{\csname COS@\number#3\endcsname}\edef#2{\csname SIN@\number#3\endcsname}}
\ctr@ld@f\def\c@lCS#1{{\mili@u=#1\p@rtent=\@ne%
    \relax\ifdim\mili@u<\z@\red@ng<-\else\red@ng>+\fi\f@ctech=\p@rtent%
    \relax\ifdim\mili@u<\z@\mili@u=-\mili@u\f@ctech=-\f@ctech\fi\c@@lCS}}
\ctr@ld@f\def\c@@lCS{\v@lmin=\mili@u\c@rre=-\mili@u\advance\c@rre\DemiPI@deg\v@lmax=\c@rre%
    \mili@u\@@lxxi\mili@u\divide\mili@u\@@mmmmlxviii%
    \edef\v@larg{\repdecn@mb{\mili@u}}\mili@u=-\v@larg\mili@u%
    \edef\v@lmxde{\repdecn@mb{\mili@u}}%
    \c@rre\@@lxxi\c@rre\divide\c@rre\@@mmmmlxviii%
    \edef\v@largC{\repdecn@mb{\c@rre}}\c@rre=-\v@largC\c@rre%
    \edef\v@lmxdeC{\repdecn@mb{\c@rre}}%
    \fctc@s\mili@u\v@lmin\global\result@t\p@rtent\v@leur%
    \let\t@mp=\v@larg\let\v@larg=\v@largC\let\v@largC=\t@mp%
    \let\t@mp=\v@lmxde\let\v@lmxde=\v@lmxdeC\let\v@lmxdeC=\t@mp%
    \fctc@s\c@rre\v@lmax\global\result@@t\f@ctech\v@leur}
\ctr@ld@f\def\fctc@s#1#2{\v@leur=#1\relax\ifdim#2<\@lxxxv\p@\cosser@h\else\sinser@t\fi}
\ctr@ld@f\def\cosser@h{\advance\v@leur\@lvi\p@\divide\v@leur\@lvi%
    \v@leur=\v@lmxde\v@leur\advance\v@leur\@xxx\p@%
    \v@leur=\v@lmxde\v@leur\advance\v@leur\@ccclx\p@%
    \v@leur=\v@lmxde\v@leur\advance\v@leur\@dccxx\p@\divide\v@leur\@dccxx}
\ctr@ld@f\def\sinser@t{\v@leur=\v@lmxdeC\p@\advance\v@leur\@vi\p@%
    \v@leur=\v@largC\v@leur\divide\v@leur\@vi}
\ctr@ld@f\def\red@ng#1#2{\relax\ifdim\mili@u#1#2\DemiPI@deg\advance\mili@u#2-\PI@deg%
    \p@rtent=-\p@rtent\red@ng#1#2\fi}
\ctr@ld@f\def\pr@c@lCS#1#2#3{\ctr@lcsn@m{COS@\number#3 }%
    \expandafter\xdef\csname COS@\number#3\endcsname{#1}%
    \expandafter\xdef\csname SIN@\number#3\endcsname{#2}}
\pr@c@lCS{1}{0}{0}
\pr@c@lCS{0.7071}{0.7071}{45}\pr@c@lCS{0.7071}{-0.7071}{-45}
\pr@c@lCS{0}{1}{90}          \pr@c@lCS{0}{-1}{-90}
\pr@c@lCS{-1}{0}{180}        \pr@c@lCS{-1}{0}{-180}
\pr@c@lCS{0}{-1}{270}        \pr@c@lCS{0}{1}{-270}
\ctr@ld@f\def\invers@#1#2{{\v@leur=#2\maxim@m{\v@lmax}{-\v@leur}{\v@leur}%
    \f@ctech=\@ne\m@inv@rs%
    \multiply\v@leur\f@ctech\edef\v@lv@leur{\repdecn@mb{\v@leur}}%
    \p@rtentiere{\p@rtent}{\v@leur}\v@lmin=\p@\divide\v@lmin\p@rtent%
    \inv@rs@\multiply\v@lmax\f@ctech\global\result@t=\v@lmax}#1=\result@t}
\ctr@ld@f\def\m@inv@rs{\ifdim\v@lmax<\p@\multiply\v@lmax\t@n\multiply\f@ctech\t@n\m@inv@rs\fi}
\ctr@ld@f\def\inv@rs@{\v@lmax=-\v@lmin\v@lmax=\v@lv@leur\v@lmax%
    \advance\v@lmax\tw@ pt\v@lmax=\repdecn@mb{\v@lmin}\v@lmax%
    \delt@=\v@lmax\advance\delt@-\v@lmin\ifdim\delt@<\z@\delt@=-\delt@\fi%
    \ifdim\delt@>\epsil@n\v@lmin=\v@lmax\inv@rs@\fi}
\ctr@ld@f\def\minim@m#1#2#3{\relax\ifdim#2<#3#1=#2\else#1=#3\fi}
\ctr@ld@f\def\maxim@m#1#2#3{\relax\ifdim#2>#3#1=#2\else#1=#3\fi}
\ctr@ld@f\def\p@rtentiere#1#2{#1=#2\divide#1by65536 }
\ctr@ld@f\def\r@undint#1#2{{\v@leur=#2\divide\v@leur\t@n\p@rtentiere{\p@rtent}{\v@leur}%
    \v@leur=\p@rtent pt\global\result@t=\t@n\v@leur}#1=\result@t}
\ctr@ld@f\def\sqrt@#1#2{{\v@leur=#2%
    \minim@m{\v@lmin}{\p@}{\v@leur}\maxim@m{\v@lmax}{\p@}{\v@leur}%
    \f@ctech=\@ne\m@sqrt@\sqrt@@%
    \mili@u=\v@lmin\advance\mili@u\v@lmax\divide\mili@u\tw@\multiply\mili@u\f@ctech%
    \global\result@t=\mili@u}#1=\result@t}
\ctr@ld@f\def\m@sqrt@{\ifdim\v@leur>\dcq@\divide\v@leur\c@nt\v@lmax=\v@leur%
    \multiply\f@ctech\t@n\m@sqrt@\fi}
\ctr@ld@f\def\sqrt@@{\mili@u=\v@lmin\advance\mili@u\v@lmax\divide\mili@u\tw@%
    \c@rre=\repdecn@mb{\mili@u}\mili@u%
    \ifdim\c@rre<\v@leur\v@lmin=\mili@u\else\v@lmax=\mili@u\fi%
    \delt@=\v@lmax\advance\delt@-\v@lmin\ifdim\delt@>\epsil@n\sqrt@@\fi}
\ctr@ld@f\def\extrairelepremi@r#1\de#2{\expandafter\lepremi@r#2@#1#2}
\ctr@ld@f\def\lepremi@r#1,#2@#3#4{\def#3{#1}\def#4{#2}\ignorespaces}
\ctr@ld@f\def\@cfor#1:=#2\do#3{%
  \edef\@fortemp{#2}%
  \ifx\@fortemp\empty\else\@cforloop#2,\@nil,\@nil\@@#1{#3}\fi}
\ctr@ln@m\@nextwhile
\ctr@ld@f\def\@cforloop#1,#2\@@#3#4{%
  \def#3{#1}%
  \ifx#3\Fig@nnil\let\@nextwhile=\Fig@fornoop\else#4\relax\let\@nextwhile=\@cforloop\fi%
  \@nextwhile#2\@@#3{#4}}

\ctr@ld@f\def\@ecfor#1:=#2\do#3{%
  \def\@@cfor{\@cfor#1:=}%
  \edef\@@@cfor{#2}%
  \expandafter\@@cfor\@@@cfor\do{#3}}
\ctr@ld@f\def\Fig@nnil{\@nil}
\ctr@ld@f\def\Fig@fornoop#1\@@#2#3{}
\ctr@ln@m\list@@rg
\ctr@ld@f\def\trtlis@rg#1#2{\def\list@@rg{#1}%
    \@ecfor\p@rv@l:=\list@@rg\do{\expandafter#2\p@rv@l|}}
\ctr@ld@f\def\trtlis@rgtok#1{\let@xte={}\let\n@xt\addt@t@xt\addt@t@xt #1}
\ctr@ln@m\M@cro
\ctr@ln@m\n@xt
\ctr@ld@f\def\addt@t@xt#1{\if#1|\let\n@xt\relax\else%
    \if#1,\expandafter\M@cro\the\let@xte|\let@xte={}%
    \else\let@xte=\expandafter{\the\let@xte #1}\fi\fi\n@xt}
\ctr@ln@w{newbox}\b@xvisu
\ctr@ln@w{newtoks}\let@xte
\ctr@ln@w{newif}\ifitis@K
\ctr@ln@w{newcount}\s@mme
\ctr@ln@w{newcount}\l@mbd@un \ctr@ln@w{newcount}\l@mbd@de
\ctr@ln@w{newcount}\superc@ntr@l\superc@ntr@l=\@ne        
\ctr@ln@w{newcount}\typec@ntr@l\typec@ntr@l=\superc@ntr@l 
\ctr@ln@w{newdimen}\v@lX  \ctr@ln@w{newdimen}\v@lY  \ctr@ln@w{newdimen}\v@lZ
\ctr@ln@w{newdimen}\v@lXa \ctr@ln@w{newdimen}\v@lYa \ctr@ln@w{newdimen}\v@lZa
\ctr@ln@w{newdimen}\unit@\unit@=\p@ 
\ctr@ld@f\def\unit@util{pt}
\ctr@ld@f\def\ptT@ptps{0.996264 }
\ctr@ld@f\def\ptpsT@pt{1.00375 }
\ctr@ld@f\def\ptT@unit@{1} 
\ctr@ld@f\def\setunit@#1{\def\unit@util{#1}\setunit@@#1:\invers@{\result@t}{\unit@}%
    \edef\ptT@unit@{\repdecn@mb\result@t}}
\ctr@ld@f\def\setunit@@#1#2:{\ifcat#1a\unit@=\@ne#1#2\else\unit@=#1#2\fi}
\ctr@ld@f\def\d@fm@cdim#1#2{{\v@leur=#2\v@leur=\ptT@unit@\v@leur\xdef#1{\repdecn@mb\v@leur}}}
\ctr@ln@w{newif}\ifBdingB@x\BdingB@xtrue
\ctr@ln@w{newdimen}\c@@rdXmin \ctr@ln@w{newdimen}\c@@rdYmin  
\ctr@ln@w{newdimen}\c@@rdXmax \ctr@ln@w{newdimen}\c@@rdYmax
\ctr@ld@f\def\b@undb@x#1#2{\ifBdingB@x%
    \relax\ifdim#1<\c@@rdXmin\global\c@@rdXmin=#1\fi%
    \relax\ifdim#2<\c@@rdYmin\global\c@@rdYmin=#2\fi%
    \relax\ifdim#1>\c@@rdXmax\global\c@@rdXmax=#1\fi%
    \relax\ifdim#2>\c@@rdYmax\global\c@@rdYmax=#2\fi\fi}
\ctr@ld@f\def\b@undb@xP#1{{\Figg@tXY{#1}\b@undb@x{\v@lX}{\v@lY}}}
\ctr@ld@f\def\ellBB@x#1;#2,#3(#4,#5,#6){{\s@uvc@ntr@l\et@tellBB@x%
    \setc@ntr@l{2}\figptell-2::#1;#2,#3(#4,#6)\b@undb@xP{-2}%
    \figptell-2::#1;#2,#3(#5,#6)\b@undb@xP{-2}%
    \c@ssin{\C@}{\S@}{#6}\v@lmin=\C@ pt\v@lmax=\S@ pt%
    \mili@u=#3\v@lmin\delt@=#2\v@lmax\arct@n\v@leur(\delt@,\mili@u)%
    \mili@u=-#3\v@lmax\delt@=#2\v@lmin\arct@n\c@rre(\delt@,\mili@u)%
    \v@leur=\rdT@deg\v@leur\advance\v@leur-\DePI@deg%
    \c@rre=\rdT@deg\c@rre\advance\c@rre-\DePI@deg%
    \v@lmin=#4pt\v@lmax=#5pt%
    \loop\ifdim\v@leur<\v@lmax\ifdim\v@leur>\v@lmin%
    \edef\@ngle{\repdecn@mb\v@leur}\figptell-2::#1;#2,#3(\@ngle,#6)%
    \b@undb@xP{-2}\fi\advance\v@leur\PI@deg\repeat%
    \loop\ifdim\c@rre<\v@lmax\ifdim\c@rre>\v@lmin%
    \edef\@ngle{\repdecn@mb\c@rre}\figptell-2::#1;#2,#3(\@ngle,#6)%
    \b@undb@xP{-2}\fi\advance\c@rre\PI@deg\repeat%
    \resetc@ntr@l\et@tellBB@x}\ignorespaces}
\ctr@ld@f\def\initb@undb@x{\c@@rdXmin=\maxdimen\c@@rdYmin=\maxdimen%
    \c@@rdXmax=-\maxdimen\c@@rdYmax=-\maxdimen}
\ctr@ld@f\def\c@ntr@lnum#1{%
    \relax\ifnum\typec@ntr@l=\@ne%
    \ifnum#1<\z@%
    \immediate\write16{*** Forbidden point number (#1). Abort.}\end\fi\fi%
    \set@bjc@de{#1}}
\ctr@ln@m\objc@de
\ctr@ld@f\def\set@bjc@de#1{\edef\objc@de{@BJ\ifnum#1<\z@ M\romannumeral-#1\else\romannumeral#1\fi}}
\s@mme=\m@ne\loop\ifnum\s@mme>-19
  \set@bjc@de{\s@mme}\ctr@lcsn@m\objc@de\ctr@lcsn@m{\objc@de T}
\advance\s@mme\m@ne\repeat
\s@mme=\@ne\loop\ifnum\s@mme<6
  \set@bjc@de{\s@mme}\ctr@lcsn@m\objc@de\ctr@lcsn@m{\objc@de T}
\advance\s@mme\@ne\repeat
\ctr@ld@f\def\setc@ntr@l#1{\ifnum\superc@ntr@l>#1\typec@ntr@l=\superc@ntr@l%
    \else\typec@ntr@l=#1\fi}
\ctr@ld@f\def\resetc@ntr@l#1{\global\superc@ntr@l=#1\setc@ntr@l{#1}}
\ctr@ld@f\def\s@uvc@ntr@l#1{\edef#1{\the\superc@ntr@l}}
\ctr@ln@m\c@lproscal
\ctr@ld@f\def\c@lproscalDD#1[#2,#3]{{\Figg@tXY{#2}%
    \edef\Xu@{\repdecn@mb{\v@lX}}\edef\Yu@{\repdecn@mb{\v@lY}}\Figg@tXY{#3}%
    \global\result@t=\Xu@\v@lX\global\advance\result@t\Yu@\v@lY}#1=\result@t}
\ctr@ld@f\def\c@lproscalTD#1[#2,#3]{{\Figg@tXY{#2}\edef\Xu@{\repdecn@mb{\v@lX}}%
    \edef\Yu@{\repdecn@mb{\v@lY}}\edef\Zu@{\repdecn@mb{\v@lZ}}%
    \Figg@tXY{#3}\global\result@t=\Xu@\v@lX\global\advance\result@t\Yu@\v@lY%
    \global\advance\result@t\Zu@\v@lZ}#1=\result@t}
\ctr@ld@f\def\c@lprovec#1{%
    \det@rmC\v@lZa(\v@lX,\v@lY,\v@lmin,\v@lmax)%
    \det@rmC\v@lXa(\v@lY,\v@lZ,\v@lmax,\v@leur)%
    \det@rmC\v@lYa(\v@lZ,\v@lX,\v@leur,\v@lmin)%
    \Figv@ctCreg#1(\v@lXa,\v@lYa,\v@lZa)}
\ctr@ld@f\def\det@rm#1[#2,#3]{{\Figg@tXY{#2}\Figg@tXYa{#3}%
    \delt@=\repdecn@mb{\v@lX}\v@lYa\advance\delt@-\repdecn@mb{\v@lY}\v@lXa%
    \global\result@t=\delt@}#1=\result@t}
\ctr@ld@f\def\det@rmC#1(#2,#3,#4,#5){{\global\result@t=\repdecn@mb{#2}#5%
    \global\advance\result@t-\repdecn@mb{#3}#4}#1=\result@t}
\ctr@ld@f\def\getredf@ctDD#1(#2,#3){{\maxim@m{\v@lXa}{-#2}{#2}\maxim@m{\v@lYa}{-#3}{#3}%
    \maxim@m{\v@lXa}{\v@lXa}{\v@lYa}
    \ifdim\v@lXa>\@xci pt\divide\v@lXa\@xci%
    \p@rtentiere{\p@rtent}{\v@lXa}\advance\p@rtent\@ne\else\p@rtent=\@ne\fi%
    \global\result@tent=\p@rtent}#1=\result@tent\ignorespaces}
\ctr@ld@f\def\getredf@ctTD#1(#2,#3,#4){{\maxim@m{\v@lXa}{-#2}{#2}\maxim@m{\v@lYa}{-#3}{#3}%
    \maxim@m{\v@lZa}{-#4}{#4}\maxim@m{\v@lXa}{\v@lXa}{\v@lYa}%
    \maxim@m{\v@lXa}{\v@lXa}{\v@lZa}
    \ifdim\v@lXa>\@lxxiv pt\divide\v@lXa\@lxxiv%
    \p@rtentiere{\p@rtent}{\v@lXa}\advance\p@rtent\@ne\else\p@rtent=\@ne\fi%
    \global\result@tent=\p@rtent}#1=\result@tent\ignorespaces}
\ctr@ln@m\getredf@ctB
\ctr@ld@f\def\getredf@ctBDD#1{\getredf@ctDD#1(\v@lX,\v@lY)}
\ctr@ld@f\def\getredf@ctBTD#1{\getredf@ctTD#1(\v@lX,\v@lY,\v@lZ)}
\ctr@ld@f\def\FigptintercircB@zDD#1:#2:#3,#4[#5,#6,#7,#8]{{\s@uvc@ntr@l\et@tfigptintercircB@zDD%
    \setc@ntr@l{2}\figvectPDD-1[#5,#8]\Figg@tXY{-1}\getredf@ctDD\f@ctech(\v@lX,\v@lY)%
    \mili@u=#4\unit@\divide\mili@u\f@ctech\c@rre=\repdecn@mb{\mili@u}\mili@u%
    \figptBezierDD-5::#3[#5,#6,#7,#8]%
    \v@lmin=#3\p@\v@lmax=\v@lmin\advance\v@lmax0.1\p@%
    \loop\edef\T@{\repdecn@mb{\v@lmax}}\figptBezierDD-2::\T@[#5,#6,#7,#8]%
    \figvectPDD-1[-5,-2]\n@rmeucCDD{\delt@}{-1}\ifdim\delt@<\c@rre\v@lmin=\v@lmax%
    \advance\v@lmax0.1\p@\repeat%
    \loop\mili@u=\v@lmin\advance\mili@u\v@lmax%
    \divide\mili@u\tw@\edef\T@{\repdecn@mb{\mili@u}}\figptBezierDD-2::\T@[#5,#6,#7,#8]%
    \figvectPDD-1[-5,-2]\n@rmeucCDD{\delt@}{-1}\ifdim\delt@>\c@rre\v@lmax=\mili@u%
    \else\v@lmin=\mili@u\fi\v@leur=\v@lmax\advance\v@leur-\v@lmin%
    \ifdim\v@leur>\epsil@n\repeat\figptcopyDD#1:#2/-2/%
    \resetc@ntr@l\et@tfigptintercircB@zDD}\ignorespaces}
\ctr@ln@m\figptinterlines
\ctr@ld@f\def\inters@cDD#1:#2[#3,#4;#5,#6]{{\s@uvc@ntr@l\et@tinters@cDD%
    \setc@ntr@l{2}\vecunit@{-1}{#4}\vecunit@{-2}{#6}%
    \Figg@tXY{-1}\setc@ntr@l{1}\Figg@tXYa{#3}%
    \edef\A@{\repdecn@mb{\v@lX}}\edef\B@{\repdecn@mb{\v@lY}}%
    \v@lmin=\B@\v@lXa\advance\v@lmin-\A@\v@lYa%
    \Figg@tXYa{#5}\setc@ntr@l{2}\Figg@tXY{-2}%
    \edef\C@{\repdecn@mb{\v@lX}}\edef\D@{\repdecn@mb{\v@lY}}%
    \v@lmax=\D@\v@lXa\advance\v@lmax-\C@\v@lYa%
    \delt@=\A@\v@lY\advance\delt@-\B@\v@lX%
    \invers@{\v@leur}{\delt@}\edef\v@ldelta{\repdecn@mb{\v@leur}}%
    \v@lXa=\A@\v@lmax\advance\v@lXa-\C@\v@lmin%
    \v@lYa=\B@\v@lmax\advance\v@lYa-\D@\v@lmin%
    \v@lXa=\v@ldelta\v@lXa\v@lYa=\v@ldelta\v@lYa%
    \setc@ntr@l{1}\Figp@intregDD#1:{#2}(\v@lXa,\v@lYa)%
    \resetc@ntr@l\et@tinters@cDD}\ignorespaces}
\ctr@ld@f\def\inters@cTD#1:#2[#3,#4;#5,#6]{{\s@uvc@ntr@l\et@tinters@cTD%
    \setc@ntr@l{2}\figvectNVTD-1[#4,#6]\figvectNVTD-2[#6,-1]\figvectPTD-1[#3,#5]%
    \r@pPSTD\v@leur[-2,-1,#4]\edef\v@lcoef{\repdecn@mb{\v@leur}}%
    \figpttraTD#1:{#2}=#3/\v@lcoef,#4/\resetc@ntr@l\et@tinters@cTD}\ignorespaces}
\ctr@ld@f\def\r@pPSTD#1[#2,#3,#4]{{\Figg@tXY{#2}\edef\Xu@{\repdecn@mb{\v@lX}}%
    \edef\Yu@{\repdecn@mb{\v@lY}}\edef\Zu@{\repdecn@mb{\v@lZ}}%
    \Figg@tXY{#3}\v@lmin=\Xu@\v@lX\advance\v@lmin\Yu@\v@lY\advance\v@lmin\Zu@\v@lZ%
    \Figg@tXY{#4}\v@lmax=\Xu@\v@lX\advance\v@lmax\Yu@\v@lY\advance\v@lmax\Zu@\v@lZ%
    \invers@{\v@leur}{\v@lmax}\global\result@t=\repdecn@mb{\v@leur}\v@lmin}%
    #1=\result@t}
\ctr@ln@m\n@rminf
\ctr@ld@f\def\n@rminfDD#1#2{{\Figg@tXY{#2}\maxim@m{\v@lX}{\v@lX}{-\v@lX}%
    \maxim@m{\v@lY}{\v@lY}{-\v@lY}\maxim@m{\global\result@t}{\v@lX}{\v@lY}}%
    #1=\result@t}
\ctr@ld@f\def\n@rminfTD#1#2{{\Figg@tXY{#2}\maxim@m{\v@lX}{\v@lX}{-\v@lX}%
    \maxim@m{\v@lY}{\v@lY}{-\v@lY}\maxim@m{\v@lZ}{\v@lZ}{-\v@lZ}%
    \maxim@m{\v@lX}{\v@lX}{\v@lY}\maxim@m{\global\result@t}{\v@lX}{\v@lZ}}%
    #1=\result@t}
\ctr@ln@m\n@rmeucC
\ctr@ld@f\def\n@rmeucCDD#1#2{\Figg@tXY{#2}\divide\v@lX\f@ctech\divide\v@lY\f@ctech%
    #1=\repdecn@mb{\v@lX}\v@lX\v@lX=\repdecn@mb{\v@lY}\v@lY\advance#1\v@lX}
\ctr@ld@f\def\n@rmeucCTD#1#2{\Figg@tXY{#2}%
    \divide\v@lX\f@ctech\divide\v@lY\f@ctech\divide\v@lZ\f@ctech%
    #1=\repdecn@mb{\v@lX}\v@lX\v@lX=\repdecn@mb{\v@lY}\v@lY\advance#1\v@lX%
    \v@lX=\repdecn@mb{\v@lZ}\v@lZ\advance#1\v@lX}
\ctr@ln@m\n@rmeucSV
\ctr@ld@f\def\n@rmeucSVDD#1#2{{\Figg@tXY{#2}%
    \v@lXa=\repdecn@mb{\v@lX}\v@lX\v@lYa=\repdecn@mb{\v@lY}\v@lY%
    \advance\v@lXa\v@lYa\sqrt@{\global\result@t}{\v@lXa}}#1=\result@t}
\ctr@ld@f\def\n@rmeucSVTD#1#2{{\Figg@tXY{#2}\v@lXa=\repdecn@mb{\v@lX}\v@lX%
    \v@lYa=\repdecn@mb{\v@lY}\v@lY\v@lZa=\repdecn@mb{\v@lZ}\v@lZ%
    \advance\v@lXa\v@lYa\advance\v@lXa\v@lZa\sqrt@{\global\result@t}{\v@lXa}}#1=\result@t}
\ctr@ln@m\n@rmeuc
\ctr@ld@f\def\n@rmeucDD#1#2{{\Figg@tXY{#2}\getredf@ctDD\f@ctech(\v@lX,\v@lY)%
    \divide\v@lX\f@ctech\divide\v@lY\f@ctech%
    \v@lXa=\repdecn@mb{\v@lX}\v@lX\v@lYa=\repdecn@mb{\v@lY}\v@lY%
    \advance\v@lXa\v@lYa\sqrt@{\global\result@t}{\v@lXa}%
    \global\multiply\result@t\f@ctech}#1=\result@t}
\ctr@ld@f\def\n@rmeucTD#1#2{{\Figg@tXY{#2}\getredf@ctTD\f@ctech(\v@lX,\v@lY,\v@lZ)%
    \divide\v@lX\f@ctech\divide\v@lY\f@ctech\divide\v@lZ\f@ctech%
    \v@lXa=\repdecn@mb{\v@lX}\v@lX%
    \v@lYa=\repdecn@mb{\v@lY}\v@lY\v@lZa=\repdecn@mb{\v@lZ}\v@lZ%
    \advance\v@lXa\v@lYa\advance\v@lXa\v@lZa\sqrt@{\global\result@t}{\v@lXa}%
    \global\multiply\result@t\f@ctech}#1=\result@t}
\ctr@ln@m\vecunit@
\ctr@ld@f\def\vecunit@DD#1#2{{\Figg@tXY{#2}\getredf@ctDD\f@ctech(\v@lX,\v@lY)%
    \divide\v@lX\f@ctech\divide\v@lY\f@ctech%
    \Figv@ctCreg#1(\v@lX,\v@lY)\n@rmeucSV{\v@lYa}{#1}%
    \invers@{\v@lXa}{\v@lYa}\edef\v@lv@lXa{\repdecn@mb{\v@lXa}}%
    \v@lX=\v@lv@lXa\v@lX\v@lY=\v@lv@lXa\v@lY%
    \Figv@ctCreg#1(\v@lX,\v@lY)\multiply\v@lYa\f@ctech\global\result@t=\v@lYa}}
\ctr@ld@f\def\vecunit@TD#1#2{{\Figg@tXY{#2}\getredf@ctTD\f@ctech(\v@lX,\v@lY,\v@lZ)%
    \divide\v@lX\f@ctech\divide\v@lY\f@ctech\divide\v@lZ\f@ctech%
    \Figv@ctCreg#1(\v@lX,\v@lY,\v@lZ)\n@rmeucSV{\v@lYa}{#1}%
    \invers@{\v@lXa}{\v@lYa}\edef\v@lv@lXa{\repdecn@mb{\v@lXa}}%
    \v@lX=\v@lv@lXa\v@lX\v@lY=\v@lv@lXa\v@lY\v@lZ=\v@lv@lXa\v@lZ%
    \Figv@ctCreg#1(\v@lX,\v@lY,\v@lZ)\multiply\v@lYa\f@ctech\global\result@t=\v@lYa}}
\ctr@ld@f\def\vecunitC@TD[#1,#2]{\Figg@tXYa{#1}\Figg@tXY{#2}%
    \advance\v@lX-\v@lXa\advance\v@lY-\v@lYa\advance\v@lZ-\v@lZa\c@lvecunitTD}
\ctr@ld@f\def\vecunitCV@TD#1{\Figg@tXY{#1}\c@lvecunitTD}
\ctr@ld@f\def\c@lvecunitTD{\getredf@ctTD\f@ctech(\v@lX,\v@lY,\v@lZ)%
    \divide\v@lX\f@ctech\divide\v@lY\f@ctech\divide\v@lZ\f@ctech%
    \v@lXa=\repdecn@mb{\v@lX}\v@lX%
    \v@lYa=\repdecn@mb{\v@lY}\v@lY\v@lZa=\repdecn@mb{\v@lZ}\v@lZ%
    \advance\v@lXa\v@lYa\advance\v@lXa\v@lZa\sqrt@{\v@lYa}{\v@lXa}%
    \invers@{\v@lXa}{\v@lYa}\edef\v@lv@lXa{\repdecn@mb{\v@lXa}}%
    \v@lX=\v@lv@lXa\v@lX\v@lY=\v@lv@lXa\v@lY\v@lZ=\v@lv@lXa\v@lZ}
\ctr@ln@m\figgetangle
\ctr@ld@f\def\figgetangleDD#1[#2,#3,#4]{\ifGR@cri{\s@uvc@ntr@l\et@tfiggetangleDD\setc@ntr@l{2}%
    \figvectPDD-1[#2,#3]\figvectPDD-2[#2,#4]\vecunit@{-1}{-1}%
    \c@lproscalDD\delt@[-2,-1]\figvectNVDD-1[-1]\c@lproscalDD\v@leur[-2,-1]%
    \arct@n\v@lmax(\delt@,\v@leur)\v@lmax=\rdT@deg\v@lmax%
    \ifdim\v@lmax<\z@\advance\v@lmax\DePI@deg\fi\xdef#1{\repdecn@mb{\v@lmax}}%
    \resetc@ntr@l\et@tfiggetangleDD}\ignorespaces\fi}
\ctr@ld@f\def\figgetangleTD#1[#2,#3,#4,#5]{\ifGR@cri{\s@uvc@ntr@l\et@tfiggetangleTD\setc@ntr@l{2}%
    \figvectPTD-1[#2,#3]\figvectPTD-2[#2,#5]\figvectNVTD-3[-1,-2]%
    \figvectPTD-2[#2,#4]\figvectNVTD-4[-3,-1]%
    \vecunit@{-1}{-1}\c@lproscalTD\delt@[-2,-1]\c@lproscalTD\v@leur[-2,-4]%
    \arct@n\v@lmax(\delt@,\v@leur)\v@lmax=\rdT@deg\v@lmax%
    \ifdim\v@lmax<\z@\advance\v@lmax\DePI@deg\fi\xdef#1{\repdecn@mb{\v@lmax}}%
    \resetc@ntr@l\et@tfiggetangleTD}\ignorespaces\fi}    
\ctr@ld@f\def\figgetdist#1[#2,#3]{\ifGR@cri{\s@uvc@ntr@l\et@tfiggetdist\setc@ntr@l{2}%
    \figvectP-1[#2,#3]\n@rmeuc{\v@lX}{-1}\v@lX=\ptT@unit@\v@lX\xdef#1{\repdecn@mb{\v@lX}}%
    \resetc@ntr@l\et@tfiggetdist}\ignorespaces\fi}
\ctr@ld@f\def\figget#1=#2[#3]{\keln@mun#1|%
    \def\n@mref{a}\ifx\l@debut\n@mref\figgetangle#2[#3]\else
    \def\n@mref{d}\ifx\l@debut\n@mref\figgetdist#2[#3]\else
    \W@rnmeskwd{figget}{#1}\fi\fi\ignorespaces}
\ctr@ld@f\def\Figg@tT#1{\c@ntr@lnum{#1}%
    {\expandafter\expandafter\expandafter\extr@ctT\csname\objc@de\endcsname:%
     \ifnum\B@@ltxt=\z@\ptn@me{#1}\else\csname\objc@de T\endcsname\fi}}
\ctr@ld@f\def\extr@ctT#1,#2,#3/#4:{\def\B@@ltxt{#3}}
\ctr@ld@f\def\Figg@tXY#1{\c@ntr@lnum{#1}%
    \expandafter\expandafter\expandafter\extr@ctC\csname\objc@de\endcsname:}
\ctr@ln@m\extr@ctC
\ctr@ld@f\def\extr@ctCDD#1/#2,#3,#4:{\v@lX=#2\v@lY=#3}
\ctr@ld@f\def\extr@ctCTD#1/#2,#3,#4:{\v@lX=#2\v@lY=#3\v@lZ=#4}
\ctr@ld@f\def\Figg@tXYa#1{\c@ntr@lnum{#1}%
    \expandafter\expandafter\expandafter\extr@ctCa\csname\objc@de\endcsname:}
\ctr@ln@m\extr@ctCa
\ctr@ld@f\def\extr@ctCaDD#1/#2,#3,#4:{\v@lXa=#2\v@lYa=#3}
\ctr@ld@f\def\extr@ctCaTD#1/#2,#3,#4:{\v@lXa=#2\v@lYa=#3\v@lZa=#4}
\ctr@ln@m\t@xt@
\ctr@ld@f\def\figinit#1{\t@stc@tcodech@nge\initpr@lim\Figinit@#1,:\initpss@ttings\ignorespaces}
\ctr@ld@f\def\Figinit@#1,#2:{\setunit@{#1}\def\t@xt@{#2}\ifx\t@xt@\empty\else\Figinit@@#2:\fi}
\ctr@ld@f\def\Figinit@@#1#2:{\if#12 \else\Figs@tproj{#1}\initTD@\fi}
\ctr@ln@w{newif}\ifTr@isDim
\ctr@ld@f\def\UnD@fined{UNDEFINED}
\ctr@ln@m\@utoFN
\ctr@ln@m\@utoFInDone
\ctr@ln@m\disob@unit
\ctr@ld@f\def\initpr@lim{\initb@undb@x\figsetmark{}\figsetptname{$A_{##1}$}\def\Sc@leFact{1}%
    \initDD@\figsetroundcoord{yes}\GR@critrue\expandafter\setupd@te\D@FTupdate:%
    \edef\disob@unit{\UnD@fined}\edef\t@rgetpt{\UnD@fined}\gdef\@utoFInDone{1}\gdef\@utoFN{0}}
\ctr@ld@f\def\initDD@{\Tr@isDimfalse%
    \ifPDFm@ke%
     \let\Ps@rcerc=\Ps@rcercBz%
     \let\Ps@rell=\Ps@rellBz%
    \fi
    \let\c@lDCUn=\c@lDCUnDD%
    \let\c@lDCDeux=\c@lDCDeuxDD%
    \let\c@ldefproj=\relax%
    \let\c@lproscal=\c@lproscalDD%
    \let\c@lprojSP=\relax%
    \let\extr@ctC=\extr@ctCDD%
    \let\extr@ctCa=\extr@ctCaDD%
    \let\extr@ctCF=\extr@ctCFDD%
    \let\Figp@intreg=\Figp@intregDD%
    \let\Figpts@xes=\Figpts@xesDD%
    \let\getredf@ctB=\getredf@ctBDD%
    \let\n@rmeucSV=\n@rmeucSVDD\let\n@rmeuc=\n@rmeucDD\let\n@rmeucC\n@rmeucCDD\let\n@rminf=\n@rminfDD%
    \let\pr@dMatV=\pr@dMatVDD%
    \let\Q@@xes=\Q@@xesDD%
    \let\vecunit@=\vecunit@DD%
    \let\figcoord=\figcoordDD%
    \let\figgetangle=\figgetangleDD%
    \let\figpt=\figptDD%
    \let\figptBezier=\figptBezierDD%
    \let\figptbary=\figptbaryDD%
    \let\figptcirc=\figptcircDD%
    \let\figptcircumcenter=\figptcircumcenterDD%
    \let\figptcopy=\figptcopyDD%
    \let\figptcurvcenter=\figptcurvcenterDD%
    \let\figptell=\figptellDD%
    \let\figptendnormal=\figptendnormalDD%
    \let\figptinterlineplane=\figptinterlineplaneDD%
    \let\figptinterlines=\inters@cDD%
    \let\figptorthocenter=\figptorthocenterDD%
    \let\figptorthoprojline=\figptorthoprojlineDD%
    \let\figptorthoprojplane=\figptorthoprojplaneDD%
    \let\figptrot=\figptrotDD%
    \let\figptscontrol=\figptscontrolDD%
    \let\figptsintercirc=\figptsintercircDD%
    \let\figptsinterlinell=\figptsinterlinellDD%
    \let\figptsorthoprojline=\figptsorthoprojlineDD%
    \let\figptorthoprojplane=\figptorthoprojplaneDD%
    \let\figptsrot=\figptsrotDD%
    \let\figptssym=\figptssymDD%
    \let\figptstra=\figptstraDD%
    \let\figptsym=\figptsymDD%
    \let\figpttraC=\figpttraCDD%
    \let\figpttra=\figpttraDD%
    \let\figptvisilimSL=\figptvisilimSLDD%
    \let\figsetobdist=\figsetobdistDD%
    \let\figsettarget=\figsettargetDD%
    \let\figsetview=\figsetviewDD%
    \let\figvectDBezier=\figvectDBezierDD%
    \let\figvectN=\figvectNDD%
    \let\figvectNV=\figvectNVDD%
    \let\figvectP=\figvectPDD%
    \let\figvectU=\figvectUDD%
    \let\figdrawarccircP=\Q@arccircPDD%
    \let\figdrawarccirc=\Q@arccircDD%
    \let\figdrawarcell=\Q@arcellDD%
    \let\figdrawarcellPA=\Q@arcellPADD%
    \let\figdrawarrowBezier=\Q@arrowBezierDD%
    \let\figdrawarrowcircP=\Q@arrowcircPDD%
    \let\figdrawarrowcirc=\Q@arrowcircDD%
    \let\figdrawarrowhead=\Q@arrowheadDD%
    \let\figdrawarrow=\Q@arrowDD%
    \let\figdrawBezier=\Q@BezierDD%
    \let\figdrawcirc=\Q@circDD%
    \let\figdrawcurve=\Q@curveDD%
    \let\figdrawnormal=\Q@normalDD%
    }
\ctr@ld@f\def\initTD@{\Tr@isDimtrue\initb@undb@xTD\newt@rgetptfalse\newdis@bfalse%
    \let\c@lDCUn=\c@lDCUnTD%
    \let\c@lDCDeux=\c@lDCDeuxTD%
    \let\c@ldefproj=\c@ldefprojTD%
    \let\c@lproscal=\c@lproscalTD%
    \let\extr@ctC=\extr@ctCTD%
    \let\extr@ctCa=\extr@ctCaTD%
    \let\extr@ctCF=\extr@ctCFTD%
    \let\Figp@intreg=\Figp@intregTD%
    \let\Figpts@xes=\Figpts@xesTD%
    \let\getredf@ctB=\getredf@ctBTD%
    \let\n@rmeucSV=\n@rmeucSVTD\let\n@rmeuc=\n@rmeucTD\let\n@rmeucC\n@rmeucCTD\let\n@rminf=\n@rminfTD%
    \let\pr@dMatV=\pr@dMatVTD%
    \let\Q@@xes=\Q@@xesTD%
    \let\vecunit@=\vecunit@TD%
    \let\figcoord=\figcoordTD%
    \let\figgetangle=\figgetangleTD%
    \let\figpt=\figptTD%
    \let\figptBezier=\figptBezierTD%
    \let\figptbary=\figptbaryTD%
    \let\figptcirc=\figptcircTD%
    \let\figptcircumcenter=\figptcircumcenterTD%
    \let\figptcopy=\figptcopyTD%
    \let\figptcurvcenter=\figptcurvcenterTD%
    \let\figptinterlineplane=\figptinterlineplaneTD%
    \let\figptinterlines=\inters@cTD%
    \let\figptorthocenter=\figptorthocenterTD%
    \let\figptorthoprojline=\figptorthoprojlineTD%
    \let\figptorthoprojplane=\figptorthoprojplaneTD%
    \let\figptrot=\figptrotTD%
    \let\figptscontrol=\figptscontrolTD%
    \let\figptsintercirc=\figptsintercircTD%
    \let\figptsorthoprojline=\figptsorthoprojlineTD%
    \let\figptsorthoprojplane=\figptsorthoprojplaneTD%
    \let\figptsrot=\figptsrotTD%
    \let\figptssym=\figptssymTD%
    \let\figptstra=\figptstraTD%
    \let\figptsym=\figptsymTD%
    \let\figpttraC=\figpttraCTD%
    \let\figpttra=\figpttraTD%
    \let\figptvisilimSL=\figptvisilimSLTD%
    \let\figsetobdist=\figsetobdistTD%
    \let\figsettarget=\figsettargetTD%
    \let\figsetview=\figsetviewTD%
    \let\figvectDBezier=\figvectDBezierTD%
    \let\figvectN=\figvectNTD%
    \let\figvectNV=\figvectNVTD%
    \let\figvectP=\figvectPTD%
    \let\figvectU=\figvectUTD%
    \let\figdrawarccircP=\Q@arccircPTD%
    \let\figdrawarccirc=\Q@arccircTD%
    \let\figdrawarcell=\Q@arcellTD%
    \let\figdrawarcellPA=\Q@arcellPATD%
    \let\figdrawarrowBezier=\Q@arrowBezierTD%
    \let\figdrawarrowcircP=\Q@arrowcircPTD%
    \let\figdrawarrowcirc=\Q@arrowcircTD%
    \let\figdrawarrowhead=\Q@arrowheadTD%
    \let\figdrawarrow=\Q@arrowTD%
    \let\figdrawBezier=\Q@BezierTD%
    \let\figdrawcirc=\Q@circTD%
    \let\figdrawcurve=\Q@curveTD%
    }
\ctr@ld@f\def\un@v@ilable#1{\immediate\write16{*** The macro #1 is not available in the current context.}}
\ctr@ld@f\def\figinsert#1{{\def\t@xt@{#1}\relax%
    \ifx\t@xt@\empty\ifnum\@utoFInDone>\z@\Figinsert@\DefGIfilen@me,:\fi%
    \else\expandafter\FiginsertNu@#1 :\fi}\ignorespaces}
\ctr@ld@f\def\FiginsertNu@#1 #2:{\def\t@xt@{#1}\relax\ifx\t@xt@\empty\def\t@xt@{#2}%
    \ifx\t@xt@\empty\ifnum\@utoFInDone>\z@\Figinsert@\DefGIfilen@me,:\fi%
    \else\FiginsertNu@#2:\fi\else\expandafter\FiginsertNd@#1 #2:\fi}
\ctr@ld@f\def\FiginsertNd@#1#2:{\ifcat#1a\Figinsert@#1#2,:\else%
    \ifnum\@utoFInDone>\z@\Figinsert@\DefGIfilen@me,#1#2,:\fi\fi}
\ctr@ln@m\Sc@leFact
\ctr@ld@f\def\Figinsert@#1,#2:{\def\t@xt@{#2}\ifx\t@xt@\empty\xdef\Sc@leFact{1}\else%
    \X@rgdeux@#2\xdef\Sc@leFact{\@rgdeux}\fi%
    \Figdisc@rdLTS{#1}{\t@xt@}\@psfgetbb{\t@xt@}%
    \v@lX=\@psfllx\p@\v@lX=\ptpsT@pt\v@lX\v@lX=\Sc@leFact\v@lX%
    \v@lY=\@psflly\p@\v@lY=\ptpsT@pt\v@lY\v@lY=\Sc@leFact\v@lY%
    \b@undb@x{\v@lX}{\v@lY}%
    \v@lX=\@psfurx\p@\v@lX=\ptpsT@pt\v@lX\v@lX=\Sc@leFact\v@lX%
    \v@lY=\@psfury\p@\v@lY=\ptpsT@pt\v@lY\v@lY=\Sc@leFact\v@lY%
    \b@undb@x{\v@lX}{\v@lY}%
    \ifPDFm@ke\Figinclud@PDF{\t@xt@}{\Sc@leFact}\else%
    \v@lX=\c@nt pt\v@lX=\Sc@leFact\v@lX\edef\F@ct{\repdecn@mb{\v@lX}}%
    \ifx\TeXturesonMacOSltX\special{postscriptfile #1 vscale=\F@ct\space hscale=\F@ct}%
    \else\includegraphics{#1}\fi\fi%
    \message{[\t@xt@]}\ignorespaces}
\ctr@ld@f\def\Figdisc@rdLTS#1#2{\expandafter\Figdisc@rdLTS@#1 :#2}
\ctr@ld@f\def\Figdisc@rdLTS@#1 #2:#3{\def#3{#1}\relax\ifx#3\empty\expandafter\Figdisc@rdLTS@#2:#3\fi}
\ctr@ld@f\def\figinsertE#1{\FiginsertE@#1,:\ignorespaces}
\ctr@ld@f\def\FiginsertE@#1,#2:{{\def\t@xt@{#2}\ifx\t@xt@\empty\xdef\Sc@leFact{1}\else%
    \X@rgdeux@#2\xdef\Sc@leFact{\@rgdeux}\fi%
    \Figdisc@rdLTS{#1}{\t@xt@}\pdfximage{\t@xt@}%
    \setbox\Gb@x=\hbox{\pdfrefximage\pdflastximage}%
    \v@lX=\z@\v@lY=-\Sc@leFact\dp\Gb@x\b@undb@x{\v@lX}{\v@lY}%
    \advance\v@lX\Sc@leFact\wd\Gb@x\advance\v@lY\Sc@leFact\dp\Gb@x%
    \advance\v@lY\Sc@leFact\ht\Gb@x\b@undb@x{\v@lX}{\v@lY}%
    \v@lX=\Sc@leFact\wd\Gb@x\pdfximage width \v@lX {\t@xt@}%
    \rlap{\pdfrefximage\pdflastximage}\message{[\t@xt@]}}\ignorespaces}
\ctr@ld@f\def\X@rgdeux@#1,{\edef\@rgdeux{#1}}
\ctr@ln@m\figpt
\ctr@ld@f\def\figptDD#1:#2(#3,#4){\ifGR@cri\c@ntr@lnum{#1}%
    {\v@lX=#3\unit@\v@lY=#4\unit@\Fig@dmpt{#2}{\z@}}\ignorespaces\fi}
\ctr@ld@f\def\Fig@dmpt#1#2{\def\t@xt@{#1}\ifx\t@xt@\empty\def\B@@ltxt{\z@}%
    \else\expandafter\gdef\csname\objc@de T\endcsname{#1}\def\B@@ltxt{\@ne}\fi%
    \expandafter\xdef\csname\objc@de\endcsname{\ifitis@vect@r\C@dCl@svect%
    \else\C@dCl@spt\fi,\z@,\B@@ltxt/\the\v@lX,\the\v@lY,#2}}
\ctr@ld@f\def\C@dCl@spt{P}
\ctr@ld@f\def\C@dCl@svect{V}
\ctr@ln@m\c@@rdYZ
\ctr@ln@m\c@@rdY
\ctr@ld@f\def\figptTD#1:#2(#3,#4){\ifGR@cri\c@ntr@lnum{#1}%
    \def\c@@rdYZ{#4,0,0}\extrairelepremi@r\c@@rdY\de\c@@rdYZ%
    \extrairelepremi@r\c@@rdZ\de\c@@rdYZ%
    {\v@lX=#3\unit@\v@lY=\c@@rdY\unit@\v@lZ=\c@@rdZ\unit@\Fig@dmpt{#2}{\the\v@lZ}%
    \b@undb@xTD{\v@lX}{\v@lY}{\v@lZ}}\ignorespaces\fi}
\ctr@ln@m\Figp@intreg
\ctr@ld@f\def\Figp@intregDD#1:#2(#3,#4){\c@ntr@lnum{#1}%
    {\result@t=#4\v@lX=#3\v@lY=\result@t\Fig@dmpt{#2}{\z@}}\ignorespaces}
\ctr@ld@f\def\Figp@intregTD#1:#2(#3,#4){\c@ntr@lnum{#1}%
    \def\c@@rdYZ{#4,\z@,\z@}\extrairelepremi@r\c@@rdY\de\c@@rdYZ%
    \extrairelepremi@r\c@@rdZ\de\c@@rdYZ%
    {\v@lX=#3\v@lY=\c@@rdY\v@lZ=\c@@rdZ\Fig@dmpt{#2}{\the\v@lZ}%
    \b@undb@xTD{\v@lX}{\v@lY}{\v@lZ}}\ignorespaces}
\ctr@ln@m\figptBezier
\ctr@ld@f\def\figptBezierDD#1:#2:#3[#4,#5,#6,#7]{\ifGR@cri{\s@uvc@ntr@l\et@tfigptBezierDD%
    \FigptBezier@#3[#4,#5,#6,#7]\Figp@intregDD#1:{#2}(\v@lX,\v@lY)%
    \resetc@ntr@l\et@tfigptBezierDD}\ignorespaces\fi}
\ctr@ld@f\def\figptBezierTD#1:#2:#3[#4,#5,#6,#7]{\ifGR@cri{\s@uvc@ntr@l\et@tfigptBezierTD%
    \FigptBezier@#3[#4,#5,#6,#7]\Figp@intregTD#1:{#2}(\v@lX,\v@lY,\v@lZ)%
    \resetc@ntr@l\et@tfigptBezierTD}\ignorespaces\fi}
\ctr@ld@f\def\FigptBezier@#1[#2,#3,#4,#5]{\setc@ntr@l{2}%
    \edef\T@{#1}\v@leur=\p@\advance\v@leur-#1pt\edef\UNmT@{\repdecn@mb{\v@leur}}%
    \figptcopy-4:/#2/\figptcopy-3:/#3/\figptcopy-2:/#4/\figptcopy-1:/#5/%
    \l@mbd@un=-4 \l@mbd@de=-\thr@@\p@rtent=\m@ne\c@lDecast%
    \l@mbd@un=-4 \l@mbd@de=-\thr@@\p@rtent=-\tw@\c@lDecast%
    \l@mbd@un=-4 \l@mbd@de=-\thr@@\p@rtent=-\thr@@\c@lDecast\Figg@tXY{-4}}
\ctr@ln@m\c@lDCUn
\ctr@ld@f\def\c@lDCUnDD#1#2{\Figg@tXY{#1}\v@lX=\UNmT@\v@lX\v@lY=\UNmT@\v@lY%
    \Figg@tXYa{#2}\advance\v@lX\T@\v@lXa\advance\v@lY\T@\v@lYa%
    \Figp@intregDD#1:(\v@lX,\v@lY)}
\ctr@ld@f\def\c@lDCUnTD#1#2{\Figg@tXY{#1}\v@lX=\UNmT@\v@lX\v@lY=\UNmT@\v@lY\v@lZ=\UNmT@\v@lZ%
    \Figg@tXYa{#2}\advance\v@lX\T@\v@lXa\advance\v@lY\T@\v@lYa\advance\v@lZ\T@\v@lZa%
    \Figp@intregTD#1:(\v@lX,\v@lY,\v@lZ)}
\ctr@ld@f\def\c@lDecast{\relax\ifnum\l@mbd@un<\p@rtent\c@lDCUn{\l@mbd@un}{\l@mbd@de}%
    \advance\l@mbd@un\@ne\advance\l@mbd@de\@ne\c@lDecast\fi}
\ctr@ld@f\def\figptmap#1:#2=#3/#4/#5/{\ifGR@cri{\s@uvc@ntr@l\et@tfigptmap%
    \setc@ntr@l{2}\figvectP-1[#4,#3]\Figg@tXY{-1}%
    \pr@dMatV/#5/\figpttra#1:{#2}=#4/1,-1/%
    \resetc@ntr@l\et@tfigptmap}\ignorespaces\fi}
\ctr@ln@m\pr@dMatV
\ctr@ld@f\def\pr@dMatVDD/#1,#2;#3,#4/{\v@lXa=#1\v@lX\advance\v@lXa#2\v@lY%
    \v@lYa=#3\v@lX\advance\v@lYa#4\v@lY\Figv@ctCreg-1(\v@lXa,\v@lYa)}
\ctr@ld@f\def\pr@dMatVTD/#1,#2,#3;#4,#5,#6;#7,#8,#9/{%
    \v@lXa=#1\v@lX\advance\v@lXa#2\v@lY\advance\v@lXa#3\v@lZ%
    \v@lYa=#4\v@lX\advance\v@lYa#5\v@lY\advance\v@lYa#6\v@lZ%
    \v@lZa=#7\v@lX\advance\v@lZa#8\v@lY\advance\v@lZa#9\v@lZ%
    \Figv@ctCreg-1(\v@lXa,\v@lYa,\v@lZa)}
\ctr@ln@m\figptbary
\ctr@ld@f\def\figptbaryDD#1:#2[#3;#4]{\ifGR@cri{\edef\list@num{#3}\extrairelepremi@r\p@int\de\list@num%
    \s@mme=\z@\@ecfor\c@ef:=#4\do{\advance\s@mme\c@ef}%
    \edef\listec@ef{#4,0}\extrairelepremi@r\c@ef\de\listec@ef%
    \Figg@tXY{\p@int}\divide\v@lX\s@mme\divide\v@lY\s@mme%
    \multiply\v@lX\c@ef\multiply\v@lY\c@ef%
    \@ecfor\p@int:=\list@num\do{\extrairelepremi@r\c@ef\de\listec@ef%
           \Figg@tXYa{\p@int}\divide\v@lXa\s@mme\divide\v@lYa\s@mme%
           \multiply\v@lXa\c@ef\multiply\v@lYa\c@ef%
           \advance\v@lX\v@lXa\advance\v@lY\v@lYa}%
    \Figp@intregDD#1:{#2}(\v@lX,\v@lY)}\ignorespaces\fi}
\ctr@ld@f\def\figptbaryTD#1:#2[#3;#4]{\ifGR@cri{\edef\list@num{#3}\extrairelepremi@r\p@int\de\list@num%
    \s@mme=\z@\@ecfor\c@ef:=#4\do{\advance\s@mme\c@ef}%
    \edef\listec@ef{#4,0}\extrairelepremi@r\c@ef\de\listec@ef%
    \Figg@tXY{\p@int}\divide\v@lX\s@mme\divide\v@lY\s@mme\divide\v@lZ\s@mme%
    \multiply\v@lX\c@ef\multiply\v@lY\c@ef\multiply\v@lZ\c@ef%
    \@ecfor\p@int:=\list@num\do{\extrairelepremi@r\c@ef\de\listec@ef%
           \Figg@tXYa{\p@int}\divide\v@lXa\s@mme\divide\v@lYa\s@mme\divide\v@lZa\s@mme%
           \multiply\v@lXa\c@ef\multiply\v@lYa\c@ef\multiply\v@lZa\c@ef%
           \advance\v@lX\v@lXa\advance\v@lY\v@lYa\advance\v@lZ\v@lZa}%
    \Figp@intregTD#1:{#2}(\v@lX,\v@lY,\v@lZ)}\ignorespaces\fi}
\ctr@ld@f\def\figptbaryR#1:#2[#3;#4]{\ifGR@cri{%
    \v@leur=\z@\@ecfor\c@ef:=#4\do{\maxim@m{\v@lmax}{\c@ef pt}{-\c@ef pt}%
    \ifdim\v@lmax>\v@leur\v@leur=\v@lmax\fi}%
    \ifdim\v@leur<\p@\f@ctech=\@M\else\ifdim\v@leur<\t@n\p@\f@ctech=\@m\else%
    \ifdim\v@leur<\c@nt\p@\f@ctech=\c@nt\else\ifdim\v@leur<\@m\p@\f@ctech=\t@n\else%
    \f@ctech=\@ne\fi\fi\fi\fi%
    \def\listec@ef{0}%
    \@ecfor\c@ef:=#4\do{\sc@lec@nvRI{\c@ef pt}\edef\listec@ef{\listec@ef,\the\s@mme}}%
    \extrairelepremi@r\c@ef\de\listec@ef\figptbary#1:#2[#3;\listec@ef]}\ignorespaces\fi}
\ctr@ld@f\def\sc@lec@nvRI#1{\v@leur=#1\p@rtentiere{\s@mme}{\v@leur}\advance\v@leur-\s@mme\p@%
    \multiply\v@leur\f@ctech\p@rtentiere{\p@rtent}{\v@leur}%
    \multiply\s@mme\f@ctech\advance\s@mme\p@rtent}
\ctr@ln@m\figptcirc
\ctr@ld@f\def\figptcircDD#1:#2:#3;#4(#5){\ifGR@cri{\s@uvc@ntr@l\et@tfigptcircDD%
    \c@lptellDD#1:{#2}:#3;#4,#4(#5)\resetc@ntr@l\et@tfigptcircDD}\ignorespaces\fi}
\ctr@ld@f\def\figptcircTD#1:#2:#3,#4,#5;#6(#7){\ifGR@cri{\s@uvc@ntr@l\et@tfigptcircTD%
    \setc@ntr@l{2}\c@lExtAxes#3,#4,#5(#6)\figptellP#1:{#2}:#3,-4,-5(#7)%
    \resetc@ntr@l\et@tfigptcircTD}\ignorespaces\fi}
\ctr@ln@m\figptcircumcenter
\ctr@ld@f\def\figptcircumcenterDD#1:#2[#3,#4,#5]{\ifGR@cri{\s@uvc@ntr@l\et@tfigptcircumcenterDD%
    \setc@ntr@l{2}\figvectNDD-5[#3,#4]\figptbaryDD-3:[#3,#4;1,1]%
                  \figvectNDD-6[#4,#5]\figptbaryDD-4:[#4,#5;1,1]%
    \resetc@ntr@l{2}\inters@cDD#1:{#2}[-3,-5;-4,-6]%
    \resetc@ntr@l\et@tfigptcircumcenterDD}\ignorespaces\fi}
\ctr@ld@f\def\figptcircumcenterTD#1:#2[#3,#4,#5]{\ifGR@cri{\s@uvc@ntr@l\et@tfigptcircumcenterTD%
    \setc@ntr@l{2}\figvectNTD-1[#3,#4,#5]%
    \figvectPTD-3[#3,#4]\figvectNVTD-5[-1,-3]\figptbaryTD-3:[#3,#4;1,1]%
    \figvectPTD-4[#4,#5]\figvectNVTD-6[-1,-4]\figptbaryTD-4:[#4,#5;1,1]%
    \resetc@ntr@l{2}\inters@cTD#1:{#2}[-3,-5;-4,-6]%
    \resetc@ntr@l\et@tfigptcircumcenterTD}\ignorespaces\fi}
\ctr@ln@m\figptcopy
\ctr@ld@f\def\figptcopyDD#1:#2/#3/{\ifGR@cri{\Figg@tXY{#3}%
    \Figp@intregDD#1:{#2}(\v@lX,\v@lY)}\ignorespaces\fi}
\ctr@ld@f\def\figptcopyTD#1:#2/#3/{\ifGR@cri{\Figg@tXY{#3}%
    \Figp@intregTD#1:{#2}(\v@lX,\v@lY,\v@lZ)}\ignorespaces\fi}
\ctr@ln@m\figptcurvcenter
\ctr@ld@f\def\figptcurvcenterDD#1:#2:#3[#4,#5,#6,#7]{\ifGR@cri{\s@uvc@ntr@l\et@tfigptcurvcenterDD%
    \setc@ntr@l{2}\c@lcurvradDD#3[#4,#5,#6,#7]\edef\Sprim@{\repdecn@mb{\result@t}}%
    \figptBezierDD-1::#3[#4,#5,#6,#7]\figpttraDD#1:{#2}=-1/\Sprim@,-5/%
    \resetc@ntr@l\et@tfigptcurvcenterDD}\ignorespaces\fi}
\ctr@ld@f\def\figptcurvcenterTD#1:#2:#3[#4,#5,#6,#7]{\ifGR@cri{\s@uvc@ntr@l\et@tfigptcurvcenterTD%
    \setc@ntr@l{2}\figvectDBezierTD -5:1,#3[#4,#5,#6,#7]%
    \figvectDBezierTD -6:2,#3[#4,#5,#6,#7]\vecunit@TD{-5}{-5}%
    \edef\Sprim@{\repdecn@mb{\result@t}}\figvectNVTD-1[-6,-5]%
    \figvectNVTD-5[-5,-1]\c@lproscalTD\v@leur[-6,-5]%
    \invers@{\v@leur}{\v@leur}\v@leur=\Sprim@\v@leur\v@leur=\Sprim@\v@leur%
    \figptBezierTD-1::#3[#4,#5,#6,#7]\edef\Sprim@{\repdecn@mb{\v@leur}}%
    \figpttraTD#1:{#2}=-1/\Sprim@,-5/\resetc@ntr@l\et@tfigptcurvcenterTD}\ignorespaces\fi}
\ctr@ld@f\def\c@lcurvradDD#1[#2,#3,#4,#5]{{\figvectDBezierDD -5:1,#1[#2,#3,#4,#5]%
    \figvectDBezierDD -6:2,#1[#2,#3,#4,#5]\vecunit@DD{-5}{-5}%
    \edef\Sprim@{\repdecn@mb{\result@t}}\figvectNVDD-5[-5]\c@lproscalDD\v@leur[-6,-5]%
    \invers@{\v@leur}{\v@leur}\v@leur=\Sprim@\v@leur\v@leur=\Sprim@\v@leur%
    \global\result@t=\v@leur}}
\ctr@ln@m\figptell
\ctr@ld@f\def\figptellDD#1:#2:#3;#4,#5(#6,#7){\ifGR@cri{\s@uvc@ntr@l\et@tfigptell%
    \c@lptellDD#1::#3;#4,#5(#6)\figptrotDD#1:{#2}=#1/#3,#7/%
    \resetc@ntr@l\et@tfigptell}\ignorespaces\fi}
\ctr@ld@f\def\c@lptellDD#1:#2:#3;#4,#5(#6){\c@ssin{\C@}{\S@}{#6}\v@lmin=\C@ pt\v@lmax=\S@ pt%
    \v@lmin=#4\v@lmin\v@lmax=#5\v@lmax%
    \edef\Xc@mp{\repdecn@mb{\v@lmin}}\edef\Yc@mp{\repdecn@mb{\v@lmax}}%
    \setc@ntr@l{2}\figvectC-1(\Xc@mp,\Yc@mp)\figpttraDD#1:{#2}=#3/1,-1/}
\ctr@ld@f\def\figptellP#1:#2:#3,#4,#5(#6){\ifGR@cri{\s@uvc@ntr@l\et@tfigptellP%
    \setc@ntr@l{2}\figvectP-1[#3,#4]\figvectP-2[#3,#5]%
    \v@leur=#6pt\c@lptellP{#3}{-1}{-2}\figptcopy#1:{#2}/-3/%
    \resetc@ntr@l\et@tfigptellP}\ignorespaces\fi}
\ctr@ln@m\@ngle
\ctr@ld@f\def\c@lptellP#1#2#3{\edef\@ngle{\repdecn@mb\v@leur}\c@ssin{\C@}{\S@}{\@ngle}%
    \figpttra-3:=#1/\C@,#2/\figpttra-3:=-3/\S@,#3/}
\ctr@ln@m\figptendnormal
\ctr@ld@f\def\figptendnormalDD#1:#2:#3,#4[#5,#6]{\ifGR@cri{\s@uvc@ntr@l\et@tfigptendnormal%
    \Figg@tXYa{#5}\Figg@tXY{#6}%
    \advance\v@lX-\v@lXa\advance\v@lY-\v@lYa%
    \setc@ntr@l{2}\Figv@ctCreg-1(\v@lX,\v@lY)\vecunit@{-1}{-1}\Figg@tXY{-1}%
    \delt@=#3\unit@\maxim@m{\delt@}{\delt@}{-\delt@}\edef\l@ngueur{\repdecn@mb{\delt@}}%
    \v@lX=\l@ngueur\v@lX\v@lY=\l@ngueur\v@lY%
    \delt@=\p@\advance\delt@-#4pt\edef\l@ngueur{\repdecn@mb{\delt@}}%
    \figptbaryR-1:[#5,#6;#4,\l@ngueur]\Figg@tXYa{-1}%
    \advance\v@lXa\v@lY\advance\v@lYa-\v@lX%
    \setc@ntr@l{1}\Figp@intregDD#1:{#2}(\v@lXa,\v@lYa)\resetc@ntr@l\et@tfigptendnormal}%
    \ignorespaces\fi}
\ctr@ld@f\def\figptexcenter#1:#2[#3,#4,#5]{\ifGR@cri{\let@xte={-}
    \Figptexinsc@nter#1:#2[#3,#4,#5]}\ignorespaces\fi}
\ctr@ld@f\def\figptincenter#1:#2[#3,#4,#5]{\ifGR@cri{\let@xte={}
    \Figptexinsc@nter#1:#2[#3,#4,#5]}\ignorespaces\fi}
\ctr@ld@f\let\figptinscribedcenter=\figptincenter
\ctr@ld@f\def\Figptexinsc@nter#1:#2[#3,#4,#5]{%
    \figgetdist\LA@[#4,#5]\figgetdist\LB@[#3,#5]\figgetdist\LC@[#3,#4]%
    \figptbaryR#1:{#2}[#3,#4,#5;\the\let@xte\LA@,\LB@,\LC@]}
\ctr@ln@m\figptinterlineplane
\ctr@ld@f\def\figptinterlineplaneDD{\un@v@ilable{figptinterlineplane}}
\ctr@ld@f\def\figptinterlineplaneTD#1:#2[#3,#4;#5,#6]{\ifGR@cri{\s@uvc@ntr@l\et@tfigptinterlineplane%
    \setc@ntr@l{2}\figvectPTD-1[#3,#5]\vecunit@TD{-2}{#6}%
    \r@pPSTD\v@leur[-2,-1,#4]\edef\v@lcoef{\repdecn@mb{\v@leur}}%
    \figpttraTD#1:{#2}=#3/\v@lcoef,#4/\resetc@ntr@l\et@tfigptinterlineplane}\ignorespaces\fi}
\ctr@ln@m\figptorthocenter
\ctr@ld@f\def\figptorthocenterDD#1:#2[#3,#4,#5]{\ifGR@cri{\s@uvc@ntr@l\et@tfigptorthocenterDD%
    \setc@ntr@l{2}\figvectNDD-3[#3,#4]\figvectNDD-4[#4,#5]%
    \resetc@ntr@l{2}\inters@cDD#1:{#2}[#5,-3;#3,-4]%
    \resetc@ntr@l\et@tfigptorthocenterDD}\ignorespaces\fi}
\ctr@ld@f\def\figptorthocenterTD#1:#2[#3,#4,#5]{\ifGR@cri{\s@uvc@ntr@l\et@tfigptorthocenterTD%
    \setc@ntr@l{2}\figvectNTD-1[#3,#4,#5]%
    \figvectPTD-2[#3,#4]\figvectNVTD-3[-1,-2]%
    \figvectPTD-2[#4,#5]\figvectNVTD-4[-1,-2]%
    \resetc@ntr@l{2}\inters@cTD#1:{#2}[#5,-3;#3,-4]%
    \resetc@ntr@l\et@tfigptorthocenterTD}\ignorespaces\fi}
\ctr@ln@m\figptorthoprojline
\ctr@ld@f\def\figptorthoprojlineDD#1:#2=#3/#4,#5/{\ifGR@cri{\s@uvc@ntr@l\et@tfigptorthoprojlineDD%
    \setc@ntr@l{2}\figvectPDD-3[#4,#5]\figvectNVDD-4[-3]\resetc@ntr@l{2}%
    \inters@cDD#1:{#2}[#3,-4;#4,-3]\resetc@ntr@l\et@tfigptorthoprojlineDD}\ignorespaces\fi}
\ctr@ld@f\def\figptorthoprojlineTD#1:#2=#3/#4,#5/{\ifGR@cri{\s@uvc@ntr@l\et@tfigptorthoprojlineTD%
    \setc@ntr@l{2}\figvectPTD-1[#4,#3]\figvectPTD-2[#4,#5]\vecunit@TD{-2}{-2}%
    \c@lproscalTD\v@leur[-1,-2]\edef\v@lcoef{\repdecn@mb{\v@leur}}%
    \figpttraTD#1:{#2}=#4/\v@lcoef,-2/\resetc@ntr@l\et@tfigptorthoprojlineTD}\ignorespaces\fi}
\ctr@ln@m\figptorthoprojplane
\ctr@ld@f\def\figptorthoprojplaneDD{\un@v@ilable{figptorthoprojplane}}
\ctr@ld@f\def\figptorthoprojplaneTD#1:#2=#3/#4,#5/{\ifGR@cri{\s@uvc@ntr@l\et@tfigptorthoprojplane%
    \setc@ntr@l{2}\figvectPTD-1[#3,#4]\vecunit@TD{-2}{#5}%
    \c@lproscalTD\v@leur[-1,-2]\edef\v@lcoef{\repdecn@mb{\v@leur}}%
    \figpttraTD#1:{#2}=#3/\v@lcoef,-2/\resetc@ntr@l\et@tfigptorthoprojplane}\ignorespaces\fi}
\ctr@ld@f\def\figpthom#1:#2=#3/#4,#5/{\ifGR@cri{\s@uvc@ntr@l\et@tfigpthom%
    \setc@ntr@l{2}\figvectP-1[#4,#3]\figpttra#1:{#2}=#4/#5,-1/%
    \resetc@ntr@l\et@tfigpthom}\ignorespaces\fi}
\ctr@ld@f\def\figptinv#1:#2=#3/#4,#5/{\ifGR@cri{\s@uvc@ntr@l\et@tfigptinv%
    \setc@ntr@l{2}\figvectP-1[#4,#3]\Figg@tXY{-1}%
    \getredf@ctB\f@ctech\n@rmeucC{\delt@}{-1}%
    \delt@=\ptT@unit@\delt@\delt@=\ptT@unit@\delt@%
    \invers@{\delt@}{\delt@}\multiply\f@ctech\f@ctech\divide\delt@\f@ctech%
    \delt@=#5\delt@\edef\v@lcoef{\repdecn@mb{\delt@}}\figpttra#1:{#2}=#4/\v@lcoef,-1/%
    \resetc@ntr@l\et@tfigptinv}\ignorespaces\fi}
\ctr@ln@m\figptrot
\ctr@ld@f\def\figptrotDD#1:#2=#3/#4,#5/{\ifGR@cri{\s@uvc@ntr@l\et@tfigptrotDD%
    \c@ssin{\C@}{\S@}{#5}\setc@ntr@l{2}\figvectPDD-1[#4,#3]\Figg@tXY{-1}%
    \v@lXa=\C@\v@lX\advance\v@lXa-\S@\v@lY%
    \v@lYa=\S@\v@lX\advance\v@lYa\C@\v@lY%
    \Figv@ctCreg-1(\v@lXa,\v@lYa)\figpttraDD#1:{#2}=#4/1,-1/%
    \resetc@ntr@l\et@tfigptrotDD}\ignorespaces\fi}
\ctr@ld@f\def\figptrotTD#1:#2=#3/#4,#5,#6/{\ifGR@cri{\s@uvc@ntr@l\et@tfigptrotTD%
    \c@ssin{\C@}{\S@}{#5}%
    \setc@ntr@l{2}\figptorthoprojplaneTD-3:=#4/#3,#6/\figvectPTD-2[-3,#3]%
    \n@rmeucTD\v@leur{-2}\ifdim\v@leur<\Cepsil@n\Figg@tXYa{#3}\else%
    \edef\v@lcoef{\repdecn@mb{\v@leur}}\figvectNVTD-1[#6,-2]%
    \Figg@tXYa{-1}\v@lXa=\v@lcoef\v@lXa\v@lYa=\v@lcoef\v@lYa\v@lZa=\v@lcoef\v@lZa%
    \v@lXa=\S@\v@lXa\v@lYa=\S@\v@lYa\v@lZa=\S@\v@lZa\Figg@tXY{-2}%
    \advance\v@lXa\C@\v@lX\advance\v@lYa\C@\v@lY\advance\v@lZa\C@\v@lZ%
    \Figg@tXY{-3}\advance\v@lXa\v@lX\advance\v@lYa\v@lY\advance\v@lZa\v@lZ\fi%
    \Figp@intregTD#1:{#2}(\v@lXa,\v@lYa,\v@lZa)\resetc@ntr@l\et@tfigptrotTD}\ignorespaces\fi}
\ctr@ln@m\figptsym
\ctr@ld@f\def\figptsymDD#1:#2=#3/#4,#5/{\ifGR@cri{\s@uvc@ntr@l\et@tfigptsymDD%
    \resetc@ntr@l{2}\figptorthoprojlineDD-5:=#3/#4,#5/\figvectPDD-2[#3,-5]%
    \figpttraDD#1:{#2}=#3/2,-2/\resetc@ntr@l\et@tfigptsymDD}\ignorespaces\fi}
\ctr@ld@f\def\figptsymTD#1:#2=#3/#4,#5/{\ifGR@cri{\s@uvc@ntr@l\et@tfigptsymTD%
    \resetc@ntr@l{2}\figptorthoprojplaneTD-3:=#3/#4,#5/\figvectPTD-2[#3,-3]%
    \figpttraTD#1:{#2}=#3/2,-2/\resetc@ntr@l\et@tfigptsymTD}\ignorespaces\fi}
\ctr@ln@m\figpttra
\ctr@ld@f\def\figpttraDD#1:#2=#3/#4,#5/{\ifGR@cri{\Figg@tXYa{#5}\v@lXa=#4\v@lXa\v@lYa=#4\v@lYa%
    \Figg@tXY{#3}\advance\v@lX\v@lXa\advance\v@lY\v@lYa%
    \Figp@intregDD#1:{#2}(\v@lX,\v@lY)}\ignorespaces\fi}
\ctr@ld@f\def\figpttraTD#1:#2=#3/#4,#5/{\ifGR@cri{\Figg@tXYa{#5}\v@lXa=#4\v@lXa\v@lYa=#4\v@lYa%
    \v@lZa=#4\v@lZa\Figg@tXY{#3}\advance\v@lX\v@lXa\advance\v@lY\v@lYa%
    \advance\v@lZ\v@lZa\Figp@intregTD#1:{#2}(\v@lX,\v@lY,\v@lZ)}\ignorespaces\fi}
\ctr@ln@m\figpttraC
\ctr@ld@f\def\figpttraCDD#1:#2=#3/#4,#5/{\ifGR@cri{\v@lXa=#4\unit@\v@lYa=#5\unit@%
    \Figg@tXY{#3}\advance\v@lX\v@lXa\advance\v@lY\v@lYa%
    \Figp@intregDD#1:{#2}(\v@lX,\v@lY)}\ignorespaces\fi}
\ctr@ld@f\def\figpttraCTD#1:#2=#3/#4,#5,#6/{\ifGR@cri{\v@lXa=#4\unit@\v@lYa=#5\unit@\v@lZa=#6\unit@%
    \Figg@tXY{#3}\advance\v@lX\v@lXa\advance\v@lY\v@lYa\advance\v@lZ\v@lZa%
    \Figp@intregTD#1:{#2}(\v@lX,\v@lY,\v@lZ)}\ignorespaces\fi}
\ctr@ld@f\def\figptsaxes#1:#2(#3){\ifGR@cri{\an@lys@xes#3,:\ifx\t@xt@\empty%
    \ifTr@isDim\Figpts@xes#1:#2(0,#3,0,#3,0,#3)\else\Figpts@xes#1:#2(0,#3,0,#3)\fi%
    \else\Figpts@xes#1:#2(#3)\fi}\ignorespaces\fi}
\ctr@ln@m\Figpts@xes
\ctr@ld@f\def\Figpts@xesDD#1:#2(#3,#4,#5,#6){%
    \s@mme=#1\figpttraC\the\s@mme:$x$=#2/#4,0/%
    \advance\s@mme\@ne\figpttraC\the\s@mme:$y$=#2/0,#6/}
\ctr@ld@f\def\Figpts@xesTD#1:#2(#3,#4,#5,#6,#7,#8){%
    \s@mme=#1\figpttraC\the\s@mme:$x$=#2/#4,0,0/%
    \advance\s@mme\@ne\figpttraC\the\s@mme:$y$=#2/0,#6,0/%
    \advance\s@mme\@ne\figpttraC\the\s@mme:$z$=#2/0,0,#8/}
\ctr@ld@f\def\figptsmap#1=#2/#3/#4/{\ifGR@cri{\s@uvc@ntr@l\et@tfigptsmap%
    \setc@ntr@l{2}\def\list@num{#2}\s@mme=#1%
    \@ecfor\p@int:=\list@num\do{\figvectP-1[#3,\p@int]\Figg@tXY{-1}%
    \pr@dMatV/#4/\figpttra\the\s@mme:=#3/1,-1/\advance\s@mme\@ne}%
    \resetc@ntr@l\et@tfigptsmap}\ignorespaces\fi}
\ctr@ln@m\figptscontrol
\ctr@ld@f\def\figptscontrolDD#1[#2,#3,#4,#5]{\ifGR@cri{\s@uvc@ntr@l\et@tfigptscontrolDD\setc@ntr@l{2}%
    \v@lX=\z@\v@lY=\z@\Figtr@nptDD{-5}{#2}\Figtr@nptDD{2}{#5}%
    \divide\v@lX\@vi\divide\v@lY\@vi%
    \Figtr@nptDD{3}{#3}\Figtr@nptDD{-1.5}{#4}\Figp@intregDD-1:(\v@lX,\v@lY)%
    \v@lX=\z@\v@lY=\z@\Figtr@nptDD{2}{#2}\Figtr@nptDD{-5}{#5}%
    \divide\v@lX\@vi\divide\v@lY\@vi\Figtr@nptDD{-1.5}{#3}\Figtr@nptDD{3}{#4}%
    \s@mme=#1\advance\s@mme\@ne\Figp@intregDD\the\s@mme:(\v@lX,\v@lY)%
    \figptcopyDD#1:/-1/\resetc@ntr@l\et@tfigptscontrolDD}\ignorespaces\fi}
\ctr@ld@f\def\figptscontrolTD#1[#2,#3,#4,#5]{\ifGR@cri{\s@uvc@ntr@l\et@tfigptscontrolTD\setc@ntr@l{2}%
    \v@lX=\z@\v@lY=\z@\v@lZ=\z@\Figtr@nptTD{-5}{#2}\Figtr@nptTD{2}{#5}%
    \divide\v@lX\@vi\divide\v@lY\@vi\divide\v@lZ\@vi%
    \Figtr@nptTD{3}{#3}\Figtr@nptTD{-1.5}{#4}\Figp@intregTD-1:(\v@lX,\v@lY,\v@lZ)%
    \v@lX=\z@\v@lY=\z@\v@lZ=\z@\Figtr@nptTD{2}{#2}\Figtr@nptTD{-5}{#5}%
    \divide\v@lX\@vi\divide\v@lY\@vi\divide\v@lZ\@vi\Figtr@nptTD{-1.5}{#3}\Figtr@nptTD{3}{#4}%
    \s@mme=#1\advance\s@mme\@ne\Figp@intregTD\the\s@mme:(\v@lX,\v@lY,\v@lZ)%
    \figptcopyTD#1:/-1/\resetc@ntr@l\et@tfigptscontrolTD}\ignorespaces\fi}
\ctr@ld@f\def\Figtr@nptDD#1#2{\Figg@tXYa{#2}\v@lXa=#1\v@lXa\v@lYa=#1\v@lYa%
    \advance\v@lX\v@lXa\advance\v@lY\v@lYa}
\ctr@ld@f\def\Figtr@nptTD#1#2{\Figg@tXYa{#2}\v@lXa=#1\v@lXa\v@lYa=#1\v@lYa\v@lZa=#1\v@lZa%
    \advance\v@lX\v@lXa\advance\v@lY\v@lYa\advance\v@lZ\v@lZa}
\ctr@ld@f\def\figptscontrolcurve#1,#2[#3]{\ifGR@cri{\s@uvc@ntr@l\et@tfigptscontrolcurve%
    \def\list@num{#3}\extrairelepremi@r\Ak@\de\list@num%
    \extrairelepremi@r\Ai@\de\list@num\extrairelepremi@r\Aj@\de\list@num%
    \s@mme=#1\figptcopy\the\s@mme:/\Ai@/%
    \setc@ntr@l{2}\figvectP -1[\Ak@,\Aj@]%
    \@ecfor\Ak@:=\list@num\do{\advance\s@mme\@ne\figpttra\the\s@mme:=\Ai@/\curv@roundness,-1/%
       \figvectP -1[\Ai@,\Ak@]\advance\s@mme\@ne\figpttra\the\s@mme:=\Aj@/-\curv@roundness,-1/%
       \advance\s@mme\@ne\figptcopy\the\s@mme:/\Aj@/%
       \edef\Ai@{\Aj@}\edef\Aj@{\Ak@}}\advance\s@mme-#1\divide\s@mme\thr@@%
       \xdef#2{\the\s@mme}%
    \resetc@ntr@l\et@tfigptscontrolcurve}\ignorespaces\fi}
\ctr@ln@m\figptsintercirc
\ctr@ld@f\def\figptsintercircDD#1[#2,#3;#4,#5]{\ifGR@cri{\s@uvc@ntr@l\et@tfigptsintercircDD%
    \setc@ntr@l{2}\let\c@lNVintc=\c@lNVintcDD\Figptsintercirc@#1[#2,#3;#4,#5]%
    \resetc@ntr@l\et@tfigptsintercircDD}\ignorespaces\fi}
\ctr@ld@f\def\figptsintercircTD#1[#2,#3;#4,#5;#6]{\ifGR@cri{\s@uvc@ntr@l\et@tfigptsintercircTD%
    \setc@ntr@l{2}\let\c@lNVintc=\c@lNVintcTD\vecunitC@TD[#2,#6]%
    \Figv@ctCreg-3(\v@lX,\v@lY,\v@lZ)\Figptsintercirc@#1[#2,#3;#4,#5]%
    \resetc@ntr@l\et@tfigptsintercircTD}\ignorespaces\fi}
\ctr@ld@f\def\Figptsintercirc@#1[#2,#3;#4,#5]{\figvectP-1[#2,#4]%
    \vecunit@{-1}{-1}\delt@=\result@t\f@ctech=\result@tent%
    \s@mme=#1\advance\s@mme\@ne\figptcopy#1:/#2/\figptcopy\the\s@mme:/#4/%
    \ifdim\delt@=\z@\else%
    \v@lmin=#3\unit@\v@lmax=#5\unit@\v@leur=\v@lmin\advance\v@leur\v@lmax%
    \ifdim\v@leur>\delt@%
    \v@leur=\v@lmin\advance\v@leur-\v@lmax\maxim@m{\v@leur}{\v@leur}{-\v@leur}%
    \ifdim\v@leur<\delt@%
    \divide\v@lmin\f@ctech\divide\v@lmax\f@ctech\divide\delt@\f@ctech%
    \v@lmin=\repdecn@mb{\v@lmin}\v@lmin\v@lmax=\repdecn@mb{\v@lmax}\v@lmax%
    \invers@{\v@leur}{\delt@}\advance\v@lmax-\v@lmin%
    \v@lmax=-\repdecn@mb{\v@leur}\v@lmax\advance\delt@\v@lmax\delt@=.5\delt@%
    \v@lmax=\delt@\multiply\v@lmax\f@ctech%
    \edef\t@ille{\repdecn@mb{\v@lmax}}\figpttra-2:=#2/\t@ille,-1/%
    \delt@=\repdecn@mb{\delt@}\delt@\advance\v@lmin-\delt@%
    \sqrt@{\v@leur}{\v@lmin}\multiply\v@leur\f@ctech\edef\t@ille{\repdecn@mb{\v@leur}}%
    \c@lNVintc\figpttra#1:=-2/-\t@ille,-1/\figpttra\the\s@mme:=-2/\t@ille,-1/\fi\fi\fi}
\ctr@ld@f\def\c@lNVintcDD{\Figg@tXY{-1}\Figv@ctCreg-1(-\v@lY,\v@lX)} 
\ctr@ld@f\def\c@lNVintcTD{{\Figg@tXY{-3}\v@lmin=\v@lX\v@lmax=\v@lY\v@leur=\v@lZ%
    \Figg@tXY{-1}\c@lprovec{-3}\vecunit@{-3}{-3}
    \Figg@tXY{-1}\v@lmin=\v@lX\v@lmax=\v@lY%
    \v@leur=\v@lZ\Figg@tXY{-3}\c@lprovec{-1}}} 
\ctr@ln@m\figptsinterlinell
\ctr@ld@f\def\figptsinterlinellDD#1[#2,#3,#4,#5;#6,#7]{\ifGR@cri{\s@uvc@ntr@l\et@tfigptsinterlinellDD%
    \figptcopy#1:/#6/\s@mme=#1\advance\s@mme\@ne\figptcopy\the\s@mme:/#7/%
    \v@lmin=#3\unit@\v@lmax=#4\unit@
    \setc@ntr@l{2}\figptbaryDD-4:[#6,#7;1,1]\figptsrotDD-3=-4,#7/#2,-#5/
    \Figg@tXY{-3}\Figg@tXYa{#2}\advance\v@lX-\v@lXa\advance\v@lY-\v@lYa
    \figvectP-1[-3,-2]\Figg@tXYa{-1}\figvectP-3[-4,#7]\Figptsint@rLE{#1}
    \resetc@ntr@l\et@tfigptsinterlinellDD}\ignorespaces\fi}
\ctr@ld@f\def\figptsinterlinellP#1[#2,#3,#4;#5,#6]{\ifGR@cri{\s@uvc@ntr@l\et@tfigptsinterlinellP%
    \figptcopy#1:/#5/\s@mme=#1\advance\s@mme\@ne\figptcopy\the\s@mme:/#6/\setc@ntr@l{2}%
    \figvectP-1[#2,#3]\vecunit@{-1}{-1}\v@lmin=\result@t
    \figvectP-2[#2,#4]\vecunit@{-2}{-2}\v@lmax=\result@t
    \figptbary-4:[#5,#6;1,1]
    \figvectP-3[#2,-4]\c@lproscal\v@lX[-3,-1]\c@lproscal\v@lY[-3,-2]
    \figvectP-3[-4,#6]\c@lproscal\v@lXa[-3,-1]\c@lproscal\v@lYa[-3,-2]
    \Figptsint@rLE{#1}\resetc@ntr@l\et@tfigptsinterlinellP}\ignorespaces\fi}
\ctr@ld@f\def\Figptsint@rLE#1{%
    \getredf@ctDD\f@ctech(\v@lmin,\v@lmax)%
    \getredf@ctDD\p@rtent(\v@lX,\v@lY)\ifnum\p@rtent>\f@ctech\f@ctech=\p@rtent\fi%
    \getredf@ctDD\p@rtent(\v@lXa,\v@lYa)\ifnum\p@rtent>\f@ctech\f@ctech=\p@rtent\fi%
    \divide\v@lmin\f@ctech\divide\v@lmax\f@ctech\divide\v@lX\f@ctech\divide\v@lY\f@ctech%
    \divide\v@lXa\f@ctech\divide\v@lYa\f@ctech%
    \c@rre=\repdecn@mb\v@lXa\v@lmax\mili@u=\repdecn@mb\v@lYa\v@lmin%
    \getredf@ctDD\f@ctech(\c@rre,\mili@u)%
    \c@rre=\repdecn@mb\v@lX\v@lmax\mili@u=\repdecn@mb\v@lY\v@lmin%
    \getredf@ctDD\p@rtent(\c@rre,\mili@u)\ifnum\p@rtent>\f@ctech\f@ctech=\p@rtent\fi%
    \divide\v@lmin\f@ctech\divide\v@lmax\f@ctech\divide\v@lX\f@ctech\divide\v@lY\f@ctech%
    \divide\v@lXa\f@ctech\divide\v@lYa\f@ctech%
    \v@lmin=\repdecn@mb{\v@lmin}\v@lmin\v@lmax=\repdecn@mb{\v@lmax}\v@lmax%
    \edef\G@xde{\repdecn@mb\v@lmin}\edef\P@xde{\repdecn@mb\v@lmax}%
    \c@rre=-\v@lmax\v@leur=\repdecn@mb\v@lY\v@lY\advance\c@rre\v@leur\c@rre=\G@xde\c@rre%
    \v@leur=\repdecn@mb\v@lX\v@lX\v@leur=\P@xde\v@leur\advance\c@rre\v@leur
    \v@lmin=\repdecn@mb\v@lYa\v@lmin\v@lmax=\repdecn@mb\v@lXa\v@lmax%
    \mili@u=\repdecn@mb\v@lX\v@lmax\advance\mili@u\repdecn@mb\v@lY\v@lmin
    \v@lmax=\repdecn@mb\v@lXa\v@lmax\advance\v@lmax\repdecn@mb\v@lYa\v@lmin
    \ifdim\v@lmax>\epsil@n%
    \maxim@m{\v@leur}{\c@rre}{-\c@rre}\maxim@m{\v@lmin}{\mili@u}{-\mili@u}%
    \maxim@m{\v@leur}{\v@leur}{\v@lmin}\maxim@m{\v@lmin}{\v@lmax}{-\v@lmax}%
    \maxim@m{\v@leur}{\v@leur}{\v@lmin}\p@rtentiere{\p@rtent}{\v@leur}\advance\p@rtent\@ne%
    \divide\c@rre\p@rtent\divide\mili@u\p@rtent\divide\v@lmax\p@rtent%
    \delt@=\repdecn@mb{\mili@u}\mili@u\v@leur=\repdecn@mb{\v@lmax}\c@rre%
    \advance\delt@-\v@leur\ifdim\delt@<\z@\else\sqrt@\delt@\delt@%
    \invers@\v@lmax\v@lmax\edef\Uns@rAp{\repdecn@mb\v@lmax}%
    \v@leur=-\mili@u\advance\v@leur-\delt@\v@leur=\Uns@rAp\v@leur%
    \edef\t@ille{\repdecn@mb\v@leur}\figpttra#1:=-4/\t@ille,-3/\s@mme=#1\advance\s@mme\@ne%
    \v@leur=-\mili@u\advance\v@leur\delt@\v@leur=\Uns@rAp\v@leur%
    \edef\t@ille{\repdecn@mb\v@leur}\figpttra\the\s@mme:=-4/\t@ille,-3/\fi\fi}
\ctr@ln@m\figptsorthoprojline
\ctr@ld@f\def\figptsorthoprojlineDD#1=#2/#3,#4/{\ifGR@cri{\s@uvc@ntr@l\et@tfigptsorthoprojlineDD%
    \setc@ntr@l{2}\figvectPDD-3[#3,#4]\figvectNVDD-4[-3]\resetc@ntr@l{2}%
    \def\list@num{#2}\s@mme=#1\@ecfor\p@int:=\list@num\do{%
    \inters@cDD\the\s@mme:[\p@int,-4;#3,-3]\advance\s@mme\@ne}%
    \resetc@ntr@l\et@tfigptsorthoprojlineDD}\ignorespaces\fi}
\ctr@ld@f\def\figptsorthoprojlineTD#1=#2/#3,#4/{\ifGR@cri{\s@uvc@ntr@l\et@tfigptsorthoprojlineTD%
    \setc@ntr@l{2}\figvectPTD-2[#3,#4]\vecunit@TD{-2}{-2}%
    \def\list@num{#2}\s@mme=#1\@ecfor\p@int:=\list@num\do{%
    \figvectPTD-1[#3,\p@int]\c@lproscalTD\v@leur[-1,-2]%
    \edef\v@lcoef{\repdecn@mb{\v@leur}}\figpttraTD\the\s@mme:=#3/\v@lcoef,-2/%
    \advance\s@mme\@ne}\resetc@ntr@l\et@tfigptsorthoprojlineTD}\ignorespaces\fi}
\ctr@ln@m\figptsorthoprojplane
\ctr@ld@f\def\figptsorthoprojplaneDD{\un@v@ilable{figptsorthoprojplane}}
\ctr@ld@f\def\figptsorthoprojplaneTD#1=#2/#3,#4/{\ifGR@cri{\s@uvc@ntr@l\et@tfigptsorthoprojplane%
    \setc@ntr@l{2}\vecunit@TD{-2}{#4}%
    \def\list@num{#2}\s@mme=#1\@ecfor\p@int:=\list@num\do{\figvectPTD-1[\p@int,#3]%
    \c@lproscalTD\v@leur[-1,-2]\edef\v@lcoef{\repdecn@mb{\v@leur}}%
    \figpttraTD\the\s@mme:=\p@int/\v@lcoef,-2/\advance\s@mme\@ne}%
    \resetc@ntr@l\et@tfigptsorthoprojplane}\ignorespaces\fi}
\ctr@ld@f\def\figptshom#1=#2/#3,#4/{\ifGR@cri{\s@uvc@ntr@l\et@tfigptshom%
    \setc@ntr@l{2}\def\list@num{#2}\s@mme=#1%
    \@ecfor\p@int:=\list@num\do{\figvectP-1[#3,\p@int]%
    \figpttra\the\s@mme:=#3/#4,-1/\advance\s@mme\@ne}%
    \resetc@ntr@l\et@tfigptshom}\ignorespaces\fi}
\ctr@ld@f\def\figptsinv#1=#2/#3,#4/{\ifGR@cri{\s@uvc@ntr@l\et@tfigptsinv%
    \setc@ntr@l{2}\def\list@num{#2}\s@mme=#1%
    \@ecfor\p@int:=\list@num\do{\figvectP-1[#3,\p@int]\Figg@tXY{-1}%
    \getredf@ctB\f@ctech\n@rmeucC{\delt@}{-1}%
    \delt@=\ptT@unit@\delt@\delt@=\ptT@unit@\delt@%
    \invers@{\delt@}{\delt@}\multiply\f@ctech\f@ctech\divide\delt@\f@ctech%
    \delt@=#4\delt@\edef\v@lcoef{\repdecn@mb{\delt@}}\figpttra\the\s@mme:=#3/\v@lcoef,-1/%
    \advance\s@mme\@ne}\resetc@ntr@l\et@tfigptsinv}\ignorespaces\fi}
\ctr@ln@m\figptsrot
\ctr@ld@f\def\figptsrotDD#1=#2/#3,#4/{\ifGR@cri{\s@uvc@ntr@l\et@tfigptsrotDD%
    \c@ssin{\C@}{\S@}{#4}\setc@ntr@l{2}\def\list@num{#2}\s@mme=#1%
    \@ecfor\p@int:=\list@num\do{\figvectPDD-1[#3,\p@int]\Figg@tXY{-1}%
    \v@lXa=\C@\v@lX\advance\v@lXa-\S@\v@lY%
    \v@lYa=\S@\v@lX\advance\v@lYa\C@\v@lY%
    \Figv@ctCreg-1(\v@lXa,\v@lYa)\figpttraDD\the\s@mme:=#3/1,-1/\advance\s@mme\@ne}%
    \resetc@ntr@l\et@tfigptsrotDD}\ignorespaces\fi}
\ctr@ld@f\def\figptsrotTD#1=#2/#3,#4,#5/{\ifGR@cri{\s@uvc@ntr@l\et@tfigptsrotTD%
    \c@ssin{\C@}{\S@}{#4}%
    \setc@ntr@l{2}\def\list@num{#2}\s@mme=#1%
    \@ecfor\p@int:=\list@num\do{\figptorthoprojplaneTD-3:=#3/\p@int,#5/%
    \figvectPTD-2[-3,\p@int]%
    \figvectNVTD-1[#5,-2]\n@rmeucTD\v@leur{-2}\edef\v@lcoef{\repdecn@mb{\v@leur}}%
    \Figg@tXYa{-1}\v@lXa=\v@lcoef\v@lXa\v@lYa=\v@lcoef\v@lYa\v@lZa=\v@lcoef\v@lZa%
    \v@lXa=\S@\v@lXa\v@lYa=\S@\v@lYa\v@lZa=\S@\v@lZa\Figg@tXY{-2}%
    \advance\v@lXa\C@\v@lX\advance\v@lYa\C@\v@lY\advance\v@lZa\C@\v@lZ%
    \Figg@tXY{-3}\advance\v@lXa\v@lX\advance\v@lYa\v@lY\advance\v@lZa\v@lZ%
    \Figp@intregTD\the\s@mme:(\v@lXa,\v@lYa,\v@lZa)\advance\s@mme\@ne}%
    \resetc@ntr@l\et@tfigptsrotTD}\ignorespaces\fi}
\ctr@ln@m\figptssym
\ctr@ld@f\def\figptssymDD#1=#2/#3,#4/{\ifGR@cri{\s@uvc@ntr@l\et@tfigptssymDD%
    \setc@ntr@l{2}\figvectPDD-3[#3,#4]\Figg@tXY{-3}\Figv@ctCreg-4(-\v@lY,\v@lX)%
    \resetc@ntr@l{2}\def\list@num{#2}\s@mme=#1%
    \@ecfor\p@int:=\list@num\do{\inters@cDD-5:[#3,-3;\p@int,-4]\figvectPDD-2[\p@int,-5]%
    \figpttraDD\the\s@mme:=\p@int/2,-2/\advance\s@mme\@ne}%
    \resetc@ntr@l\et@tfigptssymDD}\ignorespaces\fi}
\ctr@ld@f\def\figptssymTD#1=#2/#3,#4/{\ifGR@cri{\s@uvc@ntr@l\et@tfigptssymTD%
    \setc@ntr@l{2}\vecunit@TD{-2}{#4}\def\list@num{#2}\s@mme=#1%
    \@ecfor\p@int:=\list@num\do{\figvectPTD-1[\p@int,#3]%
    \c@lproscalTD\v@leur[-1,-2]\v@leur=2\v@leur\edef\v@lcoef{\repdecn@mb{\v@leur}}%
    \figpttraTD\the\s@mme:=\p@int/\v@lcoef,-2/\advance\s@mme\@ne}%
    \resetc@ntr@l\et@tfigptssymTD}\ignorespaces\fi}
\ctr@ln@m\figptstra
\ctr@ld@f\def\figptstraDD#1=#2/#3,#4/{\ifGR@cri{\Figg@tXYa{#4}\v@lXa=#3\v@lXa\v@lYa=#3\v@lYa%
    \def\list@num{#2}\s@mme=#1\@ecfor\p@int:=\list@num\do{\Figg@tXY{\p@int}%
    \advance\v@lX\v@lXa\advance\v@lY\v@lYa%
    \Figp@intregDD\the\s@mme:(\v@lX,\v@lY)\advance\s@mme\@ne}}\ignorespaces\fi}
\ctr@ld@f\def\figptstraTD#1=#2/#3,#4/{\ifGR@cri{\Figg@tXYa{#4}\v@lXa=#3\v@lXa\v@lYa=#3\v@lYa%
    \v@lZa=#3\v@lZa\def\list@num{#2}\s@mme=#1\@ecfor\p@int:=\list@num\do{\Figg@tXY{\p@int}%
    \advance\v@lX\v@lXa\advance\v@lY\v@lYa\advance\v@lZ\v@lZa%
    \Figp@intregTD\the\s@mme:(\v@lX,\v@lY,\v@lZ)\advance\s@mme\@ne}}\ignorespaces\fi}
\ctr@ln@m\figptvisilimSL
\ctr@ld@f\def\figptvisilimSLDD{\un@v@ilable{figptvisilimSL}}
\ctr@ld@f\def\figptvisilimSLTD#1:#2[#3,#4;#5,#6]{\ifGR@cri{\s@uvc@ntr@l\et@tfigptvisilimSLTD%
    \setc@ntr@l{2}\figvectP-1[#3,#4]\n@rminf{\delt@}{-1}%
    \ifcase\CUR@proj\v@lX=\cxa@\p@\v@lY=-\p@\v@lZ=\cxb@\p@
    \Figv@ctCreg-2(\v@lX,\v@lY,\v@lZ)\figvectP-3[#5,#6]\figvectNV-1[-2,-3]%
    \or\figvectP-1[#5,#6]\vecunitCV@TD{-1}\v@lmin=\v@lX\v@lmax=\v@lY
    \v@leur=\v@lZ\v@lX=\cza@\p@\v@lY=\czb@\p@\v@lZ=\czc@\p@\c@lprovec{-1}%
    \or\c@ley@pt{-2}\figvectN-1[#5,#6,-2]\fi
    \edef\Ai@{#3}\edef\Aj@{#4}\figvectP-2[#5,\Ai@]\c@lproscal\v@leur[-1,-2]%
    \ifdim\v@leur>\z@\p@rtent=\@ne\else\p@rtent=\m@ne\fi%
    \figvectP-2[#5,\Aj@]\c@lproscal\v@leur[-1,-2]%
    \ifdim\p@rtent\v@leur>\z@\figptcopy#1:#2/#3/%
    \message{*** \BS@ figptvisilimSL: points are on the same side.}\else%
    \figptcopy-3:/#3/\figptcopy-4:/#4/%
    \loop\figptbary-5:[-3,-4;1,1]\figvectP-2[#5,-5]\c@lproscal\v@leur[-1,-2]%
    \ifdim\p@rtent\v@leur>\z@\figptcopy-3:/-5/\else\figptcopy-4:/-5/\fi%
    \divide\delt@\tw@\ifdim\delt@>\epsil@n\repeat%
    \figptbary#1:#2[-3,-4;1,1]\fi\resetc@ntr@l\et@tfigptvisilimSLTD}\ignorespaces\fi}
\ctr@ld@f\def\c@ley@pt#1{\t@stp@r\ifitis@K\v@lX=\cza@\p@\v@lY=\czb@\p@\v@lZ=\czc@\p@%
    \Figv@ctCreg-1(\v@lX,\v@lY,\v@lZ)\Figp@intreg-2:(\wd\Bt@rget,\ht\Bt@rget,\dp\Bt@rget)%
    \figpttra#1:=-2/-\disob@intern,-1/\else\end\fi}
\ctr@ld@f\def\t@stp@r{\itis@Ktrue\ifnewt@rgetpt\else\itis@Kfalse%
    \message{*** \BS@ figptvisilimXX: target point undefined.}\fi\ifnewdis@b\else%
    \itis@Kfalse\message{*** \BS@ figptvisilimXX: observation distance undefined.}\fi%
    \ifitis@K\else\message{*** This macro must be called after \BS@ figdrawbegin or after
    having set the missing parameter(s) with \BS@ figset proj()}\fi}
\ctr@ld@f\def\figscan#1(#2,#3){{\s@uvc@ntr@l\et@tfigscan\@psfgetbb{#1}\if@psfbbfound\else%
    \def\@psfllx{0}\def\@psflly{20}\def\@psfurx{540}\def\@psfury{640}\fi\figscan@{#2}{#3}%
    \resetc@ntr@l\et@tfigscan}\ignorespaces}
\ctr@ld@f\def\figscan@#1#2{%
    \unit@=\@ne bp\setc@ntr@l{2}\figsetmark{}%
    \def\minst@p{20pt}%
    \v@lX=\@psfllx\p@\v@lX=\Sc@leFact\v@lX\r@undint\v@lX\v@lX%
    \v@lY=\@psflly\p@\v@lY=\Sc@leFact\v@lY\ifdim\v@lY>\z@\r@undint\v@lY\v@lY\fi%
    \delt@=\@psfury\p@\delt@=\Sc@leFact\delt@%
    \advance\delt@-\v@lY\v@lXa=\@psfurx\p@\v@lXa=\Sc@leFact\v@lXa\v@leur=\minst@p%
    \edef\valv@lY{\repdecn@mb{\v@lY}}\edef\LgTr@it{\the\delt@}%
    \loop\ifdim\v@lX<\v@lXa\edef\valv@lX{\repdecn@mb{\v@lX}}%
    \figptDD -1:(\valv@lX,\valv@lY)\figwriten -1:\hbox{\vrule height\LgTr@it}(0)%
    \ifdim\v@leur<\minst@p\else\figsetmark{\raise-8bp\hbox{$\scriptscriptstyle\triangle$}}%
    \figwrites -1:\@ffichnb{0}{\valv@lX}(6)\v@leur=\z@\figsetmark{}\fi%
    \advance\v@leur#1pt\advance\v@lX#1pt\repeat%
    \def\minst@p{10pt}%
    \v@lX=\@psfllx\p@\v@lX=\Sc@leFact\v@lX\ifdim\v@lX>\z@\r@undint\v@lX\v@lX\fi%
    \v@lY=\@psflly\p@\v@lY=\Sc@leFact\v@lY\r@undint\v@lY\v@lY%
    \delt@=\@psfurx\p@\delt@=\Sc@leFact\delt@%
    \advance\delt@-\v@lX\v@lYa=\@psfury\p@\v@lYa=\Sc@leFact\v@lYa\v@leur=\minst@p%
    \edef\valv@lX{\repdecn@mb{\v@lX}}\edef\LgTr@it{\the\delt@}%
    \loop\ifdim\v@lY<\v@lYa\edef\valv@lY{\repdecn@mb{\v@lY}}%
    \figptDD -1:(\valv@lX,\valv@lY)\figwritee -1:\vbox{\hrule width\LgTr@it}(0)%
    \ifdim\v@leur<\minst@p\else\figsetmark{$\triangleright$\kern4bp}%
    \figwritew -1:\@ffichnb{0}{\valv@lY}(6)\v@leur=\z@\figsetmark{}\fi%
    \advance\v@leur#2pt\advance\v@lY#2pt\repeat}
\ctr@ld@f\let\figscanI=\figscan
\ctr@ld@f\def\figscan@E#1(#2,#3){{\s@uvc@ntr@l\et@tfigscan@E%
    \Figdisc@rdLTS{#1}{\t@xt@}\pdfximage{\t@xt@}%
    \setbox\Gb@x=\hbox{\pdfrefximage\pdflastximage}%
    \edef\@psfllx{0}\v@lY=-\dp\Gb@x\edef\@psflly{\repdecn@mb{\v@lY}}%
    \edef\@psfurx{\repdecn@mb{\wd\Gb@x}}%
    \v@lY=\dp\Gb@x\advance\v@lY\ht\Gb@x\edef\@psfury{\repdecn@mb{\v@lY}}%
    \figscan@{#2}{#3}\resetc@ntr@l\et@tfigscan@E}\ignorespaces}
\ctr@ld@f\def\figshowpts[#1,#2]{{\figsetmark{$\bullet$}\figsetptname{\bf ##1}%
    \p@rtent=#2\relax\ifnum\p@rtent<\z@\p@rtent=\z@\fi%
    \s@mme=#1\relax\ifnum\s@mme<\z@\s@mme=\z@\fi%
    \loop\ifnum\s@mme<\p@rtent\pt@rvect{\s@mme}%
    \ifitis@K\figwriten{\the\s@mme}:(4pt)\fi\advance\s@mme\@ne\repeat%
    \pt@rvect{\s@mme}\ifitis@K\figwriten{\the\s@mme}:(4pt)\fi}\ignorespaces}
\ctr@ld@f\def\pt@rvect#1{\set@bjc@de{#1}%
    \expandafter\expandafter\expandafter\inqpt@rvec\csname\objc@de\endcsname:}
\ctr@ld@f\def\inqpt@rvec#1#2:{\if#1\C@dCl@spt\itis@Ktrue\else\itis@Kfalse\fi}
\ctr@ld@f\def\figshowsettings{{%
    \immediate\write16{====================================================================}%
    \immediate\write16{ Current settings are (DDV means "with dynamic default value"):}%
    \immediate\write16{ --- GENERAL ---}%
    \immediate\write16{Scale factor and Unit = \unit@util\space (\the\unit@)
     \space -> \BS@ figinit{ScaleFactorUnit}}%
    \immediate\write16{Update mode = \ifGRupdatem@de yes\else no\fi
     \space-> \BS@ figset(update=yes/no) or \BS@ figsetdefault(update=yes/no)}%
    \immediate\write16{ --- WRITING ---}%
    \immediate\write16{Implicit point name = \ptn@me{i} \space-> \BS@ figset write(ptname={Name})}%
    \immediate\write16{Point marker = \the\c@nsymb \space -> \BS@ figset write(mark=Mark)}%
    \immediate\write16{Print rounded coordinates = \ifr@undcoord yes\else no\fi
     \space-> \BS@ figset write(roundcoord=yes/no)}%
    \immediate\write16{ --- GRAPHICAL (general) ---}%
    \immediate\write16{Color = \CUR@color \space-> \BS@ figset(color=ColorDefinition)}%
    \immediate\write16{Filling mode = \iffillm@de yes\else no\fi
     \space-> \BS@ figset(fillmode=yes/no)}%
    \immediate\write16{Line join = \CUR@join \space-> \BS@ figset(join=miter/round/bevel)}%
    \immediate\write16{Line style = \CUR@dash \space-> \BS@ figset(dash=Index/Pattern)}%
    \immediate\write16{Line width = \CUR@width
     \space-> \BS@ figset(width=real in PostScript units)}%
    \immediate\write16{ --- GRAPHICAL (specific) ---}%
    \immediate\write16{Altitude (all the following attributes are DDV):}%
    \immediate\write16{ Base line color =
     \ifx\DDV@blcolor\D@FTref general color\else\DDV@blcolor\fi
     \space-> \BS@ figset altitude(blcolor=ColorDefinition)}%
    \immediate\write16{ Base line style =
     \ifx\DDV@bldash\D@FTref general style\else\DDV@bldash\fi
     \space-> \BS@ figset altitude(bldash=Index/Pattern)}%
    \immediate\write16{ Base line width =
     \ifx\DDV@blwidth\D@FTref general width\else\DDV@blwidth\fi
     \space-> \BS@ figset altitude(blwidth=real in PostScript units)}%
    \immediate\write16{ Square line color =
     \ifx\DDV@sqcolor\D@FTref general color\else\DDV@sqcolor\fi
     \space-> \BS@ figset altitude(sqcolor=ColorDefinition)}%
    \immediate\write16{ Square line style =
     \ifx\DDV@sqdash\D@FTref general style\else\DDV@sqdash\fi
     \space-> \BS@ figset altitude(sqdash=Index/Pattern)}%
    \immediate\write16{ Square line width =
     \ifx\DDV@sqwidth\D@FTref general width\else\DDV@sqwidth\fi
     \space-> \BS@ figset altitude(sqwidth=real in PostScript units)}%
    \immediate\write16{Arrowhead:}%
    \immediate\write16{ (half-)Angle = \@rrowheadangle
     \space-> \BS@ figset arrowhead(angle=real in degrees)}%
    \immediate\write16{ Filling mode = \if@rrowhfill yes\else no\fi
     \space-> \BS@ figset arrowhead(fillmode=yes/no)}%
    \immediate\write16{ "Outside" = \if@rrowhout yes\else no\fi
     \space-> \BS@ figset arrowhead(out=yes/no)}%
    \immediate\write16{ Length = \@rrowheadlength
     \if@rrowratio\space(not active)\else\space(active)\fi
     \space-> \BS@ figset arrowhead(length=real in user coord.)}%
    \immediate\write16{ Ratio = \@rrowheadratio
     \if@rrowratio\space(active)\else\space(not active)\fi
     \space-> \BS@ figset arrowhead(ratio=real in [0,1])}%
    \immediate\write16{Curve:}%
    \immediate\write16{ Roundness = \curv@roundness
     \space-> \BS@ figset curve(roundness=real in [0,0.5])}%
    \immediate\write16{Flow chart:}%
    \immediate\write16{ Arrow position = \@rrowp@s
     \space-> \BS@ figset flowchart(arrowposition=real in [0,1])}%
    \immediate\write16{ Arrow reference point = \ifcase\@rrowr@fpt start\else end\fi
     \space-> \BS@ figset flowchart(arrowrefpt = start/end)}%
    \immediate\write16{ Background color = \fcbgc@lor
     \space-> \BS@ figset flowchart(bgcolor=ColorDefinition)}%
    \immediate\write16{ Line type = \ifcase\fclin@typ@ curve\else polygon\fi
     \space-> \BS@ figset flowchart(line=polygon/curve)}%
    \immediate\write16{ Padding = (\Xp@dd, \Yp@dd)
     \space-> \BS@ figset flowchart(padding = real in user coord.)}%
    \immediate\write16{\space\space\space\space(or
     \BS@ figset flowchart(xpadding=real, ypadding=real) )}%
    \immediate\write16{ Radius = \fclin@r@d
     \space-> \BS@ figset flowchart(radius=positive real in user coord.)}%
    \immediate\write16{ Shape = \fcsh@pe
     \space-> \BS@ figset flowchart(shape = rectangle, ellipse or lozenge)}%
    \immediate\write16{ Thickness color (DDV) = 
     \ifx\DDV@thickcolor\D@FTref general color\else\DDV@thickcolor\fi
     \space-> \BS@ figset flowchart(thickcolor=ColorDefinition)}%
    \immediate\write16{ Thickness = \thickn@ss
     \space-> \BS@ figset flowchart(thickness = real in user coord.)}%
    \immediate\write16{Mesh:}%
    \immediate\write16{ Diagonal = \c@ntrolmesh
     \space-> \BS@ figset mesh(diag=integer in {-1,0,1})}%
    \immediate\write16{ Lines color (DDV) =
     \ifx\DDV@meshcolor\D@FTref general color\else\DDV@meshcolor\fi
     \space-> \BS@ figset mesh(color=ColorDefinition)}%
    \immediate\write16{ Lines style (DDV) =
     \ifx\DDV@meshdash\D@FTref general style\else\DDV@meshdash\fi
     \space-> \BS@ figset mesh(dash=Index/Pattern)}%
    \immediate\write16{ Lines width (DDV) =
     \ifx\DDV@meshwidth\D@FTref general width\else\DDV@meshwidth\fi
     \space-> \BS@ figset mesh(width=real in PostScript units)}%
    \immediate\write16{Trimesh:}%
    \immediate\write16{ Lines color (DDV) =
     \ifx\DDV@tmeshcolor\D@FTref general color\else\DDV@tmeshcolor\fi
     \space-> \BS@ figset trimesh(color=ColorDefinition)}%
    \immediate\write16{ Lines style (DDV) =
     \ifx\DDV@tmeshdash\D@FTref general style\else\DDV@tmeshdash\fi
     \space-> \BS@ figset trimesh(dash=Index/Pattern)}%
    \immediate\write16{ Lines width (DDV) =
     \ifx\DDV@tmeshwidth\D@FTref general width\else\DDV@tmeshwidth\fi
     \space-> \BS@ figset trimesh(width=real in PostScript units)}%
    \ifTr@isDim%
    \immediate\write16{ --- 3D to 2D PROJECTION ---}%
    \immediate\write16{Projection : \typ@proj \space-> \BS@ figinit{ScaleFactorUnit, ProjType}}%
    \immediate\write16{Longitude (psi) = \v@lPsi \space-> \BS@ figset proj(psi=real in degrees)}%
    \ifcase\CUR@proj\immediate\write16{Depth coeff. (Lambda)
     \space = \v@lTheta \space-> \BS@ figset proj(lambda=real in [0,1])}%
    \else\immediate\write16{Latitude (theta)
     \space = \v@lTheta \space-> \BS@ figset proj(theta=real in degrees)}%
    \fi%
    \ifnum\CUR@proj=\tw@%
    \immediate\write16{Observation distance = \disob@unit
     \space-> \BS@ figset proj(dist=real in user coord.)}%
    \immediate\write16{Target point = \t@rgetpt \space-> \BS@ figset proj(targetpt=pt number)}%
     \v@lX=\ptT@unit@\wd\Bt@rget\v@lY=\ptT@unit@\ht\Bt@rget\v@lZ=\ptT@unit@\dp\Bt@rget%
    \immediate\write16{ Its coordinates are
     (\repdecn@mb{\v@lX}, \repdecn@mb{\v@lY}, \repdecn@mb{\v@lZ})}%
    \fi%
    \fi%
    \immediate\write16{====================================================================}%
    \ignorespaces}}
\ctr@ln@w{newif}\ifitis@vect@r
\ctr@ld@f\def\figvectC#1(#2,#3){{\itis@vect@rtrue\figpt#1:(#2,#3)}\ignorespaces}
\ctr@ld@f\def\Figv@ctCreg#1(#2,#3){{\itis@vect@rtrue\Figp@intreg#1:(#2,#3)}\ignorespaces}
\ctr@ln@m\figvectDBezier
\ctr@ld@f\def\figvectDBezierDD#1:#2,#3[#4,#5,#6,#7]{\ifGR@cri{\s@uvc@ntr@l\et@tfigvectDBezierDD%
    \FigvectDBezier@#2,#3[#4,#5,#6,#7]\v@lX=\c@ef\v@lX\v@lY=\c@ef\v@lY%
    \Figv@ctCreg#1(\v@lX,\v@lY)\resetc@ntr@l\et@tfigvectDBezierDD}\ignorespaces\fi}
\ctr@ld@f\def\figvectDBezierTD#1:#2,#3[#4,#5,#6,#7]{\ifGR@cri{\s@uvc@ntr@l\et@tfigvectDBezierTD%
    \FigvectDBezier@#2,#3[#4,#5,#6,#7]\v@lX=\c@ef\v@lX\v@lY=\c@ef\v@lY\v@lZ=\c@ef\v@lZ%
    \Figv@ctCreg#1(\v@lX,\v@lY,\v@lZ)\resetc@ntr@l\et@tfigvectDBezierTD}\ignorespaces\fi}
\ctr@ld@f\def\FigvectDBezier@#1,#2[#3,#4,#5,#6]{\setc@ntr@l{2}%
    \edef\T@{#2}\v@leur=\p@\advance\v@leur-#2pt\edef\UNmT@{\repdecn@mb{\v@leur}}%
    \ifnum#1=\tw@\def\c@ef{6}\else\def\c@ef{3}\fi%
    \figptcopy-4:/#3/\figptcopy-3:/#4/\figptcopy-2:/#5/\figptcopy-1:/#6/%
    \l@mbd@un=-4 \l@mbd@de=-\thr@@\p@rtent=\m@ne\c@lDecast%
    \ifnum#1=\tw@\c@lDCDeux{-4}{-3}\c@lDCDeux{-3}{-2}\c@lDCDeux{-4}{-3}\else%
    \l@mbd@un=-4 \l@mbd@de=-\thr@@\p@rtent=-\tw@\c@lDecast%
    \c@lDCDeux{-4}{-3}\fi\Figg@tXY{-4}}
\ctr@ln@m\c@lDCDeux
\ctr@ld@f\def\c@lDCDeuxDD#1#2{\Figg@tXY{#2}\Figg@tXYa{#1}%
    \advance\v@lX-\v@lXa\advance\v@lY-\v@lYa\Figp@intregDD#1:(\v@lX,\v@lY)}
\ctr@ld@f\def\c@lDCDeuxTD#1#2{\Figg@tXY{#2}\Figg@tXYa{#1}\advance\v@lX-\v@lXa%
    \advance\v@lY-\v@lYa\advance\v@lZ-\v@lZa\Figp@intregTD#1:(\v@lX,\v@lY,\v@lZ)}
\ctr@ln@m\figvectN
\ctr@ld@f\def\figvectNDD#1[#2,#3]{\ifGR@cri{\Figg@tXYa{#2}\Figg@tXY{#3}%
    \advance\v@lX-\v@lXa\advance\v@lY-\v@lYa%
    \Figv@ctCreg#1(-\v@lY,\v@lX)}\ignorespaces\fi}
\ctr@ld@f\def\figvectNTD#1[#2,#3,#4]{\ifGR@cri{\vecunitC@TD[#2,#4]\v@lmin=\v@lX\v@lmax=\v@lY%
    \v@leur=\v@lZ\vecunitC@TD[#2,#3]\c@lprovec{#1}}\ignorespaces\fi}
\ctr@ln@m\figvectNV
\ctr@ld@f\def\figvectNVDD#1[#2]{\ifGR@cri{\Figg@tXY{#2}\Figv@ctCreg#1(-\v@lY,\v@lX)}\ignorespaces\fi}
\ctr@ld@f\def\figvectNVTD#1[#2,#3]{\ifGR@cri{\vecunitCV@TD{#3}\v@lmin=\v@lX\v@lmax=\v@lY%
    \v@leur=\v@lZ\vecunitCV@TD{#2}\c@lprovec{#1}}\ignorespaces\fi}
\ctr@ln@m\figvectP
\ctr@ld@f\def\figvectPDD#1[#2,#3]{\ifGR@cri{\Figg@tXYa{#2}\Figg@tXY{#3}%
    \advance\v@lX-\v@lXa\advance\v@lY-\v@lYa%
    \Figv@ctCreg#1(\v@lX,\v@lY)}\ignorespaces\fi}
\ctr@ld@f\def\figvectPTD#1[#2,#3]{\ifGR@cri{\Figg@tXYa{#2}\Figg@tXY{#3}%
    \advance\v@lX-\v@lXa\advance\v@lY-\v@lYa\advance\v@lZ-\v@lZa%
    \Figv@ctCreg#1(\v@lX,\v@lY,\v@lZ)}\ignorespaces\fi}
\ctr@ln@m\figvectU
\ctr@ld@f\def\figvectUDD#1[#2]{\ifGR@cri{\n@rmeuc\v@leur{#2}\invers@\v@leur\v@leur%
    \delt@=\repdecn@mb{\v@leur}\unit@\edef\v@ldelt@{\repdecn@mb{\delt@}}%
    \Figg@tXY{#2}\v@lX=\v@ldelt@\v@lX\v@lY=\v@ldelt@\v@lY%
    \Figv@ctCreg#1(\v@lX,\v@lY)}\ignorespaces\fi}
\ctr@ld@f\def\figvectUTD#1[#2]{\ifGR@cri{\n@rmeuc\v@leur{#2}\invers@\v@leur\v@leur%
    \delt@=\repdecn@mb{\v@leur}\unit@\edef\v@ldelt@{\repdecn@mb{\delt@}}%
    \Figg@tXY{#2}\v@lX=\v@ldelt@\v@lX\v@lY=\v@ldelt@\v@lY\v@lZ=\v@ldelt@\v@lZ%
    \Figv@ctCreg#1(\v@lX,\v@lY,\v@lZ)}\ignorespaces\fi}
\ctr@ld@f\def\figvisu#1#2#3{\c@ldefproj\initb@undb@x\xdef\figforTeXFigno{\figforTeXnextFigno}%
    \s@mme=\figforTeXnextFigno\advance\s@mme\@ne\xdef\figforTeXnextFigno{\number\s@mme}%
    \setbox\b@xvisu=\hbox{\ifnum\@utoFN>\z@\figinsert{}\gdef\@utoFInDone{0}\fi\ignorespaces#3}%
    \gdef\@utoFInDone{1}\gdef\@utoFN{0}%
    \v@lXa=-\c@@rdYmin\v@lYa=\c@@rdYmax\advance\v@lYa-\c@@rdYmin%
    \v@lX=\c@@rdXmax\advance\v@lX-\c@@rdXmin%
    \setbox#1=\hbox{#2}\v@lY=-\v@lX\maxim@m{\v@lX}{\v@lX}{\wd#1}%
    \advance\v@lY\v@lX\divide\v@lY\tw@\advance\v@lY-\c@@rdXmin%
    \setbox#1=\vbox{\parindent\z@\hsize=\v@lX\vskip\v@lYa%
    \rlap{\hskip\v@lY\smash{\raise\v@lXa\box\b@xvisu}}%
    \def\t@xt@{#2}\ifx\t@xt@\empty\else\medskip\centerline{#2}\fi}\wd#1=\v@lX}
\ctr@ld@f\def\figDecrementFigno{{\xdef\figforTeXnextFigno{\figforTeXFigno}%
    \s@mme=\figforTeXFigno\advance\s@mme\m@ne\xdef\figforTeXFigno{\number\s@mme}}}
\ctr@ln@w{newbox}\Bt@rget\setbox\Bt@rget=\null
\ctr@ln@w{newbox}\BminTD@\setbox\BminTD@=\null
\ctr@ln@w{newbox}\BmaxTD@\setbox\BmaxTD@=\null
\ctr@ln@w{newif}\ifnewt@rgetpt\ctr@ln@w{newif}\ifnewdis@b
\ctr@ld@f\def\b@undb@xTD#1#2#3{%
    \relax\ifdim#1<\wd\BminTD@\global\wd\BminTD@=#1\fi%
    \relax\ifdim#2<\ht\BminTD@\global\ht\BminTD@=#2\fi%
    \relax\ifdim#3<\dp\BminTD@\global\dp\BminTD@=#3\fi%
    \relax\ifdim#1>\wd\BmaxTD@\global\wd\BmaxTD@=#1\fi%
    \relax\ifdim#2>\ht\BmaxTD@\global\ht\BmaxTD@=#2\fi%
    \relax\ifdim#3>\dp\BmaxTD@\global\dp\BmaxTD@=#3\fi}
\ctr@ld@f\def\c@ldefdisob{{\ifdim\wd\BminTD@<\maxdimen\v@leur=\wd\BmaxTD@\advance\v@leur-\wd\BminTD@%
    \delt@=\ht\BmaxTD@\advance\delt@-\ht\BminTD@\maxim@m{\v@leur}{\v@leur}{\delt@}%
    \delt@=\dp\BmaxTD@\advance\delt@-\dp\BminTD@\maxim@m{\v@leur}{\v@leur}{\delt@}%
    \v@leur=5\v@leur\else\v@leur=800pt\fi\c@ldefdisob@{\v@leur}}}
\ctr@ln@m\disob@intern
\ctr@ln@m\disob@
\ctr@ln@m\divf@ctproj
\ctr@ld@f\def\c@ldefdisob@#1{{\v@leur=#1\ifdim\v@leur<\p@\v@leur=800pt\fi%
    \xdef\disob@intern{\repdecn@mb{\v@leur}}%
    \delt@=\ptT@unit@\v@leur\xdef\disob@unit{\repdecn@mb{\delt@}}%
    \f@ctech=\@ne\loop\ifdim\v@leur>\t@n pt\divide\v@leur\t@n\multiply\f@ctech\t@n\repeat%
    \xdef\disob@{\repdecn@mb{\v@leur}}\xdef\divf@ctproj{\the\f@ctech}}%
    \global\newdis@btrue}
\ctr@ln@m\t@rgetpt
\ctr@ld@f\def\c@ldeft@rgetpt{\newt@rgetpttrue\def\t@rgetpt{CenterBoundBox}{%
    \delt@=\wd\BmaxTD@\advance\delt@-\wd\BminTD@\divide\delt@\tw@%
    \v@leur=\wd\BminTD@\advance\v@leur\delt@\global\wd\Bt@rget=\v@leur%
    \delt@=\ht\BmaxTD@\advance\delt@-\ht\BminTD@\divide\delt@\tw@%
    \v@leur=\ht\BminTD@\advance\v@leur\delt@\global\ht\Bt@rget=\v@leur%
    \delt@=\dp\BmaxTD@\advance\delt@-\dp\BminTD@\divide\delt@\tw@%
    \v@leur=\dp\BminTD@\advance\v@leur\delt@\global\dp\Bt@rget=\v@leur}}
\ctr@ln@m\c@ldefproj
\ctr@ld@f\def\c@ldefprojTD{\ifnewt@rgetpt\else\c@ldeft@rgetpt\fi\ifnewdis@b\else\c@ldefdisob\fi}
\ctr@ld@f\def\c@lprojcav{
    \v@lZa=\cxa@\v@lY\advance\v@lX\v@lZa%
    \v@lZa=\cxb@\v@lY\v@lY=\v@lZ\advance\v@lY\v@lZa\ignorespaces}
\ctr@ln@m\v@lcoef
\ctr@ld@f\def\c@lprojrea{
    \advance\v@lX-\wd\Bt@rget\advance\v@lY-\ht\Bt@rget\advance\v@lZ-\dp\Bt@rget%
    \v@lZa=\cza@\v@lX\advance\v@lZa\czb@\v@lY\advance\v@lZa\czc@\v@lZ%
    \divide\v@lZa\divf@ctproj\advance\v@lZa\disob@ pt\invers@{\v@lZa}{\v@lZa}%
    \v@lZa=\disob@\v@lZa\edef\v@lcoef{\repdecn@mb{\v@lZa}}%
    \v@lXa=\cxa@\v@lX\advance\v@lXa\cxb@\v@lY\v@lXa=\v@lcoef\v@lXa%
    \v@lY=\cyb@\v@lY\advance\v@lY\cya@\v@lX\advance\v@lY\cyc@\v@lZ%
    \v@lY=\v@lcoef\v@lY\v@lX=\v@lXa\ignorespaces}
\ctr@ld@f\def\c@lprojort{
    \v@lXa=\cxa@\v@lX\advance\v@lXa\cxb@\v@lY%
    \v@lY=\cyb@\v@lY\advance\v@lY\cya@\v@lX\advance\v@lY\cyc@\v@lZ%
    \v@lX=\v@lXa\ignorespaces}
\ctr@ld@f\def\Figptpr@j#1:#2/#3/{{\Figg@tXY{#3}\superc@lprojSP%
    \Figp@intregDD#1:{#2}(\v@lX,\v@lY)}\ignorespaces}
\ctr@ln@m\figsetobdist
\ctr@ld@f\def\figsetobdistDD{\un@v@ilable{figsetobdist}}
\ctr@ld@f\def\figsetobdistTD(#1){{\ifCUR@PS\W@rnmesIgn{figset proj(dist=...)}%
    \else\v@leur=#1\unit@\c@ldefdisob@{\v@leur}\fi}\ignorespaces}
\ctr@ln@m\c@lprojSP
\ctr@ln@m\CUR@proj
\ctr@ln@m\typ@proj
\ctr@ln@m\superc@lprojSP
\ctr@ld@f\def\Figs@tproj#1{%
    \if#13 \def@ultproj\else\if#1c\def@ultproj%
    \else\if#1o\xdef\CUR@proj{1}\xdef\typ@proj{orthogonal}%
         \figsetviewTD(\def@ultpsi,\def@ulttheta)%
         \global\let\c@lprojSP=\c@lprojort\global\let\superc@lprojSP=\c@lprojort%
    \else\if#1r\xdef\CUR@proj{2}\xdef\typ@proj{realistic}%
         \figsetviewTD(\def@ultpsi,\def@ulttheta)%
         \global\let\c@lprojSP=\c@lprojrea\global\let\superc@lprojSP=\c@lprojrea%
    \else\def@ultproj\message{*** Unknown projection. Cavalier projection assumed.}%
    \fi\fi\fi\fi}
\ctr@ld@f\def\def@ultproj{\xdef\CUR@proj{0}\xdef\typ@proj{cavalier}\figsetviewTD(\def@ultpsi,0.5)%
         \global\let\c@lprojSP=\c@lprojcav\global\let\superc@lprojSP=\c@lprojcav}
\ctr@ln@m\figsettarget
\ctr@ld@f\def\figsettargetDD{\un@v@ilable{figsettarget}}
\ctr@ld@f\def\figsettargetTD[#1]{{\ifCUR@PS\W@rnmesIgn{figset proj(targetpt=...)}%
    \else\global\newt@rgetpttrue\xdef\t@rgetpt{#1}\Figg@tXY{#1}\global\wd\Bt@rget=\v@lX%
    \global\ht\Bt@rget=\v@lY\global\dp\Bt@rget=\v@lZ\fi}\ignorespaces}
\ctr@ln@m\figsetview
\ctr@ld@f\def\figsetviewDD{\un@v@ilable{figsetview}}
\ctr@ld@f\def\figsetviewTD(#1){\ifCUR@PS\W@rnmesIgn{figset proj(Psi|Theta|Lambda=...)}%
     \else\Figsetview@#1,:\fi\ignorespaces}
\ctr@ld@f\def\Figsetview@#1,#2:{{\xdef\v@lPsi{#1}\def\t@xt@{#2}%
    \ifx\t@xt@\empty\def\@rgdeux{\v@lTheta}\else\X@rgdeux@#2\fi%
    \c@ssin{\costhet@}{\sinthet@}{#1}\v@lmin=\costhet@ pt\v@lmax=\sinthet@ pt%
    \ifcase\CUR@proj%
    \v@leur=\@rgdeux\v@lmin\xdef\cxa@{\repdecn@mb{\v@leur}}%
    \v@leur=\@rgdeux\v@lmax\xdef\cxb@{\repdecn@mb{\v@leur}}\v@leur=\@rgdeux pt%
    \relax\ifdim\v@leur>\p@\message{*** Lambda too large ! See \BS@ figset proj() !}\fi%
    \else%
    \v@lmax=-\v@lmax\xdef\cxa@{\repdecn@mb{\v@lmax}}\xdef\cxb@{\costhet@}%
    \ifx\t@xt@\empty\edef\@rgdeux{\def@ulttheta}\fi\c@ssin{\C@}{\S@}{\@rgdeux}%
    \v@lmax=-\S@ pt%
    \v@leur=\v@lmax\v@leur=\costhet@\v@leur\xdef\cya@{\repdecn@mb{\v@leur}}%
    \v@leur=\v@lmax\v@leur=\sinthet@\v@leur\xdef\cyb@{\repdecn@mb{\v@leur}}%
    \xdef\cyc@{\C@}\v@lmin=-\C@ pt%
    \v@leur=\v@lmin\v@leur=\costhet@\v@leur\xdef\cza@{\repdecn@mb{\v@leur}}%
    \v@leur=\v@lmin\v@leur=\sinthet@\v@leur\xdef\czb@{\repdecn@mb{\v@leur}}%
    \xdef\czc@{\repdecn@mb{\v@lmax}}\fi%
    \xdef\v@lTheta{\@rgdeux}}}
\ctr@ld@f\def\def@ultpsi{40}
\ctr@ld@f\def\def@ulttheta{25}
\ctr@ln@m\l@debut
\ctr@ln@m\n@mref
\ctr@ld@f\def\Figsetpr@j#1=#2|{\keln@mtr#1|%
    \def\n@mref{dep}\ifx\l@debut\n@mref\Figsetd@p{#2}\else
    \def\n@mref{dis}\ifx\l@debut\n@mref%
     \ifnum\CUR@proj=\tw@\figsetobdist(#2)\else\Figset@rr\fi\else
    \def\n@mref{lam}\ifx\l@debut\n@mref\Figsetd@p{#2}\else
    \def\n@mref{lat}\ifx\l@debut\n@mref\Figsetth@{#2}\else
    \def\n@mref{lon}\ifx\l@debut\n@mref\figsetview(#2)\else
    \def\n@mref{psi}\ifx\l@debut\n@mref\figsetview(#2)\else
    \def\n@mref{tar}\ifx\l@debut\n@mref%
     \ifnum\CUR@proj=\tw@\figsettarget[#2]\else\Figset@rr\fi\else
    \def\n@mref{the}\ifx\l@debut\n@mref\Figsetth@{#2}\else
    \W@rnmesAttr{figset proj}{#1}\fi\fi\fi\fi\fi\fi\fi\fi}
\ctr@ld@f\def\Figsetd@p#1{\ifnum\CUR@proj=\z@\figsetview(\v@lPsi,#1)\else\Figset@rr\fi}
\ctr@ld@f\def\Figsetth@#1{\ifnum\CUR@proj=\z@\Figset@rr\else\figsetview(\v@lPsi,#1)\fi}
\ctr@ld@f\def\Figset@rr{\message{*** \BS@ figset proj(): Attribute "\n@mref" ignored, incompatible
    with current projection}}
\ctr@ld@f\def\initb@undb@xTD{\wd\BminTD@=\maxdimen\ht\BminTD@=\maxdimen\dp\BminTD@=\maxdimen%
    \wd\BmaxTD@=-\maxdimen\ht\BmaxTD@=-\maxdimen\dp\BmaxTD@=-\maxdimen}
\ctr@ln@w{newbox}\Gb@x      
\ctr@ln@w{newbox}\Gb@xSC    
\ctr@ln@w{newtoks}\c@nsymb  
\ctr@ln@w{newif}\ifr@undcoord\ctr@ln@w{newif}\ifunitpr@sent
\ctr@ld@f\def\unssqrttw@{0.707106 }
\ctr@ld@f\def\figAst{\raise-1.15ex\hbox{$\ast$}}
\ctr@ld@f\def\figBullet{\raise-1.15ex\hbox{$\bullet$}}
\ctr@ld@f\def\figCirc{\raise-1.15ex\hbox{$\circ$}}
\ctr@ld@f\def\figDiamond{\raise-1.15ex\hbox{$\diamond$}}%
\ctr@ld@f\def\boxit#1#2{\leavevmode\hbox{\vrule\vbox{\hrule\vglue#1%
    \vtop{\hbox{\kern#1{#2}\kern#1}\vglue#1\hrule}}\vrule}}
\ctr@ld@f\def\centertext#1#2{\vbox{\hsize#1\parindent0cm%
    \leftskip=0pt plus 1fil\rightskip=0pt plus 1fil\parfillskip=0pt{#2}}}
\ctr@ld@f\def\lefttext#1#2{\vbox{\hsize#1\parindent0cm\rightskip=0pt plus 1fil#2}}
\ctr@ld@f\def\c@nterpt{\ignorespaces%
    \kern-.5\wd\Gb@xSC%
    \raise-.5\ht\Gb@xSC\rlap{\hbox{\raise.5\dp\Gb@xSC\hbox{\copy\Gb@xSC}}}%
    \kern .5\wd\Gb@xSC\ignorespaces}
\ctr@ld@f\def\b@undb@xSC#1#2{{\v@lXa=#1\v@lYa=#2%
    \v@leur=\ht\Gb@xSC\advance\v@leur\dp\Gb@xSC%
    \advance\v@lXa-.5\wd\Gb@xSC\advance\v@lYa-.5\v@leur\b@undb@x{\v@lXa}{\v@lYa}%
    \advance\v@lXa\wd\Gb@xSC\advance\v@lYa\v@leur\b@undb@x{\v@lXa}{\v@lYa}}}
\ctr@ln@m\Dist@n
\ctr@ln@m\l@suite
\ctr@ld@f\def\@keldist#1#2{\edef\Dist@n{#2}\y@tiunit{\Dist@n}%
    \ifunitpr@sent#1=\Dist@n\else#1=\Dist@n\unit@\fi}
\ctr@ld@f\def\y@tiunit#1{\unitpr@sentfalse\expandafter\y@tiunit@#1:}
\ctr@ld@f\def\y@tiunit@#1#2:{\ifcat#1a\unitpr@senttrue\else\def\l@suite{#2}%
    \ifx\l@suite\empty\else\y@tiunit@#2:\fi\fi}
\ctr@ln@m\figcoord
\ctr@ld@f\def\figcoordDD#1{{\v@lX=\ptT@unit@\v@lX\v@lY=\ptT@unit@\v@lY%
    \ifr@undcoord\ifcase#1\v@leur=0.5pt\or\v@leur=0.05pt\or\v@leur=0.005pt%
    \or\v@leur=0.0005pt\else\v@leur=\z@\fi%
    \ifdim\v@lX<\z@\advance\v@lX-\v@leur\else\advance\v@lX\v@leur\fi%
    \ifdim\v@lY<\z@\advance\v@lY-\v@leur\else\advance\v@lY\v@leur\fi\fi%
    (\@ffichnb{#1}{\repdecn@mb{\v@lX}},\ifmmode\else\thinspace\fi%
    \@ffichnb{#1}{\repdecn@mb{\v@lY}})}}
\ctr@ld@f\def\@ffichnb#1#2{{\def\@@ffich{\@ffich#1(}\edef\n@mbre{#2}%
    \expandafter\@@ffich\n@mbre)}}
\ctr@ld@f\def\@ffich#1(#2.#3){{#2\ifnum#1>\z@.\fi\def\dig@ts{#3}\s@mme=\z@%
    \loop\ifnum\s@mme<#1\expandafter\@ffichdec\dig@ts:\advance\s@mme\@ne\repeat}}
\ctr@ld@f\def\@ffichdec#1#2:{\relax#1\def\dig@ts{#20}}
\ctr@ld@f\def\figcoordTD#1{{\v@lX=\ptT@unit@\v@lX\v@lY=\ptT@unit@\v@lY\v@lZ=\ptT@unit@\v@lZ%
    \ifr@undcoord\ifcase#1\v@leur=0.5pt\or\v@leur=0.05pt\or\v@leur=0.005pt%
    \or\v@leur=0.0005pt\else\v@leur=\z@\fi%
    \ifdim\v@lX<\z@\advance\v@lX-\v@leur\else\advance\v@lX\v@leur\fi%
    \ifdim\v@lY<\z@\advance\v@lY-\v@leur\else\advance\v@lY\v@leur\fi%
    \ifdim\v@lZ<\z@\advance\v@lZ-\v@leur\else\advance\v@lZ\v@leur\fi\fi%
    (\@ffichnb{#1}{\repdecn@mb{\v@lX}},\ifmmode\else\thinspace\fi%
     \@ffichnb{#1}{\repdecn@mb{\v@lY}},\ifmmode\else\thinspace\fi%
     \@ffichnb{#1}{\repdecn@mb{\v@lZ}})}}
\ctr@ld@f\def\figsetroundcoord#1{\expandafter\Figsetr@undcoord#1:\ignorespaces}
\ctr@ld@f\def\Figsetr@undcoord#1#2:{\if#1n\r@undcoordfalse\else\r@undcoordtrue\fi}
\ctr@ld@f\def\Figsetwr@te#1=#2|{\keln@mun#1|%
    \def\n@mref{m}\ifx\l@debut\n@mref\figsetmark{#2}\else
    \def\n@mref{p}\ifx\l@debut\n@mref\figsetptname{#2}\else
    \def\n@mref{r}\ifx\l@debut\n@mref\figsetroundcoord{#2}\else
    \W@rnmesAttr{figset write}{#1}\fi\fi\fi}
\ctr@ld@f\def\figsetmark#1{\c@nsymb={#1}\setbox\Gb@xSC=\hbox{\the\c@nsymb}\ignorespaces}
\ctr@ln@m\ptn@me
\ctr@ld@f\def\figsetptname#1{\def\ptn@me##1{#1}\ignorespaces}
\ctr@ld@f\def\FigWrit@L#1:#2(#3,#4){\ignorespaces\@keldist\v@leur{#3}\@keldist\delt@{#4}%
    \C@rp@r@m\def\list@num{#1}\@ecfor\p@int:=\list@num\do{\FigWrit@pt{\p@int}{#2}}}
\ctr@ld@f\def\FigWrit@pt#1#2{\FigWp@r@m{#1}{#2}\Vc@rrect\figWp@si%
    \ifdim\wd\Gb@xSC>\z@\b@undb@xSC{\v@lX}{\v@lY}\fi\figWBB@x}
\ctr@ld@f\def\FigWp@r@m#1#2{\Figg@tXY{#1}%
    \setbox\Gb@x=\hbox{\def\t@xt@{#2}\ifx\t@xt@\empty\Figg@tT{#1}\else#2\fi}\c@lprojSP}
\ctr@ld@f\let\Vc@rrect=\relax
\ctr@ld@f\let\C@rp@r@m=\relax
\ctr@ld@f\def\figwrite[#1]#2{{\ignorespaces\def\list@num{#1}\@ecfor\p@int:=\list@num\do{%
    \setbox\Gb@x=\hbox{\def\t@xt@{#2}\ifx\t@xt@\empty\Figg@tT{\p@int}\else#2\fi}%
    \Figwrit@{\p@int}}}\ignorespaces}
\ctr@ld@f\def\Figwrit@#1{\Figg@tXY{#1}\c@lprojSP%
    \rlap{\kern\v@lX\raise\v@lY\hbox{\unhcopy\Gb@x}}\v@leur=\v@lY%
    \advance\v@lY\ht\Gb@x\b@undb@x{\v@lX}{\v@lY}\advance\v@lX\wd\Gb@x%
    \v@lY=\v@leur\advance\v@lY-\dp\Gb@x\b@undb@x{\v@lX}{\v@lY}}
\ctr@ld@f\def\figwritec[#1]#2{{\ignorespaces\def\list@num{#1}%
    \@ecfor\p@int:=\list@num\do{\Figwrit@c{\p@int}{#2}}}\ignorespaces}
\ctr@ld@f\def\Figwrit@c#1#2{\FigWp@r@m{#1}{#2}%
    \rlap{\kern\v@lX\raise\v@lY\hbox{\rlap{\kern-.5\wd\Gb@x%
    \raise-.5\ht\Gb@x\hbox{\raise.5\dp\Gb@x\hbox{\unhcopy\Gb@x}}}}}%
    \v@leur=\ht\Gb@x\advance\v@leur\dp\Gb@x%
    \advance\v@lX-.5\wd\Gb@x\advance\v@lY-.5\v@leur\b@undb@x{\v@lX}{\v@lY}%
    \advance\v@lX\wd\Gb@x\advance\v@lY\v@leur\b@undb@x{\v@lX}{\v@lY}}
\ctr@ld@f\def\figwritep[#1]{{\ignorespaces\def\list@num{#1}\setbox\Gb@x=\hbox{\c@nterpt}%
    \@ecfor\p@int:=\list@num\do{\Figwrit@{\p@int}}}\ignorespaces}
\ctr@ld@f\def\figwritew#1:#2(#3){\figwritegcw#1:{#2}(#3,0pt)}
\ctr@ld@f\def\figwritee#1:#2(#3){\figwritegce#1:{#2}(#3,0pt)}
\ctr@ld@f\def\figwriten#1:#2(#3){{\def\Vc@rrect{\v@lZ=\v@leur\advance\v@lZ\dp\Gb@x}%
    \Figwrit@NS#1:{#2}(#3)}\ignorespaces}
\ctr@ld@f\def\figwrites#1:#2(#3){{\def\Vc@rrect{\v@lZ=-\v@leur\advance\v@lZ-\ht\Gb@x}%
    \Figwrit@NS#1:{#2}(#3)}\ignorespaces}
\ctr@ld@f\def\Figwrit@NS#1:#2(#3){\let\figWp@si=\FigWp@siNS\let\figWBB@x=\FigWBB@xNS%
    \FigWrit@L#1:{#2}(#3,0pt)}
\ctr@ld@f\def\FigWp@siNS{\rlap{\kern\v@lX\raise\v@lY\hbox{\rlap{\kern-.5\wd\Gb@x%
    \raise\v@lZ\hbox{\unhcopy\Gb@x}}\c@nterpt}}}
\ctr@ld@f\def\FigWBB@xNS{\advance\v@lY\v@lZ%
    \advance\v@lY-\dp\Gb@x\advance\v@lX-.5\wd\Gb@x\b@undb@x{\v@lX}{\v@lY}%
    \advance\v@lY\ht\Gb@x\advance\v@lY\dp\Gb@x%
    \advance\v@lX\wd\Gb@x\b@undb@x{\v@lX}{\v@lY}}
\ctr@ld@f\def\figwritenw#1:#2(#3){{\let\figWp@si=\FigWp@sigW\let\figWBB@x=\FigWBB@xgWE%
    \def\C@rp@r@m{\v@leur=\unssqrttw@\v@leur\delt@=\v@leur%
    \ifdim\delt@=\z@\delt@=\epsil@n\fi}\let@xte={-}\FigWrit@L#1:{#2}(#3,0pt)}\ignorespaces}
\ctr@ld@f\def\figwritesw#1:#2(#3){{\let\figWp@si=\FigWp@sigW\let\figWBB@x=\FigWBB@xgWE%
    \def\C@rp@r@m{\v@leur=\unssqrttw@\v@leur\delt@=-\v@leur%
    \ifdim\delt@=\z@\delt@=-\epsil@n\fi}\let@xte={-}\FigWrit@L#1:{#2}(#3,0pt)}\ignorespaces}
\ctr@ld@f\def\figwritene#1:#2(#3){{\let\figWp@si=\FigWp@sigE\let\figWBB@x=\FigWBB@xgWE%
    \def\C@rp@r@m{\v@leur=\unssqrttw@\v@leur\delt@=\v@leur%
    \ifdim\delt@=\z@\delt@=\epsil@n\fi}\let@xte={}\FigWrit@L#1:{#2}(#3,0pt)}\ignorespaces}
\ctr@ld@f\def\figwritese#1:#2(#3){{\let\figWp@si=\FigWp@sigE\let\figWBB@x=\FigWBB@xgWE%
    \def\C@rp@r@m{\v@leur=\unssqrttw@\v@leur\delt@=-\v@leur%
    \ifdim\delt@=\z@\delt@=-\epsil@n\fi}\let@xte={}\FigWrit@L#1:{#2}(#3,0pt)}\ignorespaces}
\ctr@ld@f\def\figwritegw#1:#2(#3,#4){{\let\figWp@si=\FigWp@sigW\let\figWBB@x=\FigWBB@xgWE%
    \let@xte={-}\FigWrit@L#1:{#2}(#3,#4)}\ignorespaces}
\ctr@ld@f\def\figwritege#1:#2(#3,#4){{\let\figWp@si=\FigWp@sigE\let\figWBB@x=\FigWBB@xgWE%
    \let@xte={}\FigWrit@L#1:{#2}(#3,#4)}\ignorespaces}
\ctr@ld@f\def\FigWp@sigW{\v@lXa=\z@\v@lYa=\ht\Gb@x\advance\v@lYa\dp\Gb@x%
    \ifdim\delt@>\z@\relax%
    \rlap{\kern\v@lX\raise\v@lY\hbox{\rlap{\kern-\wd\Gb@x\kern-\v@leur%
          \raise\delt@\hbox{\raise\dp\Gb@x\hbox{\unhcopy\Gb@x}}}\c@nterpt}}%
    \else\ifdim\delt@<\z@\relax\v@lYa=-\v@lYa%
    \rlap{\kern\v@lX\raise\v@lY\hbox{\rlap{\kern-\wd\Gb@x\kern-\v@leur%
          \raise\delt@\hbox{\raise-\ht\Gb@x\hbox{\unhcopy\Gb@x}}}\c@nterpt}}%
    \else\v@lXa=-.5\v@lYa%
    \rlap{\kern\v@lX\raise\v@lY\hbox{\rlap{\kern-\wd\Gb@x\kern-\v@leur%
          \raise-.5\ht\Gb@x\hbox{\raise.5\dp\Gb@x\hbox{\unhcopy\Gb@x}}}\c@nterpt}}%
    \fi\fi}
\ctr@ld@f\def\FigWp@sigE{\v@lXa=\z@\v@lYa=\ht\Gb@x\advance\v@lYa\dp\Gb@x%
    \ifdim\delt@>\z@\relax%
    \rlap{\kern\v@lX\raise\v@lY\hbox{\c@nterpt\kern\v@leur%
          \raise\delt@\hbox{\raise\dp\Gb@x\hbox{\unhcopy\Gb@x}}}}%
    \else\ifdim\delt@<\z@\relax\v@lYa=-\v@lYa%
    \rlap{\kern\v@lX\raise\v@lY\hbox{\c@nterpt\kern\v@leur%
          \raise\delt@\hbox{\raise-\ht\Gb@x\hbox{\unhcopy\Gb@x}}}}%
    \else\v@lXa=-.5\v@lYa%
    \rlap{\kern\v@lX\raise\v@lY\hbox{\c@nterpt\kern\v@leur%
          \raise-.5\ht\Gb@x\hbox{\raise.5\dp\Gb@x\hbox{\unhcopy\Gb@x}}}}%
    \fi\fi}
\ctr@ld@f\def\FigWBB@xgWE{\advance\v@lY\delt@%
    \advance\v@lX\the\let@xte\v@leur\advance\v@lY\v@lXa\b@undb@x{\v@lX}{\v@lY}%
    \advance\v@lX\the\let@xte\wd\Gb@x\advance\v@lY\v@lYa\b@undb@x{\v@lX}{\v@lY}}
\ctr@ld@f\def\figwritegcw#1:#2(#3,#4){{\let\figWp@si=\FigWp@sigcW\let\figWBB@x=\FigWBB@xgcWE%
    \let@xte={-}\FigWrit@L#1:{#2}(#3,#4)}\ignorespaces}
\ctr@ld@f\def\figwritegce#1:#2(#3,#4){{\let\figWp@si=\FigWp@sigcE\let\figWBB@x=\FigWBB@xgcWE%
    \let@xte={}\FigWrit@L#1:{#2}(#3,#4)}\ignorespaces}
\ctr@ld@f\def\FigWp@sigcW{\rlap{\kern\v@lX\raise\v@lY\hbox{\rlap{\kern-\wd\Gb@x\kern-\v@leur%
     \raise-.5\ht\Gb@x\hbox{\raise\delt@\hbox{\raise.5\dp\Gb@x\hbox{\unhcopy\Gb@x}}}}%
     \c@nterpt}}}
\ctr@ld@f\def\FigWp@sigcE{\rlap{\kern\v@lX\raise\v@lY\hbox{\c@nterpt\kern\v@leur%
    \raise-.5\ht\Gb@x\hbox{\raise\delt@\hbox{\raise.5\dp\Gb@x\hbox{\unhcopy\Gb@x}}}}}}
\ctr@ld@f\def\FigWBB@xgcWE{\v@lZ=\ht\Gb@x\advance\v@lZ\dp\Gb@x%
    \advance\v@lX\the\let@xte\v@leur\advance\v@lY\delt@\advance\v@lY.5\v@lZ%
    \b@undb@x{\v@lX}{\v@lY}%
    \advance\v@lX\the\let@xte\wd\Gb@x\advance\v@lY-\v@lZ\b@undb@x{\v@lX}{\v@lY}}
\ctr@ld@f\def\figwritebn#1:#2(#3){{\def\Vc@rrect{\v@lZ=\v@leur}\Figwrit@NS#1:{#2}(#3)}\ignorespaces}
\ctr@ld@f\def\figwritebs#1:#2(#3){{\def\Vc@rrect{\v@lZ=-\v@leur}\Figwrit@NS#1:{#2}(#3)}\ignorespaces}
\ctr@ld@f\def\figwritebw#1:#2(#3){{\let\figWp@si=\FigWp@sibW\let\figWBB@x=\FigWBB@xbWE%
    \let@xte={-}\FigWrit@L#1:{#2}(#3,0pt)}\ignorespaces}
\ctr@ld@f\def\figwritebe#1:#2(#3){{\let\figWp@si=\FigWp@sibE\let\figWBB@x=\FigWBB@xbWE%
    \let@xte={}\FigWrit@L#1:{#2}(#3,0pt)}\ignorespaces}
\ctr@ld@f\def\FigWp@sibW{\rlap{\kern\v@lX\raise\v@lY\hbox{\rlap{\kern-\wd\Gb@x\kern-\v@leur%
          \hbox{\unhcopy\Gb@x}}\c@nterpt}}}
\ctr@ld@f\def\FigWp@sibE{\rlap{\kern\v@lX\raise\v@lY\hbox{\c@nterpt\kern\v@leur%
          \hbox{\unhcopy\Gb@x}}}}
\ctr@ld@f\def\FigWBB@xbWE{\v@lZ=\ht\Gb@x\advance\v@lZ\dp\Gb@x%
    \advance\v@lX\the\let@xte\v@leur\advance\v@lY\ht\Gb@x\b@undb@x{\v@lX}{\v@lY}%
    \advance\v@lX\the\let@xte\wd\Gb@x\advance\v@lY-\v@lZ\b@undb@x{\v@lX}{\v@lY}}
\ctr@ln@w{newread}\frf@g  \ctr@ln@w{newwrite}\fwf@g
\ctr@ln@w{newif}\ifCUR@PS
\ctr@ln@w{newif}\ifGR@cri
\ctr@ln@w{newif}\ifUse@llipse
\ctr@ln@w{newif}\ifGRdebugm@de \GRdebugm@defalse 
\ctr@ln@w{newif}\ifPDFm@ke
\ifx\pdfliteral\undefined\else\ifnum\pdfoutput>\z@\PDFm@ketrue\fi\fi
\ctr@ld@f\def\initPDF@rDVI{%
\ifPDFm@ke
 \let\figscan=\figscan@E
 \let\newGr@FN=\newGr@FNPDF
 \ctr@ld@f\def\c@mcurveto{c}
 \ctr@ld@f\def\c@mfill{f}
 \ctr@ld@f\def\c@mgsave{q}
 \ctr@ld@f\def\c@mgrestore{Q}
 \ctr@ld@f\def\c@mlineto{l}
 \ctr@ld@f\def\c@mmoveto{m}
 \ctr@ld@f\def\c@msetgray{g}     \ctr@ld@f\def\c@msetgrayStroke{G}
 \ctr@ld@f\def\c@msetcmykcolor{k}\ctr@ld@f\def\c@msetcmykcolorStroke{K}
 \ctr@ld@f\def\c@msetrgbcolor{rg}\ctr@ld@f\def\c@msetrgbcolorStroke{RG}
 \ctr@ld@f\def\d@fprimarC@lor{\CUR@color\space\CUR@colorc@md%
               \space\CUR@color\space\CUR@colorc@mdStroke}
 \ctr@ld@f\def\c@msetdash{d}
 \ctr@ld@f\def\c@msetlinejoin{j}
 \ctr@ld@f\def\c@msetlinewidth{w}
 \ctr@ld@f\def\f@gclosestroke{\immediate\write\fwf@g{s}}
 \ctr@ld@f\def\f@gfill{\immediate\write\fwf@g{\fillc@md}}
 \ctr@ld@f\def\f@gnewpath{}
 \ctr@ld@f\def\f@gstroke{\immediate\write\fwf@g{S}}
\else
 \let\figinsertE=\figinsert
 \let\newGr@FN=\newGr@FNDVI
 \ctr@ld@f\def\c@mcurveto{curveto}
 \ctr@ld@f\def\c@mfill{fill}
 \ctr@ld@f\def\c@mgsave{gsave}
 \ctr@ld@f\def\c@mgrestore{grestore}
 \ctr@ld@f\def\c@mlineto{lineto}
 \ctr@ld@f\def\c@mmoveto{moveto}
 \ctr@ld@f\def\c@msetgray{setgray}          \ctr@ld@f\def\c@msetgrayStroke{}
 \ctr@ld@f\def\c@msetcmykcolor{setcmykcolor}\ctr@ld@f\def\c@msetcmykcolorStroke{}
 \ctr@ld@f\def\c@msetrgbcolor{setrgbcolor}  \ctr@ld@f\def\c@msetrgbcolorStroke{}
 \ctr@ld@f\def\d@fprimarC@lor{\CUR@color\space\CUR@colorc@md}
 \ctr@ld@f\def\c@msetdash{setdash}
 \ctr@ld@f\def\c@msetlinejoin{setlinejoin}
 \ctr@ld@f\def\c@msetlinewidth{setlinewidth}
 \ctr@ld@f\def\f@gclosestroke{\immediate\write\fwf@g{closepath\space stroke}}
 \ctr@ld@f\def\f@gfill{\immediate\write\fwf@g{\fillc@md}}
 \ctr@ld@f\def\f@gnewpath{\immediate\write\fwf@g{newpath}}
 \ctr@ld@f\def\f@gstroke{\immediate\write\fwf@g{stroke}}
\fi}
\ctr@ld@f\def\c@pypsfile#1#2{\c@pyfil@{\immediate\write#1}{#2}}
\ctr@ld@f\def\Figinclud@PDF#1#2{\openin\frf@g=#1\pdfliteral{q #2 0 0 #2 0 0 cm}%
    \c@pyfil@{\pdfliteral}{\frf@g}\pdfliteral{Q}\closein\frf@g}
\ctr@ln@w{newif}\ifmored@ta
\ctr@ln@m\bl@nkline
\ctr@ld@f\def\c@pyfil@#1#2{\def\bl@nkline{\par}{\catcode`\%=12
    \loop\ifeof#2\mored@tafalse\else\mored@tatrue\immediate\read#2 to\tr@c
    \ifx\tr@c\bl@nkline\else#1{\tr@c}\fi\fi\ifmored@ta\repeat}}
\ctr@ld@f\def\keln@mun#1#2|{\def\l@debut{#1}\def\l@suite{#2}}
\ctr@ld@f\def\keln@mde#1#2#3|{\def\l@debut{#1#2}\def\l@suite{#3}}
\ctr@ld@f\def\keln@mtr#1#2#3#4|{\def\l@debut{#1#2#3}\def\l@suite{#4}}
\ctr@ld@f\def\keln@mqu#1#2#3#4#5|{\def\l@debut{#1#2#3#4}\def\l@suite{#5}}
\ctr@ld@f\let\@psffilein=\frf@g 
\ctr@ln@w{newif}\if@psffileok    
\ctr@ln@w{newif}\if@psfbbfound   
\ctr@ln@w{newif}\if@psfverbose   
\@psfverbosetrue
\ctr@ln@m\@psfllx \ctr@ln@m\@psflly
\ctr@ln@m\@psfurx \ctr@ln@m\@psfury
\ctr@ln@m\resetcolonc@tcode
\ctr@ld@f\def\@psfgetbb#1{\global\@psfbbfoundfalse%
\global\def\@psfllx{0}\global\def\@psflly{0}%
\global\def\@psfurx{30}\global\def\@psfury{30}%
\openin\@psffilein=#1\relax
\ifeof\@psffilein\errmessage{I couldn't open #1, will ignore it}\else
   \edef\resetcolonc@tcode{\catcode`\noexpand\:\the\catcode`\:\relax}%
   {\@psffileoktrue \chardef\other=12
    \def\do##1{\catcode`##1=\other}\dospecials \catcode`\ =10 \resetcolonc@tcode
    \loop
       \read\@psffilein to \@psffileline
       \ifeof\@psffilein\@psffileokfalse\else
          \expandafter\@psfaux\@psffileline:. \\%
       \fi
   \if@psffileok\repeat
   \if@psfbbfound\else
    \if@psfverbose\message{No bounding box comment in #1; using defaults}\fi\fi
   }\closein\@psffilein\fi}%
\ctr@ln@m\@psfbblit
\ctr@ln@m\@psfpercent
{\catcode`\%=12 \global\let\@psfpercent=
\ctr@ln@m\@psfaux
\long\def\@psfaux#1#2:#3\\{\ifx#1\@psfpercent
   \def\testit{#2}\ifx\testit\@psfbblit
      \@psfgrab #3 . . . \\%
      \@psffileokfalse
      \global\@psfbbfoundtrue
   \fi\else\ifx#1\par\else\@psffileokfalse\fi\fi}%
\ctr@ld@f\def\@psfempty{}%
\ctr@ld@f\def\@psfgrab #1 #2 #3 #4 #5\\{%
\global\def\@psfllx{#1}\ifx\@psfllx\@psfempty
      \@psfgrab #2 #3 #4 #5 .\\\else
   \global\def\@psflly{#2}%
   \global\def\@psfurx{#3}\global\def\@psfury{#4}\fi}%
\ctr@ld@f\def\PSwrit@cmd#1#2#3{{\Figg@tXY{#1}\c@lprojSP\b@undb@x{\v@lX}{\v@lY}%
    \v@lX=\ptT@ptps\v@lX\v@lY=\ptT@ptps\v@lY%
    \immediate\write#3{\repdecn@mb{\v@lX}\space\repdecn@mb{\v@lY}\space#2}}}
\ctr@ld@f\def\PSwrit@cmdS#1#2#3#4#5{{\Figg@tXY{#1}\c@lprojSP\b@undb@x{\v@lX}{\v@lY}%
    \global\result@t=\v@lX\global\result@@t=\v@lY%
    \v@lX=\ptT@ptps\v@lX\v@lY=\ptT@ptps\v@lY%
    \immediate\write#3{\repdecn@mb{\v@lX}\space\repdecn@mb{\v@lY}\space#2}}%
    \edef#4{\the\result@t}\edef#5{\the\result@@t}}
\ctr@ld@f\def\update@ttr#1#2#3{\Figdisc@rdLTS{#3}{\n@mref}%
    \ifx\n@mref\D@FTref#2{#1}\else#2{#3}\fi}
\ctr@ld@f\def\D@FTref{default}
\ctr@ld@f\def\W@rnmesAttr#1#2{%
    \immediate\write16{*** Unknown attribute: \BS@ #1(..., #2=...)}}
\ctr@ld@f\def\W@rnmeskwd#1#2{%
    \immediate\write16{*** Unknown keyword #2 in \BS@ #1}}
\ctr@ld@f\def\W@rnmesIgn#1{\immediate\write16{*** \BS@ #1 is ignored inside a
     \BS@ figdrawbegin-\BS@ figdrawend block.}}
\ctr@ld@f\def\Psset@lti#1=#2|{\keln@mtr#1|%
    \def\n@mref{blc}\ifx\l@debut\n@mref\update@ttr\D@FTref\P@setblcolor{#2}\else
    \def\n@mref{bld}\ifx\l@debut\n@mref\update@ttr\D@FTref\P@setbldash{#2}\else
    \def\n@mref{blw}\ifx\l@debut\n@mref\update@ttr\D@FTref\P@setblwidth{#2}\else
    \def\n@mref{sqc}\ifx\l@debut\n@mref\update@ttr\D@FTref\P@setsqcolor{#2}\else
    \def\n@mref{sqd}\ifx\l@debut\n@mref\update@ttr\D@FTref\P@setsqdash{#2}\else
    \def\n@mref{sqw}\ifx\l@debut\n@mref\update@ttr\D@FTref\P@setsqwidth{#2}\else
    \W@rnmesAttr{figset altitude}{#1}\fi\fi\fi\fi\fi\fi}
\ctr@ln@m\DDV@blcolor
\ctr@ld@f\def\P@setblcolor#1{\edef\DDV@blcolor{#1}}
\ctr@ln@m\DDV@bldash
\ctr@ld@f\def\P@setbldash#1{\edef\DDV@bldash{#1}}
\ctr@ln@m\DDV@blwidth
\ctr@ld@f\def\P@setblwidth#1{\edef\DDV@blwidth{#1}}
\ctr@ln@m\DDV@sqcolor
\ctr@ld@f\def\P@setsqcolor#1{\edef\DDV@sqcolor{#1}}
\ctr@ln@m\DDV@sqdash
\ctr@ld@f\def\P@setsqdash#1{\edef\DDV@sqdash{#1}}
\ctr@ln@m\DDV@sqwidth
\ctr@ld@f\def\P@setsqwidth#1{\edef\DDV@sqwidth{#1}}
\ctr@ld@f\def\figdrawaltitude#1[#2,#3,#4]{{\ifCUR@PS\ifGR@cri%
    \PSc@mment{altitude Square Dim=#1, Triangle=[#2 / #3,#4]}%
    \s@uvc@ntr@l\et@tpsaltitude\resetc@ntr@l{2}\figptorthoprojline-5:=#2/#3,#4/%
    \figvectP -1[#3,#4]\n@rminf{\v@leur}{-1}\vecunit@{-3}{-1}%
    \figvectP -1[-5,#3]\n@rminf{\v@lmin}{-1}\figvectP -2[-5,#4]\n@rminf{\v@lmax}{-2}%
    \ifdim\v@lmin<\v@lmax\s@mme=#3\else\v@lmax=\v@lmin\s@mme=#4\fi%
    \figvectP -4[-5,#2]\vecunit@{-4}{-4}\delt@=#1\unit@%
    \edef\t@ille{\repdecn@mb{\delt@}}\figpttra-1:=-5/\t@ille,-3/%
    \figptstra-3=-5,-1/\t@ille,-4/\figdrawline[#2,-5]%
    \Pss@tspecifSt{color=\DDV@sqcolor,dash=\DDV@sqdash,width=\DDV@sqwidth}%
    \figdrawline[-1,-2,-3]%
    \Psrest@reSt{color=\DDV@sqcolor,dash=\DDV@sqdash,width=\DDV@sqwidth}%
    \ifdim\v@leur<\v@lmax%
    \Pss@tspecifSt{color=\DDV@blcolor,dash=\DDV@bldash,width=\DDV@blwidth}%
    \figdrawline[-5,\the\s@mme]%
    \Psrest@reSt{color=\DDV@blcolor,dash=\DDV@bldash,width=\DDV@blwidth}%
    \fi\PSc@mment{End altitude}\resetc@ntr@l\et@tpsaltitude\fi\fi}}
\ctr@ld@f\def\Ps@rcerc#1;#2(#3,#4){\ellBB@x#1;#2,#2(#3,#4,0)%
    \f@gnewpath{\delt@=#2\unit@\delt@=\ptT@ptps\delt@%
    \BdingB@xfalse%
    \PSwrit@cmd{#1}{\repdecn@mb{\delt@}\space #3\space #4\space arc}{\fwf@g}}}
\ctr@ln@m\figdrawarccirc
\ctr@ld@f\def\Q@arccircDD#1;#2(#3,#4){\ifCUR@PS\ifGR@cri%
    \PSc@mment{arccircDD Center=#1 ; Radius=#2 (Ang1=#3, Ang2=#4)}%
    \iffillm@de\Ps@rcerc#1;#2(#3,#4)%
    \f@gfill%
    \else\Ps@rcerc#1;#2(#3,#4)\f@gstroke\fi%
    \PSc@mment{End arccircDD}\fi\fi}
\ctr@ld@f\def\Q@arccircTD#1,#2,#3;#4(#5,#6){{\ifCUR@PS\ifGR@cri\s@uvc@ntr@l\et@tpsarccircTD%
    \PSc@mment{arccircTD Center=#1,P1=#2,P2=#3 ; Radius=#4 (Ang1=#5, Ang2=#6)}%
    \setc@ntr@l{2}\c@lExtAxes#1,#2,#3(#4)\Q@arcellPATD#1,-4,-5(#5,#6)%
    \PSc@mment{End arccircTD}\resetc@ntr@l\et@tpsarccircTD\fi\fi}}
\ctr@ld@f\def\c@lExtAxes#1,#2,#3(#4){%
    \figvectPTD-5[#1,#2]\vecunit@{-5}{-5}\figvectNTD-4[#1,#2,#3]\vecunit@{-4}{-4}%
    \figvectNVTD-3[-4,-5]\delt@=#4\unit@\edef\r@yon{\repdecn@mb{\delt@}}%
    \figpttra-4:=#1/\r@yon,-5/\figpttra-5:=#1/\r@yon,-3/}
\ctr@ln@m\figdrawarccircP
\ctr@ld@f\def\Q@arccircPDD#1;#2[#3,#4]{{\ifCUR@PS\ifGR@cri\s@uvc@ntr@l\et@tpsarccircPDD%
    \PSc@mment{arccircPDD Center=#1; Radius=#2, [P1=#3, P2=#4]}%
    \Ps@ngleparam#1;#2[#3,#4]\ifdim\v@lmin>\v@lmax\advance\v@lmax\DePI@deg\fi%
    \edef\@ngdeb{\repdecn@mb{\v@lmin}}\edef\@ngfin{\repdecn@mb{\v@lmax}}%
    \figdrawarccirc#1;\r@dius(\@ngdeb,\@ngfin)%
    \PSc@mment{End arccircPDD}\resetc@ntr@l\et@tpsarccircPDD\fi\fi}}
\ctr@ld@f\def\Q@arccircPTD#1;#2[#3,#4,#5]{{\ifCUR@PS\ifGR@cri\s@uvc@ntr@l\et@tpsarccircPTD%
    \PSc@mment{arccircPTD Center=#1; Radius=#2, [P1=#3, P2=#4, P3=#5]}%
    \setc@ntr@l{2}\c@lExtAxes#1,#3,#5(#2)\figdrawarcellPP#1,-4,-5[#3,#4]%
    \PSc@mment{End arccircPTD}\resetc@ntr@l\et@tpsarccircPTD\fi\fi}}
\ctr@ld@f\def\Ps@ngleparam#1;#2[#3,#4]{\setc@ntr@l{2}%
    \figvectPDD-1[#1,#3]\vecunit@{-1}{-1}\Figg@tXY{-1}\arct@n\v@lmin(\v@lX,\v@lY)%
    \figvectPDD-2[#1,#4]\vecunit@{-2}{-2}\Figg@tXY{-2}\arct@n\v@lmax(\v@lX,\v@lY)%
    \v@lmin=\rdT@deg\v@lmin\v@lmax=\rdT@deg\v@lmax%
    \v@leur=#2pt\maxim@m{\mili@u}{-\v@leur}{\v@leur}%
    \edef\r@dius{\repdecn@mb{\mili@u}}}
\ctr@ld@f\def\Ps@rcercBz#1;#2(#3,#4){\Ps@rellBz#1;#2,#2(#3,#4,0)}
\ctr@ld@f\def\Ps@rellBz#1;#2,#3(#4,#5,#6){%
    \ellBB@x#1;#2,#3(#4,#5,#6)\BdingB@xfalse%
    \c@lNbarcs{#4}{#5}\v@leur=#4pt\setc@ntr@l{2}\figptell-13::#1;#2,#3(#4,#6)%
    \f@gnewpath\PSwrit@cmd{-13}{\c@mmoveto}{\fwf@g}%
    \s@mme=\z@\bcl@rellBz#1;#2,#3(#6)\BdingB@xtrue}
\ctr@ld@f\def\bcl@rellBz#1;#2,#3(#4){\relax%
    \ifnum\s@mme<\p@rtent\advance\s@mme\@ne%
    \advance\v@leur\delt@\edef\@ngle{\repdecn@mb\v@leur}\figptell-14::#1;#2,#3(\@ngle,#4)%
    \advance\v@leur\delt@\edef\@ngle{\repdecn@mb\v@leur}\figptell-15::#1;#2,#3(\@ngle,#4)%
    \advance\v@leur\delt@\edef\@ngle{\repdecn@mb\v@leur}\figptell-16::#1;#2,#3(\@ngle,#4)%
    \figptscontrolDD-18[-13,-14,-15,-16]%
    \PSwrit@cmd{-18}{}{\fwf@g}\PSwrit@cmd{-17}{}{\fwf@g}%
    \PSwrit@cmd{-16}{\c@mcurveto}{\fwf@g}%
    \figptcopyDD-13:/-16/\bcl@rellBz#1;#2,#3(#4)\fi}
\ctr@ld@f\def\Ps@rell#1;#2,#3(#4,#5,#6){\ellBB@x#1;#2,#3(#4,#5,#6)%
    \f@gnewpath{\v@lmin=#2\unit@\v@lmin=\ptT@ptps\v@lmin%
    \v@lmax=#3\unit@\v@lmax=\ptT@ptps\v@lmax\BdingB@xfalse%
    \PSwrit@cmd{#1}%
    {#6\space\repdecn@mb{\v@lmin}\space\repdecn@mb{\v@lmax}\space #4\space #5\space ellipse}{\fwf@g}}%
    \global\Use@llipsetrue}
\ctr@ln@m\figdrawarcell
\ctr@ld@f\def\Q@arcellDD#1;#2,#3(#4,#5,#6){{\ifCUR@PS\ifGR@cri%
    \PSc@mment{arcellDD Center=#1 ; XRad=#2, YRad=#3 (Ang1=#4, Ang2=#5, Inclination=#6)}%
    \iffillm@de\Ps@rell#1;#2,#3(#4,#5,#6)%
    \f@gfill%
    \else\Ps@rell#1;#2,#3(#4,#5,#6)\f@gstroke\fi%
    \PSc@mment{End arcellDD}\fi\fi}}
\ctr@ld@f\def\Q@arcellTD#1;#2,#3(#4,#5,#6){{\ifCUR@PS\ifGR@cri\s@uvc@ntr@l\et@tpsarcellTD%
    \PSc@mment{arcellTD Center=#1 ; XRad=#2, YRad=#3 (Ang1=#4, Ang2=#5, Inclination=#6)}%
    \setc@ntr@l{2}\figpttraC -8:=#1/#2,0,0/\figpttraC -7:=#1/0,#3,0/%
    \figvectC -4(0,0,1)\figptsrot -8=-8,-7/#1,#6,-4/\Q@arcellPATD#1,-8,-7(#4,#5)%
    \PSc@mment{End arcellTD}\resetc@ntr@l\et@tpsarcellTD\fi\fi}}
\ctr@ln@m\figdrawarcellPA
\ctr@ld@f\def\Q@arcellPADD#1,#2,#3(#4,#5){{\ifCUR@PS\ifGR@cri\s@uvc@ntr@l\et@tpsarcellPADD%
    \PSc@mment{arcellPADD Center=#1,PtAxis1=#2,PtAxis2=#3 (Ang1=#4, Ang2=#5)}%
    \setc@ntr@l{2}\figvectPDD-1[#1,#2]\vecunit@DD{-1}{-1}\v@lX=\ptT@unit@\result@t%
    \edef\XR@d{\repdecn@mb{\v@lX}}\Figg@tXY{-1}\arct@n\v@lmin(\v@lX,\v@lY)%
    \v@lmin=\rdT@deg\v@lmin\edef\Inclin@{\repdecn@mb{\v@lmin}}%
    \figgetdist\YR@d[#1,#3]\Q@arcellDD#1;\XR@d,\YR@d(#4,#5,\Inclin@)%
    \PSc@mment{End arcellPADD}\resetc@ntr@l\et@tpsarcellPADD\fi\fi}}
\ctr@ld@f\def\Q@arcellPATD#1,#2,#3(#4,#5){{\ifCUR@PS\ifGR@cri\s@uvc@ntr@l\et@tpsarcellPATD%
    \PSc@mment{arcellPATD Center=#1,PtAxis1=#2,PtAxis2=#3 (Ang1=#4, Ang2=#5)}%
    \iffillm@de\Ps@rellPATD#1,#2,#3(#4,#5)%
    \f@gfill%
    \else\Ps@rellPATD#1,#2,#3(#4,#5)\f@gstroke\fi%
    \PSc@mment{End arcellPATD}\resetc@ntr@l\et@tpsarcellPATD\fi\fi}}
\ctr@ld@f\def\Ps@rellPATD#1,#2,#3(#4,#5){\let\c@lprojSP=\relax%
    \setc@ntr@l{2}\figvectPTD-1[#1,#2]\figvectPTD-2[#1,#3]\c@lNbarcs{#4}{#5}%
    \v@leur=#4pt\c@lptellP{#1}{-1}{-2}\Figptpr@j-5:/-3/%
    \f@gnewpath\PSwrit@cmdS{-5}{\c@mmoveto}{\fwf@g}{\X@un}{\Y@un}%
    \edef\C@nt@r{#1}\s@mme=\z@\bcl@rellPATD}
\ctr@ld@f\def\bcl@rellPATD{\relax%
    \ifnum\s@mme<\p@rtent\advance\s@mme\@ne%
    \advance\v@leur\delt@\c@lptellP{\C@nt@r}{-1}{-2}\Figptpr@j-4:/-3/%
    \advance\v@leur\delt@\c@lptellP{\C@nt@r}{-1}{-2}\Figptpr@j-6:/-3/%
    \advance\v@leur\delt@\c@lptellP{\C@nt@r}{-1}{-2}\Figptpr@j-3:/-3/%
    \v@lX=\z@\v@lY=\z@\Figtr@nptDD{-5}{-5}\Figtr@nptDD{2}{-3}%
    \divide\v@lX\@vi\divide\v@lY\@vi%
    \Figtr@nptDD{3}{-4}\Figtr@nptDD{-1.5}{-6}\v@lmin=\v@lX\v@lmax=\v@lY%
    \v@lX=\z@\v@lY=\z@\Figtr@nptDD{2}{-5}\Figtr@nptDD{-5}{-3}%
    \divide\v@lX\@vi\divide\v@lY\@vi\Figtr@nptDD{-1.5}{-4}\Figtr@nptDD{3}{-6}%
    \BdingB@xfalse%
    \Figp@intregDD-4:(\v@lmin,\v@lmax)\PSwrit@cmdS{-4}{}{\fwf@g}{\X@de}{\Y@de}%
    \Figp@intregDD-4:(\v@lX,\v@lY)\PSwrit@cmdS{-4}{}{\fwf@g}{\X@tr}{\Y@tr}%
    \BdingB@xtrue\PSwrit@cmdS{-3}{\c@mcurveto}{\fwf@g}{\X@qu}{\Y@qu}%
    \B@zierBB@x{1}{\Y@un}(\X@un,\X@de,\X@tr,\X@qu)%
    \B@zierBB@x{2}{\X@un}(\Y@un,\Y@de,\Y@tr,\Y@qu)%
    \edef\X@un{\X@qu}\edef\Y@un{\Y@qu}\figptcopyDD-5:/-3/\bcl@rellPATD\fi}
\ctr@ld@f\def\c@lNbarcs#1#2{%
    \delt@=#2pt\advance\delt@-#1pt\maxim@m{\v@lmax}{\delt@}{-\delt@}%
    \v@leur=\v@lmax\divide\v@leur45 \p@rtentiere{\p@rtent}{\v@leur}\advance\p@rtent\@ne%
    \s@mme=\p@rtent\multiply\s@mme\thr@@\divide\delt@\s@mme}
\ctr@ld@f\def\figdrawarcellPP#1,#2,#3[#4,#5]{{\ifCUR@PS\ifGR@cri\s@uvc@ntr@l\et@tpsarcellPP%
    \PSc@mment{arcellPP Center=#1,PtAxis1=#2,PtAxis2=#3 [Point1=#4, Point2=#5]}%
    \setc@ntr@l{2}\figvectP-2[#1,#3]\vecunit@{-2}{-2}\v@lmin=\result@t%
    \invers@{\v@lmax}{\v@lmin}%
    \figvectP-1[#1,#2]\vecunit@{-1}{-1}\v@leur=\result@t%
    \v@leur=\repdecn@mb{\v@lmax}\v@leur\edef\AsB@{\repdecn@mb{\v@leur}}
    \c@lAngle{#1}{#4}{\v@lmin}\edef\@ngdeb{\repdecn@mb{\v@lmin}}%
    \c@lAngle{#1}{#5}{\v@lmax}\ifdim\v@lmin>\v@lmax\advance\v@lmax\DePI@deg\fi%
    \edef\@ngfin{\repdecn@mb{\v@lmax}}\figdrawarcellPA#1,#2,#3(\@ngdeb,\@ngfin)%
    \PSc@mment{End arcellPP}\resetc@ntr@l\et@tpsarcellPP\fi\fi}}
\ctr@ld@f\def\c@lAngle#1#2#3{\figvectP-3[#1,#2]%
    \c@lproscal\delt@[-3,-1]\c@lproscal\v@leur[-3,-2]%
    \v@leur=\AsB@\v@leur\arct@n#3(\delt@,\v@leur)#3=\rdT@deg#3}
\ctr@ln@w{newif}\if@rrowratio\@rrowratiotrue
\ctr@ln@w{newif}\if@rrowhfill
\ctr@ln@w{newif}\if@rrowhout
\ctr@ld@f\def\Psset@rrowhe@d#1=#2|{\keln@mun#1|%
    \def\n@mref{a}\ifx\l@debut\n@mref\update@ttr\D@FTarrowheadangle\Q@s@tarrowheadangle{#2}\else
    \def\n@mref{f}\ifx\l@debut\n@mref\update@ttr\D@FTarrowheadfill\Q@s@tarrowheadfill{#2}\else
    \def\n@mref{l}\ifx\l@debut\n@mref\update@ttr\D@FTarrowheadlength\Q@s@tarrowheadlength{#2}\else
    \def\n@mref{o}\ifx\l@debut\n@mref\update@ttr\D@FTarrowheadout\Q@s@tarrowheadout{#2}\else
    \def\n@mref{r}\ifx\l@debut\n@mref\update@ttr\D@FTarrowheadratio\Q@s@tarrowheadratio{#2}\else
    \W@rnmesAttr{figset arrowhead}{#1}\fi\fi\fi\fi\fi}
\ctr@ln@m\@rrowheadangle
\ctr@ln@m\C@AHANG \ctr@ln@m\S@AHANG \ctr@ln@m\UNSS@N
\ctr@ld@f\def\Q@s@tarrowheadangle#1{\edef\@rrowheadangle{#1}{\c@ssin{\C@}{\S@}{#1}%
    \xdef\C@AHANG{\C@}\xdef\S@AHANG{\S@}\v@lmax=\S@ pt%
    \invers@{\v@leur}{\v@lmax}\maxim@m{\v@leur}{\v@leur}{-\v@leur}%
    \xdef\UNSS@N{\the\v@leur}}}
\ctr@ld@f\def\Q@s@tarrowheadfill#1{\expandafter\set@rrowhfill#1:}
\ctr@ld@f\def\set@rrowhfill#1#2:{\if#1n\@rrowhfillfalse\else\@rrowhfilltrue\fi}
\ctr@ld@f\def\Q@s@tarrowheadout#1{\expandafter\set@rrowhout#1:}
\ctr@ld@f\def\set@rrowhout#1#2:{\if#1n\@rrowhoutfalse\else\@rrowhouttrue\fi}
\ctr@ln@m\@rrowheadlength
\ctr@ld@f\def\Q@s@tarrowheadlength#1{\edef\@rrowheadlength{#1}\@rrowratiofalse}
\ctr@ln@m\@rrowheadratio
\ctr@ld@f\def\Q@s@tarrowheadratio#1{\edef\@rrowheadratio{#1}\@rrowratiotrue}
\ctr@ln@m\D@FTarrowheadlength
\ctr@ld@f\def\figresetarrowhead{%
    \Q@s@tarrowheadangle{\D@FTarrowheadangle}%
    \Q@s@tarrowheadfill{\D@FTarrowheadfill}%
    \Q@s@tarrowheadout{\D@FTarrowheadout}%
    \Q@s@tarrowheadratio{\D@FTarrowheadratio}%
    \d@fm@cdim\D@FTarrowheadlength{\D@FTh@rdahlength}
    \Q@s@tarrowheadlength{\D@FTarrowheadlength}}
\ctr@ld@f\def\D@FTarrowheadratio{0.1}
\ctr@ld@f\def\D@FTarrowheadangle{20}
\ctr@ld@f\def\D@FTarrowheadfill{no}
\ctr@ld@f\def\D@FTarrowheadout{no}
\ctr@ld@f\def\D@FTh@rdahlength{8pt}
\ctr@ln@m\figdrawarrow
\ctr@ld@f\def\Q@arrowDD[#1,#2]{{\ifCUR@PS\ifGR@cri\s@uvc@ntr@l\et@tpsarrow%
    \PSc@mment{arrowDD [Pt1,Pt2]=[#1,#2]}\Q@s@tfillmode{no}%
    \Q@arrowheadDD[#1,#2]\setc@ntr@l{2}\figdrawline[#1,-3]%
    \PSc@mment{End arrowDD}\resetc@ntr@l\et@tpsarrow\fi\fi}}
\ctr@ld@f\def\Q@arrowTD[#1,#2]{{\ifCUR@PS\ifGR@cri\s@uvc@ntr@l\et@tpsarrowTD%
    \PSc@mment{arrowTD [Pt1,Pt2]=[#1,#2]}\resetc@ntr@l{2}%
    \Figptpr@j-5:/#1/\Figptpr@j-6:/#2/\let\c@lprojSP=\relax\Q@arrowDD[-5,-6]%
    \PSc@mment{End arrowTD}\resetc@ntr@l\et@tpsarrowTD\fi\fi}}
\ctr@ln@m\figdrawarrowhead
\ctr@ld@f\def\Q@arrowheadDD[#1,#2]{{\ifCUR@PS\ifGR@cri\s@uvc@ntr@l\et@tpsarrowheadDD%
    \if@rrowhfill\def\@hangle{-\@rrowheadangle}\else\def\@hangle{\@rrowheadangle}\fi%
    \if@rrowratio%
    \if@rrowhout\def\@hratio{-\@rrowheadratio}\else\def\@hratio{\@rrowheadratio}\fi%
    \PSc@mment{arrowheadDD Ratio=\@hratio, Angle=\@hangle, [Pt1,Pt2]=[#1,#2]}%
    \Ps@rrowhead\@hratio,\@hangle[#1,#2]%
    \else%
    \if@rrowhout\def\@hlength{-\@rrowheadlength}\else\def\@hlength{\@rrowheadlength}\fi%
    \PSc@mment{arrowheadDD Length=\@hlength, Angle=\@hangle, [Pt1,Pt2]=[#1,#2]}%
    \Ps@rrowheadfd\@hlength,\@hangle[#1,#2]%
    \fi%
    \PSc@mment{End arrowheadDD}\resetc@ntr@l\et@tpsarrowheadDD\fi\fi}}
\ctr@ld@f\def\Q@arrowheadTD[#1,#2]{{\ifCUR@PS\ifGR@cri\s@uvc@ntr@l\et@tpsarrowheadTD%
    \PSc@mment{arrowheadTD [Pt1,Pt2]=[#1,#2]}\resetc@ntr@l{2}%
    \Figptpr@j-5:/#1/\Figptpr@j-6:/#2/\let\c@lprojSP=\relax\Q@arrowheadDD[-5,-6]%
    \PSc@mment{End arrowheadTD}\resetc@ntr@l\et@tpsarrowheadTD\fi\fi}}
\ctr@ld@f\def\Ps@rrowhead#1,#2[#3,#4]{\v@leur=#1\p@\maxim@m{\v@leur}{\v@leur}{-\v@leur}%
    \ifdim\v@leur>\Cepsil@n{
    \PSc@mment{@rrowhead Ratio=#1, Angle=#2, [Pt1,Pt2]=[#3,#4]}\v@leur=\UNSS@N%
    \v@leur=\CUR@width\v@leur\v@leur=\ptpsT@pt\v@leur\delt@=.5\v@leur
    \setc@ntr@l{2}\figvectPDD-3[#4,#3]%
    \Figg@tXY{-3}\v@lX=#1\v@lX\v@lY=#1\v@lY\Figv@ctCreg-3(\v@lX,\v@lY)%
    \vecunit@{-4}{-3}\mili@u=\result@t%
    \ifdim#2pt>\z@\v@lXa=-\C@AHANG\delt@%
     \edef\c@ef{\repdecn@mb{\v@lXa}}\figpttraDD-3:=-3/\c@ef,-4/\fi%
    \edef\c@ef{\repdecn@mb{\delt@}}%
    \v@lXa=\mili@u\v@lXa=\C@AHANG\v@lXa%
    \v@lYa=\ptpsT@pt\p@\v@lYa=\CUR@width\v@lYa\v@lYa=\sDcc@ngle\v@lYa%
    \advance\v@lXa-\v@lYa\gdef\sDcc@ngle{0}%
    \ifdim\v@lXa>\v@leur\edef\c@efendpt{\repdecn@mb{\v@leur}}%
    \else\edef\c@efendpt{\repdecn@mb{\v@lXa}}\fi%
    \Figg@tXY{-3}\v@lmin=\v@lX\v@lmax=\v@lY%
    \v@lXa=\C@AHANG\v@lmin\v@lYa=\S@AHANG\v@lmax\advance\v@lXa\v@lYa%
    \v@lYa=-\S@AHANG\v@lmin\v@lX=\C@AHANG\v@lmax\advance\v@lYa\v@lX%
    \setc@ntr@l{1}\Figg@tXY{#4}\advance\v@lX\v@lXa\advance\v@lY\v@lYa%
    \setc@ntr@l{2}\Figp@intregDD-2:(\v@lX,\v@lY)%
    \v@lXa=\C@AHANG\v@lmin\v@lYa=-\S@AHANG\v@lmax\advance\v@lXa\v@lYa%
    \v@lYa=\S@AHANG\v@lmin\v@lX=\C@AHANG\v@lmax\advance\v@lYa\v@lX%
    \setc@ntr@l{1}\Figg@tXY{#4}\advance\v@lX\v@lXa\advance\v@lY\v@lYa%
    \setc@ntr@l{2}\Figp@intregDD-1:(\v@lX,\v@lY)%
    \ifdim#2pt<\z@\fillm@detrue\figdrawline[-2,#4,-1]
    \else\figptstraDD-3=#4,-2,-1/\c@ef,-4/\s@uvdash{\typ@dash}\Q@s@tdash{\D@FTdash}%
    \figdrawline[-2,-3,-1]\Q@s@tdash{\typ@dash}\fi
    \ifdim#1pt>\z@\figpttraDD-3:=#4/\c@efendpt,-4/\else\figptcopyDD-3:/#4/\fi%
    \PSc@mment{End @rrowhead}}\fi}
\ctr@ld@f\def\sDcc@ngle{0}
\ctr@ld@f\def\Ps@rrowheadfd#1,#2[#3,#4]{{%
    \PSc@mment{@rrowheadfd Length=#1, Angle=#2, [Pt1,Pt2]=[#3,#4]}%
    \setc@ntr@l{2}\figvectPDD-1[#3,#4]\n@rmeucDD{\v@leur}{-1}\v@leur=\ptT@unit@\v@leur%
    \invers@{\v@leur}{\v@leur}\v@leur=#1\v@leur\edef\R@tio{\repdecn@mb{\v@leur}}%
    \Ps@rrowhead\R@tio,#2[#3,#4]\PSc@mment{End @rrowheadfd}}}
\ctr@ln@m\figdrawarrowBezier
\ctr@ld@f\def\Q@arrowBezierDD[#1,#2,#3,#4]{{\ifCUR@PS\ifGR@cri\s@uvc@ntr@l\et@tpsarrowBezierDD%
    \PSc@mment{arrowBezierDD Control points=#1,#2,#3,#4}\setc@ntr@l{2}%
    \if@rrowratio\c@larclengthDD\v@leur,10[#1,#2,#3,#4]\else\v@leur=\z@\fi%
    \Ps@rrowB@zDD\v@leur[#1,#2,#3,#4]%
    \PSc@mment{End arrowBezierDD}\resetc@ntr@l\et@tpsarrowBezierDD\fi\fi}}
\ctr@ld@f\def\Q@arrowBezierTD[#1,#2,#3,#4]{{\ifCUR@PS\ifGR@cri\s@uvc@ntr@l\et@tpsarrowBezierTD%
    \PSc@mment{arrowBezierTD Control points=#1,#2,#3,#4}\resetc@ntr@l{2}%
    \Figptpr@j-7:/#1/\Figptpr@j-8:/#2/\Figptpr@j-9:/#3/\Figptpr@j-10:/#4/%
    \let\c@lprojSP=\relax\ifnum\CUR@proj<\tw@\Q@arrowBezierDD[-7,-8,-9,-10]%
    \else\f@gnewpath\PSwrit@cmd{-7}{\c@mmoveto}{\fwf@g}%
    \if@rrowratio\c@larclengthDD\mili@u,10[-7,-8,-9,-10]\else\mili@u=\z@\fi%
    \p@rtent=\NBz@rcs\advance\p@rtent\m@ne\subB@zierTD\p@rtent[#1,#2,#3,#4]%
    \f@gstroke%
    \advance\v@lmin\p@rtent\delt@
    \v@leur=\v@lmin\advance\v@leur0.33333 \delt@\edef\unti@rs{\repdecn@mb{\v@leur}}%
    \v@leur=\v@lmin\advance\v@leur0.66666 \delt@\edef\deti@rs{\repdecn@mb{\v@leur}}%
    \figptcopyDD-8:/-10/\c@lsubBzarc\unti@rs,\deti@rs[#1,#2,#3,#4]%
    \figptcopyDD-8:/-4/\figptcopyDD-9:/-3/\Ps@rrowB@zDD\mili@u[-7,-8,-9,-10]\fi%
    \PSc@mment{End arrowBezierTD}\resetc@ntr@l\et@tpsarrowBezierTD\fi\fi}}
\ctr@ld@f\def\c@larclengthDD#1,#2[#3,#4,#5,#6]{{\p@rtent=#2\figptcopyDD-5:/#3/%
    \delt@=\p@\divide\delt@\p@rtent\c@rre=\z@\v@leur=\z@\s@mme=\z@%
    \loop\ifnum\s@mme<\p@rtent\advance\s@mme\@ne\advance\v@leur\delt@%
    \edef\T@{\repdecn@mb{\v@leur}}\figptBezierDD-6::\T@[#3,#4,#5,#6]%
    \figvectPDD-1[-5,-6]\n@rmeucDD{\mili@u}{-1}\advance\c@rre\mili@u%
    \figptcopyDD-5:/-6/\repeat\global\result@t=\ptT@unit@\c@rre}#1=\result@t}
\ctr@ld@f\def\Ps@rrowB@zDD#1[#2,#3,#4,#5]{{\Q@s@tfillmode{no}%
    \if@rrowratio\delt@=\@rrowheadratio#1\else\delt@=\@rrowheadlength pt\fi%
    \v@leur=\C@AHANG\delt@\edef\R@dius{\repdecn@mb{\v@leur}}%
    \FigptintercircB@zDD-5::0,\R@dius[#5,#4,#3,#2]%
    \Q@s@tarrowheadlength{\repdecn@mb{\delt@}}\Q@arrowheadDD[-5,#5]%
    \let\n@rmeuc=\n@rmeucDD\figgetdist\R@dius[#5,-3]%
    \FigptintercircB@zDD-6::0,\R@dius[#5,#4,#3,#2]%
    \figptBezierDD-5::0.33333[#5,#4,#3,#2]\figptBezierDD-3::0.66666[#5,#4,#3,#2]%
    \figptscontrolDD-5[-6,-5,-3,#2]\Q@BezierDD1[-6,-5,-4,#2]}}
\ctr@ln@m\figdrawarrowcirc
\ctr@ld@f\def\Q@arrowcircDD#1;#2(#3,#4){{\ifCUR@PS\ifGR@cri\s@uvc@ntr@l\et@tpsarrowcircDD%
    \PSc@mment{arrowcircDD Center=#1 ; Radius=#2 (Ang1=#3,Ang2=#4)}%
    \Q@s@tfillmode{no}\Pscirc@rrowhead#1;#2(#3,#4)%
    \setc@ntr@l{2}\figvectPDD -4[#1,-3]\vecunit@{-4}{-4}%
    \Figg@tXY{-4}\arct@n\v@lmin(\v@lX,\v@lY)%
    \v@lmin=\rdT@deg\v@lmin\v@leur=#4pt\advance\v@leur-\v@lmin%
    \maxim@m{\v@leur}{\v@leur}{-\v@leur}%
    \ifdim\v@leur>\DemiPI@deg\relax\ifdim\v@lmin<#4pt\advance\v@lmin\DePI@deg%
    \else\advance\v@lmin-\DePI@deg\fi\fi\edef\ar@ngle{\repdecn@mb{\v@lmin}}%
    \ifdim#3pt<#4pt\figdrawarccirc#1;#2(#3,\ar@ngle)\else\figdrawarccirc#1;#2(\ar@ngle,#3)\fi%
    \PSc@mment{End arrowcircDD}\resetc@ntr@l\et@tpsarrowcircDD\fi\fi}}
\ctr@ld@f\def\Q@arrowcircTD#1,#2,#3;#4(#5,#6){{\ifCUR@PS\ifGR@cri\s@uvc@ntr@l\et@tpsarrowcircTD%
    \PSc@mment{arrowcircTD Center=#1,P1=#2,P2=#3 ; Radius=#4 (Ang1=#5, Ang2=#6)}%
    \resetc@ntr@l{2}\c@lExtAxes#1,#2,#3(#4)\let\c@lprojSP=\relax%
    \figvectPTD-11[#1,-4]\figvectPTD-12[#1,-5]\c@lNbarcs{#5}{#6}%
    \if@rrowratio\v@lmax=\degT@rd\v@lmax\edef\D@lpha{\repdecn@mb{\v@lmax}}\fi%
    \advance\p@rtent\m@ne\mili@u=\z@%
    \v@leur=#5pt\c@lptellP{#1}{-11}{-12}\Figptpr@j-9:/-3/%
    \f@gnewpath\PSwrit@cmdS{-9}{\c@mmoveto}{\fwf@g}{\X@un}{\Y@un}%
    \edef\C@nt@r{#1}\s@mme=\z@\bcl@rcircTD\f@gstroke%
    \advance\v@leur\delt@\c@lptellP{#1}{-11}{-12}\Figptpr@j-5:/-3/%
    \advance\v@leur\delt@\c@lptellP{#1}{-11}{-12}\Figptpr@j-6:/-3/%
    \advance\v@leur\delt@\c@lptellP{#1}{-11}{-12}\Figptpr@j-10:/-3/%
    \figptscontrolDD-8[-9,-5,-6,-10]%
    \if@rrowratio\c@lcurvradDD0.5[-9,-8,-7,-10]\advance\mili@u\result@t%
    \maxim@m{\mili@u}{\mili@u}{-\mili@u}\mili@u=\ptT@unit@\mili@u%
    \mili@u=\D@lpha\mili@u\advance\p@rtent\@ne\divide\mili@u\p@rtent\fi%
    \Ps@rrowB@zDD\mili@u[-9,-8,-7,-10]%
    \PSc@mment{End arrowcircTD}\resetc@ntr@l\et@tpsarrowcircTD\fi\fi}}
\ctr@ld@f\def\bcl@rcircTD{\relax%
    \ifnum\s@mme<\p@rtent\advance\s@mme\@ne%
    \advance\v@leur\delt@\c@lptellP{\C@nt@r}{-11}{-12}\Figptpr@j-5:/-3/%
    \advance\v@leur\delt@\c@lptellP{\C@nt@r}{-11}{-12}\Figptpr@j-6:/-3/%
    \advance\v@leur\delt@\c@lptellP{\C@nt@r}{-11}{-12}\Figptpr@j-10:/-3/%
    \figptscontrolDD-8[-9,-5,-6,-10]\BdingB@xfalse%
    \PSwrit@cmdS{-8}{}{\fwf@g}{\X@de}{\Y@de}\PSwrit@cmdS{-7}{}{\fwf@g}{\X@tr}{\Y@tr}%
    \BdingB@xtrue\PSwrit@cmdS{-10}{\c@mcurveto}{\fwf@g}{\X@qu}{\Y@qu}%
    \if@rrowratio\c@lcurvradDD0.5[-9,-8,-7,-10]\advance\mili@u\result@t\fi%
    \B@zierBB@x{1}{\Y@un}(\X@un,\X@de,\X@tr,\X@qu)%
    \B@zierBB@x{2}{\X@un}(\Y@un,\Y@de,\Y@tr,\Y@qu)%
    \edef\X@un{\X@qu}\edef\Y@un{\Y@qu}\figptcopyDD-9:/-10/\bcl@rcircTD\fi}
\ctr@ld@f\def\Pscirc@rrowhead#1;#2(#3,#4){{%
    \PSc@mment{circ@rrowhead Center=#1 ; Radius=#2 (Ang1=#3,Ang2=#4)}%
    \v@leur=#2\unit@\edef\s@glen{\repdecn@mb{\v@leur}}\v@lY=\z@\v@lX=\v@leur%
    \resetc@ntr@l{2}\Figv@ctCreg-3(\v@lX,\v@lY)\figpttraDD-5:=#1/1,-3/%
    \figptrotDD-5:=-5/#1,#4/%
    \figvectPDD-3[#1,-5]\Figg@tXY{-3}\v@leur=\v@lX%
    \ifdim#3pt<#4pt\v@lX=\v@lY\v@lY=-\v@leur\else\v@lX=-\v@lY\v@lY=\v@leur\fi%
    \Figv@ctCreg-3(\v@lX,\v@lY)\vecunit@{-3}{-3}%
    \if@rrowratio\v@leur=#4pt\advance\v@leur-#3pt\maxim@m{\mili@u}{-\v@leur}{\v@leur}%
    \mili@u=\degT@rd\mili@u\v@leur=\s@glen\mili@u\edef\s@glen{\repdecn@mb{\v@leur}}%
    \mili@u=#2\mili@u\mili@u=\@rrowheadratio\mili@u\else\mili@u=\@rrowheadlength pt\fi%
    \figpttraDD-6:=-5/\s@glen,-3/\v@leur=#2pt\v@leur=2\v@leur%
    \invers@{\v@leur}{\v@leur}\c@rre=\repdecn@mb{\v@leur}\mili@u
    \mili@u=\c@rre\mili@u=\repdecn@mb{\c@rre}\mili@u%
    \v@leur=\p@\advance\v@leur-\mili@u
    \invers@{\mili@u}{2\v@leur}\delt@=\c@rre\delt@=\repdecn@mb{\mili@u}\delt@%
    \xdef\sDcc@ngle{\repdecn@mb{\delt@}}
    \sqrt@{\mili@u}{\v@leur}\arct@n\v@leur(\mili@u,\c@rre)%
    \v@leur=\rdT@deg\v@leur
    \ifdim#3pt<#4pt\v@leur=-\v@leur\fi%
    \if@rrowhout\v@leur=-\v@leur\fi\edef\cor@ngle{\repdecn@mb{\v@leur}}%
    \figptrotDD-6:=-6/-5,\cor@ngle/\Q@arrowheadDD[-6,-5]%
    \PSc@mment{End circ@rrowhead}}}
\ctr@ln@m\figdrawarrowcircP
\ctr@ld@f\def\Q@arrowcircPDD#1;#2[#3,#4]{{\ifCUR@PS\ifGR@cri%
    \PSc@mment{arrowcircPDD Center=#1; Radius=#2, [P1=#3,P2=#4]}%
    \s@uvc@ntr@l\et@tpsarrowcircPDD\Ps@ngleparam#1;#2[#3,#4]%
    \ifdim\v@leur>\z@\ifdim\v@lmin>\v@lmax\advance\v@lmax\DePI@deg\fi%
    \else\ifdim\v@lmin<\v@lmax\advance\v@lmin\DePI@deg\fi\fi%
    \edef\@ngdeb{\repdecn@mb{\v@lmin}}\edef\@ngfin{\repdecn@mb{\v@lmax}}%
    \figdrawarrowcirc#1;\r@dius(\@ngdeb,\@ngfin)%
    \PSc@mment{End arrowcircPDD}\resetc@ntr@l\et@tpsarrowcircPDD\fi\fi}}
\ctr@ld@f\def\Q@arrowcircPTD#1;#2[#3,#4,#5]{{\ifCUR@PS\ifGR@cri\s@uvc@ntr@l\et@tpsarrowcircPTD%
    \PSc@mment{arrowcircPTD Center=#1; Radius=#2, [P1=#3,P2=#4,P3=#5]}%
    \figgetangleTD\@ngfin[#1,#3,#4,#5]\v@leur=#2pt%
    \maxim@m{\mili@u}{-\v@leur}{\v@leur}\edef\r@dius{\repdecn@mb{\mili@u}}%
    \ifdim\v@leur<\z@\v@lmax=\@ngfin pt\advance\v@lmax-\DePI@deg%
    \edef\@ngfin{\repdecn@mb{\v@lmax}}\fi\Q@arrowcircTD#1,#3,#5;\r@dius(0,\@ngfin)%
    \PSc@mment{End arrowcircPTD}\resetc@ntr@l\et@tpsarrowcircPTD\fi\fi}}
\ctr@ld@f\def\figdrawaxes#1(#2){{\ifCUR@PS\ifGR@cri\s@uvc@ntr@l\et@tpsaxes%
    \PSc@mment{axes Origin=#1 Range=(#2)}\an@lys@xes#2,:\resetc@ntr@l{2}%
    \ifx\t@xt@\empty\ifTr@isDim\Q@@xes#1(0,#2,0,#2,0,#2)\else\Q@@xes#1(0,#2,0,#2)\fi%
    \else\Q@@xes#1(#2)\fi\PSc@mment{End axes}\resetc@ntr@l\et@tpsaxes\fi\fi}}
\ctr@ld@f\def\an@lys@xes#1,#2:{\def\t@xt@{#2}}
\ctr@ln@m\Q@@xes
\ctr@ld@f\def\Q@@xesDD#1(#2,#3,#4,#5){%
    \figpttraC-5:=#1/#2,0/\figpttraC-6:=#1/#3,0/\Q@arrowDD[-5,-6]%
    \figpttraC-5:=#1/0,#4/\figpttraC-6:=#1/0,#5/\Q@arrowDD[-5,-6]}
\ctr@ld@f\def\Q@@xesTD#1(#2,#3,#4,#5,#6,#7){%
    \figpttraC-7:=#1/#2,0,0/\figpttraC-8:=#1/#3,0,0/\Q@arrowTD[-7,-8]%
    \figpttraC-7:=#1/0,#4,0/\figpttraC-8:=#1/0,#5,0/\Q@arrowTD[-7,-8]%
    \figpttraC-7:=#1/0,0,#6/\figpttraC-8:=#1/0,0,#7/\Q@arrowTD[-7,-8]}
\ctr@ln@m\newGr@FN
\ctr@ld@f\def\newGr@FNPDF#1{\s@mme=\Gr@FNb\advance\s@mme\@ne\xdef\Gr@FNb{\number\s@mme}}
\ctr@ld@f\def\newGr@FNDVI#1{\newGr@FNPDF{}\xdef#1{\jobname GI\Gr@FNb.anx}}
\ctr@ld@f\def\figdrawbegin#1{\newGr@FN\DefGIfilen@me\gdef\@utoFN{0}%
    \def\t@xt@{#1}\relax\ifx\t@xt@\empty\GRupdatem@detrue%
    \gdef\@utoFN{1}\Psb@ginfig\DefGIfilen@me\else\expandafter\Psb@ginfigNu@#1 :\fi}
\ctr@ld@f\def\Psb@ginfigNu@#1 #2:{\def\t@xt@{#1}\relax\ifx\t@xt@\empty\def\t@xt@{#2}%
    \ifx\t@xt@\empty\GRupdatem@detrue\gdef\@utoFN{1}\Psb@ginfig\DefGIfilen@me%
    \else\Psb@ginfigNu@#2:\fi\else\Psb@ginfig{#1}\fi}
\ctr@ln@m\PSfilen@me \ctr@ln@m\auxfilen@me
\ctr@ld@f\def\Psb@ginfig#1{\ifCUR@PS\else%
    \edef\PSfilen@me{#1}\edef\auxfilen@me{\jobname.anx}%
    \ifGRupdatem@de\GR@critrue\else\openin\frf@g=\PSfilen@me\relax%
    \ifeof\frf@g\GR@critrue\else\GR@crifalse\fi\closein\frf@g\fi%
    \CUR@PStrue\c@ldefproj\expandafter\setupd@te\D@FTupdate:%
    \ifGR@cri\initb@undb@x%
    \immediate\openout\fwf@g=\auxfilen@me\initpss@ttings\fi%
    \fi}
\ctr@ld@f\def\Gr@FNb{0}
\ctr@ld@f\def\figforTeXFileno{\Gr@FNb}
\ctr@ld@f\def\figforTeXFigno{0 }
\ctr@ld@f\def\figforTeXnextFigno{1 }
\ctr@ld@f\edef\DefGIfilen@me{\jobname GI.anx}
\ctr@ld@f\def\initpss@ttings{\figreset{altitude,arrowhead,curve,general,flowchart,mesh,trimesh}%
    \Use@llipsefalse}
\ctr@ld@f\def\B@zierBB@x#1#2(#3,#4,#5,#6){{\c@rre=\t@n\epsil@n
    \v@lmax=#4\advance\v@lmax-#5\v@lmax=\thr@@\v@lmax\advance\v@lmax#6\advance\v@lmax-#3%
    \mili@u=#4\mili@u=-\tw@\mili@u\advance\mili@u#3\advance\mili@u#5%
    \v@lmin=#4\advance\v@lmin-#3\maxim@m{\v@leur}{-\v@lmax}{\v@lmax}%
    \maxim@m{\delt@}{-\mili@u}{\mili@u}\maxim@m{\v@leur}{\v@leur}{\delt@}%
    \maxim@m{\delt@}{-\v@lmin}{\v@lmin}\maxim@m{\v@leur}{\v@leur}{\delt@}%
    \ifdim\v@leur>\c@rre\invers@{\v@leur}{\v@leur}\edef\Uns@rM@x{\repdecn@mb{\v@leur}}%
    \v@lmax=\Uns@rM@x\v@lmax\mili@u=\Uns@rM@x\mili@u\v@lmin=\Uns@rM@x\v@lmin%
    \maxim@m{\v@leur}{-\v@lmax}{\v@lmax}\ifdim\v@leur<\c@rre%
    \maxim@m{\v@leur}{-\mili@u}{\mili@u}\ifdim\v@leur<\c@rre\else%
    \invers@{\mili@u}{\mili@u}\v@leur=-0.5\v@lmin%
    \v@leur=\repdecn@mb{\mili@u}\v@leur\m@jBBB@x{\v@leur}{#1}{#2}(#3,#4,#5,#6)\fi%
    \else\delt@=\repdecn@mb{\mili@u}\mili@u\v@leur=\repdecn@mb{\v@lmax}\v@lmin%
    \advance\delt@-\v@leur\ifdim\delt@<\z@\else\invers@{\v@lmax}{\v@lmax}%
    \edef\Uns@rAp{\repdecn@mb{\v@lmax}}\sqrt@{\delt@}{\delt@}%
    \v@leur=-\mili@u\advance\v@leur\delt@\v@leur=\Uns@rAp\v@leur%
    \m@jBBB@x{\v@leur}{#1}{#2}(#3,#4,#5,#6)%
    \v@leur=-\mili@u\advance\v@leur-\delt@\v@leur=\Uns@rAp\v@leur%
    \m@jBBB@x{\v@leur}{#1}{#2}(#3,#4,#5,#6)\fi\fi\fi}}
\ctr@ld@f\def\m@jBBB@x#1#2#3(#4,#5,#6,#7){{\relax\ifdim#1>\z@\ifdim#1<\p@%
    \edef\T@{\repdecn@mb{#1}}\v@lX=\p@\advance\v@lX-#1\edef\UNmT@{\repdecn@mb{\v@lX}}%
    \v@lX=#4\v@lY=#5\v@lZ=#6\v@lXa=#7\v@lX=\UNmT@\v@lX\advance\v@lX\T@\v@lY%
    \v@lY=\UNmT@\v@lY\advance\v@lY\T@\v@lZ\v@lZ=\UNmT@\v@lZ\advance\v@lZ\T@\v@lXa%
    \v@lX=\UNmT@\v@lX\advance\v@lX\T@\v@lY\v@lY=\UNmT@\v@lY\advance\v@lY\T@\v@lZ%
    \v@lX=\UNmT@\v@lX\advance\v@lX\T@\v@lY%
    \ifcase#2\or\v@lY=#3\or\v@lY=\v@lX\v@lX=#3\fi\b@undb@x{\v@lX}{\v@lY}\fi\fi}}
\ctr@ld@f\def\PsB@zier#1[#2]{{\f@gnewpath%
    \s@mme=\z@\def\list@num{#2,0}\extrairelepremi@r\p@int\de\list@num%
    \PSwrit@cmdS{\p@int}{\c@mmoveto}{\fwf@g}{\X@un}{\Y@un}\p@rtent=#1\bclB@zier}}
\ctr@ld@f\def\bclB@zier{\relax%
    \ifnum\s@mme<\p@rtent\advance\s@mme\@ne\BdingB@xfalse%
    \extrairelepremi@r\p@int\de\list@num\PSwrit@cmdS{\p@int}{}{\fwf@g}{\X@de}{\Y@de}%
    \extrairelepremi@r\p@int\de\list@num\PSwrit@cmdS{\p@int}{}{\fwf@g}{\X@tr}{\Y@tr}%
    \BdingB@xtrue%
    \extrairelepremi@r\p@int\de\list@num\PSwrit@cmdS{\p@int}{\c@mcurveto}{\fwf@g}{\X@qu}{\Y@qu}%
    \B@zierBB@x{1}{\Y@un}(\X@un,\X@de,\X@tr,\X@qu)%
    \B@zierBB@x{2}{\X@un}(\Y@un,\Y@de,\Y@tr,\Y@qu)%
    \edef\X@un{\X@qu}\edef\Y@un{\Y@qu}\bclB@zier\fi}
\ctr@ln@m\figdrawBezier
\ctr@ld@f\def\Q@BezierDD#1[#2]{\ifCUR@PS\ifGR@cri%
    \PSc@mment{BezierDD N arcs=#1, Control points=#2}%
    \iffillm@de\PsB@zier#1[#2]%
    \f@gfill%
    \else\PsB@zier#1[#2]\f@gstroke\fi%
    \PSc@mment{End BezierDD}\fi\fi}
\ctr@ln@m\et@tpsBezierTD
\ctr@ld@f\def\Q@BezierTD#1[#2]{\ifCUR@PS\ifGR@cri\s@uvc@ntr@l\et@tpsBezierTD%
    \PSc@mment{BezierTD N arcs=#1, Control points=#2}%
    \iffillm@de\PsB@zierTD#1[#2]%
    \f@gfill%
    \else\PsB@zierTD#1[#2]\f@gstroke\fi%
    \PSc@mment{End BezierTD}\resetc@ntr@l\et@tpsBezierTD\fi\fi}
\ctr@ld@f\def\PsB@zierTD#1[#2]{\ifnum\CUR@proj<\tw@\PsB@zier#1[#2]\else\PsB@zier@TD#1[#2]\fi}
\ctr@ld@f\def\PsB@zier@TD#1[#2]{{\f@gnewpath%
    \s@mme=\z@\def\list@num{#2,0}\extrairelepremi@r\p@int\de\list@num%
    \let\c@lprojSP=\relax\setc@ntr@l{2}\Figptpr@j-7:/\p@int/%
    \PSwrit@cmd{-7}{\c@mmoveto}{\fwf@g}%
    \loop\ifnum\s@mme<#1\advance\s@mme\@ne\extrairelepremi@r\p@intun\de\list@num%
    \extrairelepremi@r\p@intde\de\list@num\extrairelepremi@r\p@inttr\de\list@num%
    \subB@zierTD\NBz@rcs[\p@int,\p@intun,\p@intde,\p@inttr]\edef\p@int{\p@inttr}\repeat}}
\ctr@ld@f\def\subB@zierTD#1[#2,#3,#4,#5]{\delt@=\p@\divide\delt@\NBz@rcs\v@lmin=\z@%
    {\Figg@tXY{-7}\edef\X@un{\the\v@lX}\edef\Y@un{\the\v@lY}%
    \s@mme=\z@\loop\ifnum\s@mme<#1\advance\s@mme\@ne%
    \v@leur=\v@lmin\advance\v@leur0.33333 \delt@\edef\unti@rs{\repdecn@mb{\v@leur}}%
    \v@leur=\v@lmin\advance\v@leur0.66666 \delt@\edef\deti@rs{\repdecn@mb{\v@leur}}%
    \advance\v@lmin\delt@\edef\trti@rs{\repdecn@mb{\v@lmin}}%
    \figptBezierTD-8::\trti@rs[#2,#3,#4,#5]\Figptpr@j-8:/-8/%
    \c@lsubBzarc\unti@rs,\deti@rs[#2,#3,#4,#5]\BdingB@xfalse%
    \PSwrit@cmdS{-4}{}{\fwf@g}{\X@de}{\Y@de}\PSwrit@cmdS{-3}{}{\fwf@g}{\X@tr}{\Y@tr}%
    \BdingB@xtrue\PSwrit@cmdS{-8}{\c@mcurveto}{\fwf@g}{\X@qu}{\Y@qu}%
    \B@zierBB@x{1}{\Y@un}(\X@un,\X@de,\X@tr,\X@qu)%
    \B@zierBB@x{2}{\X@un}(\Y@un,\Y@de,\Y@tr,\Y@qu)%
    \edef\X@un{\X@qu}\edef\Y@un{\Y@qu}\figptcopyDD-7:/-8/\repeat}}
\ctr@ld@f\def\NBz@rcs{2}
\ctr@ld@f\def\c@lsubBzarc#1,#2[#3,#4,#5,#6]{\figptBezierTD-5::#1[#3,#4,#5,#6]%
    \figptBezierTD-6::#2[#3,#4,#5,#6]\Figptpr@j-4:/-5/\Figptpr@j-5:/-6/%
    \figptscontrolDD-4[-7,-4,-5,-8]}
\ctr@ln@m\figdrawcirc
\ctr@ld@f\def\Q@circDD#1(#2){\ifCUR@PS\ifGR@cri\PSc@mment{circDD Center=#1 (Radius=#2)}%
    \Q@arccircDD#1;#2(0,360)\PSc@mment{End circDD}\fi\fi}
\ctr@ld@f\def\Q@circTD#1,#2,#3(#4){\ifCUR@PS\ifGR@cri%
    \PSc@mment{circTD Center=#1,P1=#2,P2=#3 (Radius=#4)}%
    \Q@arccircTD#1,#2,#3;#4(0,360)\PSc@mment{End circTD}\fi\fi}
\ctr@ln@m\p@urcent
{\catcode`\%=12\gdef\p@urcent{
\ctr@ld@f\def\PSc@mment#1{\ifGRdebugm@de\immediate\write\fwf@g{\p@urcent\space#1}\fi}
\ctr@ln@m\acc@louv \ctr@ln@m\acc@lfer
{\catcode`\[=1\catcode`\{=12\gdef\acc@louv[{}}
{\catcode`\]=2\catcode`\}=12\gdef\acc@lfer{}]]
\ctr@ld@f\def\PSdict@{\ifUse@llipse%
    \immediate\write\fwf@g{/ellipsedict 9 dict def ellipsedict /mtrx matrix put}%
    \immediate\write\fwf@g{/ellipse \acc@louv ellipsedict begin}%
    \immediate\write\fwf@g{ /endangle exch def /startangle exch def}%
    \immediate\write\fwf@g{ /yrad exch def /xrad exch def}%
    \immediate\write\fwf@g{ /rotangle exch def /y exch def /x exch def}%
    \immediate\write\fwf@g{ /savematrix mtrx currentmatrix def}%
    \immediate\write\fwf@g{ x y translate rotangle rotate xrad yrad scale}%
    \immediate\write\fwf@g{ 0 0 1 startangle endangle arc}%
    \immediate\write\fwf@g{ savematrix setmatrix end\acc@lfer def}%
    \fi\PShe@der{EndProlog}}
\ctr@ld@f\def\Pssetc@rve#1=#2|{\keln@mun#1|%
    \def\n@mref{r}\ifx\l@debut\n@mref\update@ttr\D@FTroundness\Q@s@troundness{#2}\else
    \W@rnmesAttr{figset curve}{#1}\fi}
\ctr@ln@m\curv@roundness
\ctr@ld@f\def\Q@s@troundness#1{\edef\curv@roundness{#1}}
\ctr@ld@f\def\D@FTroundness{0.2} 
\ctr@ln@m\figdrawcurve
\ctr@ld@f\def\Q@curveDD[#1]{{\ifCUR@PS\ifGR@cri\PSc@mment{curveDD Points=#1}%
    \s@uvc@ntr@l\et@tpscurveDD%
    \iffillm@de\Psc@rveDD\curv@roundness[#1]%
    \f@gfill%
    \else\Psc@rveDD\curv@roundness[#1]\f@gstroke\fi%
    \PSc@mment{End curveDD}\resetc@ntr@l\et@tpscurveDD\fi\fi}}
\ctr@ld@f\def\Q@curveTD[#1]{{\ifCUR@PS\ifGR@cri%
    \PSc@mment{curveTD Points=#1}\s@uvc@ntr@l\et@tpscurveTD\let\c@lprojSP=\relax%
    \iffillm@de\Psc@rveTD\curv@roundness[#1]%
    \f@gfill%
    \else\Psc@rveTD\curv@roundness[#1]\f@gstroke\fi%
    \PSc@mment{End curveTD}\resetc@ntr@l\et@tpscurveTD\fi\fi}}
\ctr@ld@f\def\Psc@rveDD#1[#2]{%
    \def\list@num{#2}\extrairelepremi@r\Ak@\de\list@num%
    \extrairelepremi@r\Ai@\de\list@num\extrairelepremi@r\Aj@\de\list@num%
    \f@gnewpath\PSwrit@cmdS{\Ai@}{\c@mmoveto}{\fwf@g}{\X@un}{\Y@un}%
    \setc@ntr@l{2}\figvectPDD -1[\Ak@,\Aj@]%
    \@ecfor\Ak@:=\list@num\do{\figpttraDD-2:=\Ai@/#1,-1/\BdingB@xfalse%
       \PSwrit@cmdS{-2}{}{\fwf@g}{\X@de}{\Y@de}%
       \figvectPDD -1[\Ai@,\Ak@]\figpttraDD-2:=\Aj@/-#1,-1/%
       \PSwrit@cmdS{-2}{}{\fwf@g}{\X@tr}{\Y@tr}\BdingB@xtrue%
       \PSwrit@cmdS{\Aj@}{\c@mcurveto}{\fwf@g}{\X@qu}{\Y@qu}%
       \B@zierBB@x{1}{\Y@un}(\X@un,\X@de,\X@tr,\X@qu)%
       \B@zierBB@x{2}{\X@un}(\Y@un,\Y@de,\Y@tr,\Y@qu)%
       \edef\X@un{\X@qu}\edef\Y@un{\Y@qu}\edef\Ai@{\Aj@}\edef\Aj@{\Ak@}}}
\ctr@ld@f\def\Psc@rveTD#1[#2]{\ifnum\CUR@proj<\tw@\Psc@rvePPTD#1[#2]\else\Psc@rveCPTD#1[#2]\fi}
\ctr@ld@f\def\Psc@rvePPTD#1[#2]{\setc@ntr@l{2}%
    \def\list@num{#2}\extrairelepremi@r\Ak@\de\list@num\Figptpr@j-5:/\Ak@/%
    \extrairelepremi@r\Ai@\de\list@num\Figptpr@j-3:/\Ai@/%
    \extrairelepremi@r\Aj@\de\list@num\Figptpr@j-4:/\Aj@/%
    \f@gnewpath\PSwrit@cmdS{-3}{\c@mmoveto}{\fwf@g}{\X@un}{\Y@un}%
    \figvectPDD -1[-5,-4]%
    \@ecfor\Ak@:=\list@num\do{\Figptpr@j-5:/\Ak@/\figpttraDD-2:=-3/#1,-1/%
       \BdingB@xfalse\PSwrit@cmdS{-2}{}{\fwf@g}{\X@de}{\Y@de}%
       \figvectPDD -1[-3,-5]\figpttraDD-2:=-4/-#1,-1/%
       \PSwrit@cmdS{-2}{}{\fwf@g}{\X@tr}{\Y@tr}\BdingB@xtrue%
       \PSwrit@cmdS{-4}{\c@mcurveto}{\fwf@g}{\X@qu}{\Y@qu}%
       \B@zierBB@x{1}{\Y@un}(\X@un,\X@de,\X@tr,\X@qu)%
       \B@zierBB@x{2}{\X@un}(\Y@un,\Y@de,\Y@tr,\Y@qu)%
       \edef\X@un{\X@qu}\edef\Y@un{\Y@qu}\figptcopyDD-3:/-4/\figptcopyDD-4:/-5/}}
\ctr@ld@f\def\Psc@rveCPTD#1[#2]{\setc@ntr@l{2}%
    \def\list@num{#2}\extrairelepremi@r\Ak@\de\list@num%
    \extrairelepremi@r\Ai@\de\list@num\extrairelepremi@r\Aj@\de\list@num%
    \Figptpr@j-7:/\Ai@/%
    \f@gnewpath\PSwrit@cmd{-7}{\c@mmoveto}{\fwf@g}%
    \figvectPTD -9[\Ak@,\Aj@]%
    \@ecfor\Ak@:=\list@num\do{\figpttraTD-10:=\Ai@/#1,-9/%
       \figvectPTD -9[\Ai@,\Ak@]\figpttraTD-11:=\Aj@/-#1,-9/%
       \subB@zierTD\NBz@rcs[\Ai@,-10,-11,\Aj@]\edef\Ai@{\Aj@}\edef\Aj@{\Ak@}}}
\ctr@ld@f\def\figdrawend{\ifCUR@PS\ifGR@cri\immediate\closeout\fwf@g%
    \immediate\openout\fwf@g=\PSfilen@me\relax%
    \ifPDFm@ke\PSBdingB@x\else%
    \immediate\write\fwf@g{\p@urcent\string!PS-Adobe-2.0 EPSF-2.0}%
    \PShe@der{Creator\string: TeX (fig4tex.tex)}%
    \PShe@der{Title\string: \PSfilen@me}%
    \PShe@der{CreationDate\string: \the\day/\the\month/\the\year}%
    \PSBdingB@x%
    \PShe@der{EndComments}\PSdict@\fi%
    \immediate\write\fwf@g{\c@mgsave}%
    \openin\frf@g=\auxfilen@me\c@pypsfile\fwf@g\frf@g\closein\frf@g%
    \immediate\write\fwf@g{\c@mgrestore}%
    \PSc@mment{End of file.}\immediate\closeout\fwf@g%
    \immediate\openout\fwf@g=\auxfilen@me\immediate\closeout\fwf@g%
    \immediate\write16{File \PSfilen@me\space created.}\fi\fi\CUR@PSfalse\GR@critrue}
\ctr@ld@f\def\PShe@der#1{\immediate\write\fwf@g{\p@urcent\p@urcent#1}}
\ctr@ld@f\def\PSBdingB@x{{\v@lX=\ptT@ptps\c@@rdXmin\v@lY=\ptT@ptps\c@@rdYmin%
     \v@lXa=\ptT@ptps\c@@rdXmax\v@lYa=\ptT@ptps\c@@rdYmax%
     \PShe@der{BoundingBox\string: \repdecn@mb{\v@lX}\space\repdecn@mb{\v@lY}%
     \space\repdecn@mb{\v@lXa}\space\repdecn@mb{\v@lYa}}}}
\ctr@ld@f\def\figdrawfcconnect[#1]{{\ifCUR@PS\ifGR@cri\PSc@mment{fcconnect Points=#1}%
    \Q@s@tfillmode{no}\s@uvc@ntr@l\et@tpsfcconnect\resetc@ntr@l{2}%
    \fcc@nnect@[#1]\resetc@ntr@l\et@tpsfcconnect\PSc@mment{End fcconnect}\fi\fi}}
\ctr@ld@f\def\fcc@nnect@[#1]{\let\N@rm=\n@rmeucDD\def\list@num{#1}%
    \extrairelepremi@r\Ai@\de\list@num\edef\pr@m{\Ai@}\v@leur=\z@\p@rtent=\@ne\c@llgtot%
    \ifcase\fclin@typ@\edef\list@num{[\pr@m,#1,\Ai@}\expandafter\figdrawcurve\list@num]%
    \else\ifdim\fclin@r@d\p@>\z@\Pslin@conge[#1]\else\figdrawline[#1]\fi\fi%
    \v@leur=\@rrowp@s\v@leur\edef\list@num{#1,\Ai@,0}%
    \extrairelepremi@r\Ai@\de\list@num\mili@u=\epsil@n\c@llgpart%
    \advance\mili@u-\epsil@n\advance\mili@u-\delt@\advance\v@leur-\mili@u%
    \ifcase\fclin@typ@\invers@\mili@u\delt@%
    \ifnum\@rrowr@fpt>\z@\advance\delt@-\v@leur\v@leur=\delt@\fi%
    \v@leur=\repdecn@mb\v@leur\mili@u\edef\v@lt{\repdecn@mb\v@leur}%
    \extrairelepremi@r\Ak@\de\list@num%
    \figvectPDD-1[\pr@m,\Aj@]\figpttraDD-6:=\Ai@/\curv@roundness,-1/%
    \figvectPDD-1[\Ak@,\Ai@]\figpttraDD-7:=\Aj@/\curv@roundness,-1/%
    \delt@=\@rrowheadlength\p@\delt@=\C@AHANG\delt@\edef\R@dius{\repdecn@mb{\delt@}}%
    \ifcase\@rrowr@fpt%
    \FigptintercircB@zDD-8::\v@lt,\R@dius[\Ai@,-6,-7,\Aj@]\Q@arrowheadDD[-5,-8]\else%
    \FigptintercircB@zDD-8::\v@lt,\R@dius[\Aj@,-7,-6,\Ai@]\Q@arrowheadDD[-8,-5]\fi%
    \else\advance\delt@-\v@leur%
    \p@rtentiere{\p@rtent}{\delt@}\edef\C@efun{\the\p@rtent}%
    \p@rtentiere{\p@rtent}{\v@leur}\edef\C@efde{\the\p@rtent}%
    \figptbaryDD-5:[\Ai@,\Aj@;\C@efun,\C@efde]\ifcase\@rrowr@fpt%
    \delt@=\@rrowheadlength\unit@\delt@=\C@AHANG\delt@\edef\t@ille{\repdecn@mb{\delt@}}%
    \figvectPDD-2[\Ai@,\Aj@]\vecunit@{-2}{-2}\figpttraDD-5:=-5/\t@ille,-2/\fi%
    \Q@arrowheadDD[\Ai@,-5]\fi}
\ctr@ld@f\def\c@llgtot{\@ecfor\Aj@:=\list@num\do{\figvectP-1[\Ai@,\Aj@]\N@rm\delt@{-1}%
    \advance\v@leur\delt@\advance\p@rtent\@ne\edef\Ai@{\Aj@}}}
\ctr@ld@f\def\c@llgpart{\extrairelepremi@r\Aj@\de\list@num\figvectP-1[\Ai@,\Aj@]\N@rm\delt@{-1}%
    \advance\mili@u\delt@\ifdim\mili@u<\v@leur\edef\pr@m{\Ai@}\edef\Ai@{\Aj@}\c@llgpart\fi}
\ctr@ld@f\def\Pslin@conge[#1]{\ifnum\p@rtent>\tw@{\def\list@num{#1}%
    \extrairelepremi@r\Ai@\de\list@num\extrairelepremi@r\Aj@\de\list@num%
    \figptcopy-6:/\Ai@/\figvectP-3[\Ai@,\Aj@]\vecunit@{-3}{-3}\v@lmax=\result@t%
    \@ecfor\Ak@:=\list@num\do{\figvectP-4[\Aj@,\Ak@]\vecunit@{-4}{-4}%
    \minim@m\v@lmin\v@lmax\result@t\v@lmax=\result@t%
    \det@rm\delt@[-3,-4]\maxim@m\mili@u{\delt@}{-\delt@}\ifdim\mili@u>\Cepsil@n%
    \ifdim\delt@>\z@\figgetangleDD\Angl@[\Aj@,\Ak@,\Ai@]\else%
    \figgetangleDD\Angl@[\Aj@,\Ai@,\Ak@]\fi%
    \v@leur=\PI@deg\advance\v@leur-\Angl@\p@\divide\v@leur\tw@%
    \edef\Angl@{\repdecn@mb\v@leur}\c@ssin{\C@}{\S@}{\Angl@}\v@leur=\fclin@r@d\unit@%
    \v@leur=\S@\v@leur\mili@u=\C@\p@\invers@\mili@u\mili@u%
    \v@leur=\repdecn@mb{\mili@u}\v@leur%
    \minim@m\v@leur\v@leur\v@lmin\edef\t@ille{\repdecn@mb{\v@leur}}%
    \figpttra-5:=\Aj@/-\t@ille,-3/\figdrawline[-6,-5]\figpttra-6:=\Aj@/\t@ille,-4/%
    \figvectNVDD-3[-3]\figvectNVDD-8[-4]\inters@cDD-7:[-5,-3;-6,-8]%
    \ifdim\delt@>\z@\figdrawarccircP-7;\fclin@r@d[-5,-6]\else\figdrawarccircP-7;\fclin@r@d[-6,-5]\fi%
    \else\figdrawline[-6,\Aj@]\figptcopy-6:/\Aj@/\fi
    \edef\Ai@{\Aj@}\edef\Aj@{\Ak@}\figptcopy-3:/-4/}\figdrawline[-6,\Aj@]}\else\figdrawline[#1]\fi}
\ctr@ld@f\def\figdrawfcnode[#1]#2{{\ifCUR@PS\ifGR@cri\PSc@mment{fcnode Points=#1}%
    \s@uvc@ntr@l\et@tpsfcnode\resetc@ntr@l{2}%
    \def\t@xt@{#2}\ifx\t@xt@\empty\def\g@tt@xt{\setbox\Gb@x=\hbox{\Figg@tT{\p@int}}}%
    \else\def\g@tt@xt{\setbox\Gb@x=\hbox{#2}}\fi%
    \v@lmin=\h@rdfcXp@dd\advance\v@lmin\Xp@dd\unit@\multiply\v@lmin\tw@%
    \v@lmax=\h@rdfcYp@dd\advance\v@lmax\Yp@dd\unit@\multiply\v@lmax\tw@%
    \Figv@ctCreg-8(\unit@,-\unit@)\def\list@num{#1}%
    \delt@=\CUR@width bp\divide\delt@\tw@%
    \fcn@de\PSc@mment{End fcnode}\resetc@ntr@l\et@tpsfcnode\fi\fi}}
\ctr@ld@f\def\d@butn@de{\g@tt@xt\v@lX=\wd\Gb@x%
    \v@lY=\ht\Gb@x\advance\v@lY\dp\Gb@x\advance\v@lX\v@lmin\advance\v@lY\v@lmax}
\ctr@ld@f\def\fcn@deE{%
    \@ecfor\p@int:=\list@num\do{\d@butn@de\v@lX=\unssqrttw@\v@lX\v@lY=\unssqrttw@\v@lY%
    \ifdim\thickn@ss\p@>\z@
    \v@lXa=\v@lX\advance\v@lXa\delt@\v@lXa=\ptT@unit@\v@lXa\edef\XR@d{\repdecn@mb\v@lXa}%
    \v@lYa=\v@lY\advance\v@lYa\delt@\v@lYa=\ptT@unit@\v@lYa\edef\YR@d{\repdecn@mb\v@lYa}%
    \arct@n\v@leur(\v@lXa,\v@lYa)\v@leur=\rdT@deg\v@leur\edef\@nglde{\repdecn@mb\v@leur}%
    {\c@lptellDD-2::\p@int;\XR@d,\YR@d(\@nglde)}
    \advance\v@leur-\PI@deg\edef\@nglun{\repdecn@mb\v@leur}%
    {\c@lptellDD-3::\p@int;\XR@d,\YR@d(\@nglun)}%
    \figptstra-6=-3,-2,\p@int/\thickn@ss,-8/\Q@s@tfillmode{yes}%
    \Pss@tspecifSt{color=\DDV@thickcolor}%
    \figdrawline[-2,-3,-6,-5]\figdrawarcell-4;\XR@d,\YR@d(\@nglun,\@nglde,0)%
    \Psrest@reSt{color=\DDV@thickcolor}\fi
    \v@lX=\ptT@unit@\v@lX\v@lY=\ptT@unit@\v@lY%
    \edef\XR@d{\repdecn@mb\v@lX}\edef\YR@d{\repdecn@mb\v@lY}%
    \Q@s@tfillmode{yes}\Pss@tspecifSt{color=\fcbgc@lor}%
    \figdrawarcell\p@int;\XR@d,\YR@d(0,360,0)%
    \Q@s@tfillmode{no}\Psrest@reSt{color=\fcbgc@lor}\figdrawarcell\p@int;\XR@d,\YR@d(0,360,0)}}
\ctr@ld@f\def\fcn@deL{\delt@=\ptT@unit@\delt@\edef\t@ille{\repdecn@mb\delt@}%
    \@ecfor\p@int:=\list@num\do{\Figg@tXYa{\p@int}\d@butn@de%
    \ifdim\v@lX>\v@lY\itis@Ktrue\else\itis@Kfalse\fi%
    \advance\v@lXa-\v@lX\Figp@intreg-1:(\v@lXa,\v@lYa)%
    \advance\v@lXa\v@lX\advance\v@lYa-\v@lY\Figp@intreg-2:(\v@lXa,\v@lYa)%
    \advance\v@lXa\v@lX\advance\v@lYa\v@lY\Figp@intreg-3:(\v@lXa,\v@lYa)%
    \advance\v@lXa-\v@lX\advance\v@lYa\v@lY\Figp@intreg-4:(\v@lXa,\v@lYa)%
    \ifdim\thickn@ss\p@>\z@
    \Figg@tXYa{\p@int}\Q@s@tfillmode{yes}\Pss@tspecifSt{color=\DDV@thickcolor}%
    \c@lpt@xt{-1}{-4}\c@lpt@xt@\v@lXa\v@lYa\v@lX\v@lY\c@rre\delt@%
    \Figp@intregDD-9:(\v@lZ,\v@lYa)\Figp@intregDD-11:(\v@lZa,\v@lYa)%
    \c@lpt@xt{-4}{-3}\c@lpt@xt@\v@lYa\v@lXa\v@lY\v@lX\delt@\c@rre%
    \Figp@intregDD-12:(\v@lXa,\v@lZ)\Figp@intregDD-10:(\v@lXa,\v@lZa)%
    \ifitis@K\figptstra-7=-9,-10,-11/\thickn@ss,-8/\figdrawline[-9,-11,-5,-6,-7]\else%
    \figptstra-7=-10,-11,-12/\thickn@ss,-8/\figdrawline[-10,-12,-5,-6,-7]\fi%
    \Psrest@reSt{color=\DDV@thickcolor}\fi
    \Q@s@tfillmode{yes}\Pss@tspecifSt{color=\fcbgc@lor}\figdrawline[-1,-2,-3,-4]%
    \Q@s@tfillmode{no}\Psrest@reSt{color=\fcbgc@lor}\figdrawline[-1,-2,-3,-4,-1]}}
\ctr@ld@f\def\c@lpt@xt#1#2{\figvectN-7[#1,#2]\vecunit@{-7}{-7}\figpttra-5:=#1/\t@ille,-7/%
    \figvectP-7[#1,#2]\Figg@tXY{-7}\c@rre=\v@lX\delt@=\v@lY\Figg@tXY{-5}}
\ctr@ld@f\def\c@lpt@xt@#1#2#3#4#5#6{\v@lZ=#6\invers@{\v@lZ}{\v@lZ}\v@leur=\repdecn@mb{#5}\v@lZ%
    \v@lZ=#2\advance\v@lZ-#4\mili@u=\repdecn@mb{\v@leur}\v@lZ%
    \v@lZ=#3\advance\v@lZ\mili@u\v@lZa=-\v@lZ\advance\v@lZa\tw@#1}
\ctr@ld@f\def\fcn@deR{\@ecfor\p@int:=\list@num\do{\Figg@tXYa{\p@int}\d@butn@de%
    \advance\v@lXa-0.5\v@lX\advance\v@lYa-0.5\v@lY\Figp@intreg-1:(\v@lXa,\v@lYa)%
    \advance\v@lXa\v@lX\Figp@intreg-2:(\v@lXa,\v@lYa)%
    \advance\v@lYa\v@lY\Figp@intreg-3:(\v@lXa,\v@lYa)%
    \advance\v@lXa-\v@lX\Figp@intreg-4:(\v@lXa,\v@lYa)%
    \ifdim\thickn@ss\p@>\z@
    \Q@s@tfillmode{yes}\Pss@tspecifSt{color=\DDV@thickcolor}%
    \Figv@ctCreg-5(-\delt@,-\delt@)\figpttra-9:=-1/1,-5/%
    \Figv@ctCreg-5(\delt@,-\delt@)\figpttra-10:=-2/1,-5/%
    \Figv@ctCreg-5(\delt@,\delt@)\figpttra-11:=-3/1,-5/%
    \figptstra-7=-9,-10,-11/\thickn@ss,-8/\figdrawline[-9,-11,-5,-6,-7]%
    \Psrest@reSt{color=\DDV@thickcolor}\fi
    \Q@s@tfillmode{yes}\Pss@tspecifSt{color=\fcbgc@lor}\figdrawline[-1,-2,-3,-4]%
    \Q@s@tfillmode{no}\Psrest@reSt{color=\fcbgc@lor}\figdrawline[-1,-2,-3,-4,-1]}}
\ctr@ld@f\def\Pssetfl@wchart#1=#2|{\keln@mtr#1|%
    \def\n@mref{arr}\ifx\l@debut\n@mref\expandafter\keln@mtr\l@suite|%
     \def\n@mref{owp}\ifx\l@debut\n@mref\update@ttr\D@FTfcarrowposition\P@setfcarrowposition{#2}\else
     \def\n@mref{owr}\ifx\l@debut\n@mref\update@ttr\D@FTfcarrowrefpt\P@setfcarrowrefpt{#2}\else
     \W@rnmesAttr{figset flowchart}{#1}\fi\fi\else%
    \def\n@mref{bgc}\ifx\l@debut\n@mref\update@ttr\D@FTfcbgcolor\P@setfcbgcolor{#2}\else
    \def\n@mref{lin}\ifx\l@debut\n@mref\update@ttr\D@FTfcline\P@setfcline{#2}\else
    \def\n@mref{pad}\ifx\l@debut\n@mref\update@ttr\D@FTfcxpadding\P@setfcxpadding{#2}%
                                       \update@ttr\D@FTfcypadding\P@setfcypadding{#2}\else
    \def\n@mref{rad}\ifx\l@debut\n@mref\update@ttr\D@FTfcradius\P@setfcradius{#2}\else
    \def\n@mref{sha}\ifx\l@debut\n@mref\update@ttr\D@FTfcshape\P@setfcshape{#2}\else
    \def\n@mref{thi}\ifx\l@debut\n@mref\expandafter\keln@mtr\l@suite|%
     \def\n@mref{ckc}\ifx\l@debut\n@mref\update@ttr\D@FTref\P@setfcthickcolor{#2}\else
     \def\n@mref{ckn}\ifx\l@debut\n@mref\update@ttr\D@FTfcthickness\P@setfcthickness{#2}\else
     \W@rnmesAttr{figset flowchart}{#1}\fi\fi\else%
    \def\n@mref{xpa}\ifx\l@debut\n@mref\update@ttr\D@FTfcxpadding\P@setfcxpadding{#2}\else
    \def\n@mref{ypa}\ifx\l@debut\n@mref\update@ttr\D@FTfcypadding\P@setfcypadding{#2}\else
    \W@rnmesAttr{figset flowchart}{#1}\fi\fi\fi\fi\fi\fi\fi\fi\fi}
\ctr@ln@m\@rrowp@s
\ctr@ld@f\def\P@setfcarrowposition#1{\edef\@rrowp@s{#1}}
\ctr@ln@m\@rrowr@fpt
\ctr@ld@f\def\P@setfcarrowrefpt#1{\setfcr@fpt#1|}
\ctr@ld@f\def\setfcr@fpt#1#2|{\if#1e\def\@rrowr@fpt{1}\else\def\@rrowr@fpt{0}\fi}
\ctr@ln@m\fcbgc@lor
\ctr@ld@f\def\P@setfcbgcolor#1{\edef\fcbgc@lor{#1}}
\ctr@ln@m\fclin@typ@
\ctr@ld@f\def\P@setfcline#1{\setfccurv@#1|}
\ctr@ld@f\def\setfccurv@#1#2|{\if#1c\def\fclin@typ@{0}\else\def\fclin@typ@{1}\fi}
\ctr@ln@m\fclin@r@d
\ctr@ld@f\def\P@setfcradius#1{\edef\fclin@r@d{#1}}
\ctr@ln@m\fcn@de \ctr@ln@m\fcsh@pe
\ctr@ln@m\h@rdfcXp@dd \ctr@ln@m\h@rdfcYp@dd
\ctr@ld@f\def\P@setfcshape#1{\setfcshap@#1|}
\ctr@ld@f\def\setfcshap@#1#2|{%
    \if#1e\let\fcn@de=\fcn@deE\def\h@rdfcXp@dd{4pt}\def\h@rdfcYp@dd{4pt}%
     \edef\fcsh@pe{ellipse}\else%
    \if#1l\let\fcn@de=\fcn@deL\def\h@rdfcXp@dd{4pt}\def\h@rdfcYp@dd{4pt}%
     \edef\fcsh@pe{lozenge}\else%
          \let\fcn@de=\fcn@deR\def\h@rdfcXp@dd{6pt}\def\h@rdfcYp@dd{6pt}%
     \edef\fcsh@pe{rectangle}\fi\fi}
\ctr@ln@m\DDV@thickcolor
\ctr@ld@f\def\P@setfcthickcolor#1{\edef\DDV@thickcolor{#1}}
\ctr@ln@m\thickn@ss
\ctr@ld@f\def\P@setfcthickness#1{\edef\thickn@ss{#1}}
\ctr@ln@m\Xp@dd
\ctr@ld@f\def\P@setfcxpadding#1{\edef\Xp@dd{#1}}
\ctr@ln@m\Yp@dd
\ctr@ld@f\def\P@setfcypadding#1{\edef\Yp@dd{#1}}
\ctr@ld@f\def\figdrawline[#1]{{\ifCUR@PS\ifGR@cri\PSc@mment{line Points=#1}%
    \let\figdrawlign@=\Pslign@P\Pslin@{#1}\PSc@mment{End line}\fi\fi}}
\ctr@ld@f\def\figdrawlineF#1{{\ifCUR@PS\ifGR@cri\PSc@mment{lineF Filename=#1}%
    \let\figdrawlign@=\Pslign@F\Pslin@{#1}\PSc@mment{End lineF}\fi\fi}}
\ctr@ld@f\def\figdrawlineC(#1){{\ifCUR@PS\ifGR@cri\PSc@mment{lineC}%
    \let\figdrawlign@=\Pslign@C\Pslin@{#1}\PSc@mment{End lineC}\fi\fi}}
\ctr@ld@f\def\Pslin@#1{\iffillm@de\figdrawlign@{#1}%
    \f@gfill%
    \else\figdrawlign@{#1}\ifx\derp@int\premp@int%
    \f@gclosestroke%
    \else\f@gstroke\fi\fi}
\ctr@ld@f\def\Pslign@P#1{\def\list@num{#1}\extrairelepremi@r\p@int\de\list@num%
    \edef\premp@int{\p@int}\f@gnewpath%
    \PSwrit@cmd{\p@int}{\c@mmoveto}{\fwf@g}%
    \@ecfor\p@int:=\list@num\do{\PSwrit@cmd{\p@int}{\c@mlineto}{\fwf@g}%
    \edef\derp@int{\p@int}}}
\ctr@ld@f\def\Pslign@F#1{\s@uvc@ntr@l\et@tPslign@F\setc@ntr@l{2}\openin\frf@g=#1\relax%
    \ifeof\frf@g\message{*** File #1 not found !}\end\else%
    \read\frf@g to\tr@c\edef\premp@int{\tr@c}\expandafter\extr@ctCF\tr@c:%
    \f@gnewpath\PSwrit@cmd{-1}{\c@mmoveto}{\fwf@g}%
    \loop\read\frf@g to\tr@c\ifeof\frf@g\mored@tafalse\else\mored@tatrue\fi%
    \ifmored@ta\expandafter\extr@ctCF\tr@c:\PSwrit@cmd{-1}{\c@mlineto}{\fwf@g}%
    \edef\derp@int{\tr@c}\repeat\fi\closein\frf@g\resetc@ntr@l\et@tPslign@F}
\ctr@ln@m\extr@ctCF
\ctr@ld@f\def\extr@ctCFDD#1 #2:{\v@lX=#1\unit@\v@lY=#2\unit@\Figp@intregDD-1:(\v@lX,\v@lY)}
\ctr@ld@f\def\extr@ctCFTD#1 #2 #3:{\v@lX=#1\unit@\v@lY=#2\unit@\v@lZ=#3\unit@%
    \Figp@intregTD-1:(\v@lX,\v@lY,\v@lZ)}
\ctr@ld@f\def\Pslign@C#1{\s@uvc@ntr@l\et@tPslign@C\setc@ntr@l{2}%
    \def\list@num{#1}\extrairelepremi@r\p@int\de\list@num%
    \edef\premp@int{\p@int}\f@gnewpath%
    \expandafter\Pslign@C@\p@int:\PSwrit@cmd{-1}{\c@mmoveto}{\fwf@g}%
    \@ecfor\p@int:=\list@num\do{\expandafter\Pslign@C@\p@int:%
    \PSwrit@cmd{-1}{\c@mlineto}{\fwf@g}\edef\derp@int{\p@int}}%
    \resetc@ntr@l\et@tPslign@C}
\ctr@ld@f\def\Pslign@C@#1 #2:{{\def\t@xt@{#1}\ifx\t@xt@\empty\Pslign@C@#2:
    \else\extr@ctCF#1 #2:\fi}}
\ctr@ld@f\def\Pssetm@sh#1=#2|{\keln@mde#1|%
    \def\n@mref{co}\ifx\l@debut\n@mref\update@ttr\D@FTref\P@setmeshcolor{#2}\else
    \def\n@mref{da}\ifx\l@debut\n@mref\update@ttr\D@FTref\P@setmeshdash{#2}\else
    \def\n@mref{di}\ifx\l@debut\n@mref\update@ttr\D@FTmeshdiag\Q@s@tmeshdiag{#2}\else
    \def\n@mref{wi}\ifx\l@debut\n@mref\update@ttr\D@FTref\P@setmeshwidth{#2}\else
    \W@rnmesAttr{figset mesh}{#1}\fi\fi\fi\fi}
\ctr@ln@m\c@ntrolmesh
\ctr@ld@f\def\Q@s@tmeshdiag#1{\edef\c@ntrolmesh{#1}}
\ctr@ld@f\def\D@FTmeshdiag{0}    
\ctr@ln@m\DDV@meshcolor
\ctr@ld@f\def\P@setmeshcolor#1{\edef\DDV@meshcolor{#1}}
\ctr@ln@m\DDV@meshdash
\ctr@ld@f\def\P@setmeshdash#1{\edef\DDV@meshdash{#1}}
\ctr@ln@m\DDV@meshwidth
\ctr@ld@f\def\P@setmeshwidth#1{\edef\DDV@meshwidth{#1}}
\ctr@ld@f\def\figdrawmesh#1,#2[#3,#4,#5,#6]{{\ifCUR@PS\ifGR@cri%
    \PSc@mment{mesh N1=#1, N2=#2, Quadrangle=[#3,#4,#5,#6]}\s@uvc@ntr@l\et@tpsmesh%
    \Pss@tspecifSt{color=\DDV@meshcolor,dash=\DDV@meshdash,width=\DDV@meshwidth}%
    \setc@ntr@l{2}%
    \ifnum#1>\@ne\Psmeshp@rt#1[#3,#4,#5,#6]\fi%
    \ifnum#2>\@ne\Psmeshp@rt#2[#4,#5,#6,#3]\fi%
    \ifnum\c@ntrolmesh>\z@\Psmeshdi@g#1,#2[#3,#4,#5,#6]\fi%
    \ifnum\c@ntrolmesh<\z@\Psmeshdi@g#2,#1[#4,#5,#6,#3]\fi%
    \Psrest@reSt{color=\DDV@meshcolor,dash=\DDV@meshdash,width=\DDV@meshwidth}%
    \figdrawline[#3,#4,#5,#6,#3]\PSc@mment{End mesh}\resetc@ntr@l\et@tpsmesh\fi\fi}}
\ctr@ld@f\def\Psmeshp@rt#1[#2,#3,#4,#5]{{\l@mbd@un=\@ne\l@mbd@de=#1\loop%
    \ifnum\l@mbd@un<#1\advance\l@mbd@de\m@ne\figptbary-1:[#2,#3;\l@mbd@de,\l@mbd@un]%
    \figptbary-2:[#5,#4;\l@mbd@de,\l@mbd@un]\figdrawline[-1,-2]\advance\l@mbd@un\@ne\repeat}}
\ctr@ld@f\def\Psmeshdi@g#1,#2[#3,#4,#5,#6]{\figptcopy-2:/#3/\figptcopy-3:/#6/%
    \l@mbd@un=\z@\l@mbd@de=#1\loop\ifnum\l@mbd@un<#1%
    \advance\l@mbd@un\@ne\advance\l@mbd@de\m@ne\figptcopy-1:/-2/\figptcopy-4:/-3/%
    \figptbary-2:[#3,#4;\l@mbd@de,\l@mbd@un]%
    \figptbary-3:[#6,#5;\l@mbd@de,\l@mbd@un]\Psmeshdi@gp@rt#2[-1,-2,-3,-4]\repeat}
\ctr@ld@f\def\Psmeshdi@gp@rt#1[#2,#3,#4,#5]{{\l@mbd@un=\z@\l@mbd@de=#1\loop%
    \ifnum\l@mbd@un<#1\figptbary-5:[#2,#5;\l@mbd@de,\l@mbd@un]%
    \advance\l@mbd@de\m@ne\advance\l@mbd@un\@ne%
    \figptbary-6:[#3,#4;\l@mbd@de,\l@mbd@un]\figdrawline[-5,-6]\repeat}}
\ctr@ln@m\figdrawnormal
\ctr@ld@f\def\Q@normalDD#1,#2[#3,#4]{{\ifCUR@PS\ifGR@cri%
    \PSc@mment{normal Length=#1, Lambda=#2 [Pt1,Pt2]=[#3,#4]}%
    \s@uvc@ntr@l\et@tpsnormal\resetc@ntr@l{2}\figptendnormal-6::#1,#2[#3,#4]%
    \figptcopyDD-5:/-1/\figdrawarrow[-5,-6]%
    \PSc@mment{End normal}\resetc@ntr@l\et@tpsnormal\fi\fi}}
\ctr@ld@f\def\figreset#1{\trtlis@rg{#1}{\Psreset@}}
\ctr@ld@f\def\Psreset@#1|{\def\t@xt@{#1}\ifx\t@xt@\empty\P@resetg@n
    \else\keln@mde#1|%
    \def\n@mref{al}\ifx\l@debut\n@mref%
        \figset altitude(blcolor=default,bldash=default,blwidth=default,%
        sqcolor=default,sqdash=default,sqwidth=default)\else
    \def\n@mref{ar}\ifx\l@debut\n@mref\figresetarrowhead\else
    \def\n@mref{cu}\ifx\l@debut\n@mref\figset curve(roundness=\D@FTroundness)\else
    \def\n@mref{ge}\ifx\l@debut\n@mref\P@resetg@n\else
    \def\n@mref{fl}\ifx\l@debut\n@mref%
        \figset flowchart(arrowp=\D@FTfcarrowposition,arrowr=\D@FTfcarrowrefpt,%
	bgcolor=\D@FTfcbgcolor,line=\D@FTfcline,radius=\D@FTfcradius,%
	shape=\D@FTfcshape,thickcolor=default,thickness=\D@FTfcthickness,%
	xpadd=\D@FTfcxpadding,ypadd=\D@FTfcypadding)\else
    \def\n@mref{me}\ifx\l@debut\n@mref\figset mesh(diag=\D@FTmeshdiag,%
        color=default,dash=default,width=default)\else
    \def\n@mref{tr}\ifx\l@debut\n@mref%
        \figset trimesh(color=default,dash=default,width=default)\else
    \W@rnmeskwd{figreset}{#1}\fi\fi\fi\fi\fi\fi\fi\fi}
\ctr@ld@f\def\P@resetg@n{\figset (color=\D@FTcolor,dash=\D@FTdash,fill=\D@FTfill,%
    join=\D@FTjoin,width=\D@FTwidth)}
\ctr@ld@f\def\figset#1(#2){\def\t@xt@{#1}\ifx\t@xt@\empty\trtlis@rg{#2}{\Pssetg@n}
    \else\keln@mde#1|%
    \def\n@mref{al}\ifx\l@debut\n@mref\trtlis@rg{#2}{\Psset@lti}\else
    \def\n@mref{ar}\ifx\l@debut\n@mref\trtlis@rg{#2}{\Psset@rrowhe@d}\else
    \def\n@mref{cu}\ifx\l@debut\n@mref\trtlis@rg{#2}{\Pssetc@rve}\else
    \def\n@mref{fl}\ifx\l@debut\n@mref\trtlis@rg{#2}{\Pssetfl@wchart}\else
    \def\n@mref{ge}\ifx\l@debut\n@mref\trtlis@rg{#2}{\Pssetg@n}\else
    \def\n@mref{me}\ifx\l@debut\n@mref\trtlis@rg{#2}{\Pssetm@sh}\else
    \def\n@mref{pr}\ifx\l@debut\n@mref\ifCUR@PS\W@rnmesIgn{figset proj(...)}%
     \else\trtlis@rg{#2}{\Figsetpr@j}\fi\else
    \def\n@mref{tr}\ifx\l@debut\n@mref\trtlis@rg{#2}{\Pssettrim@sh}\else
    \def\n@mref{wr}\ifx\l@debut\n@mref\let\M@cro=\Figsetwr@te\trtlis@rgtok{#2,|}\else
    \W@rnmeskwd{figset}{#1}\fi\fi\fi\fi\fi\fi\fi\fi\fi\fi\ignorespaces}
\ctr@ld@f\def\figsetdefault#1(#2){\ifCUR@PS\W@rnmesIgn{figsetdefault}\else%
    \def\t@xt@{#1}\ifx\t@xt@\empty\trtlis@rg{#2}{\Pssd@g@n}\else\keln@mun#1|
    \def\n@mref{a}\ifx\l@debut\n@mref\trtlis@rg{#2}{\Pssd@@rrowhe@d}\else
    \def\n@mref{c}\ifx\l@debut\n@mref\trtlis@rg{#2}{\Pssd@c@rve}\else
    \def\n@mref{g}\ifx\l@debut\n@mref\trtlis@rg{#2}{\Pssd@g@n}\else
    \def\n@mref{f}\ifx\l@debut\n@mref\trtlis@rg{#2}{\Pssd@fl@wchart}\else
    \def\n@mref{m}\ifx\l@debut\n@mref\trtlis@rg{#2}{\Pssd@m@sh}\else
    \W@rnmeskwd{figsetdefault}{#1}\fi\fi\fi\fi\fi\fi\initpss@ttings\fi}
\ctr@ld@f\def\Pssd@g@n#1=#2|{\keln@mun#1|%
    \def\n@mref{c}\ifx\l@debut\n@mref\edef\D@FTcolor{#2}\else
    \def\n@mref{d}\ifx\l@debut\n@mref\edef\D@FTdash{#2}\else
    \def\n@mref{f}\ifx\l@debut\n@mref\edef\D@FTfill{#2}\else
    \def\n@mref{j}\ifx\l@debut\n@mref\edef\D@FTjoin{#2}\else
    \def\n@mref{u}\ifx\l@debut\n@mref\edef\D@FTupdate{#2}\Q@s@tupdate{#2}\else
    \def\n@mref{w}\ifx\l@debut\n@mref\edef\D@FTwidth{#2}\else
    \W@rnmesAttr{figsetdefault}{#1}\fi\fi\fi\fi\fi\fi}
\ctr@ld@f\def\Pssd@@rrowhe@d#1=#2|{\keln@mun#1|%
    \def\n@mref{a}\ifx\l@debut\n@mref\edef\D@FTarrowheadangle{#2}\else
    \def\n@mref{f}\ifx\l@debut\n@mref\edef\D@FTarrowheadfill{#2}\else
    \def\n@mref{l}\ifx\l@debut\n@mref\y@tiunit{#2}\ifunitpr@sent%
     \edef\D@FTh@rdahlength{#2}\else\edef\D@FTh@rdahlength{#2pt}%
     \message{*** \BS@ figsetdefault (..., #1=#2, ...) : unit is missing, pt is assumed.}%
     \fi\else
    \def\n@mref{o}\ifx\l@debut\n@mref\edef\D@FTarrowheadout{#2}\else
    \def\n@mref{r}\ifx\l@debut\n@mref\edef\D@FTarrowheadratio{#2}\else
    \W@rnmesAttr{figsetdefault arrowhead}{#1}\fi\fi\fi\fi\fi}
\ctr@ld@f\def\Pssd@c@rve#1=#2|{\keln@mun#1|%
    \def\n@mref{r}\ifx\l@debut\n@mref\edef\D@FTroundness{#2}\else%
    \W@rnmesAttr{figsetdefault curve}{#1}\fi}
\ctr@ld@f\def\Pssd@fl@wchart#1=#2|{\keln@mtr#1|%
    \def\n@mref{arr}\ifx\l@debut\n@mref\expandafter\keln@mtr\l@suite|%
     \def\n@mref{owp}\ifx\l@debut\n@mref\edef\D@FTfcarrowposition{#2}\else
     \def\n@mref{owr}\ifx\l@debut\n@mref\edef\D@FTfcarrowrefpt{#2}\else
                     \W@rnmesAttr{figsetdefault flowchart}{#1}\fi\fi\else%
    \def\n@mref{bgc}\ifx\l@debut\n@mref\edef\D@FTfcbgcolor{#2}\else
    \def\n@mref{lin}\ifx\l@debut\n@mref\edef\D@FTfcline{#2}\else
    \def\n@mref{pad}\ifx\l@debut\n@mref\edef\D@FTfcxpadding{#2}%
                    \edef\D@FTfcypadding{#2}\else
    \def\n@mref{rad}\ifx\l@debut\n@mref\edef\D@FTfcradius{#2}\else
    \def\n@mref{sha}\ifx\l@debut\n@mref\edef\D@FTfcshape{#2}\else
    \def\n@mref{thi}\ifx\l@debut\n@mref\expandafter\keln@mtr\l@suite|%
     \def\n@mref{ckn}\ifx\l@debut\n@mref\edef\D@FTfcthickness{#2}\else
                     \W@rnmesAttr{figsetdefault flowchart}{#1}\fi\else%
    \def\n@mref{xpa}\ifx\l@debut\n@mref\edef\D@FTfcxpadding{#2}\else
    \def\n@mref{ypa}\ifx\l@debut\n@mref\edef\D@FTfcypadding{#2}\else
    \W@rnmesAttr{figsetdefault flowchart}{#1}\fi\fi\fi\fi\fi\fi\fi\fi\fi}
\ctr@ld@f\def\D@FTfcarrowposition{0.5}
\ctr@ld@f\def\D@FTfcarrowrefpt{start}
\ctr@ld@f\def\D@FTfcbgcolor{1}
\ctr@ld@f\def\D@FTfcline{polygon}
\ctr@ld@f\def\D@FTfcradius{0}
\ctr@ld@f\def\D@FTfcshape{rectangle}
\ctr@ld@f\def\D@FTfcthickness{0}
\ctr@ld@f\def\D@FTfcxpadding{0}
\ctr@ld@f\def\D@FTfcypadding{0}
\ctr@ld@f\def\Pssd@m@sh#1=#2|{\keln@mun#1|%
    \def\n@mref{d}\ifx\l@debut\n@mref\edef\D@FTmeshdiag{#2}\else%
    \W@rnmesAttr{figsetdefault mesh}{#1}\fi}
\ctr@ln@w{newif}\iffillm@de
\ctr@ld@f\def\Q@s@tfillmode#1{\expandafter\setfillm@de#1:}
\ctr@ld@f\def\setfillm@de#1#2:{\if#1n\fillm@defalse\else\fillm@detrue\fi}
\ctr@ld@f\def\D@FTfill{no}     
\ctr@ln@w{newif}\ifGRupdatem@de
\ctr@ld@f\def\Q@s@tupdate#1{\ifCUR@PS\W@rnmesIgn{figset (update=...)}%
    \else\expandafter\setupd@te#1:\fi}
\ctr@ld@f\def\setupd@te#1#2:{\if#1n\GRupdatem@defalse\else\GRupdatem@detrue\fi}
\ctr@ld@f\def\D@FTupdate{no}     
\ctr@ln@m\CUR@color \ctr@ln@m\CUR@colorc@md
\ctr@ld@f\def\s@uvcolor#1{\edef#1{\CUR@color}}
\ctr@ld@f\def\D@FTcolor{0}       
\ctr@ld@f\def\Pssetc@lor#1{\ifGR@cri\result@tent=\@ne\expandafter\c@lnbV@l#1 :%
    \def\CUR@color{}\def\CUR@colorc@md{}%
    \ifcase\result@tent\or\Q@s@tgray{#1}\or\or\Q@s@trgb{#1}\or\Q@s@tcmyk{#1}\fi\fi}
\ctr@ln@m\CUR@colorc@mdStroke
\ctr@ld@f\def\Q@s@tcmyk#1{\ifGR@cri\def\CUR@color{#1}\def\CUR@colorc@md{\c@msetcmykcolor}%
    \def\CUR@colorc@mdStroke{\c@msetcmykcolorStroke}%
    \ifCUR@PS\PSc@mment{setcmyk Color=#1}\us@primarC@lor\fi\fi}
\ctr@ld@f\def\Q@s@trgb#1{\ifGR@cri\def\CUR@color{#1}\def\CUR@colorc@md{\c@msetrgbcolor}%
    \def\CUR@colorc@mdStroke{\c@msetrgbcolorStroke}%
    \ifCUR@PS\PSc@mment{setrgb Color=#1}\us@primarC@lor\fi\fi}
\ctr@ld@f\def\Q@s@tgray#1{\ifGR@cri\def\CUR@color{#1}\def\CUR@colorc@md{\c@msetgray}%
    \def\CUR@colorc@mdStroke{\c@msetgrayStroke}%
    \ifCUR@PS\PSc@mment{setgray Gray level=#1}\us@primarC@lor\fi\fi}
\ctr@ln@m\fillc@md
\ctr@ld@f\def\us@primarC@lor{\immediate\write\fwf@g{\d@fprimarC@lor}%
    \let\fillc@md=\prfillc@md}
\ctr@ld@f\def\prfillc@md{\d@fprimarC@lor\space\c@mfill}
\ctr@ld@f\def\c@lnbV@l#1 #2:{\def\t@xt@{#1}\relax\ifx\t@xt@\empty\c@lnbV@l#2:
    \else\c@lnbV@l@#1 #2:\fi}
\ctr@ld@f\def\c@lnbV@l@#1 #2:{\def\t@xt@{#2}\ifx\t@xt@\empty%
    \def\t@xt@{#1}\ifx\t@xt@\empty\advance\result@tent\m@ne\fi
    \else\advance\result@tent\@ne\c@lnbV@l@#2:\fi}
\ctr@ld@f\def\Blackcmyk{0 0 0 1}
\ctr@ld@f\def\Whitecmyk{0 0 0 0}
\ctr@ld@f\def\Cyancmyk{1 0 0 0}
\ctr@ld@f\def\Magentacmyk{0 1 0 0}
\ctr@ld@f\def\Yellowcmyk{0 0 1 0}
\ctr@ld@f\def\Redcmyk{0 1 1 0}
\ctr@ld@f\def\Greencmyk{1 0 1 0}
\ctr@ld@f\def\Bluecmyk{1 1 0 0}
\ctr@ld@f\def\Graycmyk{0 0 0 0.50}
\ctr@ld@f\def\BrickRedcmyk{0 0.89 0.94 0.28} 
\ctr@ld@f\def\Browncmyk{0 0.81 1 0.60} 
\ctr@ld@f\def\ForestGreencmyk{0.91 0 0.88 0.12} 
\ctr@ld@f\def\Goldenrodcmyk{ 0 0.10 0.84 0} 
\ctr@ld@f\def\Marooncmyk{0 0.87 0.68 0.32} 
\ctr@ld@f\def\Orangecmyk{0 0.61 0.87 0} 
\ctr@ld@f\def\Purplecmyk{0.45 0.86 0 0} 
\ctr@ld@f\def\RoyalBluecmyk{1. 0.50 0 0} 
\ctr@ld@f\def\Violetcmyk{0.79 0.88 0 0} 
\ctr@ld@f\def\Blackrgb{0 0 0}
\ctr@ld@f\def\Whitergb{1 1 1}
\ctr@ld@f\def\Redrgb{1 0 0}
\ctr@ld@f\def\Greenrgb{0 1 0}
\ctr@ld@f\def\Bluergb{0 0 1}
\ctr@ld@f\def\Cyanrgb{0 1 1}
\ctr@ld@f\def\Magentargb{1 0 1}
\ctr@ld@f\def\Yellowrgb{1 1 0}
\ctr@ld@f\def\Grayrgb{0.5 0.5 0.5}
\ctr@ld@f\def\Chocolatergb{0.824 0.412 0.118}
\ctr@ld@f\def\DarkGoldenrodrgb{0.722 0.525 0.043}
\ctr@ld@f\def\DarkOrangergb{1 0.549 0}
\ctr@ld@f\def\Firebrickrgb{0.698 0.133 0.133}
\ctr@ld@f\def\ForestGreenrgb{0.133 0.545 0.133}
\ctr@ld@f\def\Goldrgb{1 0.843 0}
\ctr@ld@f\def\HotPinkrgb{1 0.412 0.706}
\ctr@ld@f\def\Maroonrgb{0.690 0.188 0.376}
\ctr@ld@f\def\Pinkrgb{1 0.753 0.796}
\ctr@ld@f\def\RoyalBluergb{0.255 0.412 0.882}
\ctr@ld@f\def\Pssetg@n#1=#2|{\keln@mun#1|%
    \def\n@mref{c}\ifx\l@debut\n@mref\update@ttr\D@FTcolor\Pssetc@lor{#2}\else
    \def\n@mref{d}\ifx\l@debut\n@mref\update@ttr\D@FTdash\Q@s@tdash{#2}\else
    \def\n@mref{f}\ifx\l@debut\n@mref\update@ttr\D@FTfill\Q@s@tfillmode{#2}\else
    \def\n@mref{j}\ifx\l@debut\n@mref\update@ttr\D@FTjoin\Q@s@tjoin{#2}\else
    \def\n@mref{u}\ifx\l@debut\n@mref\update@ttr\D@FTupdate\Q@s@tupdate{#2}\else
    \def\n@mref{w}\ifx\l@debut\n@mref\update@ttr\D@FTwidth\Q@s@twidth{#2}\else
    \W@rnmesAttr{figset}{#1}\fi\fi\fi\fi\fi\fi}
\ctr@ln@m\CUR@dash
\ctr@ld@f\def\s@uvdash#1{\edef#1{\CUR@dash}}
\ctr@ld@f\def\D@FTdash{1}        
\ctr@ld@f\def\Q@s@tdash#1{\ifGR@cri\edef\CUR@dash{#1}\ifCUR@PS\expandafter\Pssetd@sh#1 :\fi\fi}
\ctr@ld@f\def\Pssetd@shI#1{\PSc@mment{setdash Index=#1}\ifcase#1%
    \or\immediate\write\fwf@g{[] 0 \c@msetdash}
    \or\immediate\write\fwf@g{[6 2] 0 \c@msetdash}
    \or\immediate\write\fwf@g{[4 2] 0 \c@msetdash}
    \or\immediate\write\fwf@g{[2 2] 0 \c@msetdash}
    \or\immediate\write\fwf@g{[1 2] 0 \c@msetdash}
    \or\immediate\write\fwf@g{[2 4] 0 \c@msetdash}
    \or\immediate\write\fwf@g{[3 5] 0 \c@msetdash}
    \or\immediate\write\fwf@g{[3 3] 0 \c@msetdash}
    \or\immediate\write\fwf@g{[3 5 1 5] 0 \c@msetdash}
    \or\immediate\write\fwf@g{[6 4 2 4] 0 \c@msetdash}
    \fi}
\ctr@ld@f\def\Pssetd@sh#1 #2:{{\def\t@xt@{#1}\ifx\t@xt@\empty\Pssetd@sh#2:
    \else\def\t@xt@{#2}\ifx\t@xt@\empty\Pssetd@shI{#1}\else\s@mme=\@ne\def\debutp@t{#1}%
    \an@lysd@sh#2:\ifodd\s@mme\edef\debutp@t{\debutp@t\space\finp@t}\def\finp@t{0}\fi%
    \PSc@mment{setdash Pattern=#1 #2}%
    \immediate\write\fwf@g{[\debutp@t] \finp@t\space\c@msetdash}\fi\fi}}
\ctr@ld@f\def\an@lysd@sh#1 #2:{\def\t@xt@{#2}\ifx\t@xt@\empty\def\finp@t{#1}\else%
    \edef\debutp@t{\debutp@t\space#1}\advance\s@mme\@ne\an@lysd@sh#2:\fi}
\ctr@ln@m\CUR@width
\ctr@ld@f\def\s@uvwidth#1{\edef#1{\CUR@width}}
\ctr@ld@f\def\D@FTwidth{0.4}     
\ctr@ld@f\def\Q@s@twidth#1{\ifGR@cri\edef\CUR@width{#1}\ifCUR@PS%
    \PSc@mment{setwidth Width=#1}\immediate\write\fwf@g{#1 \c@msetlinewidth}\fi\fi}
\ctr@ln@m\CUR@join
\ctr@ld@f\def\s@uvjoin#1{\edef#1{\CUR@join}}
\ctr@ld@f\def\D@FTjoin{miter}   
\ctr@ld@f\def\Q@s@tjoin#1{\ifGR@cri\edef\CUR@join{#1}\ifCUR@PS\expandafter\Pssetj@in#1:\fi\fi}
\ctr@ld@f\def\Pssetj@in#1#2:{\PSc@mment{setjoin join=#1}%
    \if#1r\def\t@xt@{1}\else\if#1b\def\t@xt@{2}\else\def\t@xt@{0}\fi\fi%
    \immediate\write\fwf@g{\t@xt@\space\c@msetlinejoin}}
\ctr@ld@f\def\Pss@tspecifSt#1{\trtlis@rg{#1}{\Pss@tspecifSt@}}
\ctr@ld@f\def\Pss@tspecifSt@#1=#2|{\keln@mun#1|%
    \def\n@mref{c}\ifx\l@debut\n@mref\def\n@mref{#2}\ifx\n@mref\D@FTref\else%
     \s@uvcolor{\typ@color}\Pssetc@lor{#2}\fi\else
    \def\n@mref{d}\ifx\l@debut\n@mref\def\n@mref{#2}\ifx\n@mref\D@FTref\else%
     \s@uvdash{\typ@dash}\Q@s@tdash{#2}\fi\else
    \def\n@mref{j}\ifx\l@debut\n@mref\def\n@mref{#2}\ifx\n@mref\D@FTref\else%
     \s@uvjoin{\typ@join}\Q@s@tjoin{#2}\fi\else
    \def\n@mref{w}\ifx\l@debut\n@mref\def\n@mref{#2}\ifx\n@mref\D@FTref\else%
     \s@uvwidth{\typ@width}\Q@s@twidth{#2}\fi\else
    \W@rnmeskwd{Pss@tspecifSt}{#1}\fi\fi\fi\fi}
\ctr@ld@f\def\Psrest@reSt#1{\trtlis@rg{#1}{\Psrest@reSt@}}
\ctr@ld@f\def\Psrest@reSt@#1=#2|{\keln@mun#1|%
    \def\n@mref{c}\ifx\l@debut\n@mref\def\n@mref{#2}\ifx\n@mref\D@FTref\else%
     \Pssetc@lor{\typ@color}\fi\else
    \def\n@mref{d}\ifx\l@debut\n@mref\def\n@mref{#2}\ifx\n@mref\D@FTref\else%
     \Q@s@tdash{\typ@dash}\fi\else
    \def\n@mref{j}\ifx\l@debut\n@mref\def\n@mref{#2}\ifx\n@mref\D@FTref\else%
     \Q@s@tjoin{\typ@join}\fi\else
    \def\n@mref{w}\ifx\l@debut\n@mref\def\n@mref{#2}\ifx\n@mref\D@FTref\else%
     \Q@s@twidth{\typ@width}\fi\else
    \W@rnmeskwd{Psrest@reSt}{#1}\fi\fi\fi\fi}
\ctr@ld@f\def\Pssettrim@sh#1=#2|{\keln@mde#1|%
    \def\n@mref{co}\ifx\l@debut\n@mref\update@ttr\D@FTref\P@settmeshcolor{#2}\else
    \def\n@mref{da}\ifx\l@debut\n@mref\update@ttr\D@FTref\P@settmeshdash{#2}\else
    \def\n@mref{wi}\ifx\l@debut\n@mref\update@ttr\D@FTref\P@settmeshwidth{#2}\else
    \W@rnmesAttr{figset trimesh}{#1}\fi\fi\fi}
\ctr@ln@m\DDV@tmeshcolor
\ctr@ld@f\def\P@settmeshcolor#1{\edef\DDV@tmeshcolor{#1}}
\ctr@ln@m\DDV@tmeshdash
\ctr@ld@f\def\P@settmeshdash#1{\edef\DDV@tmeshdash{#1}}
\ctr@ln@m\DDV@tmeshwidth
\ctr@ld@f\def\P@settmeshwidth#1{\edef\DDV@tmeshwidth{#1}}
\ctr@ld@f\def\figdrawtrimesh#1[#2,#3,#4]{{\ifCUR@PS\ifGR@cri%
    \PSc@mment{trimesh Type=#1, Triangle=[#2,#3,#4]}%
    \s@uvc@ntr@l\et@tpstrimesh\ifnum#1>\@ne%
    \Pss@tspecifSt{color=\DDV@tmeshcolor,dash=\DDV@tmeshdash,width=\DDV@tmeshwidth}%
    \setc@ntr@l{2}%
    \Pstrimeshp@rt#1[#2,#3,#4]\Pstrimeshp@rt#1[#3,#4,#2]\Pstrimeshp@rt#1[#4,#2,#3]%
    \Psrest@reSt{color=\DDV@tmeshcolor,dash=\DDV@tmeshdash,width=\DDV@tmeshwidth}%
    \fi\figdrawline[#2,#3,#4,#2]%
    \PSc@mment{End trimesh}\resetc@ntr@l\et@tpstrimesh\fi\fi}}
\ctr@ld@f\def\Pstrimeshp@rt#1[#2,#3,#4]{{\l@mbd@un=\@ne\l@mbd@de=#1\loop\ifnum\l@mbd@de>\@ne%
    \advance\l@mbd@de\m@ne\figptbary-1:[#2,#3;\l@mbd@de,\l@mbd@un]%
    \figptbary-2:[#2,#4;\l@mbd@de,\l@mbd@un]\figdrawline[-1,-2]%
    \advance\l@mbd@un\@ne\repeat}}
\initpr@lim\initpss@ttings\initPDF@rDVI
\ctr@ln@w{newbox}\figBoxA
\ctr@ln@w{newbox}\figBoxB
\ctr@ln@w{newbox}\figBoxC
\catcode`\@=12

\pssetdefault(update=yes)
\newbox\figbox

\part{Introduction}
\label{part:1}
\chapter{Introduction of the problem and main results}
\label{sec:intro}

In this work we investigate the ground state energy of the magnetic Laplacian associated with a large magnetic field, posed on a bounded three-dimensional domain and completed by Neumann boundary conditions. This problem can be obtained by linearization from a Ginzburg-Landau equation modelling the surface superconductivity in presence of an exterior magnetic field the intensity of which is close to a (large) critical value, see {\it e.g.} \cite{BeSt98,BauPhTan98,DpFeSt00}. Then the works \cite{LuPan99,HePa03,FouHe05,FouHe06-2,BoFou07,FouHe09} highlight the link between the bottom of the spectrum of a semiclassical Schr\"odinger operator with magnetic field with the behavior of the minimizer of the Ginzburg-Landau functional.
The operator can also be viewed as a Schr\"odinger operator with magnetic field. The problematics of large magnetic field for the magnetic Laplacian is trivially equivalent to the semiclassical limit of the  Schr\"odinger operator as the small parameter $h$ tends to $0$. This problem has been addressed in numerous works in various situations (smooth two- or three-dimensional domains,  see {\it e.g. }the papers \cite{BeSt98,DpFeSt00,LuPan00,HeMo01,HeMo04,Ray09-3d} and the book \cite{FouHe10}, and polygonal domains in dimension 2, see {\it e.g.} \cite{Ja01,Pan02,BonDau06,BoDauMaVial07}). Much less is known for corner three-dimensional domains,  see {\it e.g.} \cite{Pan02,PoRay12}, and this is our aim to provide a unified treatment of smooth and corner domains, possibly in any space dimension $n$. As we will see, we have succeeded at this level of generality for $n=2$ and $3$, and have also obtained somewhat less precise results for any dimension $n$.

The semiclassical limit of the ground state energy is provided by the infimum of \emph{local energies}\index{Local ground energy} defined at each point of the closure of the domain. Local energies are ground state energies of adapted \emph{tangent operators}\index{Tangent operator} at each point. The notion of tangent operator fits in with the problematic that one wants to solve. For example if one is interested in Fredholm theory for elliptic boundary value problems, tangent operators are obtained by taking the principal part of the operator frozen at each point. Another example is the semiclassical limit of the Schr\"odinger operator with electric field. For a rough estimate, tangent operators are then obtained by freezing the electric field at each point, and, for more information on the semiclassical limit, the Hessian at each point has to be included in the tangent operator.

In our situation, tangent operators are obtained by freezing the magnetic field at each point, that is, taking the \emph{linear part of the magnetic potential} at each point. The domain on which the tangent operator is acting is the tangent model\index{Tangent domain} domain at this point. For smooth domains, this notion is obvious (the full space if the point is sitting inside the domain, and the tangent half-space if the point belongs to the boundary). For corner domains, various infinite cones have to be added to the collection of tangent domains.

Almost all known results concerning the semiclassical limit of the ground state energy rely on an {\it a priori} knowledge (or assumptions) on where the local energy is minimal. For instance, this is known if the domain is smooth, or if it is a polygon with openings $\le\frac\pi2$ and constant magnetic field. By contrast, for three-dimensional polyhedra, possible configurations involving edges and corners are much more intricate, and nowadays this is impossible to know where the local energy attains its minimum. Up until recently, it was not even known whether the infimum is attained. 

In this work, we investigate the behavior of the local energy in general 3D corner domains and we prove in particular that it attains its minimum. The properties that we show allow us to obtain an asymptotics with remainder for the ground state energy of the Schr\"odinger operator with magnetic field. 
In some situations, the remainder is optimal. 
We also have partial results for the natural class of $n$-dimensional corner domains.
Let us now present our problematics and results in more detail.

\section{The magnetic Laplacian and its lowest eigenvalue}
The Schr\"odinger operator with magnetic field (also called magnetic Laplacian) in a $n$-dimensional space takes the form
\[
   (-i\nabla+\bA)^2 = \sum_{j=1}^n(-i\partial_{x_{j}}+A_{j})^2,
\]
where $\bA=(A_{1},\ldots,A_{n})$ is a given vector field and $\partial_{x_{j}}$ is the partial derivatives with respect to $x_j$ with $\bx=(x_1,\ldots,x_n)$ denoting Cartesian variables. The field $\bA$ represents the magnetic potential.
When set on a domain $\Omega$ of $\R^n$, this elliptic operator is completed by the magnetic Neumann boundary conditions $(-i\nabla+\bA)\psi\cdot\bn=0$ on $\partial\Omega$, where $\bn$ denotes the unit normal vector to the boundary. We assume in the whole work that the field $\bA$ is twice differentiable on the closure $\overline\Omega$ of $\Omega$, which we write:
\begin{equation}
\label{eq:W2}
   \bA \in \sC^{2}(\overline{\Omega}).
\end{equation} 
This Neumann realization is denoted by $\OP(\bA,\Omega)$. If $\Omega$ is bounded with a Lipschitz boundary\index{Lipschitz domain}\footnote{Or more generally if $\Omega$ is a finite union of bounded Lipschitz domains, cf.\ \cite[Chapter 1]{MazyaBook11} for instance.}, the form domain of $\OP(\bA,\Omega)$ is the standard Sobolev space $H^1(\Omega)$ and $\OP(\bA,\Omega)$ is self-adjoint, non negative, and with compact resolvent. A ground state of $\OP(\bA,\Omega)$ is an eigenpair $(\lambda,\psi)$
associated with the lowest eigenvalue $\lambda$. 
If $\Omega$ is simply connected, its eigenvalues only depend on the magnetic field defined as follows, cf.\ \cite[\S 1.1]{FouHe10}. 
If $\omega_{\bA}$ denotes the 1-form associated with the vector field $\bA$
\begin{equation}\label{eq:omegaA}
   \omega_{\bA}=\sum_{j=1}^n A_{j} \,\rd x_{j},
\end{equation}
the corresponding 2-form $\sigma_{\bB}$ 
\begin{equation}\label{eq:sigmaB}
  \sigma_{\bB}=\rd\omega_{\bA}=\sum_{j<k}B_{jk} \,\rd x_{j}\wedge \rd x_{k}
\end{equation}
is called the magnetic field.
In dimension $n=2$ or $n=3$, $\sigma_{\bB}$ can be identified with 
\begin{equation}
\label{eq:B}
   \bB = \curl\bA.
\end{equation}
When the domain $\Omega$ is simply connected (which will be assumed everywhere unless otherwise stated), the eigenvectors corresponding to two different instances of $\bA$ for the same $\bB$ are obtained from each other by a {\em gauge transform}\index{Gauge transform} and the eigenvalues depend on $\bB$ only. 

Introducing a (small) parameter $h>0$ and setting
\[
   \OP_{h}(\bA,\Omega) = (-ih\nabla+\bA)^2
   \quad\mbox{with magnetic Neumann b.c. on $\partial\Omega$},
\]
we get the relation
\begin{equation}
\label{eq:h2}
   \OP_{h}(\bA,\Omega) = h^2 \OP\Big(\frac{\bA}{h},\Omega\Big)
\end{equation}
linking the problem with large magnetic field to the semiclassical limit $h\to0$ for the Schr\"odinger operator with magnetic potential. Reminding that eigenvalues depend only on the magnetic field, we denote by $\lambda_h=\lambda_h(\bB,\Omega)$ the smallest eigenvalue of $\OP_{h}(\bA,\Omega)$ and by $\psi_h$ an associated eigenvector, so that
\begin{equation}
\label{eq:pbh}
   \begin{cases}
   (-ih\nabla+\bA)^2\psi_h=\lambda_h\psi_h\ \ &\mbox{in}\ \ \Omega \,,\\
   (-ih\nabla+\bA)\psi_{h}\cdot\bn=0\ \ &\mbox{on}\ \ \partial\Omega\,.
   \end{cases}
\end{equation}
The behavior of $\lambda_h(\bB,\Omega)$ as $h\to0$ clearly provide equivalent information about the lowest eigenvalue of $\OP(\breve\bA,\Omega)$ when $\breve\bB$ is large, especially in the parametric case when $\breve\bB = B\ee\bB$ where the real number $B$ tends to $+\infty$ and $\bB$ is a chosen reference magnetic field.

From now on, we consider that $\bB$ is fixed. We assume that it is smooth enough  and, unless otherwise mentioned, does not vanish on $\overline\Omega$. 
The question of the semiclassical behavior of $\lambda_h(\bB,\Omega)$ has been considered in many papers for a variety of domains, with constant or variable magnetic fields: Smooth domains \cite{BeSt98,LuPan99-2, HeMo01, FouHe06,Ara07, Ray09} and polygons \cite{Ja01,Pan02,Bo05, BonDau06, BoDauMaVial07} in dimension $n=2$, and mainly smooth domains \cite{LuPan00, HeMo02, HeMo04, Ray09-3d, FouHe10} in dimension $n=3$. Until now, three-dimensional non-smooth domains were only addressed in two particular configurations---rectangular cuboids \cite{Pan02} and lenses \cite[Chap. 8]{PoTh} and \cite{PoRay12}, with special orientations of the magnetic field (that is supposed to be constant). We give more detail and references about the state of the art in Chapter \ref{sec:art}.

\section{Local ground state energies}
Let us make precise what we call local energy in the three-dimensional setting. The domains that we are considering are members of a very general class of corner domains\index{Corner domain} defined by recursion over the dimension $n$ (these definitions are set in Chapter \ref{sec:chain}).
In the three-dimensional case, each point $\bx$ in the closure of a corner domain $\Omega$
is associated with a dilation invariant, tangent open\index{Tangent cone} set $\Pi_\bx$, according to the following cases:
\begin{enumerate}
\item If $\bx$ is an interior point, $\Pi_\bx=\R^3$,
\item If $\bx$ belongs to a {\em face $\bff$} ({\it i.e.}, a connected component of the smooth part of $\partial\Omega$), $\Pi_\bx$ is a half-space,
\item If $\bx$ belongs to an {\em edge} $\be$, $\Pi_\bx$ is an infinite wedge,
\item If $\bx$ is a {\em vertex} $\bv$, $\Pi_\bx$ is an infinite cone.
\end{enumerate}
Let $\bB_\bx$ be the magnetic field frozen at $\bx$. The tangent operator at $\bx$ is the magnetic Laplacian $\OP(\bA_\bx \ee,\Pi_\bx)$ where $\bA_\bx$ is the linear approximation\index{Linearized magnetic potential} of $\bA$ at $\bx$, so that
\[
   \curl\bA_\bx = \bB_\bx \,.
\] 
We define the \emph{local energy}\index{Local ground energy} $\En(\bB_\bx \ee,\Pi_\bx)$ at $\bx$ as the ground state energy of the tangent operator\index{Tangent operator} $\OP(\bA_\bx \ee,\Pi_\bx)$ and we introduce the global quantity {\it (lowest local energy)}\index{Lowest local energy}
\begin{equation}
\label{eq:s}
   \sE(\bB \ee,\Omega) := \inf_{\bx\in\overline\Omega} \En(\bB_\bx \ee,\Pi_\bx).
\end{equation}
One of our objectives is to show the existence of a minimizer for these ground state energies, achieved by a certain tangent cone\index{Tangent cone} associated with suitable generalized eigenfunctions, as we will specify later on.

The tangent operators $\OP(\bA_\bx \ee,\Pi_\bx)$ are magnetic Laplacians set on unbounded domains and with constant magnetic field. So they have mainly an essential spectrum and, only in some cases when $\bx$ is a vertex, discrete spectrum. This fact makes it difficult to study continuity properties of the ground energy and to construct quasimodes for the initial operator.

In the regular case, the tangent operators are magnetic Laplacians associated respectively with interior points and boundary points, acting respectively on the full space and on half-spaces. The spectrum of the operator on the full space is well-known and corresponds to Landau modes. The case of the half-spaces has also been investigated for a long time (\cite{LuPan00, HeMo04}): The ground state energy depends now on the angle between the (constant) magnetic field and the boundary of the half-space. 
It is continuous and increasing with this angle, so that the ground state is minimal for a magnetic field tangent to the boundary, and maximal for a magnetic field normal to the boundary. In all cases, it is possible to find a bounded generalized eigenfunction satisfying locally the boundary conditions. 

For two dimensional domains with corners, new tangent model operators have to be considered, now acting on infinite sectors (\cite{Pan02, Bo05}). For openings $\le\frac\pi2$, the ground state energy is an eigenvalue strictly less than in the regular case for the same value of $\bB$. But for larger openings in 2D and conical or polyhedral singularities in 3D, it becomes harder to compare ground state energies, and for a given tangent operator, it is not clear whether there exist associated generalized eigenfunctions. Moreover, it is not clear anymore whether the infimum of the ground state energies over all tangent operators is reached.

In this work, for two or three dimensions of space,  we provide positive answers to the questions of existence for a minimum  in \eqref{eq:s} and for related generalized eigenvectors associated with the minimum energy.  First we have proved very general continuity and semicontinuity properties for the function $\bx\mapsto\En(\bB_\bx \ee,\Pi_\bx)$ as described now.
Let $\gF$ be the set of faces\index{Face} $\bff$, $\gE$ the set of edges\index{Edge} $\be$ and $\gV$ the set of vertices\index{Vertex} of $\Omega$. They form a partition of the closure of $\Omega$, called stratification\index{Stratification}
\begin{equation}
\label{eq:stratif}
   \overline\Omega = \Omega \cup \big(\bigcup_{\bff\in\gF}\ \bff\ \big)
   \cup \big(\bigcup_{\be\in\gE}\ \be\ \big)
   \cup \big(\bigcup_{\bv\in\gV}\ \bv\ \big) .
\end{equation}
The sets $\Omega$, $\bff$, $\be$ and $\bv$ are open sets called the strata\index{Stratum} of $\overline\Omega$, compare with \cite{MazyaPlamenevskii77} and \cite[Ch. 9]{NazarovPlamenevskii94}. We denote them generically by $\bt$ and their set by $\gT$. Note that strata do not contain their boundaries: faces do not include edges or vertices, and edges do not include vertices. We will show the following facts
\begin{enumerate}[{\quad}(a)]
\item For each stratum $\bt\in\gT$, the function $\bx\mapsto \En(\bB_\bx \ee,\Pi_\bx)$ is continuous on $\bt$. 
\item The function $\bx\mapsto \En(\bB_\bx \ee,\Pi_\bx)$ is lower semicontinuous on $\overline\Omega$.
\end{enumerate}
As a consequence, the infimum determining the limit $\sE(\bB,\Omega)$ in \eqref{eq:s} is a minimum
\begin{equation}
\label{eq:s,min}
   \sE(\bB \ee,\Omega) = \min_{\bx\in\overline\Omega} \,\En(\bB_\bx \ee,\Pi_\bx) \,.
\end{equation}
From this we can deduce in particular that $\sE(\bB,\Omega)>0$ as soon as $\bB$ does not vanish on $\overline\Omega$.

But we need more than properties a) and b) to show an upper bound for $\lambda_h(\bB,\Omega)$ as $h\to0$. We need to construct quasimodes\index{Quasimode} whatever is the geometry of $\Omega$ near the minimizers of the local energy. For this we define a second level of energy attached to each point $\bx\in\overline\Omega$ which we denote by $\seE(\bB_{\bx},\Pi_{\bx})$ and call {\it energy on tangent substructures}\index{Energy on tangent substructures}. This quantity has been introduced on the emblematic example of edges in \cite{Pop13}: If $\bx$ belongs to an edge, then $\Pi_\bx$ is a wedge. This wedge has two faces defining two half-spaces $\Pi^\pm_\bx$ in a natural way: This provides, in addition with the full space $\R^3$, what we call the \emph{tangent substructures}\index{Tangent substructure} of $\Pi_\bx$. In this situation $\seE(\bB_{\bx},\Pi_{\bx})$ is defined as
\[
   \seE(\bB_{\bx},\Pi_{\bx}) = 
   \min\big\{\En(\bB_\bx \ee,\Pi^+_\bx),\En(\bB_\bx \ee,\Pi^-_\bx),\En(\bB_\bx \ee,\R^3) \big\}.
\]
For a general point $\bx\in\overline\Omega$, $\seE(\bB_{\bx},\Pi_{\bx})$ is the infimum of local energies associated with the tangent substructures of $\Pi_\bx$, that is all cones $\Pi_\by$ associated with points $\by\in\overline\Pi_\bx\setminus\bt_0$ where $\bt_0$ is the stratum of $\Pi_\bx$ containing the origin (for the example of a wedge, $\bt_0$ is its edge). 
Equivalently, $\seE(\bB_{\bx},\Pi_{\bx})$ yields $\liminf_{\by\to \bx} E(\bB_{\bx},\Pi_{\by})$ for points $\by\in\overline\Omega$ that are not in the same stratum as $\bx$. We show that $ \En(\bB_\bx,\Pi_\bx) \leq \seE(\bB_\bx,\Pi_\bx)$. This may be understood as a monotonicity property of the ground state energy for a tangent cone and its tangent substructures. 

The quantity $\seE(\bB_\bx,\Pi_\bx)$ has a spectral interpretation: For a vertex $\bx$ of $\Omega$,  $\seE(\bB_\bx,\Pi_\bx)$ is the bottom of the essential spectrum of $H(\bA_\bx,\Pi_\bx)$ so that if $\En(\bB_\bx,\Pi_\bx) < \seE(\bB_\bx,\Pi_\bx)$, there exists an eigenfunction associated with $\En(\bB_\bx,\Pi_\bx)$. For $\bx$ other than a vertex, the interpretation of $\seE(\bB_\bx,\Pi_\bx)$ is less standard: We show that if $\En(\bB,\Pi_\bx) < \seE(\bB_\bx,\Pi_\bx)$, then there exists a bounded \emph{generalized eigenvector}\index{Generalized eigenvector} associated with $\En(\bB_\bx,\Pi_\bx)$. 

However, it remains possible that $\En(\bB_\bx,\Pi_\bx)$ equals $\seE(\bB_\bx,\Pi_\bx)$. This case seems at first glance to be problematic, but we provide a solution issued from the recursive properties of  corner domains: We show that there always exists a tangent substructure of $\Pi_\bx$ providing generalized eigenfunctions for the same level of energy.

\section{Asymptotic formulas with remainders}
\subsubsection*{Case of 3D domains}
A thorough investigation of local energies $\En(\bB_\bx \ee,\Pi_\bx)$ and $\seE(\bB_{\bx},\Pi_{\bx})$ allows us to find asymptotic formulas with remainders for the ground state energy $\lambda_h(\bB \ee,\Omega)$ of the magnetic Laplacian on any 3D corner domain $\Omega$ as $h\to0$. Our remainders depend on the singularities of $\Omega$: The convergence rate is improved in the case of \emph{polyhedral domains}\index{Polyhedral domain} in which, by contrast with conical domains, the main curvatures at any smooth point of the boundary remain uniformly bounded. Figure \ref{F1} gives several examples of corner domains: Both edge domains in Figure \ref{F1B} are polyhedral, such as the Fichera corner in the left part of Figure \ref{F1C}, whereas the three other domains (Figure \ref{F1A} and Figure \ref{F1C}-right) have conical points where one main curvature tends to infinity.
\begin{figure}[ht]

\noindent\centering
    \begin{subfigure}[h]{0.32\textwidth}
\includegraphics[keepaspectratio=true,width=5.3cm]{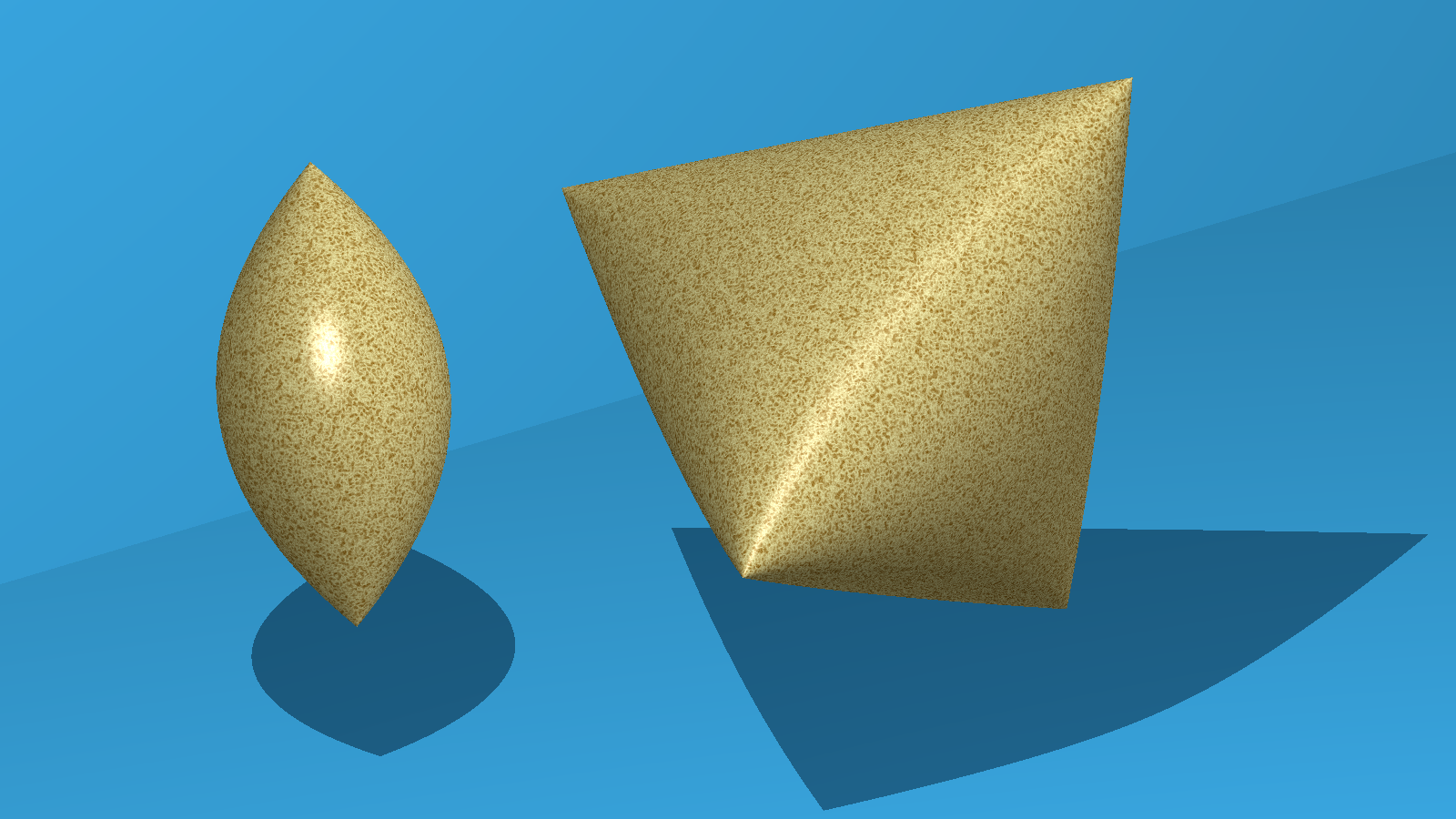}
        \caption{Domains with corners}
        \label{F1A}
    \end{subfigure}
\
            \begin{subfigure}[h]{0.32\textwidth}
\includegraphics[keepaspectratio=true,width=5.3cm]{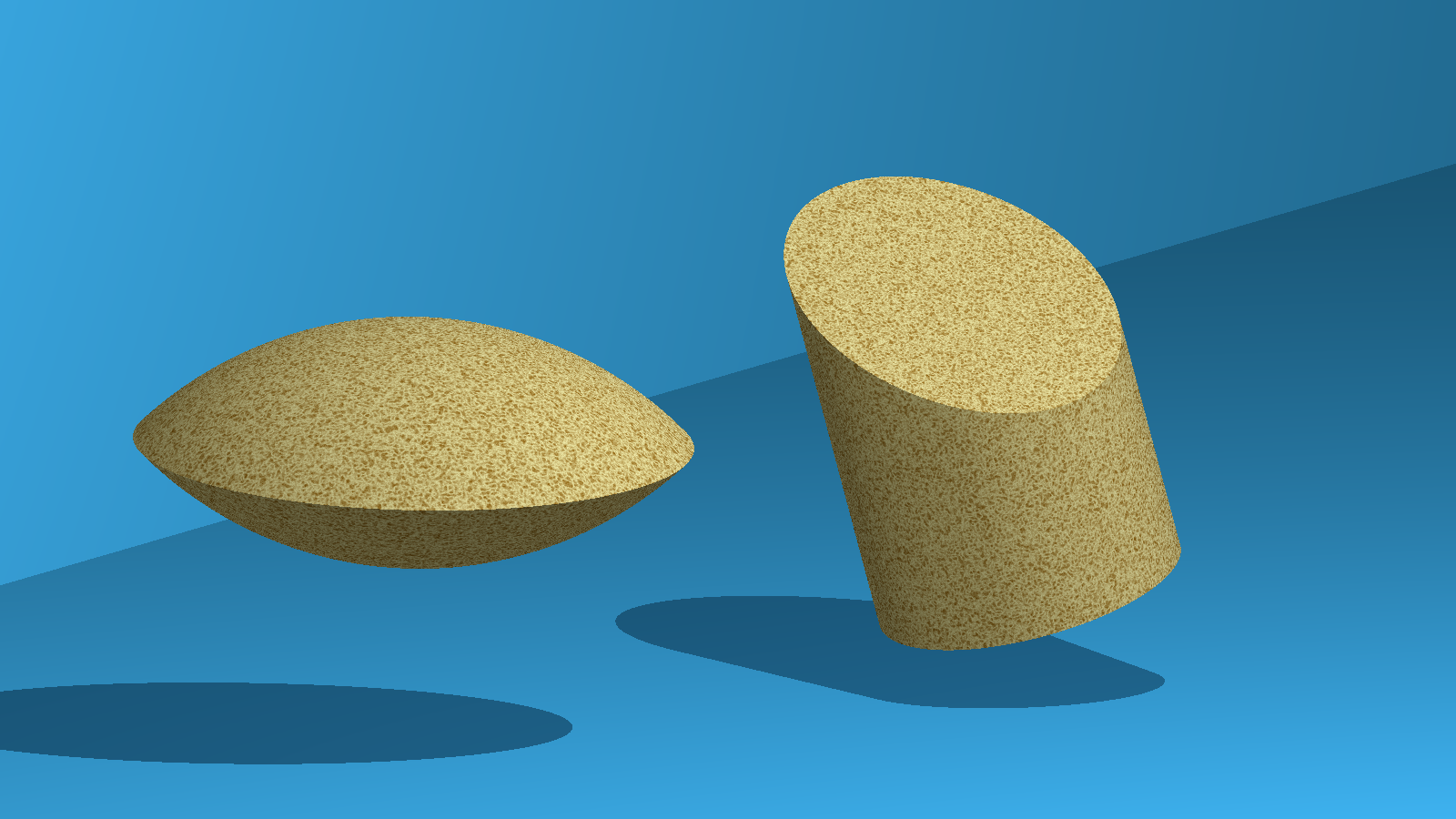}
        \caption{Domains with edges}
        \label{F1B}
    \end{subfigure}
    \
        \begin{subfigure}[h]{0.32\textwidth}
\includegraphics[keepaspectratio=true,width=5.3cm]{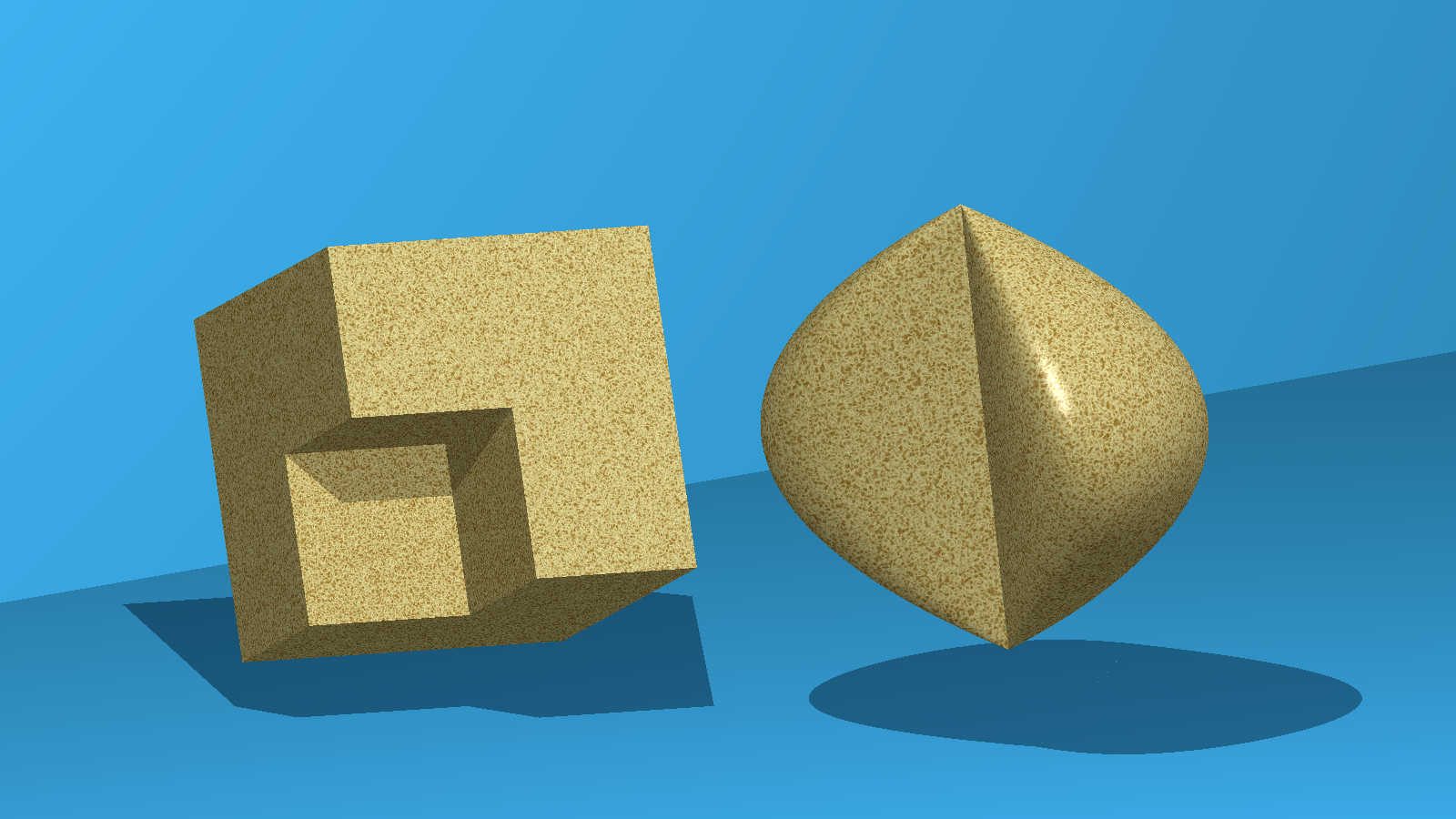}
        \caption{Domains with both}
        \label{F1C}
    \end{subfigure}
\caption{Examples of 3D corner domains (Figures drawn by M. Costabel with POV-ray)}
\label{F1}

\end{figure}

Our main results can be stated as follows (Theorems~\ref{T:generalLB} and \ref{T:generalUB}) as $h\to0$
\begin{equation}
\label{eq:conv}
   \big|\lambda_h(\bB \ee,\Omega) - h\ee \sE(\bB \ee,\Omega) \big|\le
   \begin{cases}
   C_\Omega \big(1+ \|\bA\|^2_{W^{2,\infty}(\Omega)}\big) \, h^{11/10} ,
    &\mbox{$\Omega$ corner domain,} 
   \\[0.5ex]
   C_\Omega \big(1+ \|\bA\|^2_{W^{2,\infty}(\Omega)}\big) \, h^{5/4} ,
    &\mbox{$\Omega$ polyhedral domain.}
   \end{cases}
\end{equation}
Here the constant $C_\Omega$ only depends on the domain $\Omega$ (and not on $\bA$, nor on $h$), and $\|\bA\|_{W^{2,\infty}(\Omega)}$ denotes the standard $L^\infty$ Sobolev norm on $\sC^{2}(\overline{\Omega})$:
\[
   \|\bA\|_{W^{2,\infty}(\Omega)} = 
   \max_{1\le j\le n} \max_{|\alpha|\le2} \|\partial^\alpha_\bx A_j \|_{L^{\infty}(\Omega)} .
\]
Note that the lower bound in \eqref{eq:conv} for the polyhedral case coincides with the one obtained in the smooth case in dimensions 2 and 3 when no further assumptions are imposed, cf.\ Section \ref{SS:dim3} below.

Besides, if $\bB$ vanishes somewhere in $\overline\Omega$, the lowest local energy $\sE(\bB \ee,\Omega)$ is zero, and we obtain the upper bound in any 3D corner domain $\Omega$ (Theorem \ref{T:generalUB})
\begin{equation}
\label{eq:convBnul}
  \lambda_h(\bB \ee,\Omega) \le    C_\Omega \big(1+ \|\bA\|^2_{W^{2,\infty}(\Omega)}\big) \, h^{4/3} ,
\end{equation}
which, in view of \cite{HeMo96,DoRa13}, is optimal. Indeed, we also improve the upper bound in \eqref{eq:conv} recovering the power $h^{4/3}$ for general potentials that are $3$ times differentiable in polyhedral domains, namely
\begin{equation}
\label{eq:convimp}
   \lambda_h(\bB \ee,\Omega) \le h\ee \sE(\bB \ee,\Omega) +
   \begin{cases}
   C_\Omega \big(1+ \|\bA\|^2_{W^{3,\infty}(\Omega)}\big) \, h^{9/8} ,
    &\mbox{$\Omega$ corner domain,} 
   \\[0.5ex]
   C_\Omega \big(1+ \|\bA\|^2_{W^{3,\infty}(\Omega)}\big) \, h^{4/3} ,
    &\mbox{$\Omega$ polyhedral domain.}
   \end{cases}
\end{equation}
Note that the $h^{4/3}$ rate was known for smooth three-dimensional domains, \cite[Proposition 6.1 \& Remark 6.2]{HeMo04} and that \eqref{eq:convimp} extends this result to polyhedral domains without loss.

Two-dimensional corner domains are curvilinear polygons. 
The curvature of their boundary satisfies the same property of uniform boundedness 
than polyhedral domains. That is why the asymptotic formulas with remainder in $h^{5/4}$ (and even $h^{4/3}$ for the upper bound) are valid.

With the point of view of large magnetic fields in the parametric case $\breve\bB = B\ee\bB$, the identity \eqref{eq:h2} used with $h=B^{-1}$ provides  
\begin{equation}
\label{lienBh}
\lambda(\breve\bB,\Omega)=B^2\lambda_{B^{-1}}(\bB,\Omega), 
\end{equation}
therefore \eqref{eq:conv} yields obviously as $B\to\infty$
\begin{equation}
\label{eq:convB}
   \big|\lambda(\breve\bB \ee,\Omega) - B\ee \sE(\bB \ee,\Omega) \big|\le 
   \begin{cases}
    C_\Omega \big(1+ \|\bA\|^2_{W^{2,\infty}(\Omega)}\big) \, B^{9/10}, 
    &\mbox{$\Omega$ corner domain,} 
    \\[0.5ex]
   C_\Omega \big(1+ \|\bA\|^2_{W^{2,\infty}(\Omega)}\big) \, B^{3/4}, 
    &\mbox{$\Omega$ polyhedral domain,} 
\end{cases}
\end{equation}
where $\bA$ is a potential associated with $\bB$. Note that $B\ee \sE(\bB \ee,\Omega)=\ee \sE(\breve\bB \ee,\Omega)$ by homogeneity. In the same spirit, improved upper bounds \eqref{eq:convimp} can be written as
\begin{equation}
\label{eq:convimpB}
   \lambda(\breve\bB \ee,\Omega) \le B\ee \sE(\bB \ee,\Omega)  +
   \begin{cases}
   C_\Omega \big(1+ \|\bA\|^2_{W^{3,\infty}(\Omega)}\big) \, B^{7/8} ,
    &\mbox{$\Omega$ corner domain,} 
   \\[0.5ex]
   C_\Omega \big(1+ \|\bA\|^2_{W^{3,\infty}(\Omega)}\big) \, B^{2/3} ,
    &\mbox{$\Omega$ polyhedral domain.}
   \end{cases}
\end{equation}

\subsubsection*{Estimates involving $\bB$ only}
In formulas \eqref{eq:conv} the remainder estimates depend on the magnetic potential $\bA$. 
It is possible to obtain estimates depending on the magnetic field $\bB$ and not on the potential as long as $\Omega$ is simply connected. For this, we consider $\bB$ as a datum and associate a potential $\bA$ with it. Operators $\sA:\bB\mapsto\bA$ lifting the curl ({\it i.e.}, such that $\curl\circ\,\sA=\Id$) and satisfying suitable estimates have been considered in the literature. We quote \cite{CostabelMcIntosh10} in which it is proved that such lifting can be constructed as a pseudo-differential operator of order $-1$. As a consequence $\sA$ is continuous between H\"older classes of non integer order:
\[
   \forall\ell\in\N,\ \forall\alpha\in(0,1),\quad
   \exists K_{\ell,\alpha} >0,\quad
   \|\sA\bB\|_{W^{\ell+1+\alpha,\infty}(\Omega)} \le K_{\ell,\alpha}  
   \|\bB\|_{W^{\ell+\alpha,\infty}(\Omega)} \,.
\]
Choosing $\bA=\sA\bB$ with $\ell=2$ and $\alpha>0$ in \eqref{eq:conv}, or with $\ell=3$ and $\alpha>0$ in \eqref{eq:convimp}, we obtain remainder estimates depending on $\bB$ only.

\subsubsection*{Generalization to $n$-dimensional corner domains}
We have also obtained a weaker result valid in any space dimension $n$, $n\ge4$. Combining Sections \ref{ss:genup} and \ref{ss:genlow} we can see that the quotient $\lambda_{h}(\bB \ee,\Omega)/h$ converges to $\sE(\bB \ee,\Omega)$ as $h\to0$ and that a general lower bound with remainder is valid. For a $n$-dimensional polyhedral domain, this lower bound is the same as in dimension 3:
\begin{equation}
\label{eq:polyhnD}
   - C_\Omega \big(1+ \|\bA\|^2_{W^{2,\infty}(\Omega)}\big) \, h^{5/4} 
   \le  \lambda_h(\bB \ee,\Omega) - h\ee \sE(\bB \ee,\Omega).
\end{equation}

\subsubsection*{Generalization to non simply connected domains}
If $\Omega$ is not simply connected, the first eigenvalue of the operator $\OP(\bA,\Omega)$ will depend on $\bA$, and not only on $\bB$. A manifestation of this is the Aharonov Bohm effect, see \cite{Helffer88a} for instance. Our results \eqref{eq:conv}--\eqref{eq:convBnul} still hold for the first eigenvalue $\lambda_h=\lambda_h(\bA,\Omega)$ of $\OP_{h}(\bA,\Omega)$. Note that, by contrast, the ground state energies of tangent operators $\OP(\bA_\bx,\Pi_\bx)$ only depend on the (constant) magnetic field $\bB_\bx$ because the potential $\bA_\bx$ is linear by definition. Therefore the lowest local energy only depends on the magnetic field and can still be denoted by $\sE(\bB,\Omega)$ even in the non simply connected case.

\section{Contents}
Our work is organized in five parts. Part \ref{part:1} is introductory and contains two chapters, the present introduction and Chapter \ref{sec:art} where we review related literature. Part \ref{part:2} is devoted to the relevant classes of corner domains and associated model tangent structures. The proof of lower bounds for the quotient $\lambda_{h}(\bB \ee,\Omega)/h$ is also presented in this part since it does not require finer tools. In Part \ref{part:3} we investigate more specific features of the (two- and) three-dimensional model magnetic Laplacians, and prove several different upper bounds. Part \ref{part:4} deals with improvements and generalizations in various directions. The last part gather appendices.\\

Let us give more details on the contents of the core parts (\ref{part:2} to \ref{part:4}) of this work.

\subsubsection*{Part \ref{part:2}}
In Chapter \ref{sec:chain} we define recursively our class of corner domains $\Omega$ in dimension $n$, alongside with their tangent cones $\Pi_\bx$ and singular chains\index{Singular chain} $\dx=(\bx_0,\bx_1,\ldots)$. We particularize these notions in the case $n=3$ and prove weighted estimates for the local maps\index{Local map} and their derivatives. The weights are powers of the distance to conical vertices\index{Conical point} around which one main curvature blows up. We investigate a special class of functions acting on singular chains. The local energy enters this class. 

In Chapter~\ref{sec:chgvar}, we introduce the tangent operators\index{Tangent operator} attached to each magnetic Laplacian on a corner domain and establish weighted estimates of the linearization error. We deduce a rough general upper bound for the quotient $\lambda_{h}(\bB \ee,\Omega)/h$ for corner domains in any dimension $n\ge2$.

In Chapter \ref{sec:low} we prove the lower bound 
$h\ee \sE(\bB \ee,\Omega) - C h^{11/10} \le \lambda_h(\bB \ee,\Omega)$ for general 3D corner domains by an IMS-type formula\index{IMS formula} based on a two-scale partition of unity\index{Partition of unity}. In the particular case of polyhedra, a one-scale standard partition suffices, which yields the improved lower bound  $h\ee \sE(\bB \ee,\Omega) - C h^{5/4} \le \lambda_h(\bB \ee,\Omega)$. We can generalize these lower bounds to any dimension $n$, letting appear the power $1+1/(3\cdot2^{\nu+1}-2)$ of $h$ with an integer $\nu\in[0,n]$ depending on the corner domain $\Omega$.

\subsubsection*{Part \ref{part:3}}
In Chapter \ref{sec:tax} we introduce the lowest energy $\seE(\bB,\Pi)$\index{Local ground energy} on tangent substructures of a model cone $\Pi$ associated with a constant magnetic field $\bB$. Then we classify magnetic model problems on three-dimensional model cones (taxonomy)\index{Taxonomy}: We characterize as precisely as possible their ground state energy, their lowest energy on tangent substructures, and their essential spectrum. 

We show in Chapter~\ref{sec:dicho} one of the most original results of our work, in view of the construction of quasimodes: To each point $\bx_0$ in $\overline\Omega$ are associated its tangent structures $\Pi_\dx$ characterized by a singular chain $\dx$ originating at $\bx_0$. Among them, there exists one for which the tangent operator $\OP(\bA_\dx,\Pi_\dx)$ possesses suitable bounded generalized eigenvectors (said \emph{admissible})\index{Admissible Generalized Eigenvector} with the same energy as the local energy at $\bx_0$:
\[
    \En(\bB_{\dx} \ee,\Pi_{\dx}) = \En(\bB_{\bx_0} \ee,\Pi_{\bx_0}).
\]

Chapter~\ref{sec:sci} is devoted to the investigation of various continuity properties of the local ground energy $\En(\bB_{\bx},\Pi_{\bx})$.

In Chapter \ref{sec:up}, by a construction of quasimodes based on admissible generalized eigenvectors for tangent problems, we prove the upper bounds 
\begin{equation}
\label{eq:up}
   \lambda_h(\bB \ee,\Omega) \le h\ee \sE(\bB \ee,\Omega) + C h^\kappa,
\end{equation}
with $\kappa={11/10}$ or $\kappa=5/4$ depending on whether $\Omega$ is a corner domain or a polyhedral domain. Our construction critically depends on the length\index{Singular chain!Length} $\nu$ of the singular chain $\dx$ that provides the generalized eigenvector. When $\nu=1$, we are in the classical situation: It suffices to concentrate the support of the quasimode around $\bx_0$, and we qualify it as \emph{sitting}\index{Quasimode!Sitting}. When $\nu=2$, the chain has the form $\dx=(\bx_0,\bx_1)$: Our quasimode is decentered in the direction provided by $\bx_1$, has a two-scale structure in general, and we qualify it as \emph{sliding}\index{Quasimode!Sliding}. When $\nu=3$, the chain has the form $\dx=(\bx_0,\bx_1,\bx_2)$ and our quasimode is \emph{doubly sliding}\index{Quasimode!Doubly sliding}. In dimension $n=3$, considering chains of length $\nu\le3$ is sufficient to conclude.

\subsubsection*{Part \ref{part:4}}
To show the improved upper bounds \eqref{eq:convimp}, we revisit,  in Chapter \ref{sec:Stab}, admissible generalized eigenvectors by analyzing the stability of their structure under perturbation. In Chapter \ref{sec:Improv}, we prove refined upper bounds of type \eqref{eq:up} with improved rates $\kappa={9/8}$ and $\kappa=4/3$ when $\Omega$ is a general corner domain and a polyhedral domain, respectively, but with a constant $C$ involving now the norm $W^{3,\infty}$ of the magnetic potential instead of the norm $W^{2,\infty}$. This proof is based on the same stratification\index{Stratification} as the previous one, combined with a new classification depending on the number of directions along which the admissible generalized eigenvector is exponentially decaying. 

In Chapter \ref{sec:conc}, we address various improvements or extensions of our results. We mention in particular the situation where one has a corner concentration\index{Corner concentration}, that is a genuine eigenvector in a tangent cone associated with the lowest local energy. This provides the existence of asymptotics as $h\to0$ for the first eigenpairs on the corner domain. We conclude our work by sketching the similarities with another, simpler, problem issued from the superconductivity, namely the Robin boundary conditions for the plain Laplace operator.

\section{Notations} 
\label{ss:not}
We denote by $\langle\cdot,\cdot\rangle_{\cO}$ the $L^2$ Hilbert product on the open set $\cO$ of $\R^n$
$$\big\langle f,g\big\rangle_{\cO}=\int_{\cO}f(\bx)\,\overline{g}(\bx)\ \rd\bx.$$
When there is no confusion, we simply write $\langle f,g\rangle$ and $\|f\|= \langle f,f\rangle^{1/2}$.

For a generic (unbounded) self-adjoint operator $L$
we denote by $\dom(L)$ its domain and $\gS(L)$ its spectrum. 
Likewise the domain of a quadratic form $q$ is denoted by $\dom(q)$.

Domains as open simply connected subsets of $\R^n$ are in general denoted by $\cO$ if they are generic, $\Pi$ if they are invariant by dilatation (cones) and $\Omega$ if they are bounded.

The quadratic forms of interest are those associated with magnetic Laplacians, namely, for a positive constant $h$, a magnetic potential $\bA\in \sC^{2}(\overline{\Omega})$, and a generic domain $\cO$
\begin{equation}
\label{D:fq}
   q_{h}[\bA,\cO](f):= \big\langle (-ih\nabla+\bA)f,(-ih\nabla+\bA)f\big\rangle_{\cO}=
   \int_{\cO}(-ih\nabla+\bA)f\cdot \overline{(-ih\nabla+\bA)f}\;\rd\bx ,
\end{equation}
with its domain 
$
   \dom(q_{h}[\bA,\cO]) = \{f\in L^2(\cO), \ (-ih\nabla+\bA)f\in L^2(\cO)\}
$.
For a bounded domain $\Omega$, $\dom(q_{h}[\bA,\Omega])$ coincides with $H^1(\Omega)$. For $h=1$, we omit the index $h$, denoting the quadratic form by $q[\bA,\cO]$. In the same way we introduce the following notation for Rayleigh quotients\index{Rayleigh quotient|textbf}
\index{Rayleigh quotient!$\QR_{h}$|textbf}
\begin{equation}
\label{eq:RayQuot}
   \QR_{h}[\bA,\cO](f) = \frac{q_{h}[\bA,\cO](f)}{\langle f,f\rangle_{\cO}},
   \quad f\in\dom(q_{h}[\bA,\cO]),\ f\neq0,
\end{equation}
and recall that, by the min-max principle\index{Min-max principle}
\begin{equation}
\label{eq:minmax}
   \lambda_h(\bB,\Omega) = \min_{ f\in\dom(q_{h}[\bA,\Omega])\,\setminus\, \{0\}} 
   \QR_{h}[\bA,\Omega](f)\,.
\end{equation}
In relation with changes of variables, we will also use the more general form with metric:
\begin{equation}
\label{D:fqG}
   q_{h}[\bA,\cO,\rG](f)=
   \int_{\cO}(-ih\nabla+\bA)f\cdot \rG \big(\,\overline{\!(-ih\nabla+\bA)f\!}\,\big)
   \; |\rG|^{-1/2} \,\rd\bx ,
\end{equation}
where $\rG$ is a smooth function with values in $3\times3$ positive symmetric matrices and $|\rG|=\det \rG$.  Its domain is
$
   \dom(q_{h}[\bA,\cO,\rG]) = \{f\in L^2_\rG(\cO), \ \rG^{1/2}(-ih\nabla+\bA)f\in L^2_\rG(\cO)\}\,,
$
where $L^2_{\rG}(\cO)$ is the space of the square-integrable functions for the weight $|\rG|^{-1/2}$ and $\rG^{1/2}$ is the square root of the matrix $\rG$. The corresponding Rayleigh quotient is denoted by $\QR_{h}[\bA,\cO,\rG]$.

The domain of the magnetic Laplacian with Neumann boundary conditions on the  set $\cO$ is
\begin{multline}
\label{eq:dom}
   \dom\left(\OP_h(\bA \ee,\cO)\right) =
   \big\{f\in \dom(q_{h}[\bA,\cO]),\\ 
   (-ih\nabla+\bA)^2f\in L^2(\cO) \ \ \mbox{and}\ \ 
   (-ih\nabla+\bA)f\cdot\bn=0\ \mbox{on}\ \partial\cO \big\} \, .
\end{multline}

We will also use the space of the functions which are {\em locally}\footnote{Here $H^m_\loc(\overline{\cO})$ denotes for $m=0,1$ the space of functions  which are in $H^m(\cO\cap\cB)$ for any ball $\cB$.} in the domain of $\OP_h(\bA,\cO)$:
\begin{multline}
\label{D:domloc}
   \dom_{\,\loc} \left(\OP_h(\bA,\cO)\right) := 
   \{f\in H^1_\loc(\overline{\cO}), \\  
   (-ih\nabla+\bA)^2f\in H^0_\loc(\overline{\cO}) \ \ \mbox{and}\ \  
  (-ih\nabla+\bA)f\cdot\bn=0 \ \mbox{on} \ \partial\cO\}.
\end{multline}
When $h=1$, we omit the index $h$ in \eqref{eq:dom} and \eqref{D:domloc}. 

\chapter{State of the art}
\label{sec:art}

Here we collect some results from the literature about the semiclassical limit for the first eigenvalue of the magnetic Laplacian depending on the geometry of the domain and the variation of the magnetic field. We briefly mention in Section \ref{SS:wbc} the case where the domain has no boundary, or when Dirichlet boundary conditions are considered. No restriction of dimension is imposed in these cases. Then we review in more detail what is known on bounded domains with Neumann boundary conditions in dimension 2 and 3, in Sections \ref{SS:dim2} and \ref{SS:dim3}, respectively. To keep this chapter short and easy to read, we mainly focus on results related with our problematics, {\it i.e.}, the general asymptotic behavior of the ground state energy without any further assumption on the minimum local energy.

\section{Without boundary or with Dirichlet conditions}
\label{SS:wbc}
Here $M$ is either a compact Riemannian manifold without boundary or $\R^n$, and $\OP_{h}(\bA,M)$ is the magnetic Laplacian associated with the 1-form $\omega_{\bA}$ defined in \eqref{eq:omegaA}. In this general framework, the magnetic field $\bB$ is the antisymmetric matrix corresponding to the 2-form $\sigma_{\bB}$ introduced in \eqref{eq:sigmaB}. Then for each $\bx\in M$ the local energy at $\bx$ is the intensity  
\begin{equation}\label{eq.b(x)}
b(\bx) := \tfrac12 \operatorname{\rm Tr}([\bB^*(\bx)\cdot\bB(\bx)]^{1/2})
\end{equation}
and $\sE(\bB,M)=b_0:=\inf_{\bx\in M} b(\bx)$. 
It is proved by Helffer and Mohamed in \cite{HeMo96} that if $b_0$ is positive and under a condition at infinity if $M=\R^n$, then 
$$
   - Ch^{5/4} \leq
   \lambda_h(\bB \ee,M)-h\sE(\bB \ee,M) \leq Ch^{4/3}  \ .
$$
More precise results can be proved in dimension $2$ when $b$ admits a unique positive non-degenerate minimum \cite{HeKo11,RayVuN13}. Note that the cancellation case $b_{0}=0$ has also been considered in various situations, see for example \cite{HeMo96,HeKo09, DoRa13,BoRay15}. 
Finally, the case of Dirichlet boundary conditions is very close to the case without boundary, see \cite{HeMo96,HeMo01} and Section \ref{SS:Dir}.

\section{Neumann conditions in dimension 2}
\label{SS:dim2}
By contrast, when Neumann boundary conditions are considered on the boundary, the local energy drops significantly as it was established in \cite{SaGe63} by Saint-James and de Gennes as early as 1963.
In this review of the dimension $n=2$, we classify the domains into two categories: those with a regular boundary and those with a polygonal boundary.

\subsection{Regular domains}
Let $\Omega\subset \R^2$ be a regular domain and $B$ be a regular non-vanishing scalar magnetic field on $\overline\Omega$. To each $\bx\in\overline\Omega$ is associated a tangent problem. According to whether $\bx$ is an interior point or a boundary point, the tangent problem is the magnetic Laplacian on the plane $\R^2$ or the half-plane $\Pi_\bx$ tangent to $\Omega$ at $\bx$, with the constant magnetic field $B_\bx\equiv B(\bx)$. The associated spectral quantities $\En(B_\bx,\R^2)$ and $\En(B_\bx,\Pi_\bx)$ are respectively equal to $|B_\bx|$ and $|B_\bx|\Theta_{0}$ where $\Theta_{0}:=\En(1,\R^2_+)$ is a universal constant whose value is close to $0.59$ (see \cite{SaGe63}). With the quantities 
\begin{equation}
\label{D:betbprime}
   b=\inf_{\bx\in\overline \Omega} |B(\bx)|, \quad  
   b'=\inf_{\bx\in \partial\Omega}|B(\bx)|, \quad \mbox{and} \quad  
   \sE(B \ee,\Omega)=\min(b,b'\Theta_{0})
\end{equation}
the asymptotic limit
\begin{equation}
\label{Metaordre0}
   \lim_{h\to0}\frac{\lambda_h(B \ee,\Omega)}{h} = \sE(B,\Omega) 
\end{equation} 
is proved by Lu and Pan in \cite{LuPan99-2}. Improvements of this result depend on the geometry and the variation of the magnetic field as we describe now.

\subsubsection*{Constant magnetic field} 
If the magnetic field is constant and normalized to $1$, then $\sE(B \ee,\Omega)=\Theta_{0}$. The following estimate is proved by Helffer and Morame:
$$
  -Ch^{3/2} \leq \lambda_h(1 \ee,\Omega)-h\Theta_{0} \leq Ch^{3/2} \ ,
$$
for $h$ small enough \cite[\S 10]{HeMo01}, while the upper bound was already given by Bernoff and Sternberg \cite{BeSt98}. 
This result is improved in \cite[\S 11]{HeMo01} in which a two-term asymptotics is proved, showing that a remainder in $O(h^{3/2})$ is optimal. Under the additional assumption that the curvature of the boundary admits a unique and non-degenerate maximum, a complete expansion of $\lambda_h(1 \ee,\Omega)$ is provided by Fournais and Helffer \cite{FouHe06}, moreover they also give a complete asymptotic expansion of the higher eigenvalues and of the associated eigenfunctions.

\subsubsection*{Variable magnetic field}
In \cite[\S 9]{HeMo01}, several different estimates for remainders are proved, function of the place where the local energy attains its minimum: In any case
\[
   -Ch^{\kappa^-} \le \lambda_h(B \ee,\Omega)-h\sE(B \ee,\Omega) \le Ch^{\kappa^+}.
\] 
with (a) $\kappa^-=\kappa^+=2$ if the minimum is attained inside the domain and (b) $\kappa^-=5/4$, $\kappa^+=3/2$ if the minimum is attained on the boundary.
Under non-degeneracy hypotheses, the optimality in the first case (a) is a consequence of \cite{HeKo11}, whereas the eigenvalue asymptotics provided in \cite{Ray09,Ray13} yields that the upper bound in the latter case (b) is sharp. Note that in \cite{Ray13}, the full asymptotic expansion of all the low-lying eigenpairs is obtained under these hypotheses, completing the analysis from \cite{FouHe06}.

\subsection{Polygonal domains}
\label{sec:2.1.2}
Let $\Omega$ be a curvilinear polygon and let $\gV$ be the (finite) set of its vertices. 
In this case, new model operators appear on infinite sectors\index{Sector} $\Pi_\bx$ tangent to $\Omega$ at vertices $\bx\in\gV$. By homogeneity $\En(B_\bx \ee,\Pi_\bx)=|B(\bx)|\En(1 \ee,\Pi_\bx)$ and by rotation invariance, $\En(1 \ee,\Pi_\bx)$ only depends on the opening $\alpha(\bx)$ of the sector $\Pi_\bx$. Let $\cS_\alpha$ be a model sector of opening $\alpha\in(0,2\pi)$. Then
$$
   \sE(B \ee,\Omega) = \min\big(b,b'\Theta_{0},
   \min_{\bx\in\gV} |B(\bx)|\,\En(1 \ee,\cS_{\alpha(\bx)})\big) \ . 
$$
In \cite[\S 11]{Bo05}, it is proved that 
$
   -Ch^{5/4} \leq \lambda_h(B \ee,\Omega)-h\sE(B \ee,\Omega) \leq Ch^{9/8}
$.
Moreover, under the assumption that a corner attracts the minimum energy
\begin{equation}
\label{eq:sless}
   \sE(B \ee,\Omega) < \min (b,b'\Theta_{0}) ,
\end{equation}
the asymptotics provided in \cite{BonDau06} yield the sharp estimates from above and below with power $h^{3/2}$.

From \cite{Ja01,Bo05} follows that for all $\alpha\in(0,\frac\pi2]$:
\begin{equation}
\label{eq:alpha}
   \En(1 \ee,\cS_{\alpha})<\Theta_{0}.
\end{equation}
Therefore condition \eqref{eq:sless} holds for constant magnetic fields as long as there is an angle opening $\alpha_\bx\le\frac\pi2$. 
Finite element computations by Galerkin projection as presented in \cite{BoDauMaVial07} suggest that \eqref{eq:alpha} still holds for all $\alpha\in(0,\pi)$. Let us finally mention that if $\Omega$ has straight sides and $B$ is constant, the convergence of $\lambda_h(B \ee,\Omega)$ to $h \sE(B \ee,\Omega)$ is exponential \cite{BonDau06}.

The asymptotic expansion of eigenfunctions and higher eigenvalues is also performed in \cite{BonDau06} under an hypothesis on the spectrum of the tangent operators at corners. We will describe these results in more details in Section \ref{SS:CornerC}.

\section{Neumann conditions in dimension 3}
\label{SS:dim3}

\subsection{Regular domains}
\label{SS:regulardomain}
For a continuous magnetic field $\bB$ it is known (\cite{LuPan00} and \cite{HeMo02}) that \eqref{Metaordre0} holds. In that case 
$$
   \sE(\bB \ee,\Omega)=
   \min\big(\inf_{\bx\in\Omega}|\bB(\bx)|,
   \inf_{\bx\in\partial\Omega}|\bB(\bx)| \ee\sigma(\theta(\bx))\big),
$$ 
where $\theta(\bx)\in[0,\frac\pi2]$ denotes the unoriented angle between the magnetic field and the boundary at point $\bx$, and the quantity $\sigma(\theta)$ is the bottom of the spectrum of a model problem, cf.\ Section \ref{SS:HS}.

\subsubsection*{Constant magnetic field}
Here the magnetic field $\bB$ is assumed without restriction to be of unit length. Then there exists a non-empty set $\Sigma$ of $\partial\Omega$ on which $\bB(\bx)$ is tangent to the boundary, which implies that
$
   \sE(\bB \ee,\Omega)=\Theta_{0} 
$. 
Then Theorem 1.1 of \cite{HeMo04} states that
$$
  |\lambda_h(\bB \ee,\Omega)-h\sE(\bB \ee,\Omega)| \leq Ch^{4/3}.
$$
Under some extra assumptions on $\Sigma$, Theorem 1.2 of \cite{HeMo04} yields a two-term asymptotics for $\lambda_h(\bB \ee,\Omega)$ showing the optimality of the previous estimate.

\subsubsection*{Variable magnetic field}
For a smooth non-vanishing magnetic field there holds \cite[Theorem 9.1.1]{FouHe10} (see also \cite{LuPan00})
$
   |\lambda_h(\bB \ee,\Omega)-h\sE(\bB \ee,\Omega)| \leq Ch^{5/4}
$.
In \cite[Remark 6.2]{HeMo04}, the upper bound is improved to $Ch^{4/3}$.
Finally, under extra assumptions,
a three-term quasimode is constructed in \cite{Ray09-3d}, providing the sharp upper bound $C h^{3/2}$.

\subsection{Singular domains}
Until now, two examples of non-smooth domains have been addressed in the literature. In both cases, the magnetic field $\bB$ is assumed to be constant.
 
\subsubsection*{Rectangular cuboids}
This case is considered by Pan \cite{Pan02}:
The asymptotic limit \eqref{Metaordre0} holds for such a domain and there exists a vertex $\bv\in\gV$ such that $\sE(\bB \ee,\Omega)=\En(\bB \ee,\Pi_\bv)$.
Moreover, in the case where the magnetic field is tangent to a face but is not tangent to any edge, we have   
$$
   \En(\bB \ee,\Pi_\bv) < 
   \inf_{\bx\in \overline{\Omega}\setminus \gV }\En(\bB \ee,\Pi_{\bx}).
$$

\subsubsection*{Lenses} 
The domain $\Omega$ is supposed to have two faces separated by an edge $\be$ that is a regular loop contained in the plane $x_3=0$. The magnetic field considered is $\bB=(0,0,1)$.
It is proved in \cite{PoTh} that, 
if the opening angle $\alpha$ of the lens is constant and $\leq 0.38\pi$,
$$
   \inf_{\bx\in \be}\En(\bB,\Pi_{\bx}) < 
   \inf_{\bx\in \overline{\Omega}\setminus \be}\En(\bB,\Pi_{\bx})  
$$
and that the asymptotic limit \eqref{Metaordre0} holds with an estimate in $C h^{5/4}$ from above and below. 
When the opening angle of the lens is variable and under some non-degeneracy hypotheses, a complete eigenvalue asymptotics is obtained in \cite{PoRay12} resulting into the optimal error estimate in $C h^{3/2}$.

\part{Corner structure and lower bounds}
\label{part:2}

\chapter{Domains with corners and their singular chains}
\label{sec:chain}
\index{Chain|see{Singular chain}}
\index{Plane sector|see{Sector}}

Domains with corners are widely addressed in the subject of Partial Differential Equations, mainly in connection with elliptic boundary problems. The pioneering work in this area is the paper \cite{Kondratev67} by Kondrat'ev devoted to domains with conical singularities. Such a domain is locally diffeomorphic to cones with smooth sections. It is singular at a finite number of points, called vertices or corners, see Figure \ref{F1A}, p.~\pageref{F1}. Domains with edges are locally diffeomorphic to a wedge and singular points form a submanifold of the boundary, see Figure \ref{F1B}. They were addressed in \cite{Kondratev70,MazyaPlamenevskii80b,Mazzeo91} among others. A combination of corners and edges in dimension 3 or higher produces a delicate interaction of several distinct singular types, see Figure \ref{F1C}. Such domains can be classified as ``corner domains'' or ``manifold with corners''. A Fredholm theory was initiated by Maz'ya and Plamenevskii \cite{MazyaPlamenevskii73,MazyaPlamenevskii77}. Since then, different aspects have been addressed, singularities \cite{Dau88,KozlovMazyaRossmann01}, pseudodifferential calculus \cite{Schulze91,Melrose91,Melrose92,MelroseBook,Schulze11}, regularity in analytic weighted spaces \cite{GuoBabuska97a,CoDaNi12}, among many others, and without mentioning the huge literature on numerical approximation.

In this work, for the sake of completeness and for ease of further discussion, we introduce a class of corner domains\index{Corner domain} with a Cartesian structure 
in any space dimension $n$. This definition is recursive over the dimension, through two intertwining classes of domains
\begin{itemize}
\item[a)] $\gP_n$, a class of infinite open cones in $\R^n$. 
\item[b)] 
$\gD(M)$, a class of bounded connected open subsets of a smooth manifold without boundary --- actually, $M=\R^n$ or $M=\dS^n$, with $\dS^n$ the unit sphere of $\R^{n+1}$.
\end{itemize}
Such definition is in the same spirit as \cite[Section 2]{Dau88}.

\section{Tangent cones and corner domains}
\label{SS:tangent}
We call a \emph{cone}\index{Cone|textbf} any open subset $\Pi$ of $\R^n$ satisfying 
\[
   \forall\rho>0 \ \ \mbox{and}\ \ \bx\in\Pi,\quad \rho \bx\in\Pi,
\]
and the \emph{section}\index{Section of a cone|textbf} of the cone $\Pi$ is its subset $\Pi\cap\dS^{n-1}$. Note that $\dS^0=\{-1,1\}$.

\begin{definition}[\sc Tangent cone]\index{Tangent cone|textbf}
\label{D:Tangentcones}
Let $\Omega$ be an open subset of $M=\R^n$ or $M=\dS^n$. Let $\bx_{0}\in \overline{\Omega}$. 
The cone $\Pi_{\bx_{0}}$\index{Tangent cone!$\Pi_{\bx_{0}}$|textbf} is said to be \emph{tangent to} $\Omega$ at $\bx_{0}$ if there exists a local $\sC^\infty$ diffeomorphism $\diffeo^{\bx_{0}}$ which maps a neighborhood $\cU_{\bx_{0}}$ of $\bx_{0}$ in $M$ onto a neighborhood $\cV_{\bx_{0}}$ of $\bfz$ in $\R^n$ and such that
\begin{equation}
\label{eq:diffeo}
   \diffeo^{\bx_{0}}(\bx_0) = \bfz,\quad
   \diffeo^{\bx_{0}}(\cU_{\bx_{0}}\cap\Omega) = \cV_{\bx_{0}}\cap\Pi_{\bx_{0}} 
   \quad\mbox{and}\quad
   \diffeo^{\bx_{0}}(\cU_{\bx_{0}}\cap\partial\Omega) = \cV_{\bx_{0}}\cap\partial\Pi_{\bx_{0}} .
\end{equation}
We denote by $\rJ^{\bx_{0}}$\index{Jacobian!$\rJ^{\bx_{0}}$|textbf} the Jacobian\index{Jacobian|textbf} of the inverse of $\diffeo^{\bx_{0}}$, that is 
\begin{equation}\label{eq.jacob}
\rJ^{\bx_{0}}({\bv}):=\rd_{\bv}(\diffeo^{\bx_{0}})^{-1}(\bv), \qquad \forall{\bv}\in\cV_{\bx_{0}}\, .
\end{equation}
We also assume that the Jacobian at $\bfz$ is the identity matrix: $\rJ^{\bx_{0}}(\bfz)=\Id_{n}$.
The open set $\cU_{\bx_{0}}$ 
is called a \emph{map-neighborhood}\index{Map-neighborhood|textbf}
and $(\cU_{\bx_{0}},\diffeo^{\bx_{0}})$\index{Local map!$(\cU_{\bx_{0}},\diffeo^{\bx_{0}})$|textbf} a \emph{local map}\index{Local map|textbf}.

The metric\index{Metric|textbf} associated with the local map $(\cU_{\bx_{0}},\diffeo^{\bx_{0}})$ is denoted by $\rG^{\bx_0}$\index{Metric!$\rG^{\bx_0}$|textbf} and defined as
\begin{equation}
\label{eq:metric}
   \rG^{\bx_0} = (\rJ^{\bx_0})^{-1} ((\rJ^{\bx_0})^{-1})^\top.
\end{equation} 
The metric $\rG^{\bx_0}$ at $\bfz$ is the identity matrix. 
\end{definition}

Because of the constraint $\rJ^{\bx_{0}}(\bfz)=\Id_{n}$, the tangent cone $\Pi_{\bx_{0}}$ does not depend on the choice of the map-neighborhood $\cU_{\bx_{0}}$ or the local map $(\cU_{\bx_{0}},\diffeo^{\bx_{0}})$. Therefore when there exists a tangent cone to $\Omega$ at $\bx_{0}$, it is unique. Note also that the constraint $\rJ^{\bx_{0}}(\bfz)=\Id_{n}$ is not restrictive for the domains: If there exists a local map $\rJ$ at $\bx_0$ that does not fulfil this constraint, it suffices to consider the new map $\rJ(\bfz)^{-1}\circ\rJ$ to remedy this.

\begin{definition}[\sc Class of corner domains]\index{Corner domain|textbf}
\label{D:cornerdomains}
The classes of corner domains $\gD(M)$\index{Corner domain!class of --\quad$\gD(M)$\ |textbf} ($M=\R^n$ or $M=\dS^n$) and tangent cones $\gP_{n}$ are defined as follows: 

\noindent{\sc Initialization}, $n=0$: 
\begin{enumerate}
\item $\gP_0$ has one element, $\{0\}$,
\item $\gD(\dS^0)$ is formed by all (non empty) subsets of $\dS^0$.
\end{enumerate}

\noindent{\sc Recurrence}: For $n\ge1$,
\begin{enumerate}
\item $\Pi\in\gP_n$ if and only if the section of $\Pi$ belongs to $\gD(\dS^{n-1})$,
\item \label{machin}
$\Omega\in\gD(M)$ if and only if for any $\bx_{0}\in\overline\Omega$, 
there exists a tangent cone $\Pi_{\bx_{0}}\in\gP_n$ to $\Omega$ at $\bx_{0}$.
\end{enumerate}
\end{definition}

Polyhedral domains and polyhedral cones form important subclasses of $\gD(M)$ and $\gP_n$.

\begin{definition}[\sc Class of polyhedral cones and domains] \index{Polyhedral domain|textbf} \index{Polyhedral cone}
The classes of polyhedral domains $\ogD(M)$\index{Polyhedral domain!set of -- \quad$\ogD(M)$\ |textbf} ($M=\R^n$ or $M=\dS^n$) and polyhedral cones $\ogP_{n}$\index{Polyhedral cone!set of -- \quad$\ogP_{n}$\ |textbf} are defined as follows: 

\begin{enumerate}
\item The cone $\Pi\in \gP_n$ is a polyhedral cone if its boundary is contained in a finite union of subspaces of codimension $1$. We write $\Pi\in\ogP_n$.
\item  The domain $\Omega\in\gD(M)$ is a polyhedral domain if all its tangent cones $\Pi_\bx$ are polyhedral. We write $\Omega\in\ogD(M)$.
\end{enumerate}
\end{definition}

Here is a rapid description of corner domains in lower dimensions $n=1,2,3$.
\pagebreak[3]

\begin{enumerate}[(i)]
\item $n=1$
\begin{itemize}
\item[a)] The elements of $\gP_1$ are $\R$, $\R_+$ and $\R_-$.
\item[b)] The elements of $\gD(\dS^1)$ are $\dS^1$ and all open intervals $\cI\subset\dS^1$ such that $\overline \cI\neq\dS^1$.
\end{itemize}

\item $n=2$ 
\begin{itemize}
\item[a)] The elements of $\gP_2$ are $\R^2$ and all plane sectors\index{Sector|textbf} with opening $\alpha\in(0,2\pi)$, including half-planes ($\alpha=\pi$).  
\item[b)] The elements of $\gD(\R^2)$ are curvilinear polygons\index{Curvilinear polygon} with piecewise non-tangent smooth sides (corner angles $\alpha\neq0,\pi,2\pi$). Note that $\gD(\R^2)$ includes smooth domains.
\item[c)] The elements of $\gD(\dS^2)$ are $\dS^2$ and all curvilinear polygons with piecewise non-tangent smooth sides in the sphere $\dS^2$.
\end{itemize}

\item $n=3$
\begin{itemize}
\item[a)] The elements of $\gP_3$ are all cones\index{Cone} with section in $\gD(\dS^2)$. This includes $\R^3$, half-spaces\index{Half-space}, wedges\index{Wedge} and many different cones like octants\index{Octant} or circular cones\index{Circular cone}.  
\item[b)] The elements of $\gD(\R^3)$ are tangent in each point $\bx_0$ to a cone $\Pi_{\bx_0}\in\gP_3$. Note that the nature of the section of the tangent cone determines whether the 3D domain has a vertex\index{Vertex}, an edge\index{Edge}, or is regular near $\bx_0$. 
\end{itemize}

\end{enumerate}

\begin{figure}[ht]
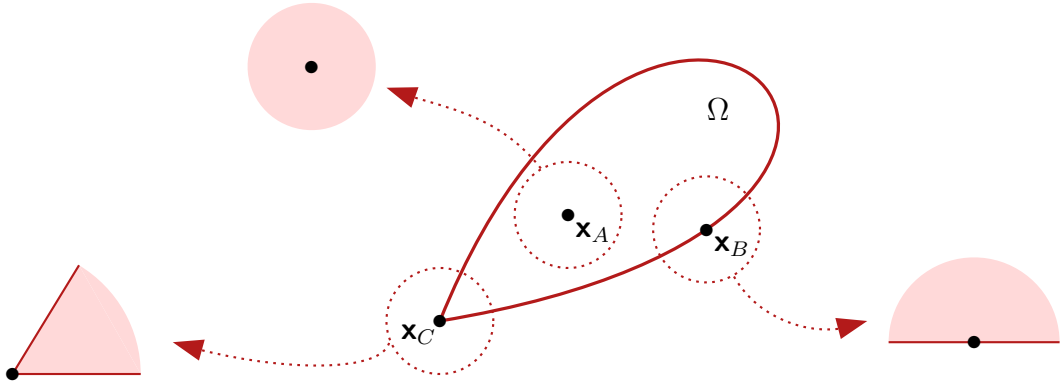

     \figinit{4pt}
     \figpt 1:(0,0)
     \figpt 2:(60,10)
     \figpt 3:(20,50)
     \figpt 4:(12,10)
     \figptBezier 5::.2[1,2,3,1]
     \figdrawbegin{}
     \figset (width=1.2)
     \figset (color = .7 .1 .1)
     \figdrawBezier 1 [1,2,3,1]
     \figset (width=.8,dash=5)
     \figset arrowhead (fill=yes,length=3)
     \figdrawcirc 1(5)
     \figdrawcirc 4(5)
     \figdrawcirc 5(5)
     \figptcirc 21::4;5 (120)
     \figpt 27:(5,20)
     \figpt 28:(0,20)
     \figpt 29:(-5,22)
     \figdrawarrowBezier [21,27,28,29]
     \figpt 22:(-12,24)
     \figset (fill=yes,color = 1. .85 .85)
     \figdrawcirc 22(6)
     \figset (fill=no,color = .7 .1 .1)
     \figptcirc 31::5;5 (300)
     \figpt 37:(30,0)
     \figpt 38:(35,-2)
     \figpt 39:(40,0)
     \figdrawarrowBezier [31,37,38,39]
     \figpt 32:(50,-2)
     \figpt 33:(42,-2)
     \figpt 34:(58,-2)
     \figset (fill=yes,color = 1. .85 .85)
     \figdrawarccirc 32; 8 (0,180)
     \figset (fill=no,color = .7 .1 .1,dash=1)
     \figdrawline [33,34]
     \figset (dash=5)
     \figptcirc 11::1;5 (200)
     \figpt 17:(-5,-5)
     \figpt 18:(-15,-5)
     \figpt 19:(-25,-2)
     \figdrawarrowBezier [11,17,18,19]
     \figpt 12:(-40,-5)
     \figgetangle \Value [1,2,3]
     \figpt 13:(-28,-5)
     \figptrot 14:=13/12,\Value/
     \figset (fill=yes,color = 1. .85 .85)
     \figgetdist \len [12,13]
     \figdrawarccircP 12; \len[13,14]
     \figdrawline[13,12,14]
     \figset (fill=no,color = .7 .1 .1,dash=1)
     \figdrawline[13,12,14]
     \figdrawend
     \figvisu{\figbox}{}{%
     \figpt 0:(25,20)
     \figwritee 0:{$\Omega$}(0pt)
     \figsetmark{$\figBullet$}
     \figwritesw 1:{$\bx_{C}$}(0.5)
     \figwritese 5:{$\bx_{B}$}(1)
     \figwritese 4:{$\bx_{A}$}(1)
     \figwritew 12:$$(0pt)
     \figwritew 22:$$(0pt)
     \figwritew 32:$$(0pt)}
     \centerline{\box\figbox}
\caption{Tangent sectors for a curvilinear polygonal domain\label{fig.dim2}}
\end{figure}

In Figure~\ref{fig.dim2}, we show an example of a domain belonging to $\gD(\R^2)$ with some of its tangent sectors. 
For the dimension $n=3$, examples are given in Figure \ref{F1}, p.~\pageref{F1}. Those in Figure \ref{F1A} have corners and are not polyhedral, whereas those in Figure \ref{F1B} have only edges and are polyhedral. In Figure \ref{F1C}, domains have both corners and edges, the first one is polyhedral whereas the second is not.
In Figure  \ref{F3} we display two of these examples with their tangent cones at one of their vertices.

\begin{figure}[ht]
\centering
    \begin{subfigure}[h]{0.48\textwidth}
\includegraphics[keepaspectratio=true,width=7.75cm]{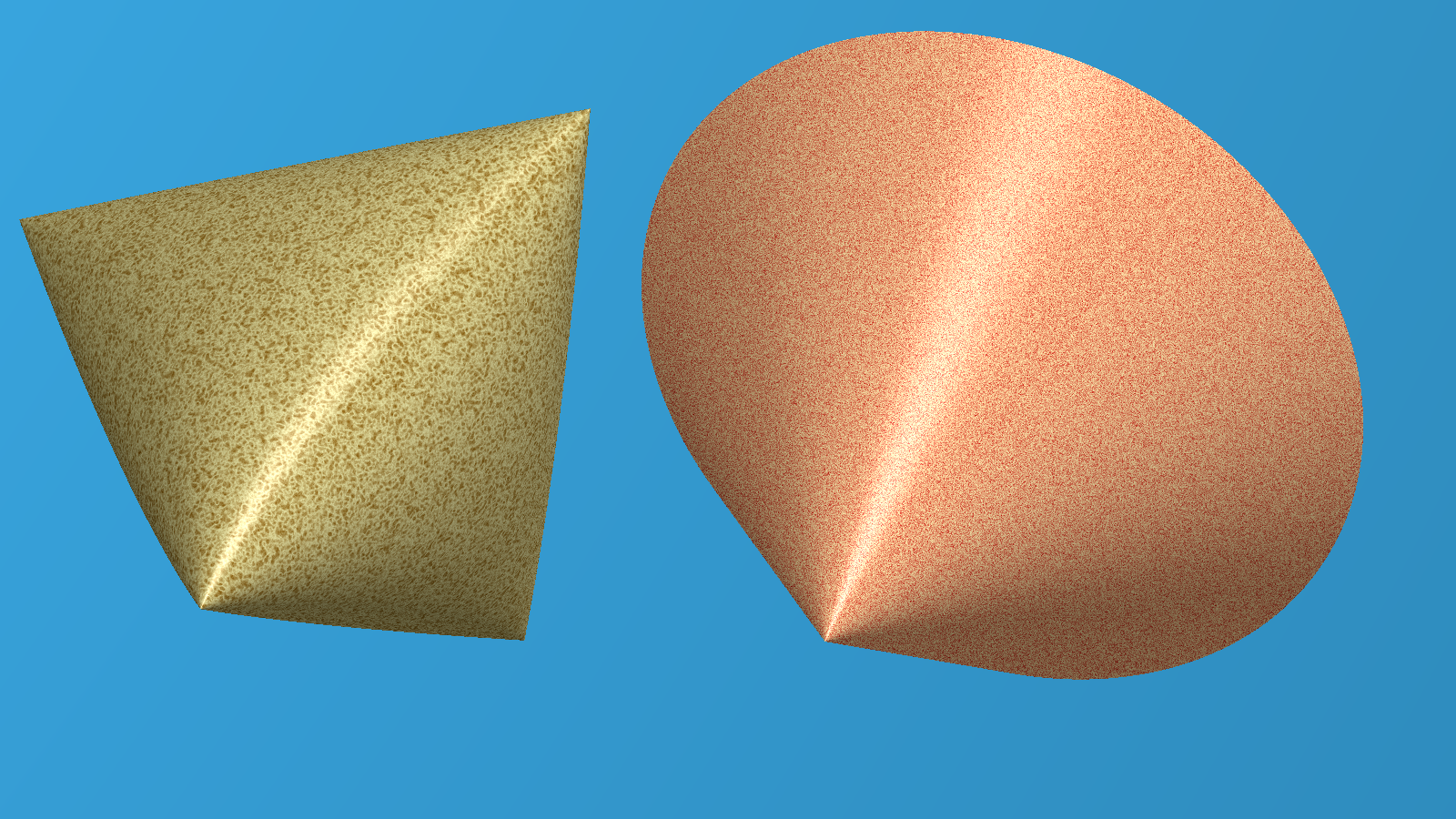}
        \caption{The Cayley tetrahedron}
        \label{F3A}
    \end{subfigure}
\
            \begin{subfigure}[h]{0.48\textwidth}
\includegraphics[keepaspectratio=true,width=7.75cm]{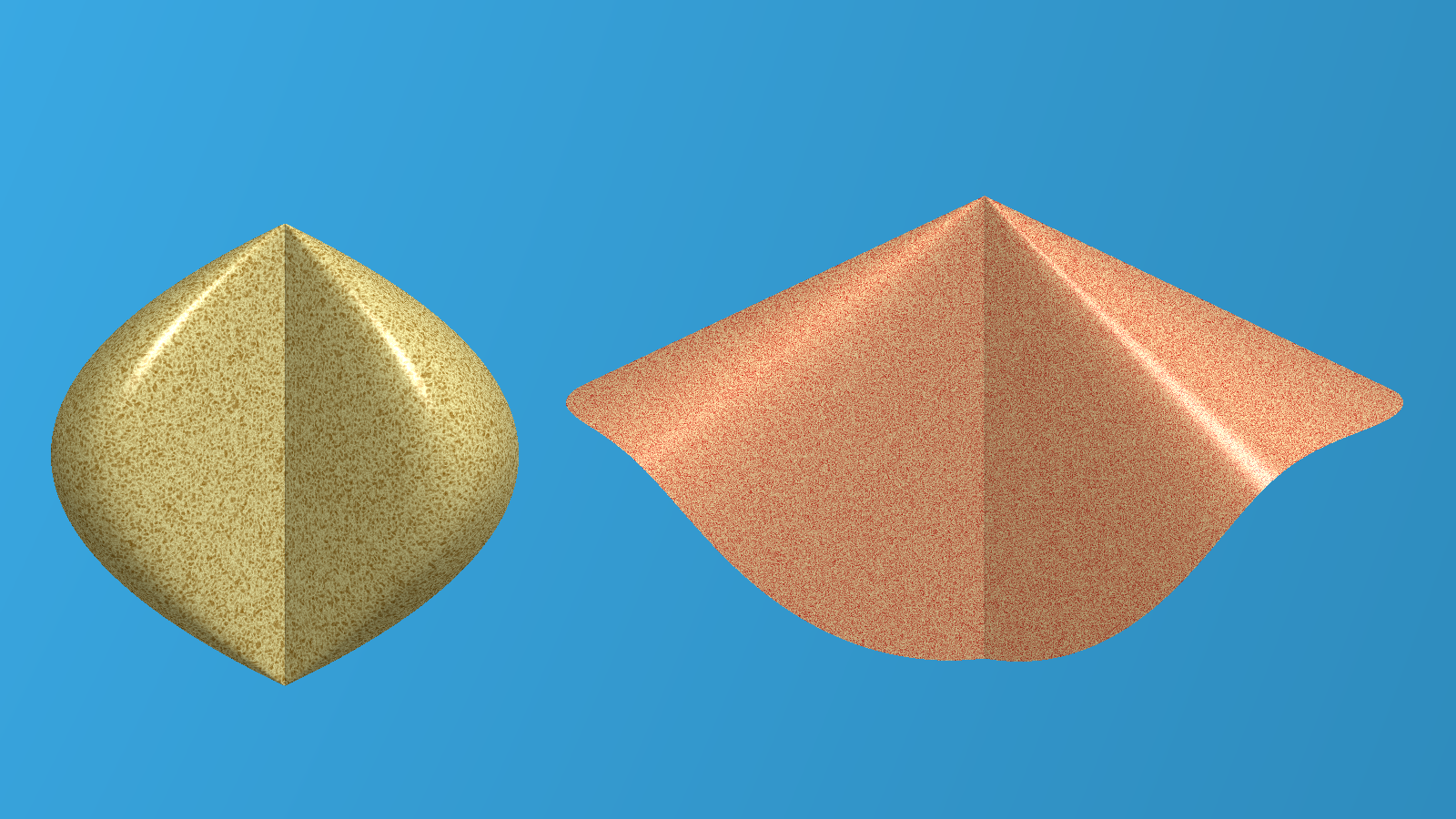}
        \caption{A domain with edge and corners}
        \label{F3B}
    \end{subfigure}
\caption{Examples of 3D corner domains with their tangent cones at vertices}
\label{F3}
\end{figure}

We will give later on (Section \ref{SS:3D}) a more exhaustive description of the class $\gD(\R^3)$ of 3D corner domains.

\begin{remark}
\label{R:2D}
In dimension 2, the cones are sectors\index{Sector}. So their sides are contained in one-dimensional subspaces, and they are ``polyhedral''. We deduce that
\begin{equation}
\label{E:poly2D}
   \gP_2=\ogP_2 \quad\mbox{and}\quad \gD(M) = \ogD(M)\ \ \mbox{for}\ \ M=\R^2 \ \mbox{or}\ \dS^2.
\end{equation} 
In dimension 3, a non-degenerate circular cone\index{Circular cone} ({\it i.e.}, different from $\R^3$ or a half-space) is not polyhedral, whereas an octant is.
\end{remark}

The recursive procedure of Definition \ref{D:cornerdomains} may generate various classes of domains. Let us give two examples:
\begin{enumerate}
\item In \cite{MazyaPlamenevskii77}, recursive sytem of {\em cylindrical coordinates} are used to define corner domains. This provides a larger class than ours. For instance, in dimension $n=2$, any piecewise smooth domain is admissible, except outward cusps. This definition of domains fits with operators that are regular with respect to such system of coordinates, and not only in Cartesian coordinates.
\item In \cite{Dau88}, cracks and slits of any dimension are admissible. The recursive definition is similar to ours with the exception that a boundary point can be associated with several distinct maps. Such a framework does not seem to be essential for our study, although this generalization would be possible.
\end{enumerate}
The definition of manifolds with corners \cite{MelroseBook} is not recursive: Manifolds are defined through an atlas \index{Atlas} of maps with domains contained in $\R^k_+\times\R^{n-k}$ for any $k=0,\ldots,n$. Any manifolds with corners that is a domain in $\R^n$ belongs to $\gD(\R^n)$ and even to $\ogD(\R^n)$ (it is polyhedral), but the converse is not true.

\begin{remark}
\label{rem:Lip}
In dimension $n=2$, any domain in $\gD(\R^n)$ has a Lipschitz boundary\index{Lipschitz domain}, but in dimension $n\ge3$, this is no longer true. However any corner domain is a finite union of Lipschitz domains, cf.\ \cite[Lemma (AA.9)]{Dau88}.
\end{remark}

\section{Admissible atlases}
We are going to introduce the notion of admissible atlas for a corner domain, so that the associated diffeomorphisms satisfy some uniformity properties. We need some definitions and preliminary results first.

\begin{notation}\label{not.N}
For $\bv\in \R^n$, we denote by $\langle \bv\rangle$ the vector space generated by $\bv$. 
For $r>0$, we denote by 
$\rN_{r}(\bv):=r^{-1}\bv $ the scaling of ratio $r^{-1}$.  
Note that $\rN_{r^{-1}}=\rN_r^{-1}$.
\end{notation}

The following lemma illustrates the coherence 
of Definition \ref{D:Tangentcones}.

\begin{lemma}
\label{L:pointtoopen}
Let $\Omega$ be an open subset of $M$ and $\bx_{0}\in \overline{\Omega}$ such that there exists a tangent cone $\Pi_{\bx_{0}}\in \gP_{n}$ to $\Omega$ at $\bx_{0}$ with map-neighborhood $\cU_{\bx_{0}}$. Then for all $\bu_0\in \cU_{\bx_{0}}\cap \overline{\Omega}$ there exists a tangent cone $\Pi_{\bu_0}\in \gP_{n}$ to $\Omega$ at $\bu_0$.
\end{lemma}

\begin{proof}
Let $\bu_0\in \cU_{\bx_{0}}\cap \overline{\Omega}$. We have to prove that there exists a tangent cone $\Pi_{\bu_0}$ at $\bu_0$ in the sense of Definition \ref{D:Tangentcones} and that $\Pi_{\bu_0}\in \gP_{n}$. Let $\widehat\Omega_{\bx_{0}}=\Pi_{\bx_{0}}\cap \dS^{n-1}$ be the section of $\Pi_{\bx_{0}}$. Let 
$(\cU_{\bx_{0}},\diffeo^{\bx_{0}})$ be a local map and
$\bv_0=\diffeo^{\bx_{0}}(\bu_0) \in \overline\Pi_{\bx_{0}}$. We denote by $(r(\bv_0),\theta(\bv_0))\in (0,+\infty)\times \overline{\widehat\Omega}_{\bx_0}$ its polar coordinates:
\begin{equation}
\label{D:thetaby}
   r(\bv_0):=\|\bv_0\|\qquad\mbox{ and }\qquad\theta(\bv_0):=\frac{\bv_0}{\| \bv_0 \|}\,.
\end{equation}
By the recursive definition there exists a tangent cone $\Pi_{\theta(\bv_0)}\in \gP_{n-1}$ to $\widehat\Omega_{\bx_{0}}$ at $\theta(\bv_0)$. Let $\diffeo^{\theta(\bv_0)}$ be an associated diffeomorphism which sends a map-neighborhood $\cU_{\theta(\bv_0)}$ of $\theta(\bv_0)$ onto a neighborhood $\cV_{\theta(\bv_0)}$ of $\bfz\in \R^{n-1}$. We may assume without restriction that there exists a $n$-dimensional ball with center $\theta(\bv_0)$ and radius $\rho_1\in(0,1)$ such that
\begin{equation}
\label{eq:cUtv0}
   \cU_{\theta(\bv_0)} = \cB(\theta(\bv_0),\rho_1) \cap \dS^{n-1}.
\end{equation}
Then we set\footnote{We distinguish between the point $\theta(\bv_{0})\in\widehat\Omega_{\bx_{0}}$ and its polar coordinates $(1,\theta(\bv_{0}))$.} $\cU_{(1,\theta(\bv_0))}=\cB(\theta(\bv_0),\rho_1)$ and define on $\cU_{(1,\theta(\bv_0))}$ the diffeomorphism---using polar coordinates $(r(\bv),\theta(\bv))$:
\begin{equation}
\label{D:diffeotheta}
\diffeo^{(1,\theta(\bv_0))}:\bv \mapsto (r(\bv)-1,\diffeo^{\theta(\bv_0)}(\theta(\bv))) \, .
\end{equation}
There holds $\rd_{(1,\theta(\bv_0))}\diffeo^{(1,\theta(\bv_0))}=\Id_{n}$. Define 
\begin{equation}
\label{D:Piy}
\Pi_{\bv_0}:=\langle \bv_0 \rangle \times \Pi_{\theta(\bv_0)} \,.
\end{equation}
 Notice that $\Pi_{\bv_0}\in \gP_{n}$. It is the tangent cone to $\Pi_{\bx_{0}}$ at the point $(1,\theta(\bv_0))$ and $\diffeo^{(1,\theta(\bv_0))}$ maps $\cU_{(1,\theta(\bv_0))}$ on a neighborhood of $\bfz\in \R^{n}$.
 Let
\begin{equation}
\label{D:diffeoy}
\diffeo^{\bv_0}:=\rN_{r(\bv_0)}^{-1}\circ \diffeo^{(1,\theta(\bv_0))} \circ \rN_{r(\bv_0)}\,.
\end{equation}
Then $\diffeo^{\bv_0}$ is a diffeomorphism defined on
\begin{equation}
\label{eq:mapv0}
   \cU_{\bv_0}:=\|\bv_0\|\ \cU_{(1,\theta(\bv_0))} = \cB(\bv_0,\rho_1 \|\bv_0\|).
\end{equation} 
Let us define
\begin{equation}
\label{eq:mapu0}
   \cU_{\bu_0} := (\diffeo^{\bx_{0}})^{-1}(\cU_{\bv_0})\,.
\end{equation} 
It is a neighborhood of $\bu_0$. 
Let 
\begin{equation}
\label{D:diffeox}
\diffeo^{\bu_0}(\bu):=
\rJ^{\bx_{0}}(\bv_0) \  (\diffeo^{\bv_0}\circ \diffeo^{\bx_{0}}(\bu)) \,
\end{equation}
be defined for $\bu\in\cU_{\bu_0}$. Note that the differential of $\diffeo^{\bu_0}$ at the point $\bu_0$ is the identity matrix $\Id_n$. 
Let us set finally 
\begin{equation}
\label{D:Pix}
\Pi_{\bu_0}:=
\rJ^{\bx_{0}}(\bv_0)(\Pi_{\bv_0}) \, .
\end{equation}
 Then the map-neighborhood $\cU_{\bu_0}$,  the diffeomorphism $\diffeo^{\bu_0}$ and the cone $\Pi_{\bu_0}$ satisfy the requirements of Definition \ref{D:Tangentcones} and $\Pi_{\bu_0}$ is the tangent cone to $\Omega$ at $\bu_0$. Since $\Pi_{\bv_0}\in \gP_{n}$, there holds $\Pi_{\bu_0}\in \gP_{n}$.
\end{proof}

\begin{remark}
\label{rem:pto}
If the tangent cone $\Pi_{\bx_0}$ is \emph{polyhedral}\index{Polyhedral cone}, the procedure for constructing $\diffeo^{\bu_0}$ can be simplified as follows: We define $\bv_0$ and its polar coordinates $(r(\bv_{0}),\theta(\bv_{0}))$ as before. Since $\Pi_{\bx_0}$ is polyhedral, the ball $\cB(\theta(\bv_0),\rho_1)$ \eqref{eq:cUtv0} is such that the set $\widetilde\cU:=\overline{\cB(\theta(\bv_0),\rho_1)}\cap\Pi_{\bx_0}$ is homogeneous with respect to $\theta(\bv_0)$, that is
\[
   \bv\in\widetilde\cU\ \ \mbox{and}\ \  \rho\in \Big[0,\frac{\rho_1}{\|\bv-\theta(\bv_{0})\|}\Big]
   \quad\Longrightarrow\quad
   \rho\bv+(1-\rho)\theta(\bv_0)\in\widetilde\cU.
\] 
The set $\widetilde\cV:=\{\bv\in\R^n|\ \bv+\theta(\bv_0) \in \widetilde\cU\}$ defines a polyhedral cone $\widetilde\Pi$ in a natural way by $\{\bv\in\R^n|\ \exists\rho>0\ \rho\bv\in\widetilde\cV\}$. Defining $\diffeo^{\bv_0}$ as the translation $\rT_{\bv_{0}}: \bv\mapsto\bv-\bv_0$, we find that $\widetilde\Pi=\Pi_{\bv_0}$. Then, with this simple definition of $\diffeo^{\bv_0}$ we still define $\diffeo^{\bu_0}$ by \eqref{D:diffeox}. On the other hand, by uniqueness of tangent cones, the new definition of $\Pi_{\bv_0}$ coincides with the old one \eqref{D:Piy}. Finally, $\Pi_{\bu_0}$ is still defined by \eqref{D:Pix}.
\end{remark}

\begin{lemma}
\label{L:eqnorm}
Let $(\cU_{\bx_0},\diffeo^{\bx_0})$ be a local map\index{Local map} with image a neighborhood $\cV_{\bx_0}$ of $\bfz$, and such that $\rJ^{\bx_0}(\bfz)=\Id_n$. There exists $r_0>0$ such that $\cB(\bfz,r_0)\subset \cV_{\bx_0}$ and for any $\bv,\bv'\in\cB(\bfz,r_0)$
\begin{equation}
\label{eq:eqnormu} 
   \|\bu'-\bu-(\bv'-\bv)\| \le 
   \tfrac12\, \|\bv'-\bv\|,\qquad\mbox{ with }\quad 
   \bu = (\diffeo^{\bx_0})^{-1}(\bv),\quad \bu' = (\diffeo^{\bx_0})^{-1}(\bv')\,.
\end{equation}
\end{lemma}

\begin{proof}
Let $r_{1}$ be such that $\bv,\bv'\in \cB(\bfz,r_{1})\subset \cV_{\bx_{0}}$. A Taylor expansion of $(\diffeo^{\bx_0})^{-1}(\bv')$ around $\bv$ gives 
$$\|(\diffeo^{\bx_0})^{-1}(\bv') - (\diffeo^{\bx_0})^{-1}(\bv) - \rJ^{\bx_0}(\bv)(\bv'-\bv)\|
\leq \tfrac12\, \| \rd \rJ^{\bx_0}\|_{L^\infty(\cB(\bfz,r_{1}))} \|\bv'-\bv\|^2.$$
Another Taylor expansion of $\rJ^{\bx_0}(\bv)$ around $\bfz$ gives 
$$
\| \rJ^{\bx_0}(\bv) - \rJ^{\bx_0}(\bfz)\|
\leq  \| \rd \rJ^{\bx_{0}}\|_{L^\infty(\cB(\bfz,r_{1}))} \|\bv\|\ .
$$
Since $(\diffeo^{\bx_0})^{-1}(\bv)=\bu$, $(\diffeo^{\bx_0})^{-1}(\bv')=\bu'$ and $\rJ^{\bx_0}(\bfz)=\Id_{n}$, we deduce
$$\|(\bu'-\bu)-(\bv'-\bv)\|\leq \Big(\| \rd \rJ^{\bx_{0}}\|_{L^\infty(\cB(\bfz,r_{1}))} \|\bv\|\ 
+\tfrac12\, \| \rd \rJ^{\bx_0}\|_{L^\infty(\cB(\bfz,r_{1}))} \|\bv'-\bv\|\Big)\ \|\bv'-\bv\|.
$$
If we choose $r_{0}\leq \min\big\{r_{1},\, 1/(4\| \rd \rJ^{\bx_{0}}\|_{L^\infty(\cB(\bfz,r_{1}))})\big\} $, we have
$$\| \rd \rJ^{\bx_{0}}\|_{L^\infty(\cB(\bfz,r_{1}))} \|\bv\|\ 
+\tfrac12\, \| \rd \rJ^{\bx_0}\|_{L^\infty(\cB(\bfz,r_{1}))} \|\bv'-\bv\| \leq \tfrac 12,\qquad
\forall \bv,\bv'\in\cB(\bfz,r_{0}),$$
which ends the proof.
\end{proof}

\begin{proposition}
\label{P:recouvrementfini}
(i) The domain $\Omega$ belongs to $\gD(\R^n)$ if and only if there exists a finite set $\gX\subset \overline{\Omega}$  satisfying the two following conditions 
\begin{enumerate}
\item For each $\bx_0\in\gX$, there exists a cone $\Pi_{\bx_0}\in \gP_{n}$ and a local map $(\cU_{\bx_{0}},\diffeo^{\bx_{0}})$ such that \eqref{eq:diffeo} holds,
\item The set $\overline\Omega$ is covered by the union of the map neighborhoods $\cU_{\bx_0}$ for $\bx_0\in\gX$. 
\end{enumerate}
\noindent
(ii) The equivalence (i) still holds if one requires that for all $\bx_0\in\gX$ and all $\bu,\bu'\in\cU_{\bx_0}$, 
\eqref{eq:eqnormu} holds.
\end{proposition}

\begin{proof}
 (i) The ``if'' direction is a consequence of the definition of $\gD(\R^n)$ and, in particular, the fact that $\overline{\Omega}$ is compact and can be covered by a finite number of map-neighborhoods. The ``only if'' direction is a consequence of the compactness of $\overline{\Omega}$ and of Lemma \ref{L:pointtoopen}.

(ii) is then a consequence of Lemma \ref{L:eqnorm} (and of the compactness of $\overline\Omega$, of course).
\end{proof}

\begin{definition}[\sc Admissible atlas]\label{def.atlas}\index{Atlas!Admissible atlas|textbf}
Let $\Omega\in \gD(M)$. An atlas $(\cU_{\bx},\diffeo^{\bx})_{\bx\in \overline{\Omega}}$ 
is called \emph{admissible} if it comes from the following recursive procedure:
 \begin{enumerate}
 \item Take a finite set $\gX\subset \overline{\Omega}$ as in Proposition \ref{P:recouvrementfini} together with the associated map-neighbor\-hoods and diffeomorphisms $(\cU_{\bx_{0}},\diffeo^{\bx_{0}})$ for $\bx_{0}\in \gX$, satisfying moreover \eqref{eq:eqnormu}. 
 \item Assume that for each $\bx_0\in\gX$ the map-neighborhood $\cU_{\bx_{0}}$  contains a ball $\cB(\bx_0,2R_{\bx_0})$ for some $R_{\bx_0}>0$ and that the balls with half-radius  $\cB(\bx_0,R_{\bx_0})$ cover $\overline\Omega$.
 \item All the other map-neighborhoods and diffeomorphisms $(\cU_{\bx},\diffeo^{\bx})$ with $\bx\in\overline{\Omega}\setminus \gX$ are constructed by the recursive procedure \eqref{D:thetaby}--\eqref{D:diffeox}, based on admissible atlases for the sections $\widehat\Omega_{\bx_0}$ associated with the set of reference points $\bx_0\in\gX$. In the polyhedral case, the straightforward construction described in Remark \ref{rem:pto} is preferred.
\end{enumerate}
\end{definition}

As a direct consequence of Lemmas \ref{L:pointtoopen}, \ref{L:eqnorm}, and Proposition \ref{P:recouvrementfini}, we obtain the existence of admissible atlases.

\begin{theorem}
\label{th:atlas}
Let $\Omega$ be a corner domain in $\gD(M)$. Then $\Omega$ admits an admissible atlas.
\end{theorem}

For an admissible atlas, we can express the derivative of the diffeomorphism as follows: Let $\bx_{0}\in \gX$, $\bu_0\in \cU_{\bx_{0}}$ and $\bv_0:=\diffeo^{\bx_{0}}(\bu_0)$. 
 Differentiating \eqref{D:diffeox}, we get
\begin{equation}
\label{E:composeediffeo}
   \forall \bv\in \cV_{\bu_0}, \quad \rJ^{\bu_0}(\bv)=
   \rJ^{\bx_{0}}(\bv)\ \rJ^{\bv_0}(\diffeo^{\bv_0}(\bv))\ (\rJ^{\bx_{0}}(\bv_0))^{-1} \, ,
\end{equation}
and \eqref{D:diffeoy} provides:
\begin{equation}
\label{E:composeediffeobis}
  \rJ^{\bv_0}(\diffeo^{\bv_0}(\bv)) = 
  \rJ^{(1,\theta(\bv_0))}\left(\diffeo^{(1,\theta(\bv_0))}
  (\tfrac{\bv}{\|\bv_0\|})\right) .
\end{equation}

\section{Estimates for local Jacobian matrices}

We give in Proposition~\ref{P:estimatejacobian} several estimates for the Jacobians\index{Jacobian} $\rJ^{\bx_{0}}$ \eqref{eq.jacob} and the metric $\rG^{\bx_{0}}$ \eqref{eq:metric} of all the diffeomorphisms  contained in an admissible atlas of a corner domain $\Omega$. All estimates are consequence of local bounds in $L^\infty$ norm on the derivative of Jacobian functions. We denote for any $\bx_0\in\overline\Omega$
\begin{equation}
\label{eq:diffJacob}
   \rK^{\bx_0}(\bv) = \rd_{\bv}\rJ^{\bx_0}(\bv),\quad \bv\in\cV_{\bx_0}\,.
\end{equation}
After considering the case of reference points $\bx_{0}\in\gX$, we deal with points $\bu_{0}\in \overline{\Omega}$ close to a reference point $\bx_0$ such that $\Pi_{\bx_{0}}\in \ogP_{n}$: in that case the quantities $\rK^{\bu_{0}}$ for $\bu_{0}\in \cU_{\bx_{0}}$ remain bounded uniformly in $\cU_{\bx_{0}}$. The next estimate is a global version of the first one when assuming that $\Omega\in \ogD(M)$. The last estimate deals with points $\bu_{0}$ close to a reference point $\bx_0$ such that the section $\widehat\Omega_{\bx_0}$ of $\Pi_{\bx_{0}}$ is polyhedral\footnote{But this does not imply that the tangent cone $\Pi_{\bx_{0}}$ is polyhedral.}: in that case we show that for $\bu_{0}\in \cU_{\bx_{0}}$, the quantity $\rK^{\bu_{0}}$ is controlled by $\|\bu_{0}-\bx_{0}\|^{-1}$.  These estimates will be useful when using change of variables on quadratic form defined on corner domains in dimension 3. An important feature of these estimates is a recursive control of their domain of validity: In each case we exhibit such domains as balls with explicit centers and implicit radii. The principle is to start from the finite number of reference points $\bx_0\in\gX$ provided by an admissible atlas and proceed with points $\bu_0$ which are not in this set using Lemma \ref{L:pointtoopen} and Remark \ref{rem:pto}. The outcome is that estimates are valid in a ball around $\bu_0$ with radius $\rho(\bu_0)$ proportional to the distance $\dist(\bu_0,\gX)$ of $\bu_0$ to the set of reference points, the proportion ratio $\rho(\widehat\bu_1)$ being a similar radius associated with the section $\widehat\Omega_{\bx_0}\in\gD(\dS^{n-1})$. 

\begin{proposition}
\label{P:estimatejacobian}
Let $\Omega\in \gD(M)$ and $(\cU_{\bx},\diffeo^{\bx})_{\bx\in \overline{\Omega}}$ be an admissible atlas with set of reference points $\gX\subset \overline{\Omega}$. Then we have the following assertions: 
\begin{enumerate}[(a)]
\item
\label{P:DLx0}
Let $\bx_{0}\in\gX$. With $R_{\bx_{0}}$ introduced in Definition~\Ref{def.atlas}, there exists $c(\bx_{0})$ such that
\begin{equation}\label{eq:KJGx0}
\begin{aligned}
&      \|\rK^{\bx_0}\|_{L^{\infty}(\cB(\bfz,R_{\bx_0}))} \le c(\bx_{0}), \\[0.5ex]
&   \|\rJ^{\bx_0}-\Id_n\|_{L^{\infty}(\cB(\bfz,r))} +
   \|\rG^{\bx_0}-\Id_n\|_{L^{\infty}(\cB(\bfz,r))} \leq r \ee c(\bx_{0}) 
   \quad\mbox{for all $r\le R_{\bx_0} $}\, .
\end{aligned}
\end{equation}

\item
\label{P:DLmetricpolyedre}
Let $\bx_{0}\in\gX$ such that $\Pi_{\bx_{0}}\in \ogP_{n}$. Then there exists a constant $c(\bx_{0})$ such that for all $\bu_0\in\overline{\Omega}\cap \cB(\bx_{0},R_{\bx_0})$, $\bu_{0}\neq\bx_{0}$,  there holds, denoting $\widehat\bu_1 := \diffeo^{\bx_0}\bu_0 / \|\diffeo^{\bx_0}\bu_0  \| \in\widehat\Omega_{\bx_{0}}$
\begin{equation}\label{eq:KJGu01}
\begin{aligned}
&      \|\rK^{\bu_0}\|_{L^{\infty}(\cB(\bfz,\rho(\bu_0)))} \le c(\bx_{0})
      \quad\mbox{with}\quad \rho(\bu_0) = \tfrac13\ee\rho(\widehat\bu_1)\,\|\bu_0-\bx_0\|, \\[0.5ex]
&   \|\rJ^{\bu_0}-\Id_n\|_{L^{\infty}(\cB(\bfz,r))} +
   \|\rG^{\bu_0}-\Id_n\|_{L^{\infty}(\cB(\bfz,r))} \leq r \ee c(\bx_{0}) 
   \quad\mbox{for all $r\le \rho(\bu_0) $}\, .
\end{aligned}
\end{equation}

\item
\label{C:DLmetricpolyedre}
Let $\Omega\in \ogD(\R^{n})$, then there exists $c(\Omega)$ such that for all $\bu_0\in \overline\Omega$,  there holds, with $\widehat\bu_1$ as above,  
\begin{equation}\label{eq:KJGu02}
\begin{aligned}
&   \|\rK^{\bu_0}\|_{L^{\infty}(\cB(\bfz,\rho(\bu_0)))} \le c(\Omega)
    \quad\mbox{with}\quad \rho(\bu_0) = \tfrac13\ee\rho(\widehat\bu_1)\dist(\bu_0,\gX), \\[0.5ex]
&   \|\rJ^{\bu_0}-\Id_n\|_{L^{\infty}(\cB(\bfz,r))} +
   \|\rG^{\bu_0}-\Id_n\|_{L^{\infty}(\cB(\bfz,r))} \leq r \ee c(\Omega) 
   \quad\mbox{for all $r\le \rho(\bu_0)$}\, .
\end{aligned}
\end{equation}

\item
\label{P:DLmetriccoin}
Let $\bx_{0}\in \gX$ be such that the section $\widehat\Omega_{\bx_{0}}=\Pi_{\bx_{0}}\cap \dS^{n-1}$ belongs to $\ogD(\dS^{n-1})$. Then there exists $c(\bx_{0})$ such that for all $\bu_0\in \overline\Omega\cap\cB(\bx_{0},R_{\bx_0})$, $\bu_{0}\neq\bx_{0}$:
\begin{equation}\label{eq:KJGu03}
\begin{aligned}
&      \|\rK^{\bu_0}\|_{L^{\infty}(\cB(\bfz,\rho(\bu_0)))} \le 
      \frac{1}{\|\bu_0-\bx_{0}\|} \ c(\bx_{0})
      \quad\mbox{with}\quad \rho(\bu_0) = \tfrac13\ee\rho(\widehat\bu_1)\,\|\bu_0-\bx_0\|, \\[0.5ex]
&   \|\rJ^{\bu_0}-\Id_n\|_{L^{\infty}(\cB(\bfz,r))} +
   \|\rG^{\bu_0}-\Id_n\|_{L^{\infty}(\cB(\bfz,r))} \leq \frac{r}{\|\bu_0-\bx_{0}\|} \,c(\bx_{0}) 
   \quad\mbox{for all $r\le \rho(\bu_0) $}\, .
\end{aligned}
\end{equation}
 \end{enumerate}
\end{proposition}

\begin{proof}
\emph{(\ref{P:DLx0})}
The estimate for $\rK^{\bx_{0}}$ in \eqref{eq:KJGx0} comes from the definition of a map-neighborhood. The bound in \eqref{eq:KJGx0} on $\rJ^{\bx_0}-\Id_n$ follows immediately because of the Taylor estimate
\begin{equation}\label{eq.Jxo-I}
   \|\rJ^{\bx_0}(\bv)-\Id_n\| \leq 
   \|\bv\| \,\|\rK^{\bx_0}\|_{L^{\infty}(\cB(\bfz,\|\bv\|))} ,\quad \bv\in\cV_{\bx_0} \,.
\end{equation}
Concerning the bound \eqref{eq:KJGx0} on $\rG^{\bx_0}-\Id_n$, we rely on the Taylor estimate
\begin{equation}
\label{eq:TaylorGx0}
   \|\rG^{\bx_0}(\bv)-\Id_n\| \leq 
   \|\bv\| \, \|\rK^{\bx_0}\|_{L^{\infty}(\cB(\bfz,\|\bv\|))} \,  
   \|(\rJ^{\bx_0})^{-1}\|_{L^{\infty}(\cB(\bfz,\|\bv\|))}^3\,.
\end{equation}

\emph{(\ref{P:DLmetricpolyedre})} 
Since $\Pi_{\bx_{0}}$ is polyhedral, we can take advantage of Remark \ref{rem:pto}: For $\bu_0$ in the ball $\cB(\bx_0,R_{\bx_0})$, the local map $(\cU_{\bu_0},\diffeo^{\bu_0})$ is defined by \eqref{eq:mapv0}--\eqref{D:diffeox} where, for some $\rho_1<1$,
\[
   \bv_0=\diffeo^{\bx_0}(\bu_0),\quad
   \cU_{\bv_0} = \cB(\bv_0,\rho_1\|\bv_0\|),\quad\mbox{and}\quad
   \diffeo^{\bv_0}(\bv) = \bv-\bv_0\,.
\]
Note that the radius $\rho_1$ is the radius $\rho(\widehat\bu_1)$ of a map neighborhood of $\widehat\bu_1:=\bv_0/\|\bv_0\|$, which plays the same role as $\rho(\bu_0)$ in one dimension less. 

We recall that our admissible atlas satisfies Condition (1) of Definition~\ref{def.atlas}. 
Applying \eqref{eq:eqnormu} with the couples $\{(\bu,\bu_{0}), (\bv,\bv_{0})\}$ and $\{(\bu_{0},\bx_{0}), (\bv_{0},\bfz)\}$, we deduce that $\cU_{\bu_0}$ contains the ball $\cB(\bu_0,\frac13 \rho_1\|\bu_0-\bx_0\|)$. On the other hand,
in this case \eqref{E:composeediffeo} reduces to
\begin{equation}
\label{eq:Ju0}
   \forall \bv\in \cV_{\bu_0}, \quad 
   \rJ^{\bu_0}(\bv)=\rJ^{\bx_{0}}(\bv)\ (\rJ^{\bx_{0}}(\bv_0))^{-1}.
\end{equation}
Thus, we deduce from the above formula that 
\begin{equation}
\label{E:majorerJpoly}
   \|\rK^{\bu_0}\|_{L^{\infty}(\cV_{\bu_0})} \leq
   \|\rK^{\bx_{0}}\|_{L^{\infty}(\cV_{\bx_0})} 
   \|(\rJ^{\bx_{0}})^{-1}\|_{L^{\infty}(\cU_{\bx_0})} \, .
\end{equation}
All of this proves estimate for $\rK^{\bu_{0}}$ in \eqref{eq:KJGu01}.

The bound in \eqref{eq:KJGu01} on $\rJ^{\bu_0}-\Id_n$ follows immediately because of the Taylor estimate \eqref{eq.Jxo-I} where $\bx_{0}$ is replaced by $\bu_{0}$.
Concerning the bound on $\rG^{\bu_0}-\Id_n$, we start from the Taylor estimate \eqref{eq:TaylorGx0} where we replace $\bx_{0}$ by $\bu_{0}$. 
It remains to bound $\|(\rJ^{\bu_0})^{-1}\|$. We note that we have, thanks to \eqref{eq:Ju0}
\[
   \rJ^{\bu_0}(\bv)^{-1} = (\rJ^{\bx_{0}}(\bv_0))\, (\rJ^{\bx_{0}}(\bv))^{-1}\ .
\]
Whence the bound \eqref{eq:KJGu01} on $\rG^{\bu_0}-\Id_n$.
 
\emph{(\ref{C:DLmetricpolyedre})} Applying Proposition \ref{P:recouvrementfini} to $\Omega\in \ogD(M)$, we deduce from \eqref{E:majorerJpoly}: 
 \begin{equation}
\label{E:majorerKpoly}
   \sup_{x\in \overline{\Omega}} \|\rK^{\bx}\|_{L^{\infty}(\cV_{\bx})} \leq 
   \max_{\bx_{0}\in \gX} \left(\|\rK^{\bx_{0}}\|_{L^{\infty}(\cV_{\bx})} 
   \|(\rJ^{\bx_{0}})^{-1}\|_{L^{\infty}(\cU_{\bx})} \right)<+\infty .
\end{equation}

\emph{(\ref{P:DLmetriccoin})} Differentiating \eqref{E:composeediffeo} with respect to $\bv$ yields
\begin{equation}
\label{E:Kbu}
   \rK^{\bu_0}(\bv) =
   \rK^{\bx_{0}}(\bv) \ \rJ^{\bv_0} (\diffeo^{\bv_0}(\bv) )\ (\rJ^{\bx_{0}}(\bv_0))^{-1} +
   \rJ^{\bx_{0}}(\bv) \ \rd_\bv \rJ^{\bv_0} (\diffeo^{\bv_0}(\bv) )\ \ (\rJ^{\bx_{0}}(\bv_0))^{-1}.
\end{equation}
Using in turn \eqref{E:composeediffeobis} we calculate
\begin{align}
   \rd_\bv \rJ^{\bv_0} (\diffeo^{\bv_0}(\bv) ) &=
   \rd_\bv \Big\{\rJ^{(1,\theta(\bv_0))}\left(\diffeo^{(1,\theta(\bv_0))} 
  (\tfrac{\bv}{\|\bv_0\|})\right) \Big\} \nonumber\\
\label{eq:dJ}
   &=\tfrac{1}{\|\bv_0\|}\ 
   \rK^{(1,\theta(\bv_0))}\left(\diffeo^{(1,\theta(\bv_0))}(\tfrac{\bv}{\|\bv_0\|}) \right)\ 
   \left(\rJ^{(1,\theta(\bv_0))}\left(
   \diffeo^{(1,\theta(\bv_0)}(\tfrac{\bv}{\|\bv_0\|})\right)\right)^{-1} \, .
\end{align}
Recall that $\diffeo^{(1,\theta)}$ is deduced from $\diffeo^{\theta}$ by formula \eqref{D:diffeotheta} on the domain $\cU_{(1,\theta(\bv_0))}=\cB(\theta(\bv_0),\rho_0)$, cf.\ \eqref{eq:cUtv0}. Therefore there exists a constant $c(\rho_0)\ge1$ such that
$$
   \|\rJ^{(1,\theta)}\|_{L^{\infty}(\cV_{(1,\theta)})} \le 
   c(\rho_0)\|\rJ^{\theta}\|_{L^{\infty}(\cV_{\theta})}
   \quad\mbox{and}\quad
   \|\rK^{(1,\theta)}\|_{L^{\infty}(\cV_{(1,\theta)})} \le
   c(\rho_0) \|\rK^{\theta}\|_{L^{\infty}(\cV_{\theta})} \, . 
$$
We deduce 
\begin{equation}
\label{E:normKbx} 
   \|\rK^{\bu_0}\| \leq  
   c'(\rho_0) \left(\|\rK^{\bx_{0}}\|\, \|\rJ^{\theta(\bv_0)}\| \,\|(\rJ^{\bx_{0}})^{-1}\| \, + \,
   \frac{\!\|(\rJ^{\theta(\bv_0)})^{-1}\|\!}{\|\bv_0\|}\ \|\rJ^{\bx_{0}}\|  \,
   \|\rK^{\theta(\bv_0)}\|\,  \|(\rJ^{\bx_{0}})^{-1}\|\right)
\end{equation}
where we have omitted the mention of the $L^{\infty}$ norms.
Since the section $\widehat\Omega_{\bx_{0}}$ belongs to $\ogD(\dS^{n-1})$, we deduce from \emph{(\ref{C:DLmetricpolyedre})} and \eqref{E:majorerKpoly} applied to the section $\widehat\Omega_{\bx_{0}}$ that 
$$
   \sup_{\theta\in\overline{\widehat\Omega}_{\bx_{0}}} \| \rJ^{\theta}\|_{L^{\infty}(\cV_{\theta})}< +\infty
   \quad \mbox{and}\quad 
   \sup_{\theta\in\overline{\widehat\Omega}_{\bx_{0}}} \| \rK^{\theta}\|_{L^{\infty}(\cV_{\theta})}< +\infty
    \, .
$$
Therefore the r.h.s. of \eqref{E:normKbx} is controlled by $c(\bx_{0})/\|\bv_0\|$. Using \eqref{eq:eqnormu} we obtain that $\|\bv_0\|\simeq\|\bu_0-\bx_{0}\|$, whence the bound \eqref{eq:KJGu03} on $\rK^{\bu_0}$. The bound \eqref{eq:KJGu03} for $\rJ^{\bu_0}-\Id_n$ follows immediately as in point \emph{(\ref{P:DLx0})}. Finally, to prove the bound on $\rG^{\bu_0}-\Id_n$, we combine the Taylor estimate \eqref{eq:TaylorGx0} (at $\bu_0$) with the estimate of $\rK^{\bu_{0}}$ in \eqref{eq:KJGu03} and the formula for $(\rJ^{\bu_0})^{-1}$
\[
   (\rJ^{\bu_0}(\bv))^{-1} = (\rJ^{\bx_{0}}(\bv_0))
   \ (\rJ^{\bv_0}(\diffeo^{\bv_0}(\bv)))^{-1}\ (\rJ^{\bx_{0}}(\bv))^{-1}\, ,
\] 
deduced from \eqref{E:composeediffeo}. It remains to use \eqref{E:composeediffeobis} to bound $(\rJ^{\bv_0}(\diffeo^{\bv_0}(\bv)))^{-1}$, which ends the proof.
\end{proof}

\begin{remark}
In dimension $n=2$, domains $\Omega\in\gD(\R^2)$ are always in case \emph{(\ref{P:DLmetricpolyedre})} or \emph{(\ref{C:DLmetricpolyedre})} of Proposition~\ref{P:estimatejacobian} since $\gD(\R^2)=\ogD(\R^2)$, cf.\ \eqref{E:poly2D}. In dimension $n=3$, Proposition~\ref{P:estimatejacobian} still covers all possibilities: Indeed, since $\gD(\dS^2)=\ogD(\dS^2)$, one is at least in case \emph{(\ref{P:DLmetriccoin})}. In higher dimensions $n\ge4$, Proposition~\ref{P:estimatejacobian} does not provide estimates for all possible singular points. General estimates would involve distance to non-discrete sets of points, see \eqref{eq:Kgen} later on. However Proposition~\ref{P:estimatejacobian} is sufficient for the core of our investigation, which, for independent reasons, is limited to dimension $n\le3$.
\end{remark}

\begin{remark}\label{R:dKu0}
We can use the computation of $\rK^{\bu_{0}}$ in the proof of Proposition~\ref{P:estimatejacobian} to obtain estimates for its differentials $\rd^{\ell}\rK^{\bu_{0}}$, $\ell=1,2,\ldots$ Note that in \eqref{E:normKbx}, the worst term is $1/\|\bv_{0}\|$. By differentiating $\ell$ times \eqref{E:Kbu}, we obtain an upper bound in $1/\|\bv_{0}\|^{\ell+1}$. Thus we have the following improvements in Proposition~\ref{P:estimatejacobian}:
\begin{enumerate}
\item In cases \emph{(\ref{P:DLx0})}, \emph{(\ref{P:DLmetricpolyedre})} and \emph{(\ref{C:DLmetricpolyedre})}, 
the estimates for $\rK^{\bx_{0}}$ and $\rK^{\bu_{0}}$ are still valid for their differentials $\rd^{\ell}\rK^{\bx_{0}}$ and $\rd^{\ell}\rK^{\bu_{0}}$, respectively.
\item Let $\bx_{0}\in \gX$ such that $\widehat\Omega_{\bx_{0}}=\Pi_{\bx_{0}}\cap \dS^{n-1}$ belongs to $\ogD(\dS^{n-1})$. Then there exists $c(\bx_{0})$ such that for all $\bu_0\in \overline\Omega\cap\cB(\bx_{0},R_{\bx_0})$, $\bu_{0}\neq\bx_{0}$, there holds, with $\widehat\bu_1 := \diffeo^{\bx_0}\bu_0 / \|\diffeo^{\bx_0}\bu_0  \|$  
\begin{equation}\label{eq:dKJGu03}
      \|\rd^{\ell}\rK^{\bu_0}\|_{L^{\infty}(\cB(\bfz,\rho(\bu_0)))} \le 
      \frac{1}{\|\bu_0-\bx_{0}\|^{\ell+1}} \ c(\bx_{0})
      \quad\mbox{with}\quad \rho(\bu_0) = \tfrac13\ee\rho(\widehat\bu_1)\,\|\bu_0-\bx_0\|.
\end{equation}
\end{enumerate}
\end{remark}

\section{Strata and singular chains}
\label{ss:chains}
In this section, we exhibit a canonical structure 
of tangent cones\index{Tangent cone} and corner domains\index{Corner domain}.
\begin{definition}
\label{def:redcone}
Let $\gO_n$\index{Orthogonal linear transformation!$\gO_n$} denote the group of orthogonal linear transformations\index{Orthogonal linear transformation} of $\R^n$. 
\begin{enumerate}[a)]
\item We say that a cone $\Pi$ is {\em equivalent}\index{Cone!Equivalence} to another cone $\Pi'$ and denote
$
   \Pi \equiv \Pi'
$
if there exists $\udiffeo\in\gO_n$ such that $\udiffeo\Pi=\Pi'$. 
\item Let $\Pi\in\gP_n$. If $\Pi$ is equivalent to $\R^{n-d}\times\Gamma$ with $\Gamma\in\gP_d$ and $d$ is minimal for such an equivalence, $\Gamma$ is said to be {\em a minimal reduced cone}\index{Minimal reduced cone|textbf} associated with $\Pi$ and we denote by $d(\Pi):=d$\index{Reduced dimension!-- of a cone $\Pi$\quad$d(\Pi)$\ } the \emph{reduced dimension}\index{Reduced dimension|textbf} of the cone $\Pi$. 
\item Let $\bx\in\overline\Omega$ and let $\Pi_\bx$ be its tangent cone. We denote by $d_0(\bx)$\index{Reduced dimension!-- of $\Omega$ at $\bx$\quad $d_{0}(\bx)$\ } the dimension of the minimal reduced cone associated with $\Pi_\bx$. We call this integer the {\em reduced dimension of $\Omega$ at $\bx$.}
\end{enumerate}
\end{definition}

\begin{remark}\label{R:isoPi}
If there exists a linear isomorphism between $\Pi$ and $\Pi'$ then $d(\Pi)=d(\Pi')$.
\end{remark}

\subsection{Recursive definition of the singular chains}
The notation $\gC(\Omega)$\index{Singular chain!$\gC(\Omega)$|textbf} represents the set of the singular chains of $\Omega$, which are defined as follows:

\begin{definition}[\sc Singular chains]
A singular chain\index{Singular chain|textbf}
$\dx=(\bx_0,\bx_1,\ldots,\bx_\pp)\in\gC(\Omega)$\index{Singular chain!$\dx$} (with $\pp$ a non negative integer) is a finite collection of points defined according to the following recursive procedure. 

\noindent {\sc Initialization}: $\bx_0\in\overline\Omega$, 
\begin{itemize}
\item Let $C_{\bx_0}$ be the tangent cone to $\Omega$ at $\bx_0$ (here $C_{\bx_0}=\Pi_{\bx_0}$).
\item Let $\Gamma_{\bx_0}\in\gP_{d_0}$ be its minimal reduced cone: 
$C_{\bx_0}=\udiffeo^0(\R^{n-d_0}\times\Gamma_{\bx_0})$.
\item Alternative:  \begin{itemize}
\item If $\pp=0$, stop here.
\item If $\pp>0$, then\footnote{If $d_{0}=0$, we have necessarily $\pp=0$.} $d_0>0$ and let $\Omega_{\bx_0}\in\gD(\dS^{d_{0}-1})$ be the section of $\Gamma_{\bx_0}$
\end{itemize}
\end{itemize}

\noindent {\sc  Recurrence}: $\bx_j\in\overline\Omega_{\bx_0,\ldots,\bx_{j-1}}\in\gD(\dS^{d_{j-1}-1})$.
If $d_{j-1}=1$, stop here ($\pp=j$). If not:
\begin{itemize}
\item Let $C_{\bx_0,\ldots,\bx_j}$ be the tangent cone to $\Omega_{\bx_0,\ldots,\bx_{j-1}}$ at $\bx_j$, 
\item Let $\Gamma_{\bx_0,\ldots,\bx_j}\in\gP_{d_j}$ be its minimal reduced cone: 
$C_{\bx_0,\ldots,\bx_j}=\udiffeo^j(\R^{d_{j-1}-1 -d_j}\times\Gamma_{\bx_0,\ldots,\bx_j})$.
\item Alternative: \begin{itemize}
\item If $\pp=j$, stop here.
\item If $\pp>j$, then $d_j>0$ and let $\Omega_{\bx_0,\ldots,\bx_j}\in\gD(\dS^{d_{j}-1})$ be the section of $\Gamma_{\bx_0,\ldots,\bx_j}$.
\end{itemize}
\end{itemize}
\end{definition}

Note that $n\ge d_0>d_1>\ldots>d_\pp$. Hence $\pp\le n$. Note also that for $\pp=0$, we obtain the trivial one element chain $(\bx_0)$ for any $\bx_0\in\overline\Omega$.

\begin{notation}\label{def.Cx}
For any $\bx\in\overline\Omega$, we denote by $\gC_\bx(\Omega)$\index{Singular chain!$\gC_\bx(\Omega)$} the subset of chains $\dx\in \gC(\Omega)$ originating at $\bx$, {\it i.e.}, the set of chains $\dx=(\bx_0,\ldots,\bx_\pp)$ with $\bx_0=\bx$. Note that the one element chain $(\bx)$ belongs to $\gC_\bx(\Omega)$. We also set\index{Singular chain!$\gC^*_\bx(\Omega)$}
\begin{equation}
\label{eq:gCx}
   \gC^*_\bx(\Omega) = \{\dx\in\gC_\bx(\Omega),\  \pp>0\} = \gC_\bx(\Omega)\setminus\{(\bx)\}.
\end{equation}
\end{notation}

We set finally, with the notation $\langle \by\rangle$ for the vector space generated by $\by$,\index{Tangent cone!$\Pi_\dx$}
\begin{equation}
\label{eq:PiX}
   \Pi_\dx = 
   \begin{cases}
     C_{\bx_0} =\Pi_{\bx_0} & \ \mbox{if} \ \  \pp=0,\\[0.5ex]
     \udiffeo^0 \big(\R^{n-d_0}\times\langle \bx_1\rangle \times 
     C_{\bx_0,\bx_1} \big)  &\ \mbox{if} \ \  \pp=1, \\[0.5ex]
     \udiffeo^0 \Big(\R^{n-d_0}\times\langle \bx_1\rangle \times  \ldots \times
     \udiffeo^{\pp-1} \big(\R^{d_{\pp-2}-1-d_{\pp-1}}\times\langle \bx_\pp\rangle \times 
     C_{\bx_0,\ldots,\bx_{\pp}}\big)
     \ldots \Big)  &\ \mbox{if} \ \  \pp\ge2.
   \end{cases}
\end{equation}
Note that if $d_\pp=0$, the cone $C_{\bx_0,\ldots,\bx_{\pp}}$ coincides with $\R^{d_{\pp-1}-1}$, leading to $\Pi_\dx=\R^n$.  

\begin{definition}
\label{def:chaineq}
Let $\dx=(\bx_0,\ldots,\bx_\pp)$ be a chain in $\gC(\Omega)$.
\begin{enumerate}[(i)]
\item The cone $\Pi_{\dx}$ defined in \eqref{eq:PiX} is called a \emph{tangent structure}\index{Tangent structure|textbf} [of $\Omega$] at $\bx_0$, and if $\dx\neq(\bx_0)$, $\Pi_{\dx}$ is called a \emph{tangent substructure}\index{Tangent substructure|textbf} of $\Pi_{\bx_0}$.
\item Let $\dx'=(\bx'_0,\ldots,\bx'_{\pp'})$ be another chain in $\gC(\Omega)$. We say that $\dx'$ is equivalent\index{Singular chain!Equivalence} to $\dx$ if $\bx'_0=\bx_0$ and $\Pi_{\dx'}=\Pi_{\dx}$.
\end{enumerate}
\end{definition}
This notion of equivalence is well suited to the class of operators that we consider in this paper.

The reader interested in examples of singular chains may find in Section \ref{SS:Sc3D} an enumeration of all possible singular chains in dimension 3, with reference to Figure \ref{F1} for illustration.

\subsection{Strata of a corner domain}\index{Stratum|textbf}
We introduce a partition of $\overline\Omega$ according to the value of the reduced dimension $d_0$ at each point.
\begin{definition}
Let $\Omega\in\gD(\R^n)$. For $d\in\{0,\ldots,n\}$, let
\begin{equation}
\label{eq:Ad}
   \gA_d(\Omega) = \{\bx\in\overline\Omega,\quad d_0(\bx)=d\},
\end{equation}
where $d_0(\bx)$ si the reduced dimension of $\Omega$ at $\bx$, see Definition \ref{def:redcone}. We call {\em stratum}, or {\em $d$-stratum} of $\overline\Omega$ any connected components of $\gA_d(\Omega)$. The strata are generically denoted by $\bt$\index{Stratum!$\bt$} and their set by $\gT$\index{Stratum!$\gT$}.
\end{definition}

\noindent{\bf Particular cases:}
\begin{itemize}
\item $\gA_0(\Omega)$ coincides with $\Omega$.
\item $\gA_1(\Omega)$ is the subset of $\partial\Omega$ of the regular points\index{Regular point} of the boundary (the corresponding strata being the faces in dimension $n=3$ and the sides in dimension $n=2$).
\item If $n=2$, $\gA_2(\Omega)$ is the set of corners\index{Corner}.
\item If $n=3$, $\gA_2(\Omega)$ is the set of edge points\index{Edge}.
\item If $n=3$, $\gA_3(\Omega)$ is the set of corners.
\end{itemize}

\begin{proposition}
\label{prop:stata}
Let $\bt\in \gA_{d}(\Omega)$ be a stratum. Then $\bt$ is a smooth submanifold\footnote{This means that for each $\bx_0\in\bt$ there exists a neighborhood $\cU\subset\bt$ of $\bx_0$ and an associate local diffeomorphism from $\cU$ onto an open set in $\R^{n-d}$.} of codimension $d$. In particular $\gA_{n}(\Omega)$ is a finite subset of $\partial\Omega$.
\end{proposition}

Thus, the strata of a corner domain have a structure of manifold ``from inside'', but not up to the boundary in general. By contrast, the strata of a manifold with corners are themselves manifold with corners.

\begin{proof}
Let $\bx_{0}\in \bt$ and $(\cU_{\bx_{0}},\diffeo^{\bx_{0}})$ be an associated local map. The tangent cone at $\bx_{0}$ writes $\Pi_{\bx_{0}}=\udiffeo\left(\R^{n-d}\times \Gamma_{\bx_{0}}\right)$, with 
$\Gamma_{\bx_{0}}\in\gP_{d}$. For simplicity, we may assume that $\udiffeo=\Id_{n}$. Denote by 
$\pi$ the orthogonal projection on $\R^{n-d}$ and set $\pi^{\perp}:=\Id_{n}-\pi$.
Let $\bu\in \cU_{\bx_{0}}$ and $\bv=\diffeo^{\bx_{0}}(\bu)$. 
According as $\pi^{\perp}(\bv)$ is $0$ or not, the tangent cone $\Pi_{\bv}$ at $\bv$ to $\Pi_{\bx_0}$ has distinct expressions. 
\begin{enumerate}
\item If $\pi^{\perp}(\bv)=0$, then $\diffeo^{\bv}$ can be taken as the translation by $\bv$ and $\Pi_{\bv}=\Pi_{\bx_{0}}$. 
\item If $\pi^{\perp}(\bv)\neq0$, we introduce the cylindrical coordinates $(r(\bv),\theta(\bv),\pi(\bv))$ of $\bv$ with:
\begin{equation}
\label{eq:cylcoord}
   r(\bv)=\|\pi^{\perp}(\bv)\|,\quad
   \theta(\bv)=\frac{\pi^{\perp}(\bv)}{\|\pi^{\perp}(\bv)\|} 
   \in \overline\Omega_{\bx_{0}}  
   \quad \mbox{with} \quad \Omega_{\bx_{0}}=\Gamma_{\bx_{0}}\cap \dS^{d-1}\, .
\end{equation}
Let $\Pi_{\theta(\bv)}\in \gP_{d-1}$ be the tangent cone to $\Omega_{\bx_{0}}$ at $\theta(\bv)$. We have, cf.\ proof of Lemma~\ref{L:pointtoopen},
\begin{equation}
\label{D:Piybis}
\Pi_{\bv}:=\R^{n-d}\times\langle \pi^{\perp}(\bv) \rangle \times \Pi_{\theta(\bv)} \,.
\end{equation}
\end{enumerate}
In any case, the tangent cone $\Pi_{\bu}$ is linked to $\Pi_{\bv}$ by the formula
$\Pi_{\bu}=\rJ^{\bx_{0}}(\bv)(\Pi_{\bv})$.
We deduce: 
\begin{enumerate}
\item If $\pi^{\perp}(\bv)=0$, then $d(\Pi_{\bu})= d(\Pi_{\bx_{0}})$ (cf.\ Remark~\ref{R:isoPi}), therefore $d_{0}(\bu)=d_{0}(\bx_{0})=d$ and $\bu\in \gA_{d}(\Omega)$.
\item If $\pi^{\perp}(\bv)\neq0$, then $d(\Pi_{\bu})=d(\Pi_{\bv})$ and we have $d_{0}(\bu) \leq d-1< d_{0}(\bx_{0})=d$.
\end{enumerate}
 Therefore $\bu\in \gA_{d}(\Omega)$ if and only if $\pi^{\perp}(\bv)=0$. We conclude that
$$\gA_{d}(\Omega)\cap\cU_{\bx_{0}}=(\diffeo^{\bx_{0}})^{-1}(\pi(\cV_{\bx_{0}})).$$
Hence the stratum $\bt$ is a smooth submanifold of codimension $d$.
\end{proof}

\begin{remark}\label{R:dgX}
Let $\Omega$ be a corner domain and $\gX$ be the set of reference points of an admissible atlas, cf.\ Definition~\ref{def.atlas}. 
Let $\bx_{0}\in\gX$. As a consequence of the above proof we find that for any $\bu_{0}\in\cB(\bx_{0},R_{\bx_{0}})$, we have the inequality $d_{0}(\bu_{0})\leq d_{0}(\bx_{0})$. Thus, in particular, the set of corners $\gA_{n}(\Omega)$ has to be contained in $\gX$.
\end{remark}

\subsection{Topology on singular chains}
Here we introduce a distance on equivalence classes of the set of chains $\gC(\Omega)$, for the equivalence already introduced in Definition \ref{def:chaineq}. This will allow to introduce natural notions of continuity and lower semicontinuity on chains.

Let us denote by $\mathsf{BGL}(n)$ the ring of linear isomorphisms $L$ with norm $\|L\|\le1$, where
\[
   \|L\| = \max_{\bx\in\R^n\setminus\{0\}} \frac{\|L\bx\|}{\|\bx\|}\,.
\]

\begin{definition}
\label{def:distchain}
Let $\dx=(\bx_{0},\ldots,\bx_{\pp})$ and $\dx'=(\bx_{0}',\ldots,\bx_{\pp'}')$ be two singular chains in $\gC(\Omega)$. We define the distance\index{Distance of chains|textbf} $\dD(\dx,\dx')\in\R_+\cup\{+\infty\}$ as \index{Distance of chains!$\dD(\dx,\dx')$}
\[
   \dD(\dx,\dx') = \|\bx_0-\bx'_0\| + \frac12\bigg\{
   \min_{\substack {L\in\mathsf{BGL}(n) \\ L\Pi_\dx = \Pi_{\dx'}}}
   \|L-\Id_n\|
      +\min_{\substack {L\in\mathsf{BGL}(n) \\ L\Pi_{\dx'} = \Pi_{\dx}}}
   \|L-\Id_n\| \bigg\}\,,
\]
where the second term is set to $+\infty$ if $\Pi_\dx$ and $\Pi_{\dx'}$ do not belong to the same orbit for the action of $\mathsf{BGL}(n)$ on $\gP_{n}$.
\end{definition}

\begin{remark}
\label{rem:distchain}
\begin{enumerate}[(a)]
\item The distance $\dD(\dx,\dx')$ is zero if and only if the chains $\dx$ and $\dx'$ are equivalent.
\item As a consequence of the proof of Proposition \ref{prop:stata}, the strata of $\overline\Omega$ are contained in orbits of the natural action of $\mathsf{BGL}(n)$ on chains. 
\item The distance between two chains $\dx$ and $\dx'$ is infinite when the associated tangent structures $\Pi_\dx$ and $\Pi_{\dx'}$ cannot be mapped from each other by a linear application. For example, this is the case when the reduced dimensions of $\Pi_\dx$ and $\Pi_{\dx'}$ are distinct. The components of $\gC(\Omega)$ separated by an infinte distance are, in certain sense, the closure of the statra, see Remark \ref{rem:str3} for a description in dimension $n=3$.
\item Inside each stratum of a polyhedral domain, the distance $\dD$ between chains of length $1$ is equivalent to the standard distance in $\R^n$. This is no longer true for strata containing conical points in their closure for the standard distance. Conical points are ``blown up'' by the distance $\dD$, cf.\ Remark \ref{rem:str3} again.
\item If $\Omega$ is a manifold with corners of dimension $n$, each tangent structure $\Pi_\dx$ is homeomorphic to $\R^d_+\times\R^{n-d}$ where $d$ is its reduced dimension. Thus the distance $\dD$ splits $\gC(\Omega)$ in at most $n+1$ components separated from each other by an infinite distance. Each of these components may contain several distinct connected components.
\end{enumerate}
\end{remark}

We define a partial order on chains.
\begin{definition}
\label{D:Ordrechain}
Let $\dx=(\bx_{0},\ldots,\bx_{\pp})$ and $\dx'=(\bx_{0}',\ldots,\bx_{\pp'}')$ be two singular chains in $\gC(\Omega)$. We say that $\dx \leq \dx'$ if 
$\pp\leq \pp'$ and $\bx_{j}=\bx_{j}'$ for all $0 \leq j\leq \pp$.
\end{definition}

\begin{theorem}
\label{th:scichain}
Let $\Omega$ be a corner domain in $\gD(M)$ with $M=\R^n$ or $\dS^n$, and $F:\gC(\Omega)\to \R$ be a function such that 
\begin{enumerate}[(i)]
\item $F$ is continuous on $\gC(\Omega)$ for the distance $\dD$
\item $F$ is order-preserving on $\gC(\Omega)$ ({\it i.e.}, $\dx \leq \dx'$ implies $F(\dx) \leq F(\dx')$).
\end{enumerate}
Then for all chain $\dx=(\bx_{0},\ldots,\bx_{\pp})\cup\{\emptyset\}$, 
the function (with the convention that $\Omega_{\emptyset}=\Omega$)
$$
   \overline{\Omega}_{\bx_0,\ldots,\bx_{\pp}} \ni\bx \longmapsto 
   F((\bx_0,\ldots,\bx_{\pp},\bx))
$$
is lower semicontinuous. In particular 
$ \overline{\Omega} \ni\bx \mapsto F((\bx))$
is lower semicontinuous.\index{Semicontinuity}
\end{theorem}

\begin{proof}
The proof is recursive over the dimension $n$.

\emph{Initialization}.
$n=1$. Let $\Omega$ belong to $\gD(M)$ with $M=\R$ or $\dS^1$. Then $\Omega$ is an open interval $(\bc,\bc')$. The chains in $\gC(\Omega)$ are
\begin{itemize}
\item $\dx=(\bx_0)$ for $\bx_0\in(\bc,\bc')$ with $\Pi_{\dx}=\R$,
\item $\dx=(\bx_0)$ for $\bx_0=\bc$ and $\bx_0=\bc'$, with $\Pi_{\dx}=\R_+$ and $\R_{-}$, respectively,
\item $\dx=(\bx_0,\bx_1)$ for $\bx_0=\bc$ or $\bx_0=\bc'$, and $\bx_1=1$, with $\Pi_{\dx}=\R$.
\end{itemize}
The function $F$ is continuous on $\gC(\Omega)$. By definition of the distance $\dD$:
\[
   \dD\big((\bx),(\bc,1)\big) = \|\bx-\bc\|
   \quad\mbox{and}\quad
   \dD\big((\bx),(\bc',1)\big) = \|\bx-\bc'\|,
   \quad\forall\bx\in(\bc,\bc')\,.
\]
Therefore, as $\bx\to\bc$, with $\bx\neq\bc$, $F((\bx))$ tends to $F((\bc,1))$. By assumption $F((\bc,1))\ge F((\bc))$, and the same at the other end $\bc'$. This proves that $F$ is lower semicontinuous on $\overline\Omega=[\bc,\bc']$.

\emph{Recurrence}.
We assume that Theorem \ref{th:scichain} holds for any dimension $n^\star<n$. Let us prove it for the dimension $n$.

a) Let $\dx_0$ be a non-empty chain in $\gC(\Omega)$. Then $\Omega_{\dx_0}$ belongs to $\gD(\dS^{n^\star})$ for a $n^\star<n$. The chains $\dy\in\gC(\Omega_{\dx_0})$ correspond to the chains $(\dx_0,\dy)$ in $\gC(\Omega)$ and the corresponding tangent substructures $\Pi_\dy\in\gP_{n^\star}$  and $\Pi_{\dx_0,\dy}\in\gP_{n}$ are linked by a relation of the type, cf.\ \eqref{eq:PiX}
\[
   \Pi_{\dx_0,\dy} = \udiffeo^0 \big(\R^{n-d_0}\times\langle \bx_1\rangle \times  \ldots \times
   \Pi_\dy \big).
\]
Hence the distances $\dD\big((\dx_0,\dy),(\dx_0,\dy') \big)$ and $\dD\big(\dy,\dy' \big)$ can  be compared:
\begin{align*}
   \dD\big((\dx_0,\dy),(\dx_0,\dy') \big) &= 
   \frac12 \bigg\{ \min_{\substack {L\in\mathsf{BGL}(n) \\ L\Pi_{\dx_0,\dy} = \Pi_{\dx_0,\dy'}}}
   \|L-\Id_n\| +\min_{\substack {L\in\mathsf{BGL}(n) \\ L\Pi_{\dx_0,\dy'} = \Pi_{\dx_0,\dy}}}
   \|L-\Id_n\| \bigg\}
   \\
   &\le \frac12 \bigg\{
   \min_{\substack {L^\star\in\mathsf{BGL}(n^\star) \\ L^\star\Pi_{\dy} = \Pi_{\dy'}}}
   \|L^\star-\Id_{n^\star}\|+\min_{\substack {L^\star\in\mathsf{BGL}(n^\star) \\ L^\star\Pi_{\dy'} = \Pi_{\dy}}}
   \|L^\star-\Id_{n^\star}\| \bigg\}
   \\
   & \le \ \dD\big(\dy,\dy' \big).
\end{align*}
Let us define the function $F^\star$ on $\gC(\Omega_{\dx_0})$ by the partial application
\[
   F^\star(\dy) = F((\dx_0,\dy)),\quad \dy\in\gC(\Omega_{\dx_0}).
\] 
Since $F$ is continuous on $\gC(\Omega)$, the above inequality between distances proves that $F^\star$ is continuous on $\gC(\Omega_{\dx_0})$. Likewise the monotonicity property is obviously transported from $F$ to $F^\star$. Therefore the recurrence assumption provides the lower semicontinuity of $F^\star$ on $\overline\Omega_{\dx_0}$, hence of $\bx\mapsto F((\dx_0,\bx))$ on the same set.

b) It remains to prove that $\bx\mapsto F((\bx))$ is lower semicontinuous on $\overline\Omega$. Let $\bx_0\in\overline\Omega$. At this point we follow the proof of Proposition \ref{prop:stata}. For any $\bu\in\cU_{\bx_0}$, we define $\pi$, $\pi^\perp$ and $\bv$ like there and encounter the same two cases:
\begin{enumerate}
\item If $\pi^{\perp}(\bv)=0$, then $\Pi_{\bv}=\Pi_{\bx_{0}}$. Hence $\Pi_{\bu}=\rJ^{\bx_{0}}(\bv)(\Pi_{\bx_0})$. Since $\rJ^{\bx_{0}}(\bv)$ tends to $\Id_n$ as $\bv\to\bfz$, the distance $\dD((\bx_0),(\bu))$ tends to $0$ as $\bu$ tends to $\bx_0$. By the continuity assumption, $F((\bu))$ tends to $F((\bx_0))$.
\item If $\pi^{\perp}(\bv)\neq0$, let $\bx_1$ be the element of $\overline\Omega_{\bx_{0}}$ defined by 
$\bx_1=\pi^{\perp}(\bv)\,\|\pi^{\perp}(\bv)\|^{-1}$.
Let $\Pi_{\bx_1}\in \gP_{d-1}$ be the tangent cone to $\Omega_{\bx_{0}}$ at $\bx_1$. We find
\begin{equation*}
\Pi_{\bv}=\R^{n-d}\times\langle \pi^{\perp}(\bv) \rangle \times \Pi_{\bx_1} = \Pi_{\bx_0,\bx_1}\,.
\end{equation*}
Hence $\Pi_{\bu}=\rJ^{\bx_{0}}(\bv)(\Pi_{\bx_0,\bx_1})$. Like before, we deduce that the distance $\dD((\bx_0,\bx_1),(\bu))$ tends to $0$ as $\bu$ tends to $\bx_0$. By the continuity assumption, $F((\bu))$ tends to $F((\bx_0,\bx_1))$, which by the monotonicity assumption, is larger than $F((\bx_0))$.
\end{enumerate}
This ends the proof of the theorem.
\end{proof}

\subsection{Singular chains and admissible atlases}
\label{sss:schain_atl}
The aim of this section is to provide an overview of map-neighborhoods and Jacobian estimates in the framework of singular chains. In their generality, these facts are not needed for our study of magnetic Laplacians, which is restricted to dimension $n\le3$ for distinct reasons that we will explain later on. Nevertheless, full generality sheds some light on the recursive process present in the very definition of admissible atlases and in the domain of validity of estimates in Proposition \ref{P:estimatejacobian}.

\subsubsection*{Chains of atlases}
Denote by $\gX(\Omega)$ the set of reference points of an admissible atlas\index{Atlas!Admissible atlas} for a corner domain $\Omega$. The chain of atlases\index{Atlas!Chain of atlases} of a corner domain $\Omega$ is defined as follows:
\begin{enumerate}
\item[(0)] Start from the set $\gX(\Omega)$ of reference points $\bx_0\in\overline\Omega$, as in Definition \ref{def.atlas}.
\item For each $\bx_0\in\gX(\Omega)$, choose an admissible atlas of the section $\Omega_{\bx_0}\in\gD(\dS^{d_0-1})$, with set $\gX(\Omega_{\bx_0})$ of reference points $\bx_1\in\overline\Omega_{\bx_0}$.
\item For each $\bx_1\in\gX(\Omega_{\bx_0})$, choose an admissible atlas of the section $\Omega_{\bx_0,\bx_1}\in\gD(\dS^{d_1-1})$, with set $\gX(\Omega_{\bx_0,\bx_1})$ of reference points $\bx_2\in\overline\Omega_{\bx_0,\bx_1}$. And so on...
\end{enumerate}

\subsubsection*{Cylindrical coordinates}
The natural coordinates associated with chains of atlases are recursively defined cylindrical coordinates\index{Cylindrical coordinates}. Let $\bu_0\in\overline\Omega$.
\begin{enumerate}
\item If $\bu_0\not\in\gX(\Omega)$, pick $\bx_0\in\gX(\Omega)$ such that $\bu_0\in\cB^n(\bx_0,R_{\bx_0})$ ($n$-dimensional ball). Then define $\bv_0=\diffeo^{\bx_0}\bu_0$ and, if $d_0>0$, its cylindrical coordinates 
\[
   \qquad\pi_0(\bv_0)\in\R^{n-d_0},\ \ 
   r(\bv_0)=\|\bv_0-\pi_0(\bv_0)\|,\ \ \mbox{and}\ \  
   \bu_1=\frac{\bv_0-\pi_0(\bv_0)}{r(\bv_0)} 
   \in \overline\Omega_{\bx_{0}} \,. 
\]
If $d_0=0$, $\pi_0=\Id_n$, then stop.
\item  If $\bu_1\not\in\gX(\Omega_{\bx_0})$, pick $\bx_1\in\gX(\Omega_{\bx_0})$ such that $\bu_1\in\cB^{d_0}(\bx_1,R_{\bx_1})\cap \dS^{d_{0}-1}$. Then define $\bv_1=\diffeo^{\bx_0,\bx_1}\bu_1$ and, if $d_1>0$, its cylindrical coordinates 
\[
   \qquad\pi_1(\bv_1)\in\R^{d_0-1-d_1},\ \ 
   r(\bv_1)=\|\bv_1-\pi_1(\bv_1)\|,\ \ \mbox{and}\ \  
   \bu_2=\frac{\bv_1-\pi_1(\bv_1)}{r(\bv_1)} 
   \in \overline\Omega_{\bx_{0},\bx_1} \,. 
\]
If $d_1=0$, $\pi_1=\Id_n$, then stop. And so on...
\end{enumerate}
Let $\bv_{\pp_*}$ be the last element of the sequence $\bv_0,\bv_1,\ldots\,$. In any case $\pp_*\le n$.

\subsubsection*{Local maps}
The local maps\index{Local map} are recursively constructed using the natural coordinates associated with chains.
\begin{enumerate}
\item[(0)] If $\bu_0=\bx_0\in\gX(\Omega)$, use the local map $(\cU_{\bx_0},\diffeo^{\bx_0})$ and stop.
\item If $\bu_0\not\in\gX(\Omega)$, a local map $(\cU_{\bu_0},\diffeo^{\bu_0})$ is defined by the formulas hereafter. The map neighborhood $\cU_{\bu_0}$ can be chosen as $(\diffeo^{\bx_{0}})^{-1}(\cU_{\bv_0})$ with 
\[
   \qquad\cU_{\bv_0} = \cB^{n-d_0}(\pi_0(\bv_0),R_{\bx_0}) 
   \, \times\, r(\bv_0)\ \cU_{(1,\bu_1)},\quad
   \cU_{(1,\bu_1)} = \cB^{d_0}(\bu_1,\rho_1),\quad
   \cU_{\bu_1} =  \cU_{(1,\bu_1)}\cap \dS^{d_0-1}.
\]
The diffeomorphism $\diffeo^{\bu_0}$ is defined by $\rJ^{\bx_{0}}(\bv_0) \  (\diffeo^{\bv_0}\circ \diffeo^{\bx_{0}})$ with
\[
   \qquad \diffeo^{\bv_0} = \big( \rT_{\pi_0(\bv_0)} \,,\,
   \rN_{r(\bv_0)}^{-1}\circ \diffeo^{(1,\bu_1)} \circ \rN_{r(\bv_0)}\big)
   \quad\mbox{and}\quad
   \diffeo^{(1,\bu_1)} = (\rT_1\,,\, \diffeo^{\bu_1})\,,
\]
where $\rT_{\pi_0(\bv_0)}$ is the translation $\bv\mapsto\bv-\pi_0(\bv_0)$ in $\R^{n-d_0}$, and $\rT_1$ is the translation by $1$ for the radius in polar coordinates. If $\bu_1=\bx_1\in\gX(\Omega_{\bx_0})$, stop. 

\item  If $\bu_1\not\in\gX(\Omega_{\bx_0})$, a local map $(\cU_{\bu_1},\diffeo^{\bu_1})$ is defined like in step (1), replacing $\bx_0$ by $\bx_1$, $\bv_0$ by $\bv_1$, $\cB^{n-d_0}$ by $\cB^{d_0-1-d_1}$, $\pi_0(\bv_0)$ by $\pi_1(\bv_1)$, $\cB^{d_0}$ by $\cB^{d_1}$, and finally $\bu_1$ by $\bu_2$\ldots
\end{enumerate}

\subsubsection*{Estimates on Jacobian matrices}
Let $\bu_0\in\overline\Omega$. As explained in Remark \ref{rem:pto}, as soon as a polyhedral cone $\Gamma_{\bx_0,\ldots,\bx_{\pp}}$ is reached in the construction, the corresponding diffeomorphism $\diffeo^{(1,\bu_{\pp+1})}$ is chosen as a translation, so it is the same for $\diffeo^{\bu_{\pp+1}}$, and the norm of its differential is bounded. By recursion, this implies the estimate for the differential $\rK^{\bu_0}$ of $\rJ^{\bu_0}$
\begin{equation}
\label{eq:Kgen}
   \| \rK^{\bu_0} \| \le \frac{c(\Omega)}{r(\bv_0)\, \cdots \,r(\bv_{\pp-1})}
\end{equation}
with the convention that if $\pp-1<0$, the denominator is $1$.The same estimate is valid if $\bu_{\pp}\in\gX(\Omega_{\bx_0,\ldots,\bx_{\pp-1}})$ with the convention that $\Omega_{\bx_0,\ldots,\bx_{\pp-1}}=\Omega$ if $\pp-1<0$. Note that $\pp=0$ for any $\bu_0$ if the domain $\Omega$ is polyhedral. In turn, the domain of validity of estimates \eqref{eq:Kgen} is (at least) a ball centered at $\bu_0$ of radius 
\begin{equation}
\label{eq:rad}
   \rho(\bu_{0}) = r(\Omega)\, r(\bv_0)\, \cdots \,r(\bv_{\pp_*})\,.
\end{equation}

\section{3D domains}\label{SS:3D}
In this section we refine our analysis for the particular case of 3D domains. In each case we provide an exhaustive description of the possible singular chains. We also determine some consequences of Proposition \ref{P:estimatejacobian}.

\subsection{Faces, edges and corners}
\begin{definition}
Let $\Omega\in \gD(\R^3)$. We denote by $\gF$ the set of the connected components of $\gA_{1}(\Omega)$  (faces)\index{Face}, $\gE$ those of $\gA_{2}(\Omega)$ (edges)\index{Edge} and $\gV$ the finite set $\gA_{3}(\Omega)$ (corners)\index{Corner}\index{Corner!set of --\quad$\gV$}.\\
Let $\bx_{0}\in \gA_{d}(\Omega)$ with $d<3$, then $\Pi_{\bx_{0}}\in \ogP_{3}$. 
Let $\bx_{0}\in \gV$, we distinguish between two cases: 
\begin{enumerate}
\item If $\Pi_{\bx_{0}}\in \ogP_{3}$, then $\bx_{0}$ is a polyhedral corner. 
\item If $\Pi_{\bx_{0}}\notin \ogP_{3}$, then $\bx_{0}$ is a conical point\index{Conical point|textbf}. We denote by $\gVc$ \index{Conical point!set of --\quad$\gVc$\ } the set of conical points.
 \end{enumerate} 
\end{definition}

Combining Proposition \ref{P:estimatejacobian} and Remark \ref{R:2D}, we obtain local estimates for the Jacobian matrix and the metric issued from changes of variables pertaining to an admissible atlas:

\begin{corollary}
\label{C:expandJ}
Let $\Omega\in \gD(\R^{3})$ and $(\cU_{\bx},\diffeo^{\bx})_{\bx\in \overline{\Omega}}$ be an admissible atlas. 
Note that the set of its reference points $\gX$ contains $\gV$ (cf.\ Remark~\Ref{R:dgX}), thus in particular the set of conical corners $\gVc$.  There exists $c(\Omega)>0$ such that
\begin{enumerate}[(a)]
\item for all $\bx_{0}\in\gX$:
\begin{equation*}
\|\rJ^{\bx_{0}}-\Id_{3}\|_{L^{\infty}(\cB(\bfz,r))} +
\|\rG^{\bx_{0}}-\Id_{3}\|_{L^{\infty}(\cB(\bfz,r))} 
\leq r \ee c(\Omega), \quad \mbox{ for all }r\leq R_{\bx_{0}} \, ,
\end{equation*}

\item for all $\bu_{0}\in \overline{\Omega}\setminus\gX$:  
\begin{equation*}
\|\rJ^{\bu_{0}}-\Id_{n}\|_{L^{\infty}(\cB(\bfz,r))} +
\|\rG^{\bu_{0}}-\Id_{n}\|_{L^{\infty}(\cB(\bfz,r))} \leq \frac{r}{d_{\gVc}(\bu_{0})} \,c(\Omega), \quad \mbox{ for all }r\leq \rho(\bu_{0}) \, ,
\end{equation*}
with $\rho(\bu_{0})$ as in Proposition~\Ref{P:estimatejacobian} and 
\begin{equation}
\label{D:domega}
d_{\gVc}(\bu_{0})=
\left\{
\begin{aligned}
&1 \quad \mbox{if} \quad \gVc=\emptyset,
\\
&\dist(\bu_{0},\gVc) \quad \mbox{else}. 
\end{aligned}
\right.
\end{equation}
\end{enumerate}
\end{corollary}

\begin{remark}
Note that estimate {\em (b)} blows up when we get closer to a conical point without reaching it, while at any conical point $\bx_0\in\gVc$, we have the good estimate {\em (a)}. This will lead to distinct analyses depending on how far $\bx_0$ is from $\gVc$.
\end{remark}

\subsection{Singular chains of 3D corner domains}
\label{SS:Sc3D}
\begin{proposition}\label{prop.3Dpolychaine}
Let $\Omega\in \gD(\R^3)$. Then chains of length $\leq3$ are sufficient to describe all equivalence\index{Singular chain!Equivalence} classes of the set of chains\index{Singular chain} $\gC(\Omega)$. 
If moreover $\Omega\in \ogD(\R^3)$, chains of length $2$ are sufficient.
\end{proposition}

\begin{proof}
Let $\bx_{0}\in\overline\Omega$. In Description~\ref{description} we enumerate all chains starting from $\bx_{0}$ with their tangent substructures according as $\bx_{0}$ is an interior point, a face point, an edge point, or a vertex. 

\begin{figure}[ht]
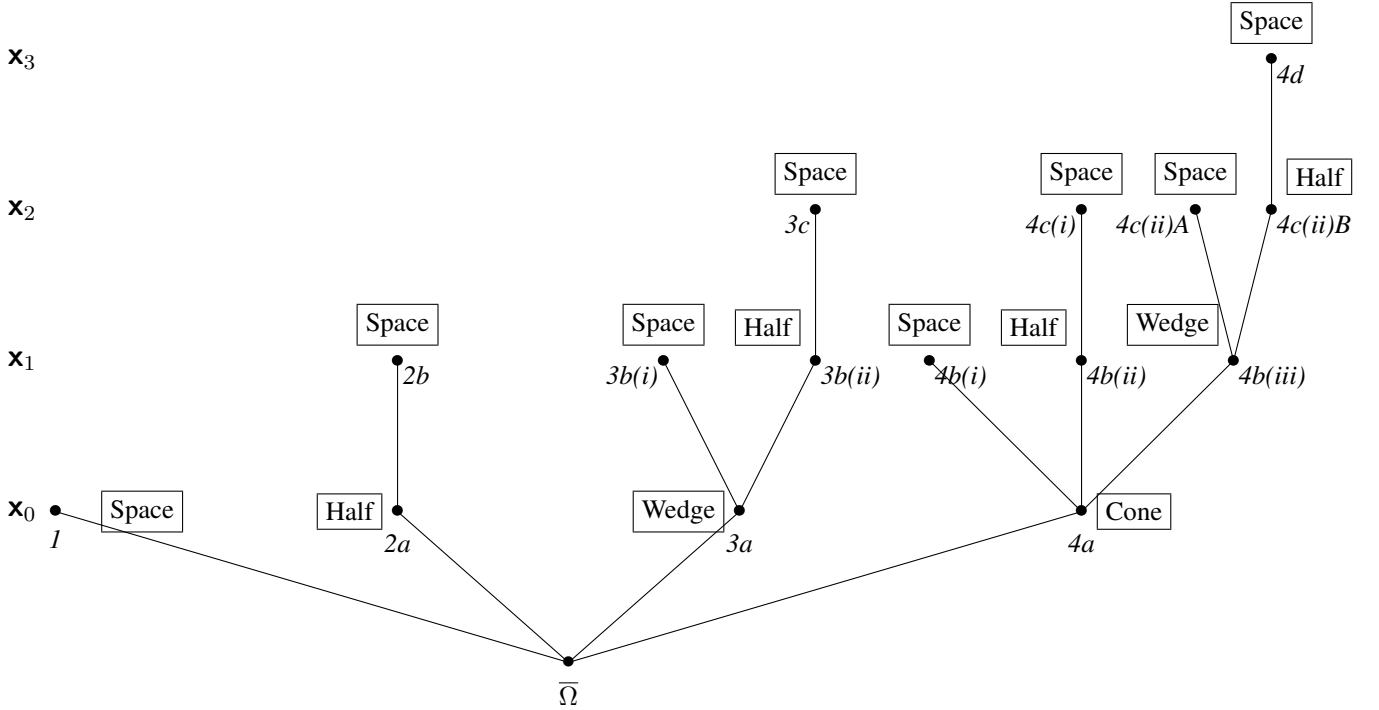

     \figinit{1mm}
     \figpt 1:(60, 0)
     \figpt 2:(60, 20)
     \figpt 3:(82.5, 0)
     \figpt 4:(70, 0)
     \figpt 5:(65, 0)
     \figvectP 10 [1,2]
     \figvectP 20 [1,3]
     \figpttra 21 := 2 / -3, 20/
     \figpttra 22 := 2 / -1, 20/
     \figpttra 23 := 2 /  1, 20/
     \figpttra 24 := 2 /  3, 20/
     \figvectP 30 [1,4]
     \figpttra 31 := 21 / 1, 10/
     \figpttra 32 := 22 / 1, 10/
     \figpttra 33 := 23 / 1, 10/
     \figpttra 34 := 33 / -1, 30/
     \figpttra 35 := 33 / 1, 30/
     \figpttra 36 := 24 / 1, 10/
     \figpttra 37 := 36 / -2, 30/
     \figpttra 38 := 36 / 2, 30/
     \figvectP 40 [1,5]
     \figpttra 48 := 31 / 1, 10/
     \figpttra 41 := 35 / 1, 10/
     \figpttra 42 := 36 / 1, 10/
     \figpttra 43 := 38 / 1, 10/
     \figpttra 44 := 43 / -1, 40/
     \figpttra 45 := 43 / 1, 40/
     \figpttra 58 := 48 / 1, 10/
     \figpttra 51 := 45 / 1, 10/

     \figdrawbegin{}
     \figdrawline [1,21]
     \figdrawline [1,22]
     \figdrawline [1,23]
     \figdrawline [1,24]
     \figdrawline [22,32]
     \figdrawline [23,34]
     \figdrawline [23,35]
     \figdrawline [24,37]
     \figdrawline [24,36]
     \figdrawline [24,38]
     \figdrawline [35,41]
     \figdrawline [36,42]
     \figdrawline [38,44]
     \figdrawline [38,45]
     \figdrawline [45,51]

     \figdrawend
     \figvisu{\figbox}{}{%
     \figwritew 21 :{$\bx_0$}(2.5)
     \figwritew 31 :{$\bx_1$}(2.5)
     \figwritew 48 :{$\bx_2$}(2.5)
     \figwritew 58 :{$\bx_3$}(2.5)
     {\footnotesize
     \figset write(mark=$\bullet$)
     \figwrites 1 :{$\overline\Omega$}(2.5)
     \figwrites 21 :{\emph{1}}(2.5)
     \figwritee 21 :{\framebox{Space}}(6)
     \figwrites 22 :{\emph{2a}}(3)
     \figwritew 22 :{\framebox{Half}}(2)
     \figwrites 23 :{\emph{3a}}(3)
     \figwritew 23 :{\framebox{Wedge}}(2)
     \figwrites 24 :{\emph{4a}}(3)
     \figwritee 24 :{\framebox{Cone}}(2)
     \figwriten 32 :{\framebox{Space}}(2)
     \figwritese 32 :{\emph{2b}}(1)
     \figwriten 34 :{\framebox{Space}}(2)
     \figwritesw 34 :{\emph{3b(i)}}(1)
     \figwritenw 35 :{\framebox{Half}}(2.8)
     \figwritese 35 :{\emph{3b(ii)}}(1)
     \figwriten 37 :{\framebox{Space}}(2)
     \figwritese 37 :{\emph{4b(i)}}(1)
     \figwritenw 36 :{\framebox{Half}}(2.8)
     \figwritese 36 :{\emph{4b(ii)}}(1)
     \figwritenw 38 :{\framebox{Wedge}}(2.8)
     \figwritese 38 :{\emph{4b(iii)}}(1)
     \figwriten 41 :{\framebox{Space}}(2)
     \figwritesw 41 :{\emph{3c}}(1)
     \figwriten 42 :{\framebox{Space}}(2)
     \figwritesw 42 :{\emph{4c(i)}}(1)
     \figwriten 44 :{\framebox{Space}}(2)
     \figwritesw 44 :{\emph{4c(ii)A}}(1)
     \figwritene 45 :{\framebox{Half}}(2.8)
     \figwritese 45 :{\emph{4c(ii)B}}(1)
     \figwriten 51 :{\framebox{Space}}(2)
     \figwritese 51 :{\emph{4d}}(1)
     }
     }
     \centerline{\box\figbox}
    \caption{The tree of singular chains with numbering according to Description \ref{description} (Half is for half-space)}
\label{F:tree}
\end{figure}

\begin{desc}(see examples \ref{E:scains3} for an illustration)\label{description}\ 
\begin{enumerate}
\item Interior point $\bx_0\in\Omega$. Only one chain in $\gC_{\bx_0}(\Omega)$: $\dx=(\bx_0)$.  $\Pi_\dx\equiv\R^3$.
\smallskip

\item  
Let $\bx_0$ belong to a face. There are two chains in $\gC_{\bx_0}(\Omega)$:
\begin{enumerate}
\item $\dx = (\bx_0)$ with $\Pi_\dx=\Pi_{\bx_0}$, the tangent half-space. $\Pi_\dx\equiv\R^2\times\R_+$.
\item $\dx = (\bx_0,\bx_1)$ where $\bx_1=1$ is the only element in $\R_+\cap\dS^0$. 
Thus $\Pi_\dx=\R^3$. 
\end{enumerate}
\smallskip

\item 
Let $\bx_0$ belong to an edge. 
There are three possible lengths for chains in $\gC_{\bx_0}(\Omega)$:
\begin{enumerate}
\item $\dx = (\bx_0)$ with $\Pi_\dx=\Pi_{\bx_0}$, the tangent wedge (which is not a half-space). The reduced cone of $\Pi_{\bx_0}$ is a sector $\Gamma_{\bx_0}$ the section of which is an interval $\cI_{\bx_0}\subset\dS^1$.
\item $\dx = (\bx_0,\bx_1)$ where $\bx_1\in \overline \cI_{\bx_0}$. 
\begin{enumerate}
\item If $\bx_1$ is interior to $\cI_{\bx_0}$, $\Pi_\dx=\R^3$. No further chain.
\item If $\bx_1$ is a boundary point of $\cI_{\bx_0}$, $\Pi_\dx$ is a half-space, containing one of the two faces $\partial^\pm\Pi_{\bx_0}$ of the wedge $\Pi_{\bx_0}$.
\end{enumerate}
\item $\dx=(\bx_0,\bx_1,\bx_2)$ where $\bx_1\in \partial \cI_{\bx_0}$, $\bx_2=1$ and $\Pi_\dx=\R^3$.
\end{enumerate}
\smallskip
\pagebreak[3]

\item 
Let $\bx_{0}$ be a corner. 
There are four possible lengths for chains in $\gC_{\bx_0}(\Omega)$:
\begin{enumerate}
\item $\dx = (\bx_0)$ with $\Pi_\dx=\Pi_{\bx_0}$, the tangent cone (which is not a wedge). It coincides with its reduced cone. Its section $\Omega_{\bx_0}$ is a polygonal domain in $\dS^2$.
\item $\dx = (\bx_0,\bx_1)$ where $\bx_1\in \overline \Omega_{\bx_0}$. 
\begin{enumerate}
\item If $\bx_1$ is interior to $\Omega_{\bx_0}$, $\Pi_\dx=\R^3$. No further chain.
\item If $\bx_1$ is in a side of $\Omega_{\bx_0}$, $\Pi_\dx$ is a half-space.
\item If $\bx_1$ is a corner of $\Omega_{\bx_0}$, $\Pi_\dx$ is a wedge. Its edge contains one of the edges of $\Pi_{\bx_0}$.
\end{enumerate}
\item $\dx=(\bx_0,\bx_1,\bx_2)$ where $\bx_1\in\partial\Omega_{\bx_0}$ 
\begin{enumerate}
\item If $\bx_1$ is in a side of $\Omega_{\bx_0}$, $\bx_2=1$, $\Pi_\dx=\R^3$. No further chain.
\item If $\bx_1$ is a corner of $\Omega_{\bx_0}$, $C_{\bx_0,\bx_1}$ is plane sector, and $\bx_2\in\overline\cI_{\bx_0,\bx_1}$ where the interval $\cI_{\bx_0,\bx_1}$ is its section.
\begin{enumerate}
\item If $\bx_2$ is an interior point of $\cI_{\bx_0,\bx_1}$, then $\Pi_\dx=\R^3$. 
\item If $\bx_2$ is a boundary point of $\cI_{\bx_0,\bx_1}$, then $\Pi_\dx$ is a half-space.
\end{enumerate}
\end{enumerate}
\item $\dx=(\bx_0,\bx_1,\bx_2,\bx_3)$ where $\bx_1$ is a corner of $\Omega_{\bx_0}$, $\bx_2\in\partial\cI_{\bx_0,\bx_1}$ and $\bx_3=1$. Then $\Pi_\dx=\R^3$.
\end{enumerate}
\end{enumerate}
\end{desc}

As a consequence of this description we may identify equivalence classes in $\gC_{\bx_0}(\Omega)$. It remains to consider edge points and corners:

--- If $\bx_0$ is an edge point, there are 4 equivalence classes: $\dx=(\bx_0)$,
$\dx=(\bx_0,\bx_1^\pm)$ with $\bx_1^-,\bx_1^+$ the ends of $\cI_{\bx_0}$,
and $\dx=(\bx_0,\bx_1^\circ)$ with $\bx_1^\circ$ any chosen point in $\cI_{\bx_0}$.

--- If $\bx_0$ is a polyhedral corner, the set of the equivalence classes of $\gC_{\bx_0}(\Omega)$ is finite according to the following description. 
Let $\bx_1^j$, $1\le j\le N$, be the corners of $\Omega_{\bx_0}$, and $\bff^j_1$, $1\le j\le N$, be its sides (notice that there are as many corners as sides). There are $2N+2$ equivalence classes:
$\dx=(\bx_0)$ (vertex),
$\dx=(\bx_0,\bx_1^j)$ with $1\le j\le N$ (edge-point limit),
$\dx=(\bx_0,\bx_1^{\circ,j})$ with $\bx_1^{\circ,j}$ any chosen point inside $\bff^j_1$ (face-point limit), and
$\dx=(\bx_0,\bx_1^\circ)$ with $\bx_1^\circ$ any chosen point in $\Omega_{\bx_0}$ (interior point limit).

--- If $\bx_0$ belongs to $\gV^\circ$, the set of chains which are face-point limits is infinite. Moreover, chains $(\bx_0,\bx_1,\bx_2)$ obtained by the general above procedure (4)-(c)-(ii)-(B) can be irreducible: Such chains represent the limit of a conical face close to an edge.
\end{proof}

\begin{example}
\label{E:scains3}
We link some cases of the enumeration in Description \ref{description} with the corner domains in Figure \ref{F1}: 
\begin{itemize}
\item Cases 1 and 2 occur for all interior points and all points inside a face (regular points of the boundary), respectively. 
\item Case 3 (points inside an edge) occurs for domains in Figure \ref{F1B} and \ref{F1C}. Together with cases 1 and 2, case 3 is sufficient to describe all singular chains of domains in Figure \ref{F1B}.
\item Case 4 (corner point) occurs in Figures \ref{F1A} and \ref{F1C}. For figure \ref{F1A}, cases 4(a), 4(bi) and 4(bii) are enough to describe all singular chains issued from a corner, whereas for Figure \ref{F1C}, all subcases of case 4 are needed.
\end{itemize}
\end{example}

\begin{remark}
\label{rem:str3}
The exhaustion of chains done in Description \ref{description} allows to figure out what are the connected components of the set of chains $\gC(\Omega)$ for the distance $\dD$:
\begin{itemize}
\item The corner chains  $\dx=(\bx_0)$ are isolated from each other.
\item Let $\be$ be an edge. The chains $\dx=(\bx_0)$ with $\bx_0\in\be$ are completed by suitable corner chains $\dx=(\bx_0,\bx_1)$ such that $\bx_0\in\partial\be$ and $\Pi_\dx$ is a wedge. The resulting set endowed with distance $\dD$ is homoemorphic to $\overline\be$ with the standard distance (except if $\be$ has only one end, in analogy with the shape of the boundary presented in Figure \ref{fig.dim2}). Convex and nonconvex edges are at infinite distance from each other.
\item If $\Omega$ is polyhedral, we have something similar for the faces: With $\bff$ a chosen face, the chains $\dx=(\bx_0)$ with $\bx_0\in\bff$ are completed by suitable edge chains $\dx=(\bx_0,\bx_1)$ and suitable corner chains $\dx=(\bx_0,\bx_1,\bx_2)$ such that $\Pi_\dx$ is a half-space. The resulting set endowed with distance $\dD$ is homoemorphic to $\overline\bff$ with the standard distance (with a few exceptions as above).
\item If $\Omega$ is not polyhedral, and if the face $\bff$ contains a conical point $\bx_0$, the contribution of the corner chains $\dx=(\bx_0,\bx_1,\bx_2)$ does not reduce to a single chain with a single half-space $\Pi_\dx$. We have now a {\em blow up} of the boundary of $\bff$ near $\bx_0$. Distinct faces are at finite nonzero distance from each other in general (the exception is when two faces share a corner and a tangent plane passing by this corner). 
\item Finally, $\overline\Omega$ is homoemorphic to the union of all chains $\dx$ starting with any $\bx_0\in\overline\Omega$ and such that $\Pi_\dx=\R^3$.
\end{itemize}
\end{remark}

\chapter{Magnetic Laplacians and their tangent operators}
\label{sec:chgvar}
Let $\bA$ be  a magnetic potential associated with the magnetic field $\bB$ on a corner domain $\Omega\in\gD(\R^3)$. We recall that $\Omega$ is assumed to be simply connected, and that the corresponding magnetic Laplacian\index{Magnetic Laplacian} is \index{Magnetic Laplacian!$\OP_{h}(\bA,\Omega)$}
$
   \OP_{h}(\bA,\Omega) = (-ih\nabla+\bA)^2
$. 
At each point $\bx_0\in\overline\Omega$ is associated a local map\index{Local map} $(\cU_{\bx_0},\diffeo^{\bx_0})$ and a tangent cone\index{Tangent cone} $\Pi_{\bx_0}$, cf.\ \eqref{eq:diffeo}. We will associate a tangent magnetic potential\index{Tangent magnetic potential} to $\Pi_{\bx_0}$ and provide formulas and estimates for the operator transformed by the local map $(\cU_{\bx_0},\diffeo^{\bx_0})$ from the magnetic Laplacian $\OP_{h}(\bA,\Omega)$.

\section{Change of variables}
\label{SS:CV}
Let $\Omega\in\gD(\R^3)$. We consider a magnetic potential $\bA\in \sC^{1}(\overline\Omega)$. Let $\bx_{0}\in\overline\Omega$. 
Let us recall that with $\bx_0$ are associated the local smooth diffeomorphism $\diffeo^{\bx_0}$ \eqref{eq:diffeo}, the Jacobian\index{Jacobian} matrix $\rJ^{\bx_0}$ \eqref{eq.jacob} of the inverse of $\diffeo^{\bx_0}$ and the associated metric\index{Metric} $\rG^{\bx_0}$ \eqref{eq:metric}. 
According to formulas \eqref{E:Atilde}--\eqref{D:tbB}, we introduce the magnetic potential $\bA^{\bx_0}$ and magnetic field $\bB^{\bx_0}=\curl\bA^{\bx_0}$ transformed by $\diffeo^{\bx_0}$ in $\cV_{\bx_0}\cap\Pi_{\bx_0}$
\begin{equation}
\label{E:ABtilde}
   \bA^{\bx_0}:=(\rJ^{\bx_0})^{\top} \big((\bA-\bA(\bx_0))\circ (\diffeo^{\bx_0})^{-1} \big) \quad 
   \mbox{and}\quad 
   \bB^{\bx_0}:=|\det \rJ^{\bx_0}|\,(\rJ^{\bx_0})^{-1} \big(\bB\circ (\diffeo^{\bx_0})^{-1} \big) \ . 
\end{equation}
We also introduce the phase shift \index{Phase shift}
\begin{equation}
\label{E:phase}
   \zeta^{\bx_0}_h(\bx) = \re^{i\langle \bA(\bx_0),\ee\bx\rangle/h},\quad \bx\in\Omega,
   \index{Phase shift!$\zeta^{\bx_0}_h(\bx)$}
\end{equation}
so that there holds for any $f$ in $H^1(\Omega)$
\begin{equation}
\label{eq:shift}
   q_{h}[\bA,\Omega](f)=q_{h}[\bA-\bA(\bx_0),\Omega](\zeta^{\bx_0}_h f).
\end{equation}
To $f\in H^1(\Omega)$ with support in $\cU_{x_0}$ we associate the function $\psi$
\begin{equation}
\label{E:psi}
   \psi:= (\zeta^{\bx_0}_h f)\circ (\diffeo^{\bx_0})^{-1},
\end{equation} 
defined in $\Pi_{\bx_0}$, with support in $\cV_{\bx_0}$. 
For any $h>0$ Lemma \ref{L:chgvar} provides the identities
\begin{equation}
\label{E:chgGx0}
   q_{h}[\bA,\Omega](f)
   = q_{h}[\bA^{\bx_{0}},\Pi_{\bx_0},\rG^{\bx_0}](\psi) 
   \quad \mbox{and} \quad 
   \| f\|_{L^2(\Omega)}=\| \psi \|_{L^2_{\rG^{\bx_0}}(\Pi_{\bx_0})}\,,
\end{equation}
where the quadratic forms $q_{h}[\bA,\Omega]$ and $q_{h}[\bA^{\bx_0},\Pi_{\bx_0},\rG^{\bx_0}]$ are defined in \eqref{D:fq}  and \eqref{D:fqG}, respectively. 
Using the Rayleigh quotient\index{Rayleigh quotient}, we immediately deduce
\begin{equation}\label{E:QRmetr}
\QR_{h}[\bA,\Omega](f)
   = \QR_{h}[\bA^{\bx_{0}},\Pi_{\bx_0},\rG^{\bx_0}](\psi).
\end{equation}

\section{Model and tangent operators}
\begin{definition}\label{def.model}
We call {\emph{model operator}}\index{Model operator} any magnetic Laplacian
$\OP(\bA,\Pi)$ where $\Pi\in\gP_{3}$ and $\bA$ is a linear potential associated with the constant magnetic field $\bB$.  
We denote by $\En(\bB \ee,\Pi)$ the bottom of the spectrum (ground state energy) of $\OP(\bA \ee,\Pi)$ and by $\lambda_{\rm ess}(\bB,\Pi)$\index{Essential spectrum!$\lambda_{\rm ess}(\bB,\Pi)$} the bottom of its essential spectrum.\index{Essential spectrum} 
\end{definition}

Let $\Omega\in\gD(\R^3)$ and $\bA\in \sC^{1}(\overline{\Omega})$. 
For each $\bx_0\in\overline\Omega$ we set
\begin{equation}
\label{eq:tangent0}
   \bB_{\bx_0} = \bB(\bx_0) \quad\mbox{and}\quad
   \bA_{\bx_0}(\bv) = \nabla\bA(\bx_0)\cdot\bv,\ \ \bv\in\Pi_{\bx_0},
\end{equation}
so that $\bB_{\bx_0}$ is the magnetic field frozen at $\bx_0$ and $\bA_{\bx_0}$\index{Tangent magnetic potential! $\bA_{\bx_0}$} the linear part\footnote{In \eqref{eq:tangent0}, $\nabla\bA$ is the $3\times3$ matrix with entries $\partial_kA_j$, $1\le j,k\le 3$, and $\cdot\,\bv$ denotes the multiplication by the column vector $\bv=(v_1,v_2,v_3)^\top$.} of the potential at $\bx_0$. 

By extension, for each singular chain $\dx=(\bx_0,\bx_1,\ldots,\bx_\pp)\in\gC(\Omega)$ we set\index{Tangent magnetic potential!$\bA_\dx$}
\begin{equation}
\label{eq:tangent}
   \bB_\dx = \bB(\bx_0) \quad\mbox{and}\quad
   \bA_\dx(\bx) = \nabla\bA(\bx_0)\cdot\bx,\ \ \bx\in\Pi_\dx.
\end{equation}
We have obviously
\[
   \curl\bA_\dx = \bB_\dx\,.
\]

\begin{definition}\label{def.tgt}
Let $\Omega\in\gD(\R^3)$ and $\bA\in \sC^{1}(\overline{\Omega})$.
Let $\dx\in\gC(\Omega)$ be a singular chain of $\Omega$. The model operator  $\OP(\bA_\dx \ee,\Pi_\dx)$\index{Tangent operator!$\OP(\bA_\dx \ee,\Pi_\dx)$} is called a \emph{tangent operator}\index{Tangent operator}. 
\end{definition}

\begin{remark}\label{rem:chainOP}
The notion of equivalence classes between singular chains as introduced in Definition \ref{def:chaineq} is sufficient for the analysis of operators $\OP_{h}(\bA,\Omega)$ in the case of magnetic fields $\bB$ smooth in Cartesian variables. Should $\bB$ be smooth in polar variables only, the whole hierarchy of singular chains would be needed.
\end{remark}

The potential $\bA_{\bx_0}$ and the field  $\bB_{\bx_0}$ are connected to the potential $\bA^{\bx_0}$ and field $\bB^{\bx_0}$ \eqref{E:ABtilde} obtained through the local map:\index{Local map} Since $\rd \diffeo^{\bx_{0}}(\bx_{0})=\Id_{3}$ by definition, we have
\begin{equation}
\label{eq:field0}
\bB^{\bx_0}(\bfz)=\bB_{\bx_{0}}\,. 
\end{equation}
Likewise, let $\bA^{\bx_0}_{\bfz}$\index{Linearized magnetic potential!$\bA^{\bx_0}_{\bfz}$} be the linear part of $\bA^{\bx_0}$ at the vertex $\bfz$ of $\Pi_{x_0}$. 
Then:
\begin{equation}
\label{eq:pot0}
   \bA^{\bx_0}(\bfz)=0 \quad\mbox{and}\quad
   \bA^{\bx_0}_{\bfz} = \bA_{\bx_{0}}.
\end{equation} 

Local and minimum energies are introduced as follows.

\begin{definition}
\label{not:En}
Let $\Omega\in\gD(\R^3)$ and $\bB\in \sC^0(\overline\Omega)$.
The application  $\bx\mapsto\En(\bB_\bx \ee,\Pi_\bx)$\index{Local ground energy!$\En(\bB_\bx \ee,\Pi_\bx)$} is called \emph{local ground energy}\index{Local ground energy} (with $\En(\bB,\Pi) $ introduced in Definition~\ref{def.model}).
We define the \emph{lowest local energy}\index{Lowest local energy} of $\bB$ on $\overline\Omega$ by
\begin{equation}\index{Lowest local energy!$\sE(\bB \ee,\Omega)$}
\label{eq:sbis}
   \sE(\bB \ee,\Omega) := \inf_{\bx\in\overline\Omega} \En(\bB_\bx \ee,\Pi_\bx).
 \end{equation}
\\[-2\baselineskip]
\mbox{} \hfill$\blacksquare$
\end{definition}
The relations with singular chains and the question whether $\sE(\bB \ee,\Omega)$ is a minimum are addressed later on Chapter \ref{sec:sci}.

\section{Linearization}\label{SS:lin}
Starting from the identity \eqref{E:chgGx0} $q_{h}[\bA,\Omega](f)
= q_{h}[\bA^{\bx_{0}},\Pi_{\bx_0},\rG^{\bx_{0}}](\psi)$, we want to compare $q_{h}[\bA^{\bx_{0}},\Pi_{\bx_0},\rG^{\bx_{0}}](\psi)$ with the term $q_{h}[\bA^{\bx_{0}}_\bfz,\Pi_{\bx_0}](\psi) = q_{h}[\bA_{\bx_{0}},\Pi_{\bx_0}](\psi)$ obtained by linearizing the potential and the metric. 

\subsection{Change of metric}
Here we compare $L^2$ norm and quadratic forms associated with the metric $\rG^{\bx_0}$, with the corresponding quantities associated with the trivial metric $\Id_3$. Like in Proposition \ref{P:estimatejacobian} and Corollary \ref{C:expandJ}, and for the same reasons, we have essentially two distinct cases, resulting into a uniform approximation in a polyhedral domain, and a controlled blow up close to conical points when they are present.

\begin{lemma}
\label{L:chgvarloc}
Let $\Omega\in \gD(\R^{3})$ and $(\cU_{\bx},\diffeo^{\bx})_{\bx\in \overline{\Omega}}$ be an admissible atlas\index{Atlas!Admissible atlas}. We recall that the set of reference points $\gX$ contains the set of conical vertices $\gVc$. 
Let $\bA\in \sC^{2}(\overline{\Omega})$ be a magnetic potential and, for $\bx_{0}\in\overline\Omega$, let $\bA^{\bx_0}$ be the potential \eqref{E:ABtilde} produced by the local map $\diffeo^{\bx_{0}}$. 
There exists $c(\Omega)$ such that 
\begin{enumerate}[(a)]
\item \label{It:cvloc} for all $\bx_{0}\in\gX$ and $r\in(0,R_{\bx_{0}})$, for all $\psi\in H^{1}(\Pi_{\bx_{0}})$ satisfying $\supp(\psi)\subset \cB(\bfz,r)$, we have:
\begin{equation}
\label{eq:psiG1}
   \begin{gathered}
   \big|q_{h}[\bA^{\bx_0},\Pi_{\bx_{0}},\rG^{\bx_0}](\psi) -
    q_{h}[\bA^{\bx_0},\Pi_{\bx_{0}}](\psi) \big|   \leq 
    c(\Omega) \,r \, q_{h}[\bA^{\bx_0},\Pi_{\bx_{0}},\rG^{\bx_0}](\psi) , \\[0.5ex]
    \big| \| \psi\|_{L^2_{\rG^{\bx_{0}}}(\Pi_{\bx_{0}})} -
    \|\psi \|_{L^2(\Pi_{\bx_{0}})} \big| \leq 
    c(\Omega)\,  r \,\| \psi\|_{L^2(\Pi_{\bx_{0}})}\,.
\end{gathered}
\end{equation}
\item \label{It:cvlocbis} for all $\bu_0\in \overline{\Omega}\setminus \gX$ and $r\in(0, \rho(\bu_{0}))$ (with $\rho(\bu_{0})$ given by Proposition~\Ref{P:estimatejacobian}), for all $\psi\in H^{1}(\Pi_{\bu_0})$ satisfying $\supp(\psi)\subset \cB(\bfz,r)$, we have:
\begin{equation}
\label{eq:psiG2}
   \begin{gathered}
   \big|q_{h}[\bA^{\bu_0},\Pi_{\bu_{0}},\rG^{\bu_0}](\psi) -
    q_{h}[\bA^{\bu_0},\Pi_{\bu_{0}}](\psi) \big|   \leq 
    c(\Omega) \,\frac{r}{d_{\gVc}(\bu_{0})} \, 
    q_{h}[\bA^{\bu_0},\Pi_{\bu_{0}},\rG^{\bu_0}](\psi) , \\[0.5ex]
    \big| \| \psi\|_{L^2_{\rG^{\bu_{0}}}(\Pi_{\bu_{0}})} -
    \|\psi \|_{L^2(\Pi_{\bu_{0}})} \big| \leq 
    c(\Omega)\,  \frac{r}{d_{\gVc}(\bu_{0})} \,\| \psi\|_{L^2(\Pi_{\bu_{0}})}\,,
\end{gathered}
\end{equation}
with $d_{\gVc}$ defined in \eqref{D:domega}.
\end{enumerate}
\end{lemma}
 
\begin{proof}
The lemma is a direct consequence of Corollary \ref{C:expandJ} providing estimates for the $L^\infty$ norm of the difference $\rG^{\bx_0}-\Id_3$. Let $\tau_i=\tau_i(\bx)$ be the eigenvalues of $\rG^{\bx_0}(\bx)$. The estimate on $\rG^{\bx_0}-\Id_3$ implies a similar estimate for $\max\{ \|\tau_i-1\|_{L^\infty},1\leq i\leq 3\}$,
which allows to compare the quadratic forms associated with $\rG^{\bx_0}$ and with $\Id_3$.
\end{proof}

Combining the identities \eqref{E:chgGx0} with Lemma~\ref{L:chgvarloc}, we see that it is equivalent to deal with $q_{h}[\bA,\Omega](f)$ or $q_{h}[\bA^{\bx_0},\Pi_{\bx_{0}}](\psi)$ modulo a well-controlled error. This will be useful later on when we will estimate the corresponding Rayleigh quotients (see Chapters~\ref{sec:low} and \ref{sec:up}). 

\subsection{Linearization of the potential}
We estimate the remainders due to the linearization\index{Linearized magnetic potential} $\bA^{\bx_{0}}_{\bfz}$\index{Linearized magnetic potential!$\bA^{\bx_0}_{\bfz}$} at the vertex $\bfz$ of the tangent cone $\Pi_{\bx_0}$ of the potential $\bA^{\bx_{0}}$ resulting from a local map. 
For this, we first use a Taylor expansion around $\bfz$ in $\Pi_{\bx_0}$.

\begin{lemma}\label{lem.TaylorA}
Let $\bx_{0}\in\overline\Omega$. For any $r>0$ such that $\cV_{\bx_0}\supset\cB(\bfz,r)$ 
\begin{equation}
\label{E:taylorA}
  \forall \bv\in\cB(\bfz,r)\cap\Pi_{\bx_{0}},\quad |\bA^{\bx_0}(\bv)-\bA^{\bx_0}_\bfz(\bv)|\leq 
  \tfrac1{2} \|\bA^{\bx_0}\|_{W^{2,\infty}(\cB(\bfz,r)\cap\Pi_{\bx_{0}})} \,|\bv|^2\, .
\end{equation}
\end{lemma}

So we have to estimate the second derivatives of the mapped potentials $\bA^{\bx_0}$.
\begin{lemma}\label{L:d2A}
Let $\Omega\in \gD(\R^{3})$ with an associated admissible atlas with set of reference points $\gX$. 
Let $\bA\in \sC^{2}(\overline{\Omega})$ be a magnetic potential. For $\bx_{0}\in\overline\Omega$, let $\bA^{\bx_0}$ be the potential \eqref{E:ABtilde}.
There exists $c(\Omega)$ such that 
\begin{enumerate}[(a)\ ]
\item \label{It:Ld2A} for all $\bx_{0}\in\gX$, 
\begin{equation}\label{E:D2Ax0}
\|\rd^2\bA^{\bx_{0}}(\bv)\| \leq c(\Omega) \|\bA\|_{W^{2,\infty}(\Omega)},\qquad
      \forall \bv\in \cB(\bfz,R_{\bx_{0}}).
\end{equation}
\item \label{It:Ld2Abis} for all $\bu_0\in \overline{\Omega}\setminus \gX$,  with $\rho(\bu_{0})$ given in Proposition~\emph{\ref{P:estimatejacobian}} and  $d_{\gVc}$ defined in \eqref{D:domega},
\begin{equation}\label{E:D2Au0}
\|\rd^2\bA^{\bu_{0}}(\bv)\| \leq c(\Omega)\left(\frac{\|\bA\|_{W^{1,\infty}(\Omega)}}{d_{\gVc}(\bu_0)}
      + \|\bA\|_{W^{2,\infty}(\Omega)}\right),\qquad
      \forall \bv\in \cB(\bfz,\rho(\bu_{0})).
\end{equation}
\end{enumerate}
\end{lemma}

\begin{proof}
Let $\bu_{0}\in\overline\Omega$. 
Differentiating twice \eqref{E:ABtilde}, we obtain, for $\bu\in\cU_{\bx_{0}}$ and $\bv=\diffeo^{\bu_{0}}(\bu)$,
\begin{equation*}
\|\rd^2\bA^{\bu_{0}}(\bv)\| 
\lesssim \|\rd\rK^{\bu_{0}}(\bv)\|\ |\bA(\bu)-\bA(\bu_{0})| 
      +\|\rK^{\bu_{0}}(\bv)\|\ \|\rJ^{\bu_{0}}(\bv)\|\ \|\rd\bA(\bu)\|
      +\|\rJ^{\bu_{0}}(\bv)\|^3\ \|\rd^2\bA(\bu)\|.
\end{equation*}
\begin{enumerate}[(a)]
\item When $\bu_{0}=\bx_{0}\in\gX$, \eqref{E:D2Ax0} is a consequence of Proposition~\ref{P:estimatejacobian} and Remark~\ref{R:dKu0} (1).
\item Let $\bu_{0}\in\overline\Omega\setminus\gX$ and $\bx_{0}\in\gX$ such that $\bu_{0}\in\cU_{\bx_{0}}$. The above inequality, Proposition~\ref{P:estimatejacobian} and Remark~\ref{R:dKu0} (2) yield for $\bv\in\cB(\bfz,\rho(\bu_{0}))$,
\begin{eqnarray*}
\|\rd^2\bA^{\bu_{0}}(\bv)\| 
&\lesssim& \frac{|\bu-\bu_{0}|}{|\bu_{0}-\bx_{0}|^2}\|\bA\|_{W^{1,\infty}}
      +\frac{1}{|\bu_{0}-\bx_{0}|}\|\bA\|_{W^{1,\infty}}
      + \|\bA\|_{W^{2,\infty}}\\
&\lesssim& \frac{1}{|\bu_{0}-\bx_{0}|}\|\bA\|_{W^{1,\infty}}
      + \|\bA\|_{W^{2,\infty}}.
\end{eqnarray*}
Here we have used the inequality $|\bu-\bu_{0}|\leq|\bu_{0}-\bx_{0}|$ which holds by construction of the admissible atlas.
\end{enumerate}
\vspace{-\baselineskip}
\end{proof}

Estimates between $\bA^{\bx_{0}}$ and $\bA^{\bx_{0}}_{\bfz}$ deduced from the combination of Lemmas~\ref{lem.TaylorA} and \ref{L:d2A} allow to compare $q_{h}[\bA^{\bx_{0}},\Pi_{\bx_{0}}](\psi)$ and $q_{h}[\bA^{\bx_{0}}_{\bfz},\Pi_{\bx_{0}}](\psi)$ via identity \eqref{eq:diffAA'}  which writes
$$
  q_{h}[\bA^{\bx_{0}},{\Pi_{\bx_{0}}}](\psi) =   
  q_{h}[\bA^{\bx_{0}}_{\bfz},\Pi_{\bx_{0}}](\psi) 
  +2\Re\big\langle (-ih\nabla+\bA^{\bx_{0}}_{\bfz})\psi,(\bA^{\bx_{0}}-\bA^{\bx_{0}}_{\bfz})\psi\big\rangle 
  + \|(\bA^{\bx_{0}}-\bA^{\bx_{0}}_{\bfz})\psi\|^2.
$$
This will be extensively used in Chapters~\ref{sec:low} and \ref{sec:up}.

\section{A general rough upper bound}
\label{ss:genup}
 Before tackling lower bounds in the next chapter, relying on the perturbation estimates provided by Lemmas \ref{L:chgvarloc} and \ref{L:d2A}, we are going to prove a very general rough upper bound
for the Rayleigh quotients\index{Rayleigh quotient} $\QR_{h}[\bA,\Omega]$ \eqref{eq:RayQuot} as $h\to0$. This proof does use any specific feature of three-dimensional problems. So we present it in the $n$-dimensional framework.

In the $n$-dimensional case, the magnetic field is a 2-form and associated magnetic potentials are 1-forms that we write by using their representation as vector fields in a canonical basis of $\R^n$, see \eqref{eq:omegaA}--\eqref{eq:sigmaB}. In dimension $n$, $\En(\bB,\Pi)$ and $\sE(\bB,\Omega)$ are defined as in Definition \ref{not:En}. 

In this context we prove a rough upper bound on the first eigenvalue of $\OP_h(\bA,\Omega)$ by using only elementary arguments. We need the following Lemma, that will also be useful later: 

\begin{lemma}
\label{L:ConsQMEps}
Let $\Omega\in \gD(\R^n)$ and let $\bA\in \sC^{2}(\overline{\Omega})$ be a magnetic potential associated with the magnetic field $\bB$. 
Let $\bx_{0}\in \overline{\Omega}$ be a chosen point and let $\varepsilon>0$. Then there exists $h_{0}>0$ such that for all $h\in (0,h_{0})$ there exists a function $f_{h}$ supported near $\bx_0$ satisfying 
$$
   \QR_{h}[\bA,\Omega](f_{h})
   \leq h\big( E(\bB_{\bx_{0}},\Pi_{\bx_{0}})+\varepsilon \big) \, ,
$$
where $E(\bB_{\bx_{0}},\Pi_{\bx_{0}})$ is the ground state energy\index{Ground state energy|textbf} of $\OP(\bA_{\bx_{0}},\Pi_{\bx_{0}})$. 
\end{lemma}

\begin{proof}
Let $(\cU_{\bx_{0}},\diffeo^{\bx_{0}})$ be a local map with $\diffeo^{\bx_{0}}:\cU_{\bx_{0}}\mapsto \cV_{\bx_{0}}\subset \Pi_{\bx_{0}}$,  cf.\ \eqref{eq:diffeo}. 
This change of variables transforms the magnetic potential into  $\bA^{\bx_{0}}$ given by \eqref{E:ABtilde}:
\begin{equation*}
   \bA^{\bx_{0}}=(\rJ^{\bx_{0}})^{\top} \big((\bA-\bA(\bx_{0}))\circ (\diffeo^{\bx_{0}})^{-1} \big) \, .
\end{equation*} 
Denote by $\bA^{\bx_{0}}_{\bfz}$ its linear part. Recall that $\curl\bA^{\bx_{0}}_{\bfz}=\bB_{\bx_{0}}$.  
By definition of $E(\bB_{\bx_{0}},\Pi_{\bx_{0}})$ there exists $\psi \in \dom(q[\bA^{\bx_{0}}_{\bfz},\Pi_{\bx_{0}}])$ a $L^2$-normalized function such that
$$
   q[\bA^{\bx_{0}}_{\bfz},\Pi_{\bx_{0}}](\psi) \le 
   \En(\bB_{\bx_{0}},\Pi_{\bx_{0}})+\tfrac{\varepsilon}{4} \, .
$$
Let us consider a smooth cut-off function $\chi$ with support in $\cB(\bfz,1)$ and equal to $1$ on $\cB(\bfz,\frac12)$. Then the functions with compact support
$
   \bx\longmapsto\chi(\tfrac{\bx}{R})\,\psi(\bx)
$
converge to $\psi$ in $\dom(q[\bA^{\bx_{0}}_{\bfz},\Pi_{\bx_{0}}])$ as $R\to\infty$. Therefore there exists $R=R(\varepsilon,\bx_0)>0$ and a new function $\psi \in \dom(q[\bA^{\bx_{0}}_{\bfz},\Pi_{\bx_{0}}])$ with support in $\cB(\bfz,R)$ which satisfies
$$
   q[\bA^{\bx_{0}}_{\bfz},\Pi_{\bx_{0}}](\psi) \le 
  \En(\bB_{\bx_{0}},\Pi_{\bx_{0}})+\tfrac{\varepsilon}{2} \, .
$$
For $h>0$, define the $L^2$-normalized function
$\psi_{h}(\bx)=h^{-n/4}\psi(h^{-1/2}\bx)$
so that, cf.\ Lemma \ref{lem.dilatation},
$$ q_h[\bA^{\bx_{0}}_{\bfz},\Pi_{\bx_{0}}](\psi_{h}) \le 
h\big(\En(\bB_{\bx_{0}},\Pi_{\bx_{0}})+\tfrac{\varepsilon}{2}\big) \, .$$
We have the inclusion $\supp(\psi_{h}) \subset \cB(\bfz,h^{1/2}R)$ and therefore there exists $h_{\varepsilon}>0$ such that for all 
$ h\in (0,h_{\varepsilon})$,   
$\supp(\psi_{h})\subset \cV_{\bx_{0}}$.
Combining \eqref{eq:diffAA'} with a Cauchy-Schwarz inequality we find
\begin{multline}
\label{E:CSMajGen}
  q_{ h}[\bA^{\bx_{0}},{\Pi_{\bx_{0}}}](\psi_{h}) 
  \leq   
  q_{ h}[\bA^{\bx_{0}}_{\bfz},\Pi_{\bx_{0}}](\psi_{h}) \\
  +2\sqrt{q_{ h}[\bA^{\bx_{0}}_{\bfz},\Pi_{\bx_{0}}](\psi_{h}) } \ 
  \|(\bA^{\bx_{0}}-\bA^{\bx_{0}}_{\bfz})\psi_{h}\|
  + \|(\bA^{\bx_{0}}-\bA^{\bx_{0}}_{\bfz})\psi_{h}\|^2.
\end{multline}
Notice now that the estimates \emph{(\ref{P:DLx0})} of Proposition \ref{P:estimatejacobian} are still valid for any chosen $\bx_{0}$ in $\overline{\Omega}$ with constants $c(\bx_0)$ and radius $R_{\bx_0}$ depending on $\bx_0$. Hence estimates \emph{(\ref{It:Ld2A})} of Lemma \ref{L:d2A} holds at $\bx_0$ with a constant $c(\bx_0)$ replacing the uniform constant $c(\Omega)$. Therefore applying Lemma~\ref{lem.TaylorA} with $r=h^{1/2}R$ we get $c=c(\varepsilon,\bx_0)>0$ such that
\begin{equation*}
\|(\bA^{\bx_{0}}-\bA^{\bx_{0}}_{\bfz})\psi_{h}\|
\leq c R^2  h\|\psi_{h}\|,\quad \forall h\in(0,h_\varepsilon).
\end{equation*}
Let $\rG^{\bx_{0}}$ be the metric associated with the change of variables (see Section \ref{SS:CV}). Again \emph{(\ref{It:cvloc})} of Lemma \ref{L:chgvarloc} is valid for all $\bx_{0}\in \overline{\Omega}$ with $c(\bx_{0})$ instead of $c(\Omega)$. Applying this with $r=h^{1/2}R$ provides another constant $c=c(\varepsilon,\bx_0)>0$ such that
\begin{align}
   \big|q_{ h}[\bA^{\bx_{0}},\Pi_{\bx_{0}},\rG^{\bx_{0}}](\psi_{h}) - q_{ h}[\bA^{\bx_{0}},\Pi_{\bx_{0}}](\psi_{h}) \big|   
    &\leq 
     c  \,Rh^{1/2} \, 
    q_{ h}[\bA^{\bx_{0}},\Pi_{\bx_{0}}](\psi_{ h}) , \\
    \left|\|\psi_{h}\|_{L^2_{\rG^{\bx_{0}}}(\Pi_{\bx_{0}})}-\|\psi_{h}\|_{L^2(\Pi_{\bx_{0}})}\right| 
    &\leq c \,Rh^{1/2} \, \|\psi_{ h}\|_{L^2(\Pi_{\bx_{0}})} \, .
    \label{E:LMMajGen}
\end{align}
According to  Section~\ref{SS:CV}  \eqref{E:ABtilde}--\eqref{E:chgGx0}, we define for $h\in (0,h_{\varepsilon})$:
$$
   f_{h}:= (\zeta^{\bx_{0}}_h)^{-1} \, \psi_{h}\circ \diffeo^{\bx_{0}} \qquad \mbox{with} \qquad
   \zeta^{\bx_{0}}_h(\bx) = \re^{i\langle \bA(\bx_{0}),\ee\bx\rangle/h},\quad 
   \bx\in \cU_{\bx_{0}}\cap \overline{\Omega} \, 
$$
and we have
\begin{equation*}
q_{h}[\bA,\Omega](f_{h})
=q_{h}[\bA^{\bx_{0}},\Pi_{\bx_{0}},\rG^{\bx_{0}}](\psi_{h})
\quad\mbox{and}\quad
\|f_{h}\|_{L^2(\Omega)}=\|\psi_{h}\|_{L^2_{\rG^{\bx_{0}}}(\Pi_{\bx_{0}})}.
\end{equation*}
Thus, combining with \eqref{E:CSMajGen}--\eqref{E:LMMajGen} we deduce
\begin{eqnarray*}
\QR_{h}[\bA,\Omega](f_{ h})
&\leq& (1+c R h^{1/2})\Big(\QR_{h}[\bA^{\bx_{0}}_{\bfz},\Pi_{\bx_{0}}](\psi_{h}) +c\left( R^2 h^{3/2} + R^4 h^{2}\right)\Big)\\[0.25ex]
&\leq& (1+cR h^{1/2})\Big( h\big(\En(\bB_{\bx_{0}},\Pi_{\bx_{0}}) 
+\tfrac{\varepsilon}{2}\big)+c\left( R^2 h^{3/2} + R^4 h^{2}\right)\Big).
\end{eqnarray*}
We can write this in the form 
$$\QR_{ h}[\bA,\Omega](f_{ h})\leq 
h\big( \En(\bB_{\bx_{0}},\Pi_{\bx_{0}})+\tfrac{\varepsilon}{2}+h^{1/2}M_\varepsilon(h)\big),$$
where $M_\varepsilon(h)$ is a bounded function for $h\in[0,h_\varepsilon]$ that depends on $\varepsilon>0$. 
We deduce the lemma by choosing $h$ so small that $h^{1/2}M_{\varepsilon}(h) \leq \frac{\varepsilon}{2}$.
\end{proof}

As a consequence of Lemma \ref{L:ConsQMEps} and the min-max principle we obtain:

\begin{proposition} 
Let $\Omega\in \gD(\R^n)$ and let $\bA\in \sC^{2}(\overline{\Omega})$ be a magnetic potential associated with the magnetic field $\bB$. 
Then the first eigenvalue $\lambda_{h}(\bB,\Omega)$ of $\OP(\bA,\Omega)$ satisfies
$$\limsup_{h\to0}\frac{\lambda_{h}(\bB,\Omega)}{h} \leq \sE(\bB,\Omega) \, .$$ 
\end{proposition}

\chapter{Lower bounds for ground state energy in corner domains}
\label{sec:low}

In this section we establish a lower bound 
for the first eigenvalue $\lambda_{h}(\bB,\Omega)$ of the magnetic Laplacian $\OP_{h}(\bA,\Omega)$ with Neumann boundary conditions.

\begin{theorem}
\label{T:generalLB}
Let $\Omega\in \gD(\R^3)$ be a corner domain\index{Corner domain}, and let $\bA\in \sC^{2}(\overline{\Omega})$ be a magnetic potential.  Then there exist $C_\Omega>0$ and $h_{0}>0$ such that for all $h\in (0,h_{0})$:
\begin{equation}
\label{eq:below}
   \lambda_h(\bB \ee,\Omega) \geq  
   \begin{cases}
   h\ee \sE(\bB \ee,\Omega) - C_\Omega \big(1+ \|\bA\|^2_{W^{2,\infty}(\Omega)}\big) \, h^{11/10} ,
    &\mbox{$\Omega$ general corner domain,} 
   \\[0.5ex]
   h\ee \sE(\bB \ee,\Omega) - C_\Omega \big(1+ \|\bA\|^2_{W^{2,\infty}(\Omega)}\big) \, h^{5/4} ,
    &\mbox{$\Omega$ polyhedral domain.}
   \end{cases}
\end{equation}
We recall that the quantity $\sE(\bB,\Omega)$ is the lowest local energy\index{Lowest local energy} defined in \eqref{eq:sbis}.
\end{theorem}

\begin{remark}
If the magnetic field $\bB$ vanishes, then $\sE(\bB,\Omega)=0$ and Theorem~\ref{T:generalLB} is obvious. 
By contrast, if $\bB$ does not vanish on $\overline\Omega$, we will see in Corollary~\ref{cor.sE>0} that $\sE(\bB,\Omega)>0$.
\end{remark}

\subsubsection*{Structure of the proof}
The proof proceeds from an IMS partition argument coupled with the analysis of  remainders due to the cut-off effects, the local maps and the linearization of the potential. 
The less classical piece of the analysis is our special construction of cut-off functions in regions close to conical points $\bx_0\in\gVc$, where a second, smaller, scale is introduced.

We choose first an admissible atlas\index{Atlas!Admissible atlas} on $\overline\Omega$ according to Definition \ref{def.atlas} and we recall that the conical points are part of the set $\gX$ of its reference points.

\subsubsection*{Splitting off the conical points}
We start with a (smooth) macro partition of unity\index{Partition of unity} on $\overline\Omega$, independent of $h$, $(\Xi_0,(\Xi_{\bx})_{\bx\in\gVc})$ which aims at separating the conical points, {\it i.e.}, such that 
\begin{itemize}
\item $\overline{\supp\Xi_{0}}\cap \gVc =\emptyset$,
\item for any $\bx_{0}\in\gVc$, $\supp\Xi_{\bx_{0}}
\subset \cB(\bx_{0},R_{\bx_{0}})$.
\end{itemize}
Here $R_{\bx_0}$ is the radius associated with the reference point $\bx_0$ in the admissible atlas.
In the polyhedral case, {\it i.e.}, when $\gVc=\emptyset$, we simply set $\Xi_0\equiv1$.

For any $f\in H^1(\Omega)$ IMS formula\index{IMS formula} (see Lemma~\ref{lem:IMS}) gives 
\begin{align}
q_{h}[\bA,\Omega](f) 
&= q_{h}[\bA,\Omega](\Xi_{0}f) +\sum_{\bx\in\gVc} q_{h}[\bA,\Omega](\Xi_{\bx}f) -
h^2\Big(\| (\nabla\Xi_{0}) f\|^2+\sum_{\bx\in\gVc}\|(\nabla\Xi_{\bx}) f\|^2\Big)
\nonumber\\
&\geq q_{h}[\bA,\Omega](\Xi_{0}f) +\sum_{\bx_{0}\in\gVc} q_{h}[\bA,\Omega](\Xi_{\bx}f) -C h^2\|f\|^2.\label{eq.IMSmacro}
\end{align}
In Section~\ref{SS.minopoly}, we give a lower bound of $q_{h}[\bA,\Omega](\Xi_{0}f)$. In the polyhedral case, this will finish the proof.
Section~\ref{SS.minoconic} is devoted to conical points and estimates of $q_{h}[\bA,\Omega](\Xi_{\bx}f)$.

\section{Estimates outside conical points}\label{SS.minopoly}
Here we prove a lower bound for $q_{h}[\bA,\Omega](\Xi_{0}f)$. 

\subsubsection*{IMS localization}
Let $\delta\in(0,\frac12)$ be an exponent which will be determined later on. 
Now, we make a $h$-dependent partition of $\supp\Xi_{0}\cap\overline\Omega$ with size $h^\delta$. 
Relying on Lemma \ref{lem:IMScov}, we can choose for $0<h\le h_0$ ($h_0$ small enough) a finite set $\sC(h)$ of points $\bc\in\overline\Omega$ together with radii $\rho_\bc$ equivalent to $h^\delta$ (with uniformity as $h\to0$) such that
\begin{enumerate}
\item The union of balls $\cB(\bc,\rho_\bc)$ covers $\supp\Xi_{0}\cap\overline\Omega$
\item Each ball $\cB(\bc,2\rho_\bc)$ is contained in a map-neighborhood of the admissible atlas
\item The finite covering condition holds
\end{enumerate}
Relying on Lemma \ref{lem:IMSpart}, we choose an associate partition of unity\index{Partition of unity} $\big(\tronc\big)_{\bc\in\sC(h)}$ such that 
$$
   \tronc\in \sC^{\infty}_0(\cB(\bc,\rho_{\bc})),\quad\forall\bc\in\sC(h) 
   \qquad\mbox{and}\qquad
   \Xi_{0} \sum_{\bc\in\sC(h)} \tronc^2=\Xi_{0}\quad\mbox{on}\quad \overline{\Omega},  
$$
and satisfying the uniform estimate of gradients 
\begin{equation}
\label{E:controlegradtronc}
   \exists C>0,\quad \forall h\in (0,h_{0}), \quad \forall \bc\in\sC(h),\quad 
   \|\nabla\tronc\|_{L^{\infty}(\Omega)} \leq C h^{-\delta} \, . 
\end{equation}
The IMS formula\index{IMS formula} (see Lemma \ref{lem:IMS}) provides for all $f\in H^1(\Omega)$
\[
   q_{h}[\bA,\Omega](\Xi_{0}f) = \sum_{\bc\in\sC(h)} q_{h}[\bA,\Omega](\tronc\, \Xi_{0}f)
   - h^2 \sum_{\bc\in\sC(h)} \|\nabla\tronc\ \Xi_{0}f\|_{L^2(\Omega)}^2
\]
and using \eqref{E:controlegradtronc} we get $C=C(\Omega)>0$ such that
\begin{equation}
\label{E:minorationpartition}
   q_{h}[\bA,\Omega](\Xi_{0}f) \geq \sum_{\bc\in\sC(h)} q_{h}[\bA,\Omega](\tronc\ee \Xi_{0}f)
   -C\, h^{2-2\delta}\|\Xi_{0}f\|_{L^2(\Omega)}^2 \, .
\end{equation}

\subsubsection*{Local control of the energy}
For each center $\bc\in\sC(h)$, we are going to bound from below the term $q_{h}[\bA,\Omega](\tronc\, \Xi_{0}f)$ appearing in \eqref{E:minorationpartition}. By construction $\supp(\tronc\, \Xi_{0}f)$ is contained in the map-neighborhood $\cU_{\bc}$. Using \eqref{E:phase} and \eqref{E:psi}, we set
\begin{equation}
\label{E:psic}
   \psi_{\bc}:= (\zeta^{\bc}_h \,\tronc\, \Xi_{0}f)\circ (\diffeo^{\bc})^{-1},
   \qquad\mbox{ with } \qquad  
   \zeta^{\bc}_h(\bx) = \re^{i\langle \bA(\bc),\ee\bx\rangle/h}.\index{Phase shift}
\end{equation}
According to \eqref{E:chgGx0} with $\bx_{0}$ replaced by $\bc$, we have
\begin{equation}
\label{E:chgGc}
   q_{h}[\bA,\Omega](\tronc\ee \Xi_{0}f)
   = q_{h}[\bA^{\bc},\Pi_{\bc},\rG^{\bc}](\psi_{\bc}) 
   \quad \mbox{and} \quad 
   \| \tronc\, \Xi_{0}f\|_{L^2(\Omega)}=\| \psi_{\bc} \|_{L^2_{\rG^{\bc}}(\Pi_{\bc})}\,.
\end{equation}
In order to replace the metric $\rG^{\bc}$ by the identity, we apply Lemma~\ref{L:chgvarloc} with $r\simeq h^{\delta}$. Using that the distance $d_{\gVc}$ to conical points is bounded from below by a positive number on $\supp\Xi_0$, we obtain the existence of a constant $c(\Omega)>0$ such that for all centers $\bc\in\sC(h)$ 
\begin{equation}
\label{E:f-psi}
   \QR_{h}[\bA^{\bc},\Pi_{\bc},\rG^{\bc}](\psi_{\bc})
   \geq (1-c(\Omega)h^{\delta}) \QR_{h}[\bA^{\bc},\Pi_{\bc}](\psi_{\bc})\,.
\end{equation}
We now want to replace $\bA^{\bc}$ in the above Rayleigh quotient\index{Rayleigh quotient} by its linear part $\bA^{\bc}_{\bfz}$ at $\bfz$. For this we use identity \eqref{eq:diffAA'} with $\psi=\psi_{\bc}$ and $\cO=\Pi_{\bc}$:
\begin{multline}
\label{eq:diff}
  q_{h}[\bA^{\bc},{\Pi_{\bc}}](\psi_{\bc}) =   
  q_{h}[\bA^{\bc}_{\bfz},\Pi_{\bc}](\psi_{\bc}) \\
  +2\Re\big\langle (-ih\nabla+\bA^{\bc}_{\bfz})\psi_{\bc},(\bA^{\bc}-\bA^{\bc}_{\bfz})\psi_{\bc}\big\rangle 
  + \|(\bA^{\bc}-\bA^{\bc}_{\bfz})\psi_{\bc}\|^2.
\end{multline}
This yields
$
  q_{h}[\bA^{\bc},{\Pi_{\bc}}](\psi_{\bc}) \geq   
  q_{h}[\bA^{\bc}_{\bfz},\Pi_{\bc}](\psi_{\bc})
  -2  \left(q_{h}[\bA^{\bc}_{\bfz},\Pi_{\bc}](\psi_{\bc})\right)^{1/2}
  \|(\bA^{\bc}-\bA^{\bc}_{\bfz})\psi_{\bc}\|
$
by Cauchy-Schwarz inequality, leading to the parametric estimate (based on inequality $2ab\leq \eta a^2+\eta^{-1}b^2$)
\begin{equation}\label{eq.minoeta}
   \forall \eta>0, \quad 
   q_{h}[\bA^{\bc},{\Pi_{\bc}}](\psi_{\bc}) \geq  
   (1-\eta) q_{h}[\bA^{\bc}_{\bfz},\Pi_{\bc}](\psi_{\bc})
  -\eta^{-1}\|(\bA^{\bc}-\bA^{\bc}_{\bfz})\psi_{\bc}\|^2\,.
\end{equation}
Since $\curl \bA^{\bc}_{\bfz}=\bB_{\bc}$, we have the lower bound by the minimum local energy at $\bc$:
\begin{align}
\label{eq.hEnc}
   q_{h}[\bA^{\bc}_{\bfz},\Pi_{\bc}](\psi_{\bc}) 
   &\geq    h\ee\En(\bB_{\bc},\Pi_{\bc})\|\psi_{\bc}\|^2\,
   \\
   &\geq    h\ee\sE(\bB,\Omega)\|\psi_{\bc}\|^2\, .
\label{eq.hEn}
\end{align}
According to Lemmas~\ref{lem.TaylorA} and \ref{L:d2A} (note that $d_{\gVc}\ge r_0>0$ on $\supp\Xi_0$), we have
\begin{equation}
\label{eq.AmAc}
   \|(\bA^{\bc}-\bA^{\bc}_{\bfz})\psi_{\bc}\| \leq 
   c(\Omega) \| \bA \|_{W^{2,\infty}({\Omega})} h^{2\delta} \|\psi_{\bc}\|\ . 
\end{equation}
Combining \eqref{eq.minoeta}--\eqref{eq.AmAc} we deduce for all $\eta>0$:
$$
   q_{h}[\bA^{\bc},{\Pi_{\bc}}](\psi_{\bc}) \geq   
   (1-\eta)h\ee\sE(\bB,\Omega)\|\psi_{\bc}\|^2
  -\eta^{-1} h^{4\delta} c(\Omega)^2 \| \bA\|_{W^{2,\infty}({\Omega})}^2 \|\psi_{\bc}\|^2 .
$$
Choosing $\eta=h^{2\delta-\frac12}$ to equilibrate $\eta h$ and $\eta^{-1} h^{4\delta}$, we get the following lower bound
\begin{equation}
\label{eq:psijh}
  q_{h}[\bA^{\bc},{\Pi_{\bc}}](\psi_{\bc})  
    \geq \left(h\ee\sE(\bB,\Omega)-C_\Omega \big(1+\| \bA\|_{W^{2,\infty}({\Omega})}^2\big)
   h^{2\delta+\frac12}\right)\|\psi_{\bc}\|^2,\quad \forall\bc\in\sC(h).
\end{equation}
\subsubsection*{Conclusion}
Combining the previous localized estimate \eqref{eq:psijh} with \eqref{E:f-psi} we deduce:
\begin{equation}\label{eq.fjpolyh}
   q_{h}[\bA,\Omega](\tronc\,\Xi_{0}f)  \geq
   \left(h\sE(\bB,\Omega)-C_\Omega(1+\| \bA\|_{W^{2,\infty}({\Omega})}^2)
   (h^{2\delta+\frac12}+h^{1+\delta}) \right)  \|\tronc\,\Xi_{0}f\|^2.
\end{equation}
Summing up in $\bc\in\sC(h)$, we obtain
\begin{equation}
\label{E:minoration99}
   \frac{\sum_{\bc\in\sC(h)} q_{h}[\bA,\Omega](\tronc\,\Xi_{0}f)}{\|\Xi_{0}f\|^2_{L^2(\Omega)}}  \geq
   h\sE(\bB,\Omega)-C_\Omega(1+\| \bA\|_{W^{2,\infty}({\Omega})}^2)
   (h^{2\delta+\frac12}+h^{1+\delta}).
\end{equation}
Using \eqref{E:minorationpartition}, we get another constant $C_\Omega>0$ such that for all $f\in H^1(\Omega)$,
\begin{equation}\label{eq.minoO'}
   \QR_{h}[\bA,\Omega](\Xi_{0}f)  \geq 
   h\sE(\bB,\Omega)-C_\Omega(1+\| \bA\|_{W^{2,\infty}({\Omega})}^2)
   \big(h^{2\delta+\frac12}+h^{1+\delta}+h^{2-2\delta}\big).
\end{equation}
In the polyhedral case, $\Xi_0\equiv1$ and the remainders are optimized by taking $\delta=\frac{3}{8}$ in \eqref{eq.minoO'}, which implies Theorem \ref{T:generalLB} in this case.

\section{Estimates near conical points}\label{SS.minoconic}
Let $\bx_{0}\in\gVc$\index{Conical point}. We estimate $q_{h}[\bA,\Omega](\Xi_{\bx_{0}}f)$ from below. 

\subsubsection*{IMS partition}
For $h>0$ small enough we construct a special covering of the support of $\Xi_{\bx_{0}}$. We recall that this support is included in the ball $\cB(\bx_0,R_{\bx_0})$. We cover $\cB(\bx_0,R_{\bx_0})\cap\overline\Omega$ by a finite collection of $h$-dependent balls $\cB(\bc,\rho_\bc)$:
\begin{itemize}
\item The first ball is centered at $\bx_0$ itself and its radius is $2h^{\de0}$: $\cB(\bc,\rho_\bc) = \cB(\bx_0,2h^{\de0})$. Here the exponent $\de0\in(0,\frac12)$ will be chosen later on.
\item The other balls $\cB(\bc,\rho_\bc)$ cover the annular region $h^{\de0}\le|\bx-\bx_0|<R_{\bx_0}$ and their radii are $\simeq h^{\de0+\de1}$ where the new exponent $\de1>0$ is such that $\de0+\de1<\frac12$ and will be also chosen later on. Thanks to Lemma \ref{lem:IMScov} the set $\sC(h,\bx_0)$ of the centers and the corresponding radii can be taken so that the conditions of this lemma are satisfied (inclusion in map-neighborhoods, finite covering), see previous case Section\ref{SS.minopoly}.
\end{itemize}
So this covering contains a ``large'' ball centered at the corner and a whole bunch of smaller ones covering the remaining part.

Relying on Lemma \ref{lem:IMSpart}, we choose an associate partition of unity $\big(\tronc\big)_{\bc\in\{\bx_{0}\}\cup\sC(h,\bx_{0})}$ such that 
$$
   \tronc\in \sC^{\infty}_0(\cB(\bc,\rho_{\bc})), \ \ \forall\bc\in\{\bx_{0}\}\cup\sC(h,\bx_{0}),
   \qquad\mbox{and}\qquad
   \Xi_{\bx_0} \hskip-0.5em \sum_{\bc\in\{\bx_{0}\}\cup\sC(h,\bx_0)} \hskip-0.5em
   \tronc^2=\Xi_{\bx_0}\quad\mbox{on}\quad \overline{\Omega},  
$$
and satisfying the following uniform estimate of gradients for all $h\in (0,h_{0})$: 
\begin{equation}
\label{E:contgradtronc}
   \mbox{for}\ \ \bc=\bx_0,\ \  
   \|\nabla\tronc\|_{L^{\infty}(\Omega)} \leq C h^{-\de0} 
    \quad \mbox{and}\quad
    \forall \bc\in\sC(h,\bx_0),\ \  
   \|\nabla\tronc\|_{L^{\infty}(\Omega)} \leq C h^{-\de0-\de1} \, . 
\end{equation}
Using the IMS formula (see Lemma~\ref{lem:IMS}), we have like previously in \eqref{E:minorationpartition}
\begin{equation}
\label{E:minoIMS}
   q_{h}[\bA,\Omega](\Xi_{\bx_0}f) 
   \geq q_{h}[\bA,\Omega](\xi_{\bx_{0}}\ \Xi_{\bx_0}f)+\hskip-0.5em\sum_{\bc\in\sC(h,\bx_{0})}
   \hskip-0.5em q_{h}[\bA,\Omega](\tronc\ \Xi_{\bx_0}f) -
   C h^{2-2(\de0+\de1)} 
   \|\Xi_{\bx_0}f\|^2. 
\end{equation}

\subsubsection*{Local control of the energy}
When $\bc=\bx_0$, we can proceed in the same way as in the polyhedral case due to the ``good'' estimates stated in Lemma~\ref{L:chgvarloc}\,\eqref{It:cvloc} and Lemma \ref{L:d2A}\,\eqref{It:Ld2A}. So we obtain a similar estimate as in \eqref{eq.fjpolyh}: There exists a constant $C=C(\Omega)$ such that for any function $f\in H^1(\Omega)$ 
\begin{equation}\label{eq.minoJ1}
   q_{h}[\bA,\Omega](\xi_{\bx_0}\Xi_{\bx_0}f)  \geq
   \left(h\sE(\bB,\Omega)- C (1+\| \bA\|_{W^{2,\infty}({\Omega})}^2)
   (h^{2\de0+\frac12}+h^{1+\de0}) \right)  \|\xi_{\bx_0}\Xi_{\bx_0}f\|^2.
\end{equation}

When $\bc\in\sC(h,\bx_{0})$, we have to revisit the arguments leading from \eqref{E:psic} to the final individual estimate \eqref{eq.fjpolyh}. First we define $\psi_\bc$ like in \eqref{E:psic}, replacing the cut-off $\Xi_0$ by $\Xi_{\bx_0}$. Then we have \eqref{E:chgGc} \emph{mutatis mutandis}. Next we have to use Lemma \ref{L:chgvarloc}\,\eqref{It:cvlocbis} with $\bu_0=\bc$ to flatten the metric. Here we have to take the distance $d_{\gVc}(\bc)$ to conical points into account. By construction $d_{\gVc}(\bc)$ coincides with $|\bc-\bx_0|$, so is larger than $h^{\de0}$, while the quantity $r$ equals $\rho_\bc$, thus is $\lesssim h^{\de0+\de1}$: In short
\[
   \frac{r}{d_{\gVc}(\bc)} = \frac{\rho_\bc}{|\bc-\bx_0|} \lesssim h^{\de1}.
\]
Hence, we obtain in place of \eqref{E:f-psi}:
\begin{equation}
\label{E:f-psic}
   \QR_{h}[\bA^{\bc},\Pi_{\bc},\rG^{\bc}](\psi_{\bc}) 
   \geq (1-c(\Omega)h^{\de1}) \QR_{h}[\bA^{\bc},\Pi_{\bc}](\psi_{\bc})\,.
\end{equation}
For the linearization of the potential $\bA^\bc$, the expressions \eqref{eq:diff}--\eqref{eq.hEn} are still valid, leading to the parametric estimate
\begin{equation}\label{eq.minoetaEn}
   \forall \eta>0, \quad 
   q_{h}[\bA^{\bc},{\Pi_{\bc}}](\psi_{\bc}) \geq  
   (1-\eta) h\sE(\bB,\Omega)\|\psi_{\bc}\|^2
  -\eta^{-1}\|(\bA^{\bc}-\bA^{\bc}_{\bfz})\psi_{\bc}\|^2\,.
\end{equation}
Here we use Lemmas~\ref{lem.TaylorA} and \ref{L:d2A}\,\eqref{It:Ld2Abis} and obtain, since $\rho_\bc\lesssim h^{\de0+\de1}$ and $d_{\gVc}(\bc)\ge h^{\de0}$
\begin{equation}
\label{eq.AmAc2}
   \|(\bA^{\bc}-\bA^{\bc}_{\bfz})\psi_{\bc}\| \leq 
   c(\Omega)\frac{\rho_{\bc}^2}{d_{\gVc}(\bc)}
   \| \bA \|_{W^{2,\infty}({\Omega})} \|\psi_{\bc}\|
   \leq   c(\Omega) h^{\de0+2\de1}\| \bA \|_{W^{2,\infty}({\Omega})} \|\psi_{\bc}\|\,. 
\end{equation}
Combining \eqref{eq.minoetaEn} with \eqref{eq.AmAc2} and taking $\eta=h^{\de0+2\de1-\frac 12}$
we deduce
\begin{equation}
\label{eq:psijh2}
   q_{h}[\bA^{\bc},{\Pi_{\bc}}](\psi_{\bc}) 
   \geq \left(h\ee\sE(\bB,\Omega)-C(\Omega) \big(1+\| \bA\|_{W^{2,\infty}({\Omega})}^2\big)
   h^{\de0+2\de1+\frac12}\right)\|\psi_{\bc}\|^2 ,\quad \forall\bc\in\sC(h,\bx_0),
\end{equation}
and then with \eqref{E:f-psic} (and \eqref{E:chgGc} with $\Xi_{\bx_0}$)
\begin{equation}\label{eq.minocoin}
   q_{h}[\bA,\Omega](\xi_{\bc}\Xi_{\bx_0}f)  \geq
   \left(h\sE(\bB,\Omega)- C (1+\| \bA\|_{W^{2,\infty}({\Omega})}^2)
   (h^{\de0+2\de1+\frac12}+h^{1+\de1}) \right)  \|\xi_{\bc}\Xi_{\bx_0}f\|^2.\hskip-1.em
\end{equation}
Summing up \eqref{eq.minoJ1} and \eqref{eq.minocoin} for $\bc\in\sC(h,\bx_0)$, and combining with the IMS formula, we deduce
\begin{equation}\label{eq.minoconic}
\QR_{h}[\bA,\Omega](\Xi_{\bx_{0}}f)
\geq h \sE(\bB,\Omega)-C( h^{2\de0+\frac12}+h^{1+\delta_{0}}+h^{\frac 12+\de0+2\de1}+h^{1+\de1} +h^{2-2(\de0+\de1)}),
\end{equation}
with $C=c(\Omega)   (1+\| \bA\|_{W^{2,\infty}({\Omega})}^2)$.

\subsubsection*{Conclusion}
Combining \eqref{eq.IMSmacro}, \eqref{eq.minoO'} and \eqref{eq.minoconic}, we deduce
\begin{multline}\label{eq.miniglo}
\QR_{h}[\bA,\Omega](f)  
\geq h \sE(\bB,\Omega) - Ch^2
-C\left(h^{2\delta+\frac 12} +h^{1+\delta}+h^{2-2\delta}\right)\\
-C \left( h^{2\de0+\frac12}+h^{1+\delta_{0}}+h^{\frac 12+\de0+2\de1}+h^{1+\de1} +h^{2-2(\de0+\de1)}\right),
\end{multline}
with $C=c(\Omega)   (1+\| \bA\|_{W^{2,\infty}({\Omega})}^2)$. 

Remind that the error with power $\de0$ and $\de1$ only appears when $\Omega$ has conical points. 
To optimize the remainder, we first choose $\delta=3/8$. 
We have now to optimize parameters $\de0, \de1$ under the constraints $0< \de0+\de1<\frac 12$, $\de0>0$, $\de1>0$.
We have 
$$\min (1+\de0,\tfrac12+2\de0)=\tfrac12+2\de0,$$
and
$$\min (1+\de1,\tfrac12+\de0+2\de1)=\tfrac12+\de0+2\de1.$$
We are reduced to solve
$$\begin{cases}
\frac12+2\de0=\frac12+\de0+2\de1\\
\frac12+2\de0= 2-2\de0-2\de1
\end{cases}
\Longleftrightarrow\quad
\begin{cases}
2\de1 = \de0\\
\frac32=4\de0+2\de1
\end{cases}
\Longleftrightarrow \quad\de0=\frac 3{10} \ \mbox{and}\ \de1=\frac 3{20}.
$$
Then we get $C(\Omega)>0$ such that
\begin{equation}
\label{E:Minfinle}
   \forall f\in H^1(\Omega), \quad  
   \QR_{h}[\bA,\Omega](f) \geq 
   h\sE(\bB,\Omega)-C(\Omega)
   \big(1+\| \bA\|_{W^{2,\infty}({\Omega})}^2\big)h^{11/10}.
\end{equation}

For further use we extract the following corollary of the previous proof: 
\begin{corollary}
\label{C:Metalbloc}
Let $\bx_{0}\in \overline{\Omega}$ and $K:=\cB(\bx_{0},\delta)$ with $\delta>0$. We define 
$$\sE_{K}(\bB,\Omega):=\inf_{\bx\in \overline{\Omega}\cap \overline K}E(\bB_{\bx},\Pi_{\bx}) \, .$$
Then there exists $C>0$ and $h_{0}>0$ such that for all $h\in (0,h_{0})$ and for all $f\in \dom(q_{h}[\bA,\Omega])$ with support $\supp f \subset\subset K$:
$$ \QR_{h}[\bA,\Omega](f) \geq h\sE_{K}(\bB,\Omega)-C h^{11/10}\,.$$ 
\end{corollary}

\begin{proof}
The corollary is obtained by slight modifications in the above proof. First we make a covering of $\overline{\Omega}\cap K$ instead in $\overline{\Omega}$. Therefore in the lower bound \eqref{eq.hEnc}, we only have to consider $\bc\in K$, and the energy is bounded below by $\sE_{K}(\bB,\Omega)$ in \eqref{eq.hEn}. We finally reached \eqref{E:Minfinle} and deduce the Corollary. 
\end{proof}

\section{Generalization}
\label{ss:genlow}
For the proofs above, we used very little knowledge on the magnetic Lapla\-cians---essentially the change of gauge\index{Gauge transform}, the change of variables, and the perturbation identity \eqref{eq:diffAA'}. 
The finest part of the analysis is related to the corner structure. With the same approach and relying on the general estimates presented in Section \ref{sss:schain_atl}, we are able to establish lower bounds for the ground state energy of magnetic Laplacians in $n$-dimensional corner domains.

Let $\Omega\in\gD(\R^n)$, and let us introduce $\nu$ as the maximal integer such that there exists a singular chain $(\bx_0,\ldots,\bx_{\nu-1})$ of length\index{Singular chain!Length} $\nu$ with a non-polyhedral reduced cone $\Gamma_{\bx_0,\ldots,\bx_{\nu-1}}$. We make the convention that $\nu=0$ if all tangent cones are polyhedral.

Using an IMS partition on a hierarchy of balls of size $h^{\de0}$, $h^{\de0+\de1}$, \ldots, $h^{\de0+\de1+\ldots+\delta_\nu}$ according to the position of their centers, and taking advantage of estimates \eqref{eq:Kgen}, we arrive to the following collection of errors
\begin{align*}
& h^{1+\de0},\ h^{1+\de1},\ \ldots,\ h^{1+\delta_\nu} \\
& h^{\frac12+2\de0}, \ h^{\frac 12+\de0+2\de1},\ \ldots,\ 
h^{\frac 12+\de0+\ldots+\delta_{\nu-1}+2\delta_\nu} \\
&h^{2-2(\de0+\de1+\ldots+\delta_\nu)},
\end{align*}
which is optimized choosing
\[
   \delta_k = 2^{\nu-k}\delta_\nu,\ k=0,\ldots,\nu,\quad\mbox{with}\quad
   \delta_\nu = \frac{3}{3\cdot 2^{\nu+2}-4}\,.
\]
The outcome is the following lower bound
\[
   \lambda_{h}(\bB,\Omega) \geq
   h\sE(\bB,\Omega)-C(\Omega) \big(1+\|\bA\|_{W^{2,\infty}(\Omega)}^2 \big)
   h^{1+1/(3\cdot2^{\nu+1}-2)}\ . 
\]
Here $\sE(\bB,\Omega)$ is the natural generalization of \eqref{eq:sbis} to $n$-dimensional domains.
The results of Theorem \ref{T:generalLB} correspond to the values $\nu=1$ and $\nu=0$. Note that the remainder $\cO(h^{5/4})$ is valid in a polyhedral domain in any dimension ($\nu=0$).

\part{Upper bounds}
\label{part:3}

\chapter{Taxonomy of model problems}
\label{sec:tax}
\index{Taxonomy|textbf}
Refined estimates for an \emph{upper bound}
of the ground state energy $\lambda_{h}(\bB,\Omega)$ will be obtained with the help of quasimode\index{Quasimode} constructions. This relies on a better knowledge of tangent model problems\index{Tangent operator} $\OP(\bA_\dx,\Pi_\dx)$ for any singular chain $\dx$ of $\Omega$. In this section, we review and, when required, complete, essential facts concerning three-dimensional model problems, that is magnetic Laplacians $\OP(\bA,\Pi)$ where $\Pi$ is a cone in $\gP_3$ and $\bA$ is a linear potential. 

With the aim of constructing quasimodes for our original problem on $\Omega$, we need (bounded) generalized eigenvectors for its tangent problems. To introduce such eigenvectors we make use of the localized domain $\dom_{\,\loc} \left(\OP(\bA,\Pi)\right)$ of the model magnetic Laplacian $\OP(\bA,\Pi)$ as introduced in \eqref{D:domloc}:

\begin{definition}[Generalized eigenvector]
\label{def:geneig}
Let $\Pi\in\gP_3$ be a cone and $\bA$ a linear magnetic potential. We call {\em generalized eigenvector}\index{Generalized eigenvector} for $\OP(\bA\ee,\Pi)$ a nonzero function $\Psi\in\dom_{\,\loc}(\OP(\bA\ee,\Pi))$ associated with a real number $\Lambda$, so that
\begin{equation}
\label{eq:geneig}
\begin{cases}
(-i\nabla+\bA)^2\Psi=\Lambda\Psi &\mbox{in } \Pi,\\
(-i\nabla+\bA)\Psi\cdot\bn =0 &\mbox{on } \partial\Pi.
\end{cases}
\end{equation}
\end{definition}

Let $\Pi\in\gP_3$ be a 3D cone and let $\bB$ be a constant magnetic field associated with a linear potential $\bA$. Let $d$ be the reduced dimension of $\Pi$ and $\Gamma\in\gP_d$ be a minimal reduced cone\index{Minimal reduced cone} associated with $\Pi$. We recall from Definition \ref{def:redcone} that this means that $\Pi\equiv\R^{3-d}\times\Gamma$ and that the dimension $d$ is minimal for such an equivalence. By analogy with Definition~\ref{def.Cx}, $\gC_\bfz(\Pi)$ denotes the set of singular chains of $\Pi$ originating at its vertex $\bfz$ and $\gC^*_\bfz(\Pi)$ is the subset of chains of length $\ge2$. Note that $\gC^*_\bfz(\Pi)$ is empty if and only if $\Pi=\R^3$, {\it i.e.}, if $d=0$. We introduce the 
energy on tangent substructures:

\begin{definition}[Energy on tangent substructures]\label{def.seE}\index{Energy on tangent substructures}
We define the quantity\index{Energy on tangent substructures!$\seE(\bB \ee,\Pi)$}
\begin{equation}
\label{eq:s*}
 \seE(\bB \ee,\Pi):=
 \begin{cases}
  \inf_{\dx\in\gC^{*}_0(\Pi)}\En(\bB \ee,\Pi_{\dx})  & \mbox{if $d>0$,} \\
  +\infty  & \mbox{if $d=0$,} \\
\end{cases}
\end{equation}
which is the infimum of the ground state energy of the magnetic Laplacian over all the singular chains of length $\geq2$.
\end{definition} 

We will see later in Chapter \ref{sec:dicho} that this quantity plays a key role 
in the existence of generalized eigenvectors that have exponential decay properties in certain directions. 

Now, in each of Sections \ref{subs:R3}--\ref{subs:coin} we consider one value of the reduced dimension\index{Reduced dimension} $d$, ranging from $0$ to $3$ and give in each case relations between the ground state energy $\En(\bB,\Pi)$ and the energy on tangent substructures $\seE(\bB,\Pi)$, and we provide generalized eigenvectors $\Psi$ if they exist. 

On the one hand, thanks to Lemma~\ref{lem.dilatation}, we may reduce the arguments  to the case of a magnetic field of unit length: $|\bB|=1$.
On the other hand, quantities $\En(\bB,\Pi)$ and $\seE(\bB,\Pi)$ are independent of a choice of Cartesian coordinates. 
Thus, once $\Pi$ and a constant magnetic field $\bB$ of unit length are chosen, we exhibit a system of Cartesian coordinates $\bx=(x_1,x_2,x_3)$ that allows the simplest possible description of the configuration $(\bB,\Pi)$. In these coordinates, the magnetic field can be viewed as a reference field, and for convenience, we denote it by $\uB=(b_{0},b_{1},b_{2})$. We also choose a corresponding reference linear potential $\uA$, since we have gauge independence by virtue of Lemma \ref{lem:gauge}.

\section{Full space \ ($d=0$) }
\label{subs:R3}
$\Pi$ is the full space. We take coordinates $\bx=(x_1,x_2,x_3)$ so that
\[
   \Pi=\R^3 \quad\mbox{and}\quad \uB = (1,0,0),
\]
and choose as reference potential
$
   \uA = (0,-\tfrac{x_{3}}{2},\tfrac{x_{2}}{2})
$.
It is classical (see \cite{LauLifIII}) that the spectrum of $ \OP(\uA \ee,\R^3)$ is $[1,+\infty)$. Therefore 
\begin{equation}
\label{E:spectrespace}
\En(\uB \ee,\R^3) = 1 \ .
\end{equation}
A generalized eigenvector associated with the ground state energy is 
\begin{equation}\label{eq:aged0}
\Psi(\bx)= \re^{-(x_{2}^2+x_{3}^2)/4} \quad\mbox{with}\quad \Lambda=1. 
\end{equation}

\section{Half-space \ ($d=1$)} \index{Half-space|textbf}
\label{SS:HS}
$\Pi$ is a half-space. We take coordinates $\bx=(x_1,x_2,x_3)$ so that
\[
   \Pi = \R^2\times\R_+ := \{(x_1,x_2,x_3)\in\R^3,\ x_{3}>0\} 
   \quad\mbox{and}\quad
   \uB=(0,b_{1},b_{2}) \ \mbox{with}\ b_1^2+b_2^2=1 \, ,
\]
and choose as reference potential
$
   \uA = (b_{1}x_{3}-b_{2}x_{2},0,0)
$.
We note that 
\begin{equation}
\label{eq:En1}
  \seE(\uB \ee,\R^2\times\R_+) = \En(\uB,\R^3)=1.
\end{equation}
There exists $\theta\in[0,2\pi)$ such that $b_1=\cos\theta$ and $b_2=\sin\theta$. Due to symmetries we can reduce to $\theta\in[0,\frac{\pi}{2}]$.
Denote by $\cF_{1}$ the Fourier transform in $x_{1}$-variable and by $\tau$ the dual variable. We have:  
$$
   \cF_{1} \ H(\uA,\R^2\times\R_+)\ \cF_{1}^*=
   \int^{\bigoplus}_{\tau\in\R} \widehat \OP_{\tau}(\uA \ee,\R^2\times\R_+)\,\rd \tau . 
$$
where 
$ 
   \widehat \OP_{\tau}(\uA \ee,\R^2\times\R_+) 
   =  (\tau+b_{1}x_{3}-b_{2}x_{2})^2 - \partial_2^2  - \partial_3^2 .
$ 
We discriminate three cases:

\subsection{Tangent field}
$\theta=0$, then $\widehat \OP_{\tau}(\uA\ee,\R^2\times\R_+) = (\tau + x_3)^2 - \partial_2^2  - \partial_3^2$. Let $\xi$ be the partial Fourier variable associated with $x_{2}$. Define the operators 
$
   \widehat \OP_{\xi,\tau}(\uA\ee,\R^2\times\R_+) =  (\tau + x_{3})^2 + \xi^2 - \partial_{3}^2$
   and
   $\DG\tau = \rD_{3}^2 + (\tau + x_{3})^2\,,
$
where $\DG\tau$ \index{de Gennes operator}(sometimes called the de Gennes operator) acts on $L^2(\R_{+})$ with Neumann boundary conditions. Its first eigenvalue is denoted by $\mu(\tau)$, moreover 
\[
   \inf\gS(\widehat \OP_{\tau,\xi}(\uA\ee,\R^2\times\R_+)) = \mu(\tau)+\xi^2.
\]
From \cite{DauHe93}) we know that $\mu$ admits a unique minimum denoted by $\Theta_0\simeq0.59$ for the value $\tau_0=-\sqrt{\Theta_0}$.
Hence
\begin{equation}
\label{DicoHS1}
   \En(\uB \ee,\R^2\times\R_+) = \Theta_0 < \seE(\uB \ee,\R^2\times\R_+).
\end{equation}
If $\Phi$ denotes an eigenvector of $\DG{\tau_0}$, the corresponding generalized eigenvector for $\OP(\uA,\Pi)$ is
\begin{equation}
\label{eq:d1.1}
   \Psi(\bx) = \re^{-\ri\ee\sqrt{\Theta_0}\, x_{1}} \,\Phi(x_{3})
   \quad\mbox{with}\quad \Lambda=\Theta_0.
\end{equation}

\subsection{Normal field}
$\theta=\frac\pi2$, then $\widehat \OP_{\tau}(\uA,\R^2\times\R_+) = (\tau - x_2)^2 -\partial_2^2 - \partial_3^2$. There holds
for all $\tau\in \R$, $\inf\gS(\widehat \OP_{\tau}(\uA,\R^2\times\R_+))=1$ (see \cite[Theorem 3.1]{LuPan00}), hence
\begin{equation}
\label{DicoHS2}
   \En(\uB \ee,\R^2\times\R_+) = 1 = \seE(\uB \ee,\R^2\times\R_+).
\end{equation}

\subsection{Neither tangent nor normal}
$\theta\in(0,\frac\pi2)$. Then for any $\tau\in\R$, $\widehat \OP_{\tau}(\uA,\R^2\times\R_+) $ is isospectral to $\widehat \OP_{0}(\uA,\R^2\times\R_+)$ the ground state energy of which is an eigenvalue $\sigma(\theta)<1$, cf.\ \cite{HeMo02}. We deduce
\begin{equation}
\label{DicoHS3}
   \En(\uB,\R^2\times\R_+) = \sigma(\theta) <1 = \seE(\uB \ee,\R^2\times\R_+)..
\end{equation}
This eigenvalue $\sigma(\theta)$ is associated with an exponentially decreasing eigenvector $\Phi$ that is a function of $(x_{2},x_{3})\in\R\times\R_+$. The corresponding generalized eigenvector for $\OP(\uA,\Pi)$ is
\begin{equation}
\label{eq:d1.3}
   \Psi(\bx) = \Phi(x_{2},x_{3}) \quad\mbox{with}\quad \Lambda=\sigma(\theta).
\end{equation}
We recall from the literature:
\begin{lemma}
\label{P:continuitesigma}
The function $\theta\mapsto \sigma(\theta)$ is continuous and increasing on $(0,\frac{\pi}{2})$ (\cite{HeMo02,LuPan00}). Set $\sigma(0)=\Theta_0$ and $\sigma(\frac{\pi}{2})=1$. Then the function $\theta\mapsto \sigma(\theta)$ is of class $\sC^1$ on $[0,\frac{\pi}{2}]$ (\cite{BoDauPopRay12}).
\end{lemma}

\section{Wedges \ ($d=2$)} \index{Wedge|textbf} \label{subs:diedre}
$\Pi$ is a wedge and let $\alpha\in(0,\pi)\cup(\pi,2\pi)$ denote its opening. 
Let us introduce the model sector\index{Sector} $\cS_{\alpha}$ and the model wedge $\cW_{\alpha}$
\begin{equation}
\label{eq:Wa}
   \cS_{\alpha} = 
   \begin{cases}
   \{x=(x_2,x_3),\ x_2\tan\tfrac\alpha2>|x_3|\big\}   & \mbox{if $\alpha\in(0,\pi)$} \\
   \{x=(x_2,x_3),\ x_2\tan\tfrac\alpha2>-|x_3|\big\}  & \mbox{if $\alpha\in(\pi,2\pi)$}
\end{cases} 
   \quad\mbox{and}\quad
   \cW_\alpha = \R\times\cS_\alpha \,.
\end{equation}
We take coordinates $\bx=(x_1,x_2,x_3)$ so that
\[
   \Pi=\cW_\alpha \quad\mbox{and}\quad
   \uB=(b_{0},b_{1},b_{2}) \ \mbox{with}\ b_{0}^2+b_1^2+b_2^2=1 \, ,
\]
and choose as reference potential
$
   \uA = (b_{1}x_{3}-b_{2}x_{2},0,b_{0}x_{2})\,.
$
The singular chains of $\gC^*_\bfz(\cW_\alpha)$ have three equivalence classes, cf.\ Definition \ref{def:chaineq} and Description \ref{description}~(3): The full space $\R^3$ and the two half-spaces $\Pi^\pm_\alpha$ corresponding to the two faces $\partial^\pm\cW_\alpha$ of $\cW_\alpha$. 
Thus
\[
   \seE(\uB,\cW_\alpha) = \min\{\En(\uB,\R^3) ,\, \En(\uB,\Pi^+_\alpha), \,\En(\uB,\Pi^-_\alpha)
   \}.
\]
Let $\theta^{\pm}\in[0,\frac{\pi}{2}]$ be the angle between $\uB$ and the face $\partial\Pi^\pm_{\alpha}$. We have, cf.\ Lemma \ref{P:continuitesigma},
\begin{equation}
\label{eq:s*Da}
   \seE(\uB,\cW_\alpha) = \min\{ 1,\, \sigma(\theta^+), \sigma(\theta^-)\} =
   \sigma(\min\{\theta^+, \,\theta^-\}).
\end{equation}
With $\tau$ the dual variable of $x_1$ and
\begin{equation}
\label{D:Hhatsector}
   \widehat \OP_{\tau}(\uA \ee,\cW_\alpha) 
   =   (\tau+b_{1}x_{3}-b_{2}x_{2})^2 - \partial_2 ^2 + (-i\partial_3 + b_0x_2 )^2
\end{equation}
we have 
$$
   \cF_{1}\ H(\uA,\cW_{\alpha})\ \cF_{1}^*=
   \int^{\bigoplus}_{\tau\in\R} \widehat \OP_{\tau}(\uA \ee,\cW_\alpha) \,\rd \tau \,.
$$
Thus 
\begin{equation}
\label{E:relslambda}
 \En(\uB \ee,\cW_{\alpha}) =\inf_{\tau\in\R} s(\uB,\cS_{\alpha};\tau) 
 \quad\mbox{with}\quad
  s(\uB,\cS_{\alpha};\tau) := \inf\gS( \widehat \OP_{\tau}(\uA \ee,\cW_\alpha) )\,.
\end{equation}
We quote from \cite[Theorem 3.5]{Pop13}:

\begin{lemma}
\label{lem:pop}
Let $\alpha\in(0,\pi)\cup(\pi,2\pi)$. 
There holds the inequality 
\begin{equation}
\label{DichoW}
\En(\uB,\cW_{\alpha})\leq\seE(\uB,\cW_{\alpha}) .
\end{equation}
Moreover, if $\En(\uB,\cW_{\alpha})<\seE(\uB,\cW_{\alpha})$, then the function $\tau\mapsto s(\uB,\cS_{\alpha};\tau)$ reaches its infimum. Let $\tau^{*}$ be a minimizer. Then $\En(\uB,\cW_{\alpha})$ is the first eigenvalue of the operator $\widehat \OP_{\tau^{*}}(\uA \ee,\cW_\alpha) $ and any associated eigenfunction $\Phi$ has exponential decay.
The function 
\begin{equation}\label{eq:aged2}
   \Psi(\bx) = \re^{\ri\tau^*x_{1}}\Phi(x_{2},x_{3})
\end{equation}
is a generalized eigenvector for the operator $\OP(\uA,\cW_{\alpha})$ associated with $\Lambda=\En(\uB,\cW_{\alpha})$. 
\end{lemma}

Finally, let us quote now the continuity result on wedges from \cite[Theorem 4.5]{Pop13}:
\begin{lemma}
\label{lem:contwedge}
The function $(\bB,\alpha)\mapsto \En(\bB,\cW_{\alpha})$ is continuous on $\dS^2\times ((0,\pi)\cup(\pi,2\pi))$. 
\end{lemma}

\section{3D cones \ ($d=3$)} \index{Cone} \label{subs:coin}
Denote by $\lambda_\ess(\bB,\Pi)$ the bottom of the essential spectrum  of $\OP(\bA,\Pi)$.
\begin{theorem}
\label{th:cone-ess}
Let $\Pi\in\gP_3$ be a cone with $d=3$, which means that $\Pi$ is not a wedge, nor a half-space, nor the full space. Let $\bB$ be a constant magnetic field. With the quantity $\seE(\bB,\Pi)$ introduced in \eqref{eq:s*}, we have 
\[
   \lambda_\ess(\bB,\Pi) = \seE(\bB,\Pi)\,.
\]
\end{theorem}

Recall Persson's Lemma\index{Persson's Lemma} \cite{Pers60} that gives a characterization of the bottom of the essential spectrum:
\begin{lemma}
\label{L:Persson}
Let $\Pi\in\gP_3$ and let $\bA$ be a linear magnetic potential associated with $\bB$. For $R>0$, we define
$ \dom_{0}^{R} (q[\bA,\Pi])$ as the subspace of functions $\Psi$ in $\dom (q[\bA,\Pi])$ with compact support, and $\supp\Psi\cap\cB(\bfz,R)=\emptyset$.
Then we have \index{Essential spectrum}
$$
  \lambda_\ess(\bB,\Pi)=\lim_{R\to +\infty}\left(
  \inf_{\Psi\,\in \, \dom_{0}^R (q[\bA,\Pi]) \,\setminus\, \{0\}
  }\QR[\bA,\Pi](\Psi)\right) \, .
$$
\end{lemma}

Before proving Theorem~\ref{th:cone-ess}, we show 
\begin{lemma}
\label{lem:chainL2}
Let $\Pi\in\gP_3$ be a cone with $d=3$, let $\Omega_{\bfz}=\Pi\cap \dS^2$ be its section. Then $\seE(\bB,\Pi)$ 
coincides with the infimum of the local energy over singular chains of length 2: 
\begin{equation}
\label{E:infClg2}
\seE(\bB,\Pi)=\inf_{\bx_{1}\in \overline\Omega_{\bfz}} \En(\bB,\Pi_{\bfz,\bx_{1}}) \, .
\end{equation}
\end{lemma}
\begin{proof}
For all singular chains $\dx$ and $\dx'$ in $\gC^{*}(\Pi)$ such that $\dx \leq \dx'$ , we have $E(\Pi_{\dx},\bB) \leq E(\Pi_{\dx'},\bB)$ as a consequence of \eqref{DicoHS1}, \eqref{DicoHS2}, \eqref{DicoHS3}, and \eqref{DichoW}. Hence \eqref{E:infClg2}. 
\end{proof}

\begin{proof}[Proof of Theorem \Ref{th:cone-ess}]
Combining Lemmas \ref{L:Persson} and \ref{lem.dilatation}, we get that 
\begin{equation}
\label{E:Perssonbis}
\lambda_\ess(\bB,\Pi)
=\lim_{h\to0}\left(h^{-1}
\inf_{ \Psi\,\in\,  \dom_{0}^{1} (q_{h}[\bA,\Pi]) \, \setminus\, \{0\}}
\QR_{h}[\bA,\Pi](\Psi) \right).
 \end{equation}

\paragraph*{Upper bound for $\lambda_\ess(\bB,\Pi)$}
Let $\varepsilon>0$.  By Lemma \ref{lem:chainL2} there exist $\bx\in\overline\Omega_{\bfz}$ and an associated chain $\dx=(\bfz,\bx)$ of length 2 such that 
\begin{equation}
\label{E:trouveunesuiteminimisbis}
\En(\bB,\Pi_{\dx}) < \seE(\bB,\Pi)+\varepsilon \ .
\end{equation}
Let $\bx':=2\bx$. Notice that the tangent cone to $\Pi$ at $\bx'$ is $\Pi_{\bx'}=\Pi_{\dx}$ and therefore $\En(\bB,\Pi_{\bx'})=\En(\bB,\Pi_{\dx})$. We use Lemma \ref{L:ConsQMEps} (that clearly applies even though $\Pi$ is unbounded): So there exists $h_{0}>0$ such that for all $h\in (0,h_{0})$ we can find $f_{h}$ normalized and supported near $\bx'$ satisfying $h^{-1}\QR_{h}[\bA,\Pi](f_{h}) \leq \En(\bB,\Pi_{\dx})+\varepsilon$. Since $|\bx'|=2$, we may assume without restriction that $\supp(f_{h}) \cap \cB(0,1)=\emptyset$. Combining this with \eqref{E:trouveunesuiteminimisbis} we get
$$\frac{1}{h}\QR_{h}[\bA,\Pi](f_{h}) \leq \seE(\bB,\Pi)+2\varepsilon \, ,$$
and therefore deduce from \eqref{E:Perssonbis} the upper bound of $\lambda_\ess(\bB,\Pi)$ by $\seE(\bB,\Pi)$.

\paragraph*{Lower bound for $\lambda_\ess(\bB,\Pi)$}
Notice that for all $\bx\in \overline{\Pi}\setminus\cB(0,1)$, we have $\Pi_{\bx}=\Pi_{\dx}$ where $\dx=(\bfz,\bx/|\bx|)$. Therefore (see \eqref{E:infClg2}):
$$\inf_{\bx\in \overline{\Pi}\setminus\cB(0,1)}\En(\bB,\Pi_{\bx})=\seE(\bB,\Pi) \, .$$
Then we easily deduce the lower bound from Corollary \ref{C:Metalbloc} and \eqref{E:Perssonbis}.
\end{proof}

\begin{corollary}
\label{C:expodecay}
Let $\Pi\in\gP_3$ be a cone with $d=3$. Assume that $\En(\bB,\Pi) < \seE(\bB,\Pi)$. Then any eigenfunction $\Psi$ of $\OP(\bA,\Pi)$ associated with the lowest eigenvalue $\En(\bB,\Pi)$, satisfies the following exponential decay estimates\index{Exponential decay}:
$$\forall c<\sqrt{\seE(\bB,\Pi)-\En(\bB,\Pi)}, \quad\exists C>0, \quad \|e^{c|\bx|}\Psi\| \leq C \|\Psi\|.$$
\end{corollary}

\begin{proof}
Recall that Theorem \ref{th:cone-ess} states that the bottom of the essential spectrum is $\seE(\bB,\Pi)$. Therefore we are in the standard framework for the techniques {\it \`a la} Agmon, see \cite{Ag85}, and also \cite[Section 7]{Bo05} for its application on plane sectors.
\end{proof}

\chapter{Dichotomy and substructures for model problems}
\label{sec:dicho}

Relying on the exhaustive description of model problems provided above, we arrive to one of the main results, the ``dichotomy'' Theorem \ref{th:dicho}
that states the existence of a generalized eigenvector (called {\it admissible}) living on a tangent structure of a cone $\Pi\in\gP_3$ and associated with the ground state energy. 
 In this section, the local energies $\En(\bB,\Pi_{\dx})$ related to singular chains $\dx\in\gC_{\bfz}(\Pi)$, play for the first time a major role in the analysis.

\section{Admissible Generalized Eigenvectors}
\begin{definition}[Admissible Generalized Eigenvector]\index{Admissible Generalized Eigenvector|textbf}\index{Admissible Generalized Eigenvector!AGE|textbf}
\label{D:Age}
Let $\Pi\in \gP_{3}$ be a cone. Recall that $d(\Pi)\in[0,3]$ is the dimension of its minimal reduced cone\index{Minimal reduced cone}. Let $\bA$ be a linear magnetic potential. A generalized eigenvector $\Psi$ for $\OP(\bA\ee,\Pi)$ (cf.\ Definition \ref{def:geneig}) is said to be {\em admissible} if there exist an integer $k\geq d(\Pi)$ and a rotation $\udiffeo:\bx\mapsto(\by,\bz)$ that maps $\Pi$ onto the product $\R^{3-k}\times\Upsilon$ with $\Upsilon$ a cone in $\gP_k$, and such that
\begin{equation}
\label{eq:age1}
   \Psi\circ \udiffeo^{-1} (\by,\bz) = \re^{\ri\ee\vartheta(\by,\bz)}\,\Phi(\bz)\
    \qquad \forall\by\in \R^{3-k},\ \ \forall\bz\in\Upsilon,
\end{equation}
with some real polynomial function $\vartheta$ of degree $\le2$ and some exponentially decreasing function $\Phi$, namely there exist positive constants $c_\Psi$ and $C_\Psi$ such that\index{Exponential decay}
\begin{equation}
\label{eq:agmongeneig}
\|\re^{c_\Phi|\bz|}\Phi\|_{L^2(\Upsilon)}\leq C_\Phi \|\Phi \|_{L^2(\Upsilon)}  \, .
\end{equation}
``Admissible Generalized Eigenvector'' will be shortened as AGE.
\end{definition}

The following lemma will be used for going from any tangent operator to one of the reference situations described in Chapter \ref{sec:tax}. Its proof is straightforward and relies on Lemmas \ref{lem:gauge}, \ref{L:chgvar}, \ref{lem.dilatation}, and \ref{lem:sense}.
\begin{lemma}
\label{lem:refop}
Let $\Pi\in \gP_{3}$ be a cone and $\bA$ be a linear potential. Assume that $\Psi$ is an AGE for $H(\bA,\Pi)$ associated with the energy $\En(\bB,\Pi)$, of the form \eqref{eq:age1}.
\begin{itemize}
\item[a1)]
For all $b>0$, the function
$$\Psi_{b}:\bx\mapsto \Psi(\frac{\bx}{\sqrt{b}}),$$
is an AGE for $H(b^{-1}\bA,\Pi)$ associated with the energy $\En(b^{-1}\bB,\Pi)=b^{-1}\En(\bB,\Pi)$. This AGE has the form \eqref{eq:age1} with $\udiffeo_{b}=\udiffeo$, $\vartheta_{b}(\by,\bz)=\vartheta(b^{-1/2}\by,b^{-1/2}\bz)$ and $\Phi_{b}(\bz)=\Phi(b^{-1/2}\bz)$.
 \item[a2)]The function
$$\Psi_{-}:\bx\mapsto \overline{\Psi(\bx)},$$
is an AGE for $H(-\bA,\Pi)$ associated with the energy $\En(-\bB,\Pi)=\En(\bB,\Pi)$. This AGE has the form \eqref{eq:age1}, with $\udiffeo^{-}=\udiffeo$, $\vartheta^{-}(\by,\bz)=-\vartheta(\by,\bz)$ and $\Phi^{-}(\bz)=\overline{\Phi(\bz)}$.
\item[b)] Let $\bA'$ be another linear potential such that $\curl \bA'=\curl\bA$. Then there exists a polynomial $\phi$ of degree $\leq2$ such that $\bA'=\bA+\nabla\phi$. The function\index{Gauge transform}\index{Shift function}
$$\Psi':\bx \mapsto e^{-i\phi(\bx)}\Psi(\bx) ,$$
is an AGE for $H(\bA',\Pi)$ associated with $\En(\bB,\Pi)$. This AGE has the form \eqref{eq:age1}, with $\udiffeo'=\udiffeo$, $\vartheta'=\vartheta-\phi\circ\udiffeo^{-1}$ and $\Phi'=\Phi$.
\item[c)] Let $\rJ\in\gO_3$ be a rotation, $\Pi_{\rJ}:=\rJ(\Pi)$ and $\bA_{\rJ}:=\rJ\circ \bA\circ \rJ^{-1}$. Introduce the constant magnetic field $\bB_{\rJ}=\rJ(\bB)$, so that $\curl \bA_{\rJ}=\bB_{\rJ}$. Then 
$$\Psi_{\rJ}:\bx\mapsto \Psi\circ \rJ^{-1}(\bx)$$
is an AGE for $H(\bA_{\rJ},\Pi_{\rJ})$ associated with $\En(\bB_{\rJ},\Pi_{\rJ})=\En(\bB,\Pi)$. It has the form \eqref{eq:age1}, with $\udiffeo_{\rJ}=\udiffeo\circ \rJ^{-1}$, $\vartheta_{\rJ}=\vartheta$ and $\Phi_{\rJ}=\Phi$. 
\end{itemize}
\end{lemma}

\section{Dichotomy Theorem}
\begin{theorem}[Dichotomy Theorem]\index{Dichotomy|textbf}
\label{th:dicho}
Let $\Pi\in\gP_3$ be a cone and $\bB\neq0$ be a constant magnetic field.
Let $\bA$ be any associated linear magnetic potential.\index{Linearized magnetic potential}
Recall that $\En(\bB,\Pi)$ is the ground state energy\index{Ground state energy} of $\OP(\bA,\Pi)$ and $\seE(\bB,\Pi)$ is the energy on tangent substructures\index{Energy on tangent substructures}, see Definition \Ref{def.seE}. Then,
\begin{equation}
\label{eq:comp}
   \En(\bB,\Pi) \leq \seE(\bB,\Pi), \ 
\end{equation}
and we have the dichotomy:
\begin{enumerate}[(i)]
\item If $\En(\bB,\Pi)<\seE(\bB,\Pi)$, then $\OP(\bA,\Pi)$ admits an Admissible Generalized Eigenvector associated with the value $\En(\bB,\Pi)$.\smallskip
\item If $\En(\bB,\Pi)=\seE(\bB,\Pi)$, then there exists a singular chain $\dx\in \gC^{*}_{\bfz}(\Pi)$ such that 
$$
   \En(\bB,\Pi_\dx) = \En(\bB,\Pi) \quad\mbox{and}\quad
   \En(\bB,\Pi_\dx)<\seE(\bB,\Pi_\dx).
$$
\end{enumerate}
\end{theorem}

\begin{remark}\label{rem:chaine}
In the case (ii), we note that by statement (i) applied to the cone $\Pi_\dx$, $\OP(\bA,\Pi_\dx)$ admits an AGE associated with the value $\En(\bB,\Pi)$.
\end{remark}

\begin{remark}\label{rem:Bnul}
If $\bB=0$, there is no magnetic field and $E(\Pi,\bB)=0$. An associated AGE is the constant function $\Psi\equiv1$.
\end{remark}

\begin{proof}[Proof of Theorem \Ref{th:dicho}.]
The proof relies on an exhaustion of cases based on Chapter \ref{sec:tax} combined with a hierarchical classification of model problems on tangent structures of a cone $\Pi$.

\subsubsection*{Geometrical invariance}
 Thanks to Lemma~\ref{lem:refop}, we may assume that $\bB$ is of unit length, choose any suitable Cartesian coordinates and any suitable linear potential. Hence, to prove the theorem, we may reduce to the reference configurations investigated in  Sections \ref{subs:R3}--\ref{subs:diedre}.

\subsubsection*{Algorithm of the proof}
 We first establish the theorem when $d=0$, then we apply the following analysis for increasing values of $d=d(\Pi)$ from 1 to 3:
\begin{enumerate}
\item Check inequality \eqref{eq:comp}.
\item Check assertion (i). 
\item Prove that there exists  a singular chain $\dx\in \gC^{*}_{\bfz}(\Pi)$ such that $\seE(\bB,\Pi)=\En(\bB,\Pi_{\dx})$. Since $d(\Pi_{\dx})<d$, assertion (ii) will be a consequence of the analysis made for lower dimensions.
\end{enumerate}
This procedure applied to reference problems described in Chapter \ref{sec:tax} will provide the theorem.

\subsubsection*{$d=0$}Here $\Pi=\R^3$, see Section \ref{subs:R3}. We have $\En(\bB,\R^3)=1$ and $\seE(\bB,\R^3)=+\infty$, moreover there always exists an admissible generalized eigenvector associated with 1, see \eqref{eq:aged0}. Theorem \ref{th:dicho} is proved for $d=0$. 

\subsubsection*{$d=1$} The model cone is $\R^2\times\R_+$, see Section \ref{SS:HS}. Inequality \eqref{eq:comp} has already been proved, see \eqref{DicoHS1}, \eqref{DicoHS2}, \eqref{DicoHS3}. We also know that $\En(\bB,\R^2\times\R_+)<\seE(\bB,\R^2\times\R_+)$ if and only if $\bB$ is not normal to the boundary. In this case, AGE have already been written, 
see \eqref{eq:d1.1} and \eqref{eq:d1.3}, so point (i) of Theorem \ref{th:dicho} holds in the non-normal case. When $\bB$ is normal, $\En(\bB,\R^2\times\R_+)=\seE(\bB,\R^2\times\R_+)$. The sole tangent substructure is $\R^3$ and we have $\seE(\bB,\R^2\times\R_+)=\En(\bB,\R^3)<\seE(\bB,\R^3)$ (see the above paragraph $d=0$). Therefore Theorem \ref{th:dicho} is proved for $d=1$.

\subsubsection*{$d=2$}
The model cone is the wedge $\cW_{\alpha}$, see Section \ref{subs:diedre}. Inequality \eqref{eq:comp} and assertion (i) come from Lemma \ref{lem:pop}. To deal with case (ii), 
we define $\circ\in \{-,+\}$ satisfying $\theta^{\circ}=\min(\theta^{-},\theta^{+})$ and $\Pi_{\alpha}^\circ$ as the corresponding face.  Due to \eqref{eq:s*Da}
$\seE(\bB,\cW_{\alpha})=\sigma(\theta^{\circ})=\En(\bB,\Pi_{\alpha}^{\circ})$. Therefore in case (ii) we reduce to the situation $d=1$ and Theorem \ref{th:dicho} is proved for $d=2$.

\subsubsection*{$d=3$}

Due to Theorem \ref{th:cone-ess}, we have $\seE(\bB,\Pi)=\lambda_\ess(\bB,\Pi)$ and therefore \eqref{eq:comp}. Moreover if $\En(\bB,\Pi)<\seE(\bB,\Pi)$, the existence of an eigenfunction with exponential decay is stated in Corollary \ref{C:expodecay}. Therefore (i) is proved. 

 It remains to find $\dx\in \gC_{0}^{*}(\Pi)$ such that $\seE(\bB,\Pi)=\En(\bB,\Pi_{\dx})$. Define on $\gC^{*}_{\bfz}(\Pi)$ the function $F(\dx)=\En(\bB,\Pi_{\dx})$. Let $\Omega_{\bfz}$ denotes the section of $\Pi$, define the function $F^\star$ on $\gC(\Omega_{\bfz})$ by the partial application
\[
   F^\star(\dy) = F((\bfz,\dy)),\quad \dy\in\gC(\Omega_{\bfz}).
\] 
   Since \eqref{eq:comp} has already been proved for $d\leq 2$, we have for all $\dy$ and $\dy'$ in $\gC(\Omega_{\bfz})$:
\begin{equation}
\label{E:monoFstar}
\dy \leq \dy' \implies  F^{\star}(\dy) \leq F^{\star}(\dy') \, .
\end{equation}
Let us show that $F^{\star}$ is continuous with respect to the distance $\dD$ introduced in Definition \ref{def:distchain}. Since $\Omega_{\bfz}$ has a finite number of vertices, the chains $\dy\in \gC(\Omega_{\bfz})$ such that $\Pi_{\dy}$ is a sector (and $\Pi_{\dx}=\Pi_{(\bfz,\dy)}$ is a wedge) are isolated for the topology associated with the distance $\dD$. If $\dy$ is such that $\Pi_{(\bfz,\dy)}=\R^3$, then $F^{\star}(\dy)=1$ (see \eqref{E:spectrespace}). Therefore it remains to treat the case where the tangent substructures $\Pi_{(\bfz,\dy)}$ are half-spaces. Let $\dy$ and $\dy'$ be such chains. Denote by $\theta$ (resp. $\theta'$) the unoriented angle in $[0,\frac{\pi}{2})$ between $\bB$ and $\Pi_{\dx}$ (resp. between $\bB$ and $\Pi_{\dx'}$). We have $|\theta-\theta'|\to 0$ as $\dD(\dy,\dy')\to0$. Moreover 
$$F^{\star}(\dy)-F^{\star}(\dy')=\En(\bB,\Pi_{\dx})-\En(\bB,\Pi_{\dx'})=\sigma(\theta)-\sigma(\theta') \, .$$ 
As a consequence of the continuity of the function $\sigma$, see Lemma \ref{P:continuitesigma}, we get that $F^{\star}(\dy)-F^{\star}(\dy')$ goes to 0 as $\dD(\dy,\dy')$ goes to 0. This shows that $F^{\star}$ is continuous on $\gC(\Omega_{\bfz})$. Thanks to \eqref{E:monoFstar}, we can apply Theorem \ref{th:scichain}: the function $\overline\Omega_{\bfz}\ni\bx\mapsto F^{\star}((\bx))=\En(\bB,\Pi_{\bfz,\bx})$ is lower semicontinuous on $\overline\Omega_{\bfz}$. Since $\overline\Omega_{\bfz}$ is compact, it reaches its infimum. Combining this with Lemma \ref{lem:chainL2}, we get:
 $$\exists\bx_{1}\in \overline\Omega_{\bfz}, \quad \seE(\bB,\Pi)=E(\bB,\Pi_{\bfz,\bx_{1}}) \, .$$
Therefore (ii) follows from the analysis of lower dimensions and Theorem \ref{th:dicho} is proved.
\end{proof}

\begin{remark}
Any AGE provided by case (i) of Theorem \ref{th:dicho} satisfies:\index{Corner concentration}
$$ \forall c_{\Phi} <\sqrt{\seE(\bB,\Pi)-\En(\bB,\Pi)},\quad\exists C_{\Phi}>0, \quad  \|\re^{c_\Phi|\bz|}\Phi\|_{L^2(\Upsilon)}\leq C_\Phi \|\Phi \|_{L^2(\Upsilon)}.$$

This is a consequence of the exponential decays given by \cite[Theorem 1.3]{BoDauPopRay12} for half-planes, \cite[Proposition 4.2]{Pop13} for wedges, and Corollary \ref{C:expodecay} for 3D cones.
\end{remark}

\section{Examples}
In the case $d=1$, {\it i.e.}, when the model cone $\Pi$ is a half-space, it is known whether we are in situation (i) or (ii) of the Dichotomy Theorem. This is not the case in general for model cones $\Pi$ with $d\geq2$, and only in few cases it is known whether inequality \eqref{eq:comp} is strict or not. We provide below some examples of wedges and 3D cones where $\En(\bB,\Pi)$ has been studied. In this whole section $\bB\in\dS^2$ is a constant magnetic field of unit length.

\begin{example}[Wedges]\index{Wedge}
Let $\alpha\in(0,\pi)\cup(\pi,2\pi)$.
\begin{enumerate}[(a)]
\item For $\alpha$ small enough there holds $\En(\bB,\cW_{\alpha})<\seE(\bB,\cW_{\alpha})$, see \cite{Pop13}  
and \cite[Ch.\ 7]{PoTh}. 

\item Let $\bB=(0,0,1)$ be tangent to the edge. Then $\seE(\bB,\cW_{\alpha})=\Theta_{0}$ and $\En(\bB,\cW_{\alpha})=\En(1 \ee,\cS_{\alpha})$, cf.\ Section \ref{sec:2.1.2}. According to whether the ground state energy $\En(1 \ee,\cS_{\alpha})$ of the plane sector\index{Sector} $\cS_{\alpha}$ is less than $\Theta_0$ or equal to $\Theta_0$, we are in case (i) or (ii) of the dichotomy.

\item Let $\bB$ be tangent to a face of the wedge and normal to the edge. Then $\seE(\bB,\cW_{\alpha})=\Theta_{0}$. It is proved in \cite{Pof13T} that
$\En(\bB,\cW_{\alpha})=\Theta_{0}$ for $\alpha\in[\frac{\pi}{2},\pi)$ (case (ii)). 
\end{enumerate}
\end{example}

\begin{example}[Octant]\index{Octant}
Let $\Pi=(\R_{+})^3$ be the model octant. We quote from \cite[\S 8]{Pan02}:
\begin{enumerate}[(a)]
\item If the magnetic field $\bB$ is tangent to a face but not to an edge of $\Pi$, there exists an edge $\be$ such that $\seE(\bB,\Pi)=\En(\bB,\Pi_{\be})$ and there holds $\En(\bB,\Pi)<\En(\bB,\Pi_{\be})$. We are in case (i).

\item If the magnetic field $\bB$ is tangent to an edge $\be$ of $\Pi$, $\seE(\bB,\Pi)=\En(\bB,\Pi_{\be})=\En(\bB,\Pi)$. 
Moreover by \cite[\S 4]{Pan02}, $\En(\bB,\Pi_{\be})=\En(1,\cS_{\pi/2})<\Theta_{0}=\seE(\bB,\Pi_{\be})$. We are in case (ii).
\end{enumerate}
\end{example}
 
\begin{example}[Circular cone]\index{Cone}
Let ${\mathcal C}_{\alpha}$ be the right circular cone\index{Circular cone} of angular opening $\alpha\in(0,\pi)$.
 It is proved in \cite{BoRay13,BoRay14} that
\begin{enumerate}[(a)]
\item For $\alpha$ small enough, $\En(\bB,\cC_{\alpha})<\seE(\bB,\cC_{\alpha})$.
\item If $\bB=(0,0,1)$, then $\seE(\bB,\cC_{\alpha})=\sigma(\alpha/2)$. 
\end{enumerate}
\end{example}
 
\begin{example}[Sharp cone]
The above result on circular cones is generalized in \cite{BoDaPoRa15} to sharp cones of any section in the following sense. Let $\omega$ be a curvilinear polygon in $\gD(\R^2)$ and for $\alpha>0$, let the cone $\Pi_\alpha$ be defined as
\[
   \Pi_\alpha = \left\{\bx = (x_1,x_2,x_3)\in \R^3,\quad  x_3>0\quad\mbox{and}\quad
   \frac{1}{\alpha}\left(\frac{x_{1}}{x_{3}},\frac{x_{2}}{x_{3}}\right)\in \omega\right\}.
\]
When $\alpha$ is small, $\Pi_\alpha$ can be qualified as ``sharp''. It is proved in \cite{BoDaPoRa15} that
for $\alpha$ small enough, $\En(\bB,\cC_{\alpha})<\seE(\bB,\cC_{\alpha})$.
\end{example}

\section{Scaling and truncating Admissible Generalized Eigenvectors}
AGE's are corner-stones for our construction of quasimodes.
Here, as a preparatory step towards final construction, we show a couple of useful properties when suitable scalings and cut-off are performed.

Let $\OP(\bA,\Pi)$ be a model operator that has an AGE $\Psi$ associated with the value $\Lambda$. Then for any positive $h$, the scaled function
\begin{equation}\label{eq.AGEsc}
\Psi_{h}(\bx):=\Psi\Big(\frac{\bx}{\sqrt{h}}\Big) ,\quad\mbox{for}\quad \bx\in\Pi,
\end{equation}
defines an AGE for the operator $H_{h}(\bA,\Pi)$ associated with $h\Lambda$:
\begin{equation}
\label{eq:genEPh}
\begin{cases}
   (-ih\nabla+\bA)^2\Psi_h=h\Lambda\Psi_h &\mbox{in } \Pi,\\
   (-ih\nabla+\bA)\Psi_h\cdot\bn=0 &\mbox{on } \partial\Pi.
\end{cases}
\end{equation}
We will need to localize $\Psi_h$.
For doing this, let us choose, once for all, a model cut-off function $\underline\chi\in \sC^\infty(\R^+)$ such that
\begin{equation}
\label{D:chi}
\underline\chi(r)=1\ \mbox{ if }\ r\leq 1\quad\mbox{and}\quad 
\underline\chi(r)=0\ \mbox{ if }\ r\geq 2.
\end{equation}
For any $R>0$, let $\underline\chi_{R}$ be the cut-off function defined by
$
\underline\chi_{R}(r)=\underline\chi \left(\frac{r}{R} \right),
$
and, finally
\begin{equation}\label{eq.chih}
\chi_{h}(\bx) = \underline\chi_{R}\left(\frac{|\bx|}{h^\delta} \right) 
= \underline\chi\left(\frac{|\bx|}{Rh^\delta} \right) 
\quad\mbox{ with }\quad 0\le \delta\le \frac 12\,.
\end{equation}
Here the exponent $\delta$ is the decay rate of the cut-off. It will be tuned later to optimize remainders. 

Since $\Psi_{h}$ belongs to $\dom_{\,\loc} (\OP_{h}(\bA,\Pi))$, we can rely on Lemma~\ref{lem:tronc} to obtain the following identity for the Rayleigh quotient\index{Rayleigh quotient} of $\chi_h\Psi_{h}$:\index{Cut-off|textbf}
\begin{equation}\label{eq:rhoh}
   \QR_{h}[\bA,\Pi](\chi_{h}\Psi_{h}) =
   h\Lambda + h^2\rho_{h} \qquad\mbox{with}\qquad
  \rho_{h}=\frac{\| \,|\nabla\chi_{h}|\, \Psi_{h}\|^2}{\|\chi_{h}\Psi_{h}\|^2}.
\end{equation}

The following lemma estimates the remainder $\rho_h$:

\begin{lemma}\label{lem:rhoh}
Let $\Psi$ be an AGE for the model operator $\OP(\bA,\Pi)$. Let $k$ be the number of independent decaying directions of $\Psi$, cf.\ \eqref{eq:age1}--\eqref{eq:agmongeneig}. Let $\Psi_{h}$  be the rescaled function given by \eqref{eq.AGEsc} and let $\chi_{h}$ be the cut-off function defined by \eqref{D:chi}--\eqref{eq.chih} involving parameters $R>0$ and $\delta\in[0,\frac12]$. 
Then there exist constants $C_0>0$ and $c_0>0$ depending only on $h_0>0$, $R_0>0$ and $\Psi$ such that 
$$
   \rho_{h}=\frac{\|\,|\nabla\chi_{h}|\, \Psi_{h}\|^2}{\|\chi_{h}\Psi_{h}\|^2}\leq
   \begin{cases} C_0\, h^{-2\delta} & \mbox{ if }k<3,\\[0.5ex]
   C_0\, h^{-2\delta} \,\re^{-c_0h^{\delta-1/2}} & \mbox{ if }k=3, 
   \end{cases}
   \qquad \forall R\ge R_0,\ \forall h\le h_0,\ \forall\delta\in[0,\tfrac12]\,.
$$
\end{lemma}

\begin{proof}
By assumption $\Psi(\bx) = \re^{\ri\vartheta(\by,\bz)}\,\Phi(\bz)$ for $\udiffeo\bx=(\by,\bz)\in \R^{3-k}\times \Upsilon$, where $\udiffeo$ is a suitable rotation, and there exist positive constants $c_\Psi,C_\Psi$ controlling the exponential decay of $\Phi$ in the cone $\Upsilon\in\gP_k$, cf.\ \eqref{eq:agmongeneig}.
 Let us set $T=Rh^\delta$, so that $\chi_h(\bx) = \underline{\chi}(|\bx|/T)$.

Let us first give an upper bound for $\|\,|\nabla\chi_{h}|\,\Psi_{h}\|$:\\
If $k<3$, then 
\[
   \|\,|\nabla\chi_{h}|\,\Psi_{h}\|^2 
   \leq CT^{-2} \int_{|\by|\leq 2T} \!\rd \by\ 
   \int_{\Upsilon\, \cap\, \{|\bz|\leq2T\}} 
   \left|\Phi\Big(\frac{\bz}{\sqrt h}\Big)\right|^2\rd \bz  
   = C T^{-2}\, T^{3-k} \, h^{k/2} \|\Phi\|^2_{L^2(\Upsilon)},
\] 
else, if $k=3$ 
\begin{eqnarray*}
   \|\,|\nabla\chi_{h}|\,\Psi_{h}\|^2 
   &\leq& C T^{-2} \hskip-1em\int\limits_{\Upsilon\, \cap\, \{T \leq |\bz|\leq2T\}} 
   \left|\Phi\Big(\frac{\bz}{\sqrt h}\Big)\right|^2\rd \bz \ \ = \ \ 
   C T^{-2} \, h^{k/2} \hskip-3em\int\limits_{\Upsilon\, \cap\, 
   \big\{Th^{-\frac12} \leq |\bz|\leq2Th^{-\frac12}\big\}} 
   \hskip-3em \left|\Phi (\bz)\right|^2\rd \bz \\
&\leq& C T^{-2} \, h^{k/2}  \ \re^{-2 c_\Psi T /\sqrt{h}}
   \int_{\Upsilon\, \cap\, 
   \big\{Th^{-\frac12} \leq |\bz|\leq2Th^{-\frac12}\big\}} 
   \re^{2 c|\bz|} \left|\Phi(\bz)\right|^2\rd \bz\\
&\leq& C T^{-2} \, h^{k/2}  \ \re^{-2 c_\Psi T/\sqrt{h}}\, \|\Phi\|^2_{L^2(\Upsilon)}.
\end{eqnarray*}
Let us now consider $\|\chi_{h}\Psi_{h}\|$ (we use that $2|\by|<R$ and $2|\bz|<R$ implies $|\bx|<R$):
\begin{eqnarray*}
   \|\chi_{h}\Psi_{h}\|^2  &\geq& \int_{ 2|\by|\leq T} \! \rd \by\ 
   \int_{\Upsilon\, \cap\, \{2|\bz|\leq T\}} 
   \left|\Phi\Big(\frac{\bz}{\sqrt h}\Big)\right|^2\rd \bz =
   C T^{3-k} h^{k/2} \, \int_{\Upsilon\, \cap\, \big\{2|\bz|\leq Th^{-\frac12}\big\}} 
   \left|\Phi(\bz)\right|^2\rd \bz \\
   &\geq & C T^{3-k} h^{k/2} \, \cI(Th^{-\frac12}) \, \|\Phi\|^2_{L^2(\Upsilon)}
\end{eqnarray*}
where we have set for any $S\ge0$
\[
   \cI(S) := \bigg(\int_{\Upsilon\, \cap\, \{2|\bz|\leq S\}} 
   \left|\Phi(\bz)\right|^2\rd \bz \bigg)  \bigg(
   \int_{\Upsilon} 
   \left|\Phi(\bz)\right|^2\rd \bz\bigg)^{-1}.
\]
The function $S\mapsto\cI(S)$ is continuous, non-negative and non-decreasing on $[0,+\infty)$. It is moreover {\em increasing and positive} on $(0,\infty)$ since $\Phi$, as a solution of an elliptic equation with polynomial coefficients and null right hand side, is analytic inside $\Upsilon$.  
Consequently, $\cI(T h^{-\frac12})=\cI(R h^{\delta-\frac12})$ is uniformly bounded from below for $R\geq R_{0}$, $h\in (0,h_{0})$, $\delta\in [0,\frac 12]$ and thus\index{Cut-off}
\begin{eqnarray*}
\rho_{h}&\leq& 
\begin{cases}
C T^{-2}\,\big\{\cI(Th^{-\frac12})\big\}^{-1}
\le C_0 h^{-2\delta}
&\mbox{ if }k<3,\\[0.5ex]
C T^{-2} \ \re^{-2 c_\Psi T/\sqrt{h}} \,\big\{\cI(Th^{-\frac12})\big\}^{-1}
\le C_0 h^{-2\delta} \re^{-c_0 h^{\delta-1/2}} &\mbox{ if }k=3,
\end{cases}
\end{eqnarray*}
where the constants $C_0$ and $c_0$ in the above estimation depend only on the lower bound $R_0$ on $R$, the upper bound $h_0$ on $h$, and on the model problem associated with $\bx_{0}$, provided $\delta\in[0,\frac12]$. Lemma \ref{lem:rhoh} is proved. 
\end{proof}

\begin{remark}
\label{rem:k=0}
The estimate of $\rho_h$ provided by Lemma \ref{lem:rhoh} is still true when $k=0$, {\it i.e.}, when $\Psi$ has no decay direction (but is of modulus $1$ everywhere).
\end{remark}

\chapter{Properties of the local ground state energy}
\label{sec:sci}
In this chapter we describe the regularity properties of the local ground state energy. The main result of this section is that the function $\bx\mapsto \En(\bB_{\bx},\Pi_{\bx})$\index{Local ground energy} is lower semicontinuous on a corner domain and therefore it reaches its infimum. 

\section{Lower semicontinuity}\label{sec:cont}
\begin{theorem}
\label{T:sci}
Let $\Omega\in \gD(\R^3)$ and let $\bB\in\sC^0(\overline\Omega)$ be a continuous magnetic field. Then 
the function $\Lambda_{\overline\Omega}:\bx\mapsto \En(\bB_{\bx},\Pi_{\bx})$ is lower semicontinuous\index{Semicontinuity} on $\overline{\Omega}$.
\end{theorem}
\begin{proof}
For $\dx=(\bx_{0},\ldots)\in\gC(\Omega)$, define the function $F(\dx):=E(\bB_{\bx_{0}},\Pi_{\dx})$, which coincides on the chains of length 1 with the function $\Lambda_{\overline\Omega}$: $F((\bx_0))=\Lambda_{\overline\Omega}(\bx_0)$. Recall that we have introduced a partial order on $\gC(\Omega)$, see Definition \ref{D:Ordrechain}. Then due to \eqref{eq:comp} applied to $\Pi_\dx$ for any chain $\dx$, the function $F:\gC(\Omega)\mapsto \R_{+}$ is clearly order preserving. 

Let us show that it is continuous with respect to the distance\index{Distance of chains} $\dD$ (see Definition \ref{def:distchain}). Let $\dx\in \gC(\Omega)$ and $\dx'$ tending to $\dx$. This means that $\bx_{0}'$ tends to $\bx_{0}$ in $\R^3$ and that there exists $J\in \mathsf{BGL}(3)$ tending to the identity $\Id_3$ such that $J(\Pi_{\dx})=\Pi_{\dx'}$. In particular for $\dx'$ close enough to $\dx$,  the reduced dimensions of the cones $\Pi_{\dx}$ and $\Pi_{\dx'}$ are equal: $d(\Pi_{\dx'})=d(\Pi_{\dx})$.

\begin{enumerate}[(1)]
\item If $\Pi_{\dx}=\R^3$, then $F(\dx)=|\bB_{\bx_{0}}|$ and $F(\dx')=|\bB_{\bx_{0}'}|$, and since $\bB$ is continuous, $F(\dx')$ converges toward $F(\dx)$ when $\dD(\dx',\dx)\to0$. 
\item When $\Pi_{\dx}$ is a half-space, we denote by $\theta(\dx)$ the angle between $\Pi_{\dx}$ and $\bB_{\bx_{0}}$. We have $\theta(\dx')\to \theta(\dx)$ when $\dD(\dx',\dx)\to0$. Moreover 
$$F(\dx')-F(\dx)=|\bB_{\bx_{0}'}|\sigma(\theta(\dx')) -|\bB_{\bx_{0}}|\sigma(\theta(\dx)),$$ 
therefore $F(\dx')$ tends to $F(\dx)$ due to Lemma \ref{P:continuitesigma} and the continuity of $\bB$. 
\item When $\Pi_{\dx}$ is a wedge, there exists $(\udiffeo,\udiffeo')$ in $\gO_{3}$ and $(\alpha,\alpha')$ in $(0,\pi)\cup(\pi,2\pi)$ such that $\udiffeo(\Pi_{\dx})=\cW_{\alpha}$ and $\udiffeo'(\Pi_{\dx'})=\cW_{\alpha'}$. Therefore 
$$F(\dx')-F(\dx)=\En(\udiffeo(\bB_{\bx_{0}}),\cW_{\alpha})-\En(\udiffeo'(\bB_{\bx_{0}'}),\cW_{\alpha'}),$$ 
with $\alpha'\to \alpha$ and $\udiffeo'\to \udiffeo$ when $\dD(\dx',\dx)\to0$. 
Lemma \ref{lem:contwedge} and the continuity of $\bB$ ensure that $F(\dx')$ tends to $F(\dx)$. 
\item Finally chains $\dx$ such that $\Pi_{\dx}$ is a 3D cone are of length 1 and are isolated in $\gC(\Omega)$ for the topology associated with $\dD$ (see Proposition \ref{prop:stata}).
 \end{enumerate}
 
Therefore $F$ is continuous on $\gC(\Omega)$. We apply Theorem \ref{th:scichain}: So the function $\bx\mapsto F((\bx))=\Lambda_{\overline\Omega}(\bx)$ is lower semicontinuous on $\overline{\Omega}$.
\end{proof} 

As a consequence of the above theorem, the function $\bx\mapsto \Lambda_{\overline\Omega}(\bx)$ reaches its infimum over $\overline{\Omega}$. This fact will be one of the key ingredients to prove an upper bound with remainder for $\lambda_{h}(\bB,\Omega)$ in the semiclassical limit.

\begin{remark}
Recall that any stratum\index{Stratum} $\bt\in \gT$ has a smooth submanifold structure (see Proposition \ref{prop:stata}). Denote by $\Lambda_{\bt}$ the restriction of the local ground energy to $\bt$. Then it follows from above that $\Lambda_{\bt}$ is continuous.  Moreover if $\Omega\in \ogD(\R^3)$, one can prove that $\Lambda_{\bt}$ admits a continuous extension to $\overline{\bt}$. But this is not true anymore if $\overline{\bt}$ contains a conical point.
\end{remark}

\begin{remark}
\label{R:straight}
Let $\bB$ be a constant magnetic field and $\Omega$ be a straight polyhedron. So, its faces are plane polygons and its edges are segments of lines. The following properties hold.
\begin{enumerate}[a)]
\item For each stratum $\bt \in \gT $, the function $\Lambda_{\bt}:\bt\ni\bx\mapsto \En(\bB,\Pi_{\bx})$ is constant.
\item As a consequence of \eqref{eq:comp} and of the lower semicontinuity, $\sE(\bB \ee,\Omega)$ is the minimum of the corner local energies: 
$$\sE(\bB \ee,\Omega) = \min_{\bv\in\gV}\En(\bB,\Pi_{\bv}).$$
\item A stratum $\bt\in \gT$ being chosen we have
$$
   \forall \bx\in \bt, \quad  \seE(\bB,\Pi_{\bx})=\min_{\bt'\in \gN(\bt)} \Lambda_{\bt'},
$$ 
where $\gN(\bt):= \{\bt'\in \gT,\ \bt\subset \partial \bt' \}\setminus \{\bt\}$  is the set of the strata adjacent to $\bt$.
\item As a consequence of a), c) and the Dichotomy Theorem, there exists $\bx_{0}\in \overline{\Omega}$ such that 
$$\sE(\bB,\Omega)=E(\bB,\Pi_{\bx_{0}})< \seE(\bB,\Pi_{\bx_{0}}).$$
\end{enumerate}
\end{remark}

\section{Positivity of the ground state energy} 
The classical diamagnetic inequality (see \cite{Ka72,Si76} for example) implies that the ground state energy is in general larger than the one without magnetic field, that is 0 in our case due to Neumann boundary conditions. Usually it is harder to show that this inequality is strict. A strict diamagnetic inequality has been proved for the Neumann magnetic Laplacian in a bounded regular domain, in \cite[Section 2.2]{FouHe10}. For our unbounded domains $\Pi$ with constant magnetic field, we have: 
\begin{proposition}
Let $\Pi\in \gP_{3}$ and $\bB\neq0$ be a constant magnetic field. Then $\En(\bB,\Pi)>0$.
\end{proposition}
\begin{proof}
It is enough to make the proof for magnetic field of unit length, see Lemma \ref{lem.dilatation}. Let $d\in [0,3]$ be the reduced dimension of the cone $\Pi$. If $d=0$, then $\En(\bB,\Pi)=1$ (see \eqref{E:spectrespace}). If $d=1$, then $\En(\bB,\Pi)$ is expressed with the function $\sigma$ that satisfies $\sigma(\theta) \geq\Theta_{0}>0$ for all $\theta\in [0,\frac{\pi}{2}]$, see Lemma \ref{P:continuitesigma}. When $d=2$, the strict positivity has been shown in \cite[Corollary 3.9]{Pop13}. 

Assume now that $d=3$. If we are in case (i) of Theorem \ref{th:dicho}, then there exists an eigenfunction $\Psi\in L^2(\Pi)$ for $H(\bA,\Pi)$ associated with $E(\bB,\Pi)$. Assume that $E(\bB,\Pi)=0$, then due to the standard diamagnetic inequality\index{Diamagnetic inequality} (see \cite[Lemma A]{Ka72}), we have 
$$0 \leq \int_{\Pi}\big|\nabla |\Psi|\big|^2 \leq \int_{\Pi}|(-i\nabla-\bA)\Psi|^2 =0 , $$
that leads to $\Psi=0$, which is a contradiction. If we are in case (ii) of Theorem \ref{th:dicho}, then there exists a tangent substructure $\Pi_{\dx}$ of $\Pi$ with $d(\Pi_{\dx})<3$ such that $\En(\bB,\Pi)=\En(\bB,\Pi_{\dx})$ that is strictly positive due to the analysis of the cases $d\leq2$, see above.
\end{proof}

Combining the above proposition with Theorem \ref{T:sci}, we get: 

\begin{corollary}\label{cor.sE>0}
Let $\Omega\in \gD(\R^3)$ and let $\bB\in\sC^0(\overline\Omega)$ be non-vanishing. Then we have $\sE(\bB,\Omega)>0$.
\end{corollary}

\chapter{Upper bounds for ground state energy in corner domains}
\label{sec:up}

In this section, we prove an upper bound involving error estimates that contains the same powers of $h$ as the lower bound in Theorem \ref{T:generalLB}.
\begin{theorem}
\label{T:generalUB}
Let $\Omega\in \gD(\R^3)$ be a general 3D corner domain, and let $\bA\in \sC^{2}(\overline{\Omega})$ be a magnetic potential.
\begin{enumerate}[(a)] 

\item
 Then there exist $C_\Omega>0$ and $h_{0}>0$ such that
\begin{equation}
\label{eq:above3}
   \forall h\in (0,h_{0}), \quad \lambda_{h}(\bB,\Omega) \leq
   h\sE(\bB,\Omega)+C_\Omega \big(1+\|\bA\|_{W^{2,\infty}(\Omega)}^2 \big)h^{11/10}\ .  
\end{equation}

\item  If $\Omega$ is a polyhedral domain, this upper bound is improved: 
\begin{equation}
\label{eq:above1}
   \forall h\in (0,h_{0}), \quad \lambda_{h}(\bB,\Omega) \leq
   h\sE(\bB,\Omega)+C_\Omega \big(1+\|\bA\|_{W^{2,\infty}(\Omega)}^2 \big)h^{5/4} \ . 
\end{equation}

\item  If there exists a point $\bx_0\in\overline\Omega$ such that $\bB(\bx_0)=0$, then $\sE(\bB,\Omega)=0$ and we have the optimal upper bound
\begin{equation}
\label{eq:above0}
   \forall h\in (0,h_{0}), \quad \lambda_{h}(\bB,\Omega) \leq
   C_\Omega \big(1+\|\bA\|_{W^{2,\infty}(\Omega)}^2 \big)h^{4/3} \ . 
\end{equation}

\item  If there exists a corner $\bx_{0}$ such that 
$\sE(\bB,\Omega)=\En(\bB_{\bx_{0}},\Pi_{\bx_{0}})<\seE(\bB_{\bx_{0}},\Pi_{\bx_{0}})$
then 
\begin{equation}
\label{eq:corner}
   \forall h\in (0,h_{0}), \quad \lambda_{h}(\bB,\Omega) \leq 
   h\sE(\bB,\Omega)+C_\Omega(1+\|\bA\|_{W^{2,\infty}(\Omega)}^2)\, h^{3/2}|\log h| \, .
\end{equation}

\item If $\Omega$ is a straight polyhedron and $\bB$ is constant,
\begin{equation}
\label{E:straight1}
\forall h\in (0,h_{0}), \quad \lambda_h(\bB,\Omega) \leq h\sE(\bB \ee,\Omega) + C h^2 \ . 
\end{equation}
\end{enumerate}
\end{theorem}

We recall the notation $\QR_{h}[\bA,\Omega](\varphi)$ \eqref{eq:RayQuot} for Rayleigh quotients\index{Rayleigh quotient} and the min-max principle\index{Min-max principle}
\[
   \lambda_h(\bB,\Omega) = \min_{ \varphi\,\in\, H^1(\Omega) \,\setminus\, \{0\}} 
   \QR_{h}[\bA,\Omega](\varphi)\,.
\]

\section{Principles of construction for quasimodes}\index{Quasimode}
\label{SS:QM}
By lower semicontinuity (see Theorem \ref{T:sci}), the energy $\bx\mapsto \En(\bB_{\bx},\Pi_{\bx})$\index{Local ground energy} reaches its infimum over $\overline\Omega$. Let $\bx_{0}\in\overline\Omega$ be a point such that 
$$\En(\bB_{\bx_{0}},\Pi_{\bx_{0}})=\sE(\bB,\Omega).$$
By the dichotomy result (Theorem~\ref{th:dicho}) there exists a singular chain $\dx$ starting at $\bx_0$ such that  (see also notation \eqref{eq:tangent}):\index{Lowest local energy}
\[
   \En(\bB_{\dx},\Pi_{\dx}) = \En(\bB_{\bx_{0}},\Pi_{\bx_{0}})\quad\mbox{and}\quad
   \En(\bB_{\dx},\Pi_{\dx}) < \seE(\bB_{\dx},\Pi_\dx).
\]
For shortness, we denote $\Lx =\En(\bB_{\dx},\Pi_{\dx})$. Still by Theorem \ref{th:dicho}, there exists an AGE\index{Admissible Generalized Eigenvector!AGE} for the tangent model operator $\OP(\bA_{\dx},\Pi_{\dx})$ denoted by $\Psi^{\dx}$ and associated with $\Lx$ \index{Energy on tangent substructures}
\begin{equation}
\label{eq:gendx}
\begin{cases}
(-i\nabla+\bA_{\dx})^2\Psi^{\dx}=\Lx \Psi^{\dx} &\mbox{in } \Pi_{\dx},\\
(-i\nabla+\bA_{\dx})\Psi^{\dx}\cdot\bn =0 &\mbox{on } \partial\Pi_{\dx}.
\end{cases}
\end{equation}
For $h>0$, we define $\Psi^{\dx}_{h}$ by using the canonical scaling \eqref{eq.AGEsc}. This gives an AGE for the operator $\OP_{h}(\bA_{\dx},\Pi_{\dx})$ associated with the value $h\Lx$.
Let $\chi_{h}$ be the cut-off function defined by \eqref{D:chi}--\eqref{eq.chih} involving the parameter $R>0$ and the exponent $\delta\in(0,\frac12)$. Then the function
\begin{equation}\label{eq.phinu}
   (\chi_{h}\Psi^\dx_h)(\bx)  = 
   \underline\chi\left(\frac{|\bx|}{Rh^\delta} \right) 
   \Psi^\dx\Big(\frac{\bx}{\sqrt{h}}\Big) ,\quad\mbox{for}\quad \bx\in\Pi_\dx\,,
\end{equation}
is a canonical quasimode on the tangent structure $\Pi_\dx$ for the model operator $\OP_{h}(\bA_{\dx},\Pi_{\dx})$: Indeed the identity \eqref{eq:rhoh} and Lemma \ref{lem:rhoh} yield
\begin{equation}
\label{eq:psinu}
   \QR_{h}[\bA_\dx,\Pi_\dx](\chi_{h}\Psi^\dx_{h}) =
   h\Lx + \cO(h^{2-2\delta}).
\end{equation}
Let us recall that the fact that $\Psi^\dx_{h}$ belongs to $\dom_{\,\loc} (\OP_{h}(\bA_{\dx},\Pi_{\dx}))$ is essential for the validity of the identity above.

In order to prove Theorem \ref{T:generalUB}, we are going to construct a family of quasimodes $\phihX0\in H^1(\Omega)$ satisfying the estimate for $h>0$ small enough and the suitable power $\kappa$
\begin{equation}
\label{eq:qm}
   \QR_{h}[\bA,\Omega](\phihX0) \leq 
   h\Lambda_{\dx}+C_{\Omega}(1+\|\bA\|_{W^{2,\infty}(\Omega)}^2)h^{\kappa}.
\end{equation}
The rationale of this construction is to build a link between the canonical quasimode $\chi_{h}\Psi^\dx_h$ on the tangent structure $\Pi_\dx$ with our original operator $\OP_{h}(\bA,\Omega)$.

Let $\nu$ be the length of the chain $\dx$. By Proposition \ref{prop.3Dpolychaine}, we can always reduce to $\nu\le3$. We write
\[
   \dx = (\bx_{0},\ldots,\bx_{\nu-1}) \quad\mbox{with}\quad \nu\in\{1,2,3\}.
\]
Our quasimode $\phihX{0}$ will have distinct features according to the value of $\nu$: We will need $\nu-1$ intermediaries $\phihX{j}$, $0<j<\nu$, between $\phihX{0}$ and the final object $\phihX{\nu}$ defined by the truncated AGE given in  \eqref{eq.phinu}, {\it i.e.}, \index{Quasimode}
\begin{equation}\label{eq:psih}
   \phihX{\nu} = \chi_{h}\Psi^\dx_h \,.
\end{equation}
For $j=1,\ldots,\nu$, the function $\phihX{j}$ is defined in the tangent structure $\Pi_{\bx_{0},\ldots, \bx_{j-1}}$.
At a glance
\begin{enumerate}
\item[$\nu=1$] The quasimode $\phihX{0}$ is deduced from $\phihX{1} = \chi_{h}\Psi^\dx_h$ through the local map $\diffeo^{\bx_0}$. This is the classical construction: We say that the quasimode is \emph{sitting}\index{Quasimode!Sitting} because as $h\to0$ the supports of $\phihX{0}$ are included in each other and concentrate to $\bx_0$, see Figure \ref{F:sitting}.

\item[$\nu=2$] The quasimode $\phihX{0}$ is deduced from $\phihX{1}$ through the local map $\diffeo^{\bx_0}$, and $\phihX{1}$ is itself deduced from $\phihX{2} = \chi_{h}\Psi^\dx_h$ through another local map $\diffeo^{\dec_1}$ connected to the second element $\bx_1$ of the chain. We say that the quasimode is \emph{sliding}\index{Quasimode!Sliding} because as $h\to0$ the supports of $\phihX{0}$ are shifted along a direction $\dir_1$ determined by $\bx_1$. At this point, the construction will be very different depending on whether $\bx_{0}$ is a conical point or not, and we say that the quasimodes are respectively {\it hard sliding} and {\it soft sliding}, see Figure \ref{F:sliding}.

\item[$\nu=3$] The quasimode $\phihX{0}$ is still deduced from $\phihX{1}$ through $\diffeo^{\bx_0}$, and $\phihX{1}$ from $\phihX{2}$ through $\diffeo^{\dec_1}$. Finally $\phihX{2}$ is itself deduced from  $\phihX{3} = \chi_{h}\Psi^\dx_h$ through a third local map $\diffeo^{\dec_2}$ connected to the third element $\bx_2$ of the chain. We say that the quasimode is \emph{doubly sliding}\index{Quasimode!Doubly sliding} because as $h\to0$ the supports of $\phihX{0}$ are shifted along two directions $\dir_1$ and $\dir_2$ determined by $\bx_1$ and $\bx_2$, respectively. 
\end{enumerate}
At each level of these constructions, different transformations of the quadratic form will be performed. We organize them in 3 steps [a], [b], and [c]: 
\begin{enumerate}
\item[\mbox{[a]}] \  for a change of variable into a higher tangent substructure,\index{Tangent substructure}
\item[\mbox{[b]}] \  for a linearization of the metrics,\index{Metric}
\item[\mbox{[c]}] \  for a linearization of the potential.\index{Linearized magnetic potential}
\end{enumerate} 

This construction is illustrated in Figure~\ref{F:QM}.

Let us introduce some notation.
\begin{notation}
\label{not:U*Z}
\begin{enumerate}[(1)]
\item If $\diffeo$ is a diffeomorphism, let $\diffeo_*$ be the operator of composition:
$
   \diffeo_* (f) = f\circ \diffeo
$.
\item
 If $\zeta^\dec_h$ is a phase\index{Phase shift}, let $\diffeoZ^\dec_h$ be the operator of multiplication
$
   \diffeoZ^\dec_h(f) = f\,\overline{\zeta^\dec_h}
$.
\end{enumerate}
\end{notation}

We are going to define recursively functions $\phihX{j}$ assuming that $\phihX{j+1}$ is known. 
Typically, these relations will take the form
\begin{equation}\label{E:phihXj}
\phihX{j} = \diffeoZ_{h}^{\bv_{j}} \circ \diffeo_{*}^{\bv_{j}}(\phihX{j+1}).
\end{equation}

\begin{remark}
Since $\bx_0$ is determined, we can always assume that $\bx_0$ belongs to the reference set $\gX$ of an admissible atlas. The error rate that we will obtain in the end will depend on whether $\nu=1$ or is larger, and on whether $\bx_0$ is a conical point or not.
\end{remark}

\diagramstyle[height=2\baselineskip]

\begin{figure}[ht]
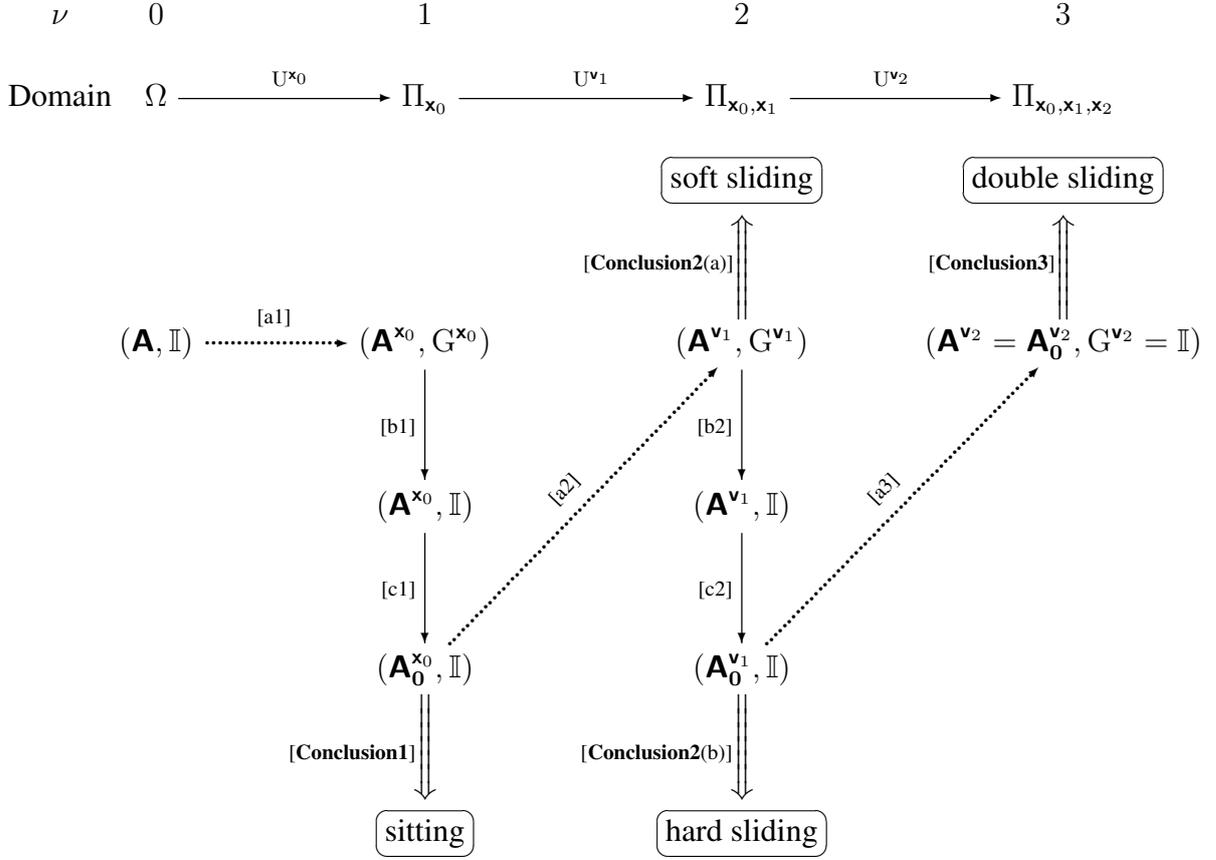

\begin{diagram}
\nu & 0 & & 1 & & 2 & & 3 \\
{\text{Domain}} & \Omega & \rTo^{\qquad\diffeo^{\bx_{0}}\quad}  
                & \Pi_{\bx_{0}} & \rTo^{\qquad\diffeo^{\dec_{1}}\quad} 
                & \Pi_{\bx_{0},\bx_{1}} & \rTo^{\quad\diffeo^{\bv_{2}}} 
                & \Pi_{\bx_{0},\bx_{1},\bx_{2}}\\
&&&&& \ovalbox{soft sliding} &&  \ovalbox{double sliding} \\
&&&&& \uImplies^{\mbox{\tiny[{\bf Conclusion2}(a)]}} && \uImplies^{\mbox{\tiny[{\bf Conclusion3}]}}\\
   & (\bA,\Id) &\rDotsto^{\mbox{\tiny[a1]}} 
   & (\bA^{\bx_{0}},\rG^{\bx_{0}}) & 
   & (\bA^{\dec_{1}},\rG^{\dec_{1}}) & 
   & (\bA^{\bv_{2}}=\bA^{\bv_{2}}_{\bfz},\rG^{\bv_{2}}=\Id) \\
   & & &\dTo^{\mbox{\tiny[b1]}} & \ruDotsto(2,4)^{\mbox{\tiny[a2]}} 
   & \dTo^{\mbox{\tiny[b2]}} &  \ruDotsto(2,4)^{\mbox{\tiny[a3]}} &  \\
   & & & (\bA^{\bx_{0}},\Id)   &  & (\bA^{\dec_{1}},\Id)  & & \\
   & & & \dTo^{\mbox{\tiny[c1]}} &  & \dTo^{\mbox{\tiny[c2]}} &  & \\
   & & & (\bA^{\bx_{0}}_{\bfz},\Id)  &  & (\bA^{\dec_{1}}_{\bfz},\Id)  &  &  \\
   & & & \dImplies^{ \mbox{\tiny[{\bf Conclusion1}]}} && \dImplies^{\mbox{\tiny[{\bf Conclusion2}(b)]}}\\
 &&&  \ovalbox{sitting} &&  \ovalbox{hard sliding}
\end{diagram}
\caption{Construction of quasimodes}
\label{F:QM}
\end{figure}

\section{First level of construction and sitting quasimodes\label{SS.phihX0}}\index{Quasimode!Sitting|textbf}
We perform the first change of variables as in Section~\ref{SS:CV}: The local diffeomorphism $\diffeo^{\bx_0}$ sends (a neighborhood of) $\bx_0$ in $\overline\Omega$ to (a neighborhood of) $\bfz$ in $\overline\Pi_{\bx_0}$.

\subsubsection*{\em[a1]} Let $\bA^{\bx_0}$ be the new potential \eqref{E:ABtilde} deduced from $\bA-\bA(\bx_0)$ by the local map $\diffeo^{\bx_0}$. Let $\zeta^{\bx_0}_h(\bx)=\re^{\ri\langle\bA(\bx_0),\,{\bx}/{h}\rangle}$, for $\bx\in\Omega$. Let us introduce the relation 
\begin{equation}\label{eq:phihX0}
\phihX0 = \diffeoZ^{\bx_0}_h \circ \diffeo_*^{\bx_0}(\phihX1),
\end{equation}
and let $r^{[1]}_h$ be the radius of the smallest ball centered at $\bfz$ containing the support of $\phihX1$ in $\overline\Pi_{\bx_0}$. The number $r^{[1]}_h$ is intended to converge to $0$ as $h$ tends to $0$, see Figure \ref{F:sitting} for a representation of the support of $\phihX0$.

\begin{figure}[ht]
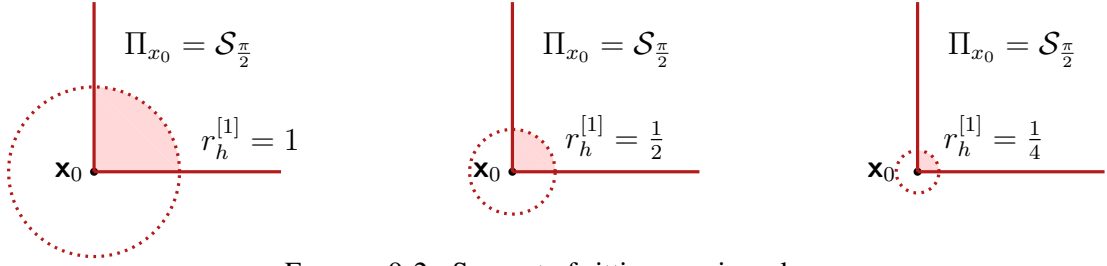

\begin{minipage}{0.32\textwidth}
\figinit{pt}
\def\R{32}
\figpt 1:(0,0)
\figpt 3:(60,0)
\figpt 4:(0,64)
\figpt 5:(70,0)
\figpt 6:(0,\R)
\figpt 7:(\R,0)
\figpt 8:(-40,-40)
\figdrawbegin{}

\figset(fillmode=yes,color = 1. .85 .85)
\figdrawline[1,6,7,1]
\figdrawarccircP 1; \R [5,4]
\figset(color=0)
\figdrawcirc 1 (1.5)
\figset(fillmode=no)
\figset (width=1.2)
\figset (color = .7 .1 .1)
\figdrawline[1,4]
\figdrawline[1,5]
\figset (dash=5)
\figdrawcirc 1 (\R)
\figdrawend
\figvisu{\figBoxA}{}
{
\figwritese 4:${\Pi_{x_{0}}=\cS_{\frac{\pi}{2}}}$ (16)
\figwriten 3:{$r^{[1]}_h=1$} (5)
\figwritesw 8:$\ $ (5)
\figwritew 1: {${\bx_0}$} (1)
}
\centerline{\box\figBoxA}
\end{minipage}
\ 
\begin{minipage}{0.32\textwidth}
\figinit{pt}
\def\R{16}
\figpt 1:(0,0)
\figpt 3:(40,0)
\figpt 4:(0,64)
\figpt 5:(70,0)
\figpt 6:(0,\R)
\figpt 7:(\R,0)
\figpt 8:(-40,-40)
\figdrawbegin{}

\figset(fillmode=yes,color = 1. .85 .85)
\figdrawline[1,6,7,1]
\figdrawarccircP 1; \R [5,4]
\figset(color=0)
\figdrawcirc 1 (1.5)
\figset(fillmode=no)
\figset (width=1.2)
\figset (color = .7 .1 .1)
\figdrawline[1,4]
\figdrawline[1,5]
\figset (dash=5)
\figdrawcirc 1 (\R)
\figdrawend
\figvisu{\figBoxA}{}
{
\figwritese 4:${\Pi_{x_{0}}=\cS_{\frac{\pi}{2}}}$ (16)
\figwriten 3:{$r^{[1]}_h=\frac12$} (5)
\figwritesw 8:$\ $ (5)
\figwritew 1: {${\bx_0}$} (1)
}
\centerline{\box\figBoxA}
\end{minipage}
\ 
\begin{minipage}{0.32\textwidth}
\figinit{pt}
\def\R{8}
\figpt 1:(0,0)
\figpt 3:(30,0)
\figpt 4:(0,64)
\figpt 5:(70,0)
\figpt 6:(0,\R)
\figpt 7:(\R,0)
\figpt 8:(-40,-40)
\figdrawbegin{}

\figset(fillmode=yes,color = 1. .85 .85)
\figdrawline[1,6,7,1]
\figdrawarccircP 1; \R [5,4]
\figset(color=0)
\figdrawcirc 1 (1.5)
\figset(fillmode=no)
\figset (width=1.2)
\figset (color = .7 .1 .1)
\figdrawline[1,4]
\figdrawline[1,5]
\figset (dash=5)
\figdrawcirc 1 (\R)
\figdrawend
\figvisu{\figBoxA}{}
{
\figwritese 4:${\Pi_{x_{0}}=\cS_{\frac{\pi}{2}}}$ (16)
\figwriten 3:{$r^{[1]}_h=\frac14$} (5)
\figwritesw 8:$\ $ (5)
\figwritew 1: {${\bx_0}$} (5)
}
\centerline{\box\figBoxA}
\end{minipage}
\vskip-2em
\caption{Support of sitting quasi-modes}
\label{F:sitting}
\end{figure}

Using \eqref{E:chgGx0}, we have
\begin{equation}\label{E:Qc1orig0}
   \QR_h[\bA,\Omega](\phihX0)
   =   \QR_h[\bA^{\bx_0},\Pi_{\bx_0},\rG^{\bx_{0}}](\phihX1). 
\end{equation}

\subsubsection*{\em[b1]}
We now linearize the metric $\rG^{\bx_{0}}$ in \eqref{E:Qc1orig0} by using Lemma \ref{L:chgvarloc}, case (\ref{It:cvloc}). We find the relation between the Rayleigh quotients
\begin{equation}\label{E:Qc1orig}
   \QR_h[\bA,\Omega](\phihX0)
   =   \QR_h[\bA^{\bx_0},\Pi_{\bx_0}](\phihX1) \,\big(1+\cO(r^{[1]}_h)\big),
\end{equation}
which implies
\begin{equation}\label{E:Qc1}
  \Big| \QR_h[\bA,\Omega](\phihX0) -
   \QR_h[\bA^{\bx_0},\Pi_{\bx_0}](\phihX1) \Big| \le 
   C_\Omega \,r^{[1]}_h\,\QR_h[\bA^{\bx_0},\Pi_{\bx_0}](\phihX1). 
\end{equation}

\subsubsection*{\em[c1]} We recall that $\bA^{\bx_0}_\bfz$ is the linear part of $\bA^{\bx_0}$ at $\bfz$. Using relation \eqref{eq:diffAA'} with $\bA=\bA^{\bx_0}$ and $\bA'=\bA^{\bx_0}_{\bfz}$ and a Cauchy-Schwarz inequality, we obtain
\begin{equation}\label{E:Qd1}
   \Big|q_h[\bA^{\bx_0},\Pi_{\bx_0}](\phihX1) -
   q_h[\bA^{\bx_0}_{\bfz},\Pi_{\bx_0}](\phihX1) \Big| \le
   2\Big( a^{[1]}_{h} \sqrt{\mu^{[1]}_{h}}+ \big(a^{[1]}_{h}\big)^2 \Big)\|\phihX1\|^2,
\end{equation}
where we have set 
\begin{equation}\label{E:Qe1}
   \mu^{[1]}_{h} = \QR_h[\bA^{\bx_0}_{\bfz},\Pi_{\bx_0}](\phihX1)
   \quad\mbox{and}\quad
   a^{[1]}_{h}=\frac{\|(\bA^{\bx_0}-\bA^{\bx_0}_{\bfz})\phihX1\|}{\|\phihX1\|}. 
\end{equation}
By Lemmas \ref{lem.TaylorA} and \ref{L:d2A} (\ref{It:Ld2A}), and since $\phihX1$ is supported in the ball $\cB(\bfz,r^{[1]}_h)$, we have
\begin{equation}\label{E:Qf1}
   a^{[1]}_h\le C(\bA)\, \big(r^{[1]}_h\big)^2\quad\mbox{with}\quad
   C(\bA) = C_\Omega\big(1+\|\bA\|_{W^{2,\infty}(\Omega)}^2 \big).
\end{equation}
Putting together \eqref{E:Qd1}--\eqref{E:Qf1}, we obtain
\begin{equation}
\label{E:QRd1}
   \Big| \QR_h[\bA^{\bx_0},\Pi_{\bx_0}](\phihX1) -
   \QR_h[\bA^{\bx_0}_{\bfz},\Pi_{\bx_0}](\phihX1) \Big| \le
   C(\bA)\Big( \big(r^{[1]}_h\big)^2 \sqrt{\mu^{[1]}_{h}}+ \big(r^{[1]}_h\big)^4 \Big). 
\end{equation}
Using the above estimate \eqref{E:QRd1}, we have
$$
r^{[1]}_h\,\QR_h[\bA^{\bx_0},\Pi_{\bx_0}](\phihX1)
\leq r^{[1]}_h\,\left(
 \QR_h[\bA^{\bx_0}_{\bfz},\Pi_{\bx_0}](\phihX1)
 +C(\bA)\Big( \big(r^{[1]}_h\big)^2 \sqrt{\mu^{[1]}_{h}}+ \big(r^{[1]}_h\big)^4 \Big)
\right).
$$
Combining this last inequality, \eqref{E:QRd1} and \eqref{E:Qc1}, we have for $r^{[1]}_h$ small enough
\begin{align}\label{E:Qg1}
\left|   \QR_h[\bA,\Omega](\phihX0) -\QR_h[\bA^{\bx_0}_{\bfz},\Pi_{\bx_0}](\phihX1)\right|
   &\leq  C(\bA)\Big(  \mu^{[1]}_{h}r^{[1]}_h 
   + \big(r^{[1]}_h\big)^2 \sqrt{\mu^{[1]}_{h}}+ \big(r^{[1]}_h\big)^4 \Big). 
\end{align}

\subsubsection*{\bf[Conclusion1]}
If $\nu=1$, we set, as already mentioned, $\phihX{1} = \chi_{h}\Psi^\dx_h$. Note that $\bA^{\bx_0}_{\bfz}$ coincides with $\bA_\dx$. To tune the cut-off $\chi_h$, we choose the exponent $\delta$ as $\de0$ and the radius $R$ as $1$. Therefore $r^{[1]}_h=\cO(h^{\de0})$ and by \eqref{eq:psinu} $\mu^{[1]}_{h}=\cO(h)$. Using \eqref{E:Qg1} and again \eqref{eq:psinu}, we deduce 
\begin{equation}\label{eq.phihX0chgvar}
   \QR_h[\bA,\Omega](\phihX0)
   \leq h\Lx + C(\bA) \big( h^{2-2\de0}+h^{1+\de0}+h^{\frac12+2\de0}+h^{4\de0} \big).
\end{equation}
So we can conclude in the sitting case. Choosing $\de0=3/8$, we optimize remainders and we get the upper bound
\[
\lambda_{h}(\bB,\Omega) \leq
   h\sE(\bB,\Omega)+C_\Omega \big(1+\|\bA\|_{W^{2,\infty}(\Omega)}^2 \big)h^{5/4} \ . 
\]

\subsubsection*{Case when $\bB(\bx_0)=0$}
If $\bB(\bx_0)=0$, the function $\Psi^\dx\equiv1$ is an AGE on $\Pi_{\bx_0}$ associated with the value $\Lambda_\dx=0$. We are in the sitting case $\nu=1$ and the estimate \eqref{E:Qg1} is still valid. But now \eqref{eq:psinu} (combined with Remark \ref{rem:k=0}) yields
\[
   \QR_{h}[\bA_\dx,\Pi_\dx](\chi_{h}\Psi^\dx_{h}) \le
   C h^{2-2\delta}.
\]
Choosing $\delta$ as $\de0$ as above, we deduce $\mu^{[1]}_{h}=\cO(h^{2-2\de0})$. Hence
\begin{equation}\label{eq.phihX0chgvar0}
   \QR_h[\bA,\Omega](\phihX0)
   \leq C \big( h^{2-2\de0}+h^{2-2\de0+\de0}+h^{1-\de0+2\de0}+h^{4\de0} \big).
\end{equation}
Choosing $\de0=1/3$, we optimize remainders and we get the upper bound
\[
\lambda_{h}(\bB,\Omega) \leq
   C_\Omega \big(1+\|\bA\|_{W^{2,\infty}(\Omega)}^2 \big)h^{4/3} \ . 
\]

\subsubsection*{Case when $\bx_0$ is a corner and $\Psi^\dx$ is an eigenvector}
Since $\En(\bB_{\bx_{0}},\Pi_{\bx_{0}})<\seE(\bB_{\bx_{0}},\Pi_{\bx_{0}})$ and $\lambda_\ess(\bB_{\bx_{0}},\Pi_{\bx_{0}})=\seE(\bB_{\bx_{0}},\Pi_{\bx_{0}})$ by Theorem \ref{th:cone-ess}, the generalized eigenfunction $\Psi^{\dx}$ of $\OP(\bA_{\bx_{0}},\Pi_{\bx_{0}})$ provided by Theorem \ref{th:dicho} is an eigenfunction and has exponential decay. Here $\dx=(\bx_0)$ and the quasimode $\phihX{0}$ is sitting. Using \eqref{E:taylorA} and Lemma~\ref{L:d2A} (\ref{It:Ld2A}), we get $C_\Omega>0$ such that
$$\forall \bx\in \supp (\phihX1), \quad |(\bA^{\bx_{0}}-\bA^{\bx_{0}}_\bfz)(\bx)| \leq C_\Omega \|\bA^{\bx_{0}}\|_{W^{2,\infty}(\supp(\phihX1))}|\bx|^2\, .$$
Using the change of variable ${\bf X}=\bx h^{-1/2}$ and the exponential decay of $\Psi^{\dx}$ we get
\begin{equation}\label{eq:ah1+}
a_{h}^{[1]}=\frac{\|(\bA^{\bx_{0}}-\bA^{\bx_{0}}_\bfz)\phihX1\|}{\|\phihX1\|} \leq C_\Omega  \|\bA^{\bx_{0}}\|_{W^{2,\infty}(\supp(\phihX1))} h.
\end{equation}
Using \eqref{E:Qd1} with estimate \eqref{eq:ah1+} and Lemma \ref{lem:rhoh}, for any $\delta\in(0,\frac12]$, we get
\begin{eqnarray*}
\QR_{h}[\bA^{\bx_{0}},\Pi_{\bx_{0}}](\phihX1)
&\leq& h\Lambda_{\dx}+ C \Big( h^{2-2\delta}\re^{-ch^{\delta-\frac 12}}
 + \| \bA \|_{W^{2,\infty}(\Omega)} h^{\frac32}+ \| \bA \|^2_{W^{2,\infty}(\Omega)} h^2 \Big)\\
&\leq& h\Lambda_{\dx}+ C \big(1+\|\bA \|^2_{W^{2,\infty}(\Omega)}\big) 
 \, \big(h^{2-2\delta}\re^{-ch^{\delta-\frac 12}}+h^{\frac32}\big).
\end{eqnarray*}
Thanks to \eqref{E:Qc1}, the quasimode $\phihX0$ satisfies
\begin{eqnarray*}
\QR_{h}[\bA,\Omega](\phihX0)
&\leq& \big(1+\cO(h^\delta)\big) \big\{h\Lambda_{\dx}+ C\big(1+\|\bA \|^2_{W^{2,\infty}(\Omega)}\big) 
 (h^{2-2\delta}\re^{-ch^{\delta-\frac 12}}+h^{3/2})\big\} \\
&\leq& h\Lambda_{\dx}+ C\big(1+\|\bA \|^2_{W^{2,\infty}(\Omega)}\big) 
 \big\{ h^{1+\delta} + h^{2-2\delta}\re^{-ch^{\delta-\frac 12}}+h^{3/2}\big\} .
\end{eqnarray*}
Here $C$ denotes various constants depending on $\Omega$ but independent from $h\le h_0$ and $\delta\le\frac12$.
We optimize this by taking $\delta=\frac12-\varepsilon(h)$ with $\varepsilon(h)$ so that
$h^{1+\delta} = h^{2-2\delta}\re^{-ch^{\delta-\frac 12}}$, {\it i.e.},
\[
   h^{\frac32-\varepsilon(h)} = h^{1+2\varepsilon(h)}\re^{-ch^{-\varepsilon(h)}} .
\] 
We find
\[
   \re^{ch^{-\varepsilon(h)}} = h^{-\frac12+3\varepsilon(h)},\quad\mbox{{\it i.e.},}\quad
   h^{-\varepsilon(h)} = \tfrac{1}{c}(-\tfrac12+3\varepsilon(h))\log h \,.
\] 
The latter equation has one solution $\varepsilon(h)$ which tends to $0$ as $h$ tends to $0$.
Replacing $h^{-\varepsilon(h)}$ by the value above in $h^{\frac32-\varepsilon(h)}$, we find that the remainder is a $\cO(h^{3/2}|\log h|)$.

\subsubsection*{Case when $\Omega$ is a straight polyhedron and $\bB$ constant}
According to Remark \ref{R:straight} d), we may assume that $(\bB,\Pi_{\bx_{0}})$ is in case (i) of the Dichotomy Theorem. We construct a sitting quasimode near $\bx_{0}$. Since the magnetic field is constant, we may associate a linear magnetic potential $\bA$. Define now $\phihX0$ from $\phihX1$ as in \eqref{eq:phihX0} and tune the cut-off by choosing $\delta=0$ and $R>0$ large enough such that the support of $\chi_{h}$ is contained in a map-neighborhood $\cV_{\bx_{0}}$ of $\bfz$ in $\Pi_{\bx_{0}}$.

Notice that $\diffeo^{\bx_{0}}$ is the translation $\bx\mapsto \bx-\bx_{0}$ and that the linear part of the potential satisfies $\bA_{\bfz}^{\bx_{0}}=\bA^{\bx_{0}}$. Therefore the error terms due to the change of variables and the linearization of the potential appearing in step [b1] are zero, and \eqref{E:Qg1} is improved in
$$ \QR_h[\bA,\Omega](\phihX0)=\QR_h[\bA^{\bx_0}_{\bfz},\Pi_{\bx_0}](\phihX1) \, .$$
Estimate \eqref{E:straight1} is then a direct consequence of identity \eqref{eq:rhoh} combined with Lemma \ref{lem:rhoh}.

\section{Second level of construction and sliding quasimodes\label{SS.phihX1}}\index{Quasimode!Sliding|textbf}
We have now to deal with the case $\nu\geq 2$. So $\dx=(\bx_0,\bx_1)$ or $(\bx_0,\bx_1,\bx_2)$.

Here we use the same notation as the introduction of singular chains in Section \ref{ss:chains}.
Let $\udiffeo^{0}\in\gO_3$ such that $\Pi_{\bx_{0}}=\udiffeo^{0}(\R^{3-d_0}\times \Gamma_{\bx_0})$ where $\Gamma_{\bx_0}$ is the reduced cone of $\Pi_{\bx_{0}}$. Let $\Omega_{\bx_0}=\Gamma_{\bx_0}\cap\dS^{d_0-1}$ be the section of $\Gamma_{\bx_0}$. By definition of chains, $\bx_{1}$ belongs to $\overline\Omega_{\bx_0}$ and let $C_{\bx_0,\bx_1}$ be the tangent cone to $\Omega_{\bx_0}$ at $\bx_1$. Then the tangent substructure $\Pi_{\bx_0,\bx_1}$ is determined by the formula
\[
  \Pi_{\bx_0,\bx_1} = \udiffeo^0 \big(\R^{3-d_0}\times\langle \bx_1\rangle \times 
  C_{\bx_0,\bx_1} \big)\,.
\]
Let us define the unit vector $\dir_1$ by the formulas
\begin{equation}
\label{def:tau}
   \underline{\dir_1}:=(0,\bx_{1})\in \R^{3-d_0}\times \Gamma_{\bx_0}
   \quad\mbox{and}\quad
   \dir_1=\udiffeo^{0}\,\underline{\dir_1}\in\overline\Pi_{\bx_{0}}\cap\dS^2.
\end{equation}
With this definition, the substructure $\Pi_{\bx_0,\bx_1}$ is the tangent cone to $\Pi_{\bx_0}$ at the point $\dir_1$. Note that in the case when $\bx_0$ is a vertex of $\Omega$, the above formulas simplify: $\Pi_{\bx_0}$ is its own reduced cone, $\Pi_{\bx_0,\bx_1} = \langle \bx_1\rangle \times C_{\bx_0,\bx_1}$, and $\dir_1$ coincides with $\bx_1$. 
 
Note also that the cone $\Pi_{\bx_0,\bx_1}$ can be the full space, a half-space or a wedge, and that $\dir_1$ gives a direction associated with $\Pi_{\bx_0,\bx_1}$ starting from the origin $\bfz$ of $\Pi_{\bx_{0}}$:
\begin{enumerate}
\item If $\Pi_{\bx_0,\bx_1}\equiv\R^3$, then $\dir_1$ belongs to the interior of $\Pi_{\bx_{0}}$.
\item If $\Pi_{\bx_0,\bx_1}\equiv\R^2\times\R_+$, then $\dir_1$ belongs to a face of $\Pi_{\bx_{0}}$.
\item If $\Pi_{\bx_0,\bx_1}\equiv\cW_{\alpha}$, then $\dir_1$ belongs to an edge  of $\Pi_{\bx_{0}}$. 
\end{enumerate}
Unless we are in the latter case ($\Pi_{\bx_0,\bx_1}$ is a wedge), the choice of $\dir_1$ is not unique.

Set $\dec_{1} = d^{[1]}_h\dir_{1}$ where $d^{[1]}_h$ is a positive quantity intended to converge to $0$ with $h$. The vector $\dec_1$ is a shift that allows to pass from the cone $\Pi_{\bx_0}$ to the substructure $\Pi_{\bx_0,\bx_1}$, which is also the tangent cone to $\Pi_{\bx_0}$ at the point $\dec_1$.
Let $\diffeo^{\dec_1}$ be a local diffeomorphism that sends 
(a neighborhood $\cU_{\dec_1}$ of) $\dec_1$ in $\Pi_{\bx_0}$ to (a neighborhood $\cV_{\dec_1}$ of) $\bfz$ in $\Pi_{\bx_0,\bx_1}$.
We can assume without restriction that $\diffeo^{\dec_1}$ is part of an admissible atlas on $\Pi_{\bx_{0}}$.

\begin{figure}[ht]
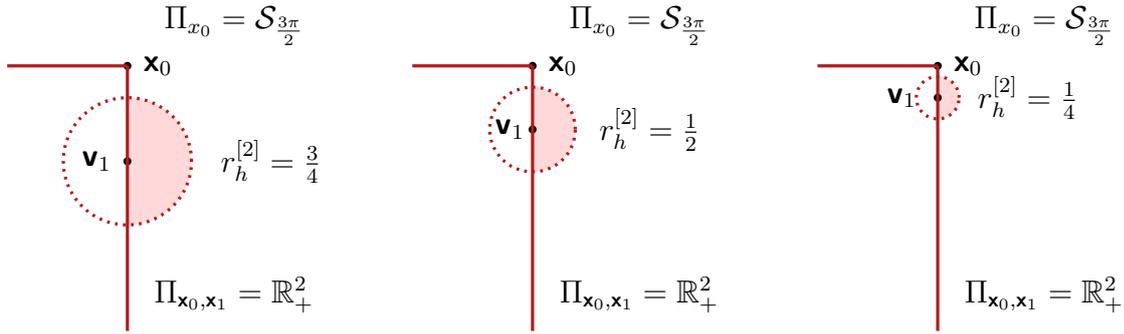

\begin{minipage}{0.32\textwidth}
\figinit{pt}
\def\R{24}
\def\c{36}
\figpt 1:(0,0)
\figpt 2:(0,-\c)
\figpt 3:(40,0)
\figpt 4:(0,-100)
\figpt 5:(-45,0)
\figpt 6:(0,\R)
\figpt 7:(\R,0)
\figpt 8:(100,20)
\figdrawbegin{}

\figset(fillmode=yes,color = 1. .85 .85)
\figdrawarccirc 2; \R  (-90,90)
\figset(color=0)
\figdrawcirc 1 (1.5)
\figdrawcirc 2 (1.5)
\figset(fillmode=no)
\figset (width=1.2)
\figset (color = .7 .1 .1)
\figdrawline[1,4]
\figdrawline[1,5]
\figset (dash=5)
\figdrawcirc 2 (\R)
\figdrawend
\figvisu{\figBoxA}{}
{
\figwritene 8: $\ $ (5)
\figwriten  3: ${\Pi_{x_{0}}=\cS_{\frac{3\pi}{2}}}$ (8)
\figwritene 4: ${\Pi_{\bx_0,\bx_1}=\mathbb{R}^{2}_{+}}$ (10)
\figwritee 1: {$\bx_0$} (3)
\figwritew 2: {$\bv_1$} (3)
\figwritee 2:{$r^{[2]}_h=\frac34$} (35)
}
\centerline{\box\figBoxA}
\end{minipage}
\ 
\begin{minipage}{0.32\textwidth}
\figinit{pt}
\def\R{16}
\def\c{24}
\figpt 1:(0,0)
\figpt 2:(0,-\c)
\figpt 3:(40,0)
\figpt 4:(0,-100)
\figpt 5:(-45,0)
\figpt 6:(0,\R)
\figpt 7:(\R,0)
\figpt 8:(100,20)
\figdrawbegin{}

\figset(fillmode=yes,color = 1. .85 .85)
\figdrawarccirc 2; \R  (-90,90)
\figset(color=0)
\figdrawcirc 1 (1.5)
\figdrawcirc 2 (1.5)
\figset(fillmode=no)
\figset (width=1.2)
\figset (color = .7 .1 .1)
\figdrawline[1,4]
\figdrawline[1,5]
\figset (dash=5)
\figdrawcirc 2 (\R)
\figdrawend
\figvisu{\figBoxA}{}
{
\figwritene 8: $\ $ (5)
\figwriten  3: ${\Pi_{x_{0}}=\cS_{\frac{3\pi}{2}}}$ (8)
\figwritene 4: ${\Pi_{\bx_0,\bx_1}=\mathbb{R}^{2}_{+}}$ (10)
\figwritee 1: {$\bx_0$} (3)
\figwritew 2: {$\bv_1$} (0)
\figwritee 2:{$r^{[2]}_h=\frac12$} (25)
}
\centerline{\box\figBoxA}
\end{minipage}
\ 
\begin{minipage}{0.32\textwidth}
\figinit{pt}
\def\R{8}
\def\c{12}
\figpt 1:(0,0)
\figpt 2:(0,-\c)
\figpt 3:(40,0)
\figpt 4:(0,-100)
\figpt 5:(-45,0)
\figpt 6:(0,\R)
\figpt 7:(\R,0)
\figpt 8:(100,20)
\figdrawbegin{}

\figset(fillmode=yes,color = 1. .85 .85)
\figdrawarccirc 2; \R  (-90,90)
\figset(color=0)
\figdrawcirc 1 (1.5)
\figdrawcirc 2 (1.5)
\figset(fillmode=no)
\figset (width=1.2)
\figset (color = .7 .1 .1)
\figdrawline[1,4]
\figdrawline[1,5]
\figset (dash=5)
\figdrawcirc 2 (\R)
\figdrawend
\figvisu{\figBoxA}{}
{
\figwritene 8: $\ $ (5)
\figwriten  3: ${\Pi_{x_{0}}=\cS_{\frac{3\pi}{2}}}$ (8)
\figwritene 4: ${\Pi_{\bx_0,\bx_1}=\mathbb{R}^{2}_{+}}$ (10)
\figwritee 1: {$\bx_0$} (3)
\figwritew 2: {$\bv_1$} (5)
\figwritee 2:{$r^{[2]}_h=\frac14$} (15)
}
\centerline{\box\figBoxA}
\end{minipage}
\caption{Sliding quasi-modes}
\label{F:sliding}
\end{figure}

\subsubsection*{\em[a2]} 
By the change of variable $\diffeo^{\dec_1}$, the potential $\bA^{\bx_0}_{\bfz}-\bA^{\bx_0}_{\bfz}(\dec_1)$ becomes $\bA^{\dec_1}$ (cf.\ \eqref{E:ABtilde})
\begin{equation*}
    \bA^{\dec_1} = 
   (\rJ^{\dec_1})^\top \Big(\big( \bA^{\bx_0}_{\bfz}-\bA^{\bx_0}_{\bfz}({\dec_{1}})\big)
    \circ (\diffeo^{\dec_1})^{-1}\Big)
   \quad\mbox{with}\quad
   \rJ^{\dec_1} = \rd (\diffeo^{\dec_1})^{-1}.
\end{equation*}
Let $\zeta^{\dec_1}_h(\bx)=\re^{\ri\langle\bA^{\bx_0}_\bfz(\dec_1),\,{\bx}/{h}\rangle}$, for $\bx\in\Pi_{\bx_0}$.
We introduce the relation
\begin{equation}\label{eq.c2}
\phihX1 = \diffeoZ^{\dec_{1}}_h \circ \diffeo_*^{\dec_1}(\phihX2),
\end{equation}
and let $r^{[2]}_h$ be the radius of the smallest ball centered at $\bfz$ containing the support of $\phihX2$ in $\overline\Pi_{\bx_0,\bx_1}$, see Figure \ref{F:sliding} for a representation of the support of $\phihX0$. This new quantity is also intended to converge to $0$ with $h$.

We now have a turning point of the algorithm: if $\bx_{0}$ is not a conical point, we use the fact that $\diffeo^{\dec_{1}}$ is a translation. Then $\rG^{\dec_{1}}=\Id$ and $\bA^{\dec_{1}}$ coincides with its linear part $\bA^{\dec_{1}}_{\bfz}$. Steps [b] and [c] are replaced by the following identity: 
\begin{equation}
\label{E:isoenergieSS}
 \QR_h[\bA_{\bfz}^{\bx_{0}},\Pi_{\bx_{0}}](\phihX1) = 
   \QR_h[\bA^{\bv_1}_{\bfz},\Pi_{\dx}](\phihX2),
   \end{equation}
and we are able to make a direct estimation of the quasimodes, see the [Conclusion2(a)] below. We will called them {\it soft sliding quasimodes}.

If $\bx_{0}$ is a conical point, we continue the algorithm as described below:
\subsubsection*{\em[b2]} 
 Using \eqref{E:chgGx0} and \eqref{eq:psiG2} in Lemma \ref{L:chgvarloc}, we find a relation between  Rayleigh quotients of the same form as \eqref{E:Qc1orig}, with $\cO(r^{[1]}_h)$ replaced by $\cO(r^{[2]}_h/d^{[1]}_h)$. Like for \eqref{E:Qc1}, we deduce
\begin{align}\label{E:Qc2}
  \Big| \QR_h[\bA^{\bx_0}_{\bfz},\Pi_{\bx_0}](\phihX1) -
   \QR_h[\bA^{\dec_1},\Pi_{\bx_0,\bx_1}](\phihX2) \Big| \lesssim
 \dfrac{r^{[2]}_h}{d^{[1]}_h} \QR_h[\bA^{\dec_1},\Pi_{\bx_0,\bx_1}](\phihX2).
\end{align}

\subsubsection*{\em[c2]} Let $\bA^{\dec_1}_\bfz$ be the linear part of $\bA^{\dec_1}$ at $\bfz\in\Pi_{\bx_0,\bx_1}$. Thus, by relation~\eqref{eq:diffAA'} and a Cauchy-Schwarz inequality, we have
\begin{equation}\label{E:Qd2}
   \Big|q_h[\bA^{\dec_1},\Pi_{\bx_0,\bx_1}](\phihX2) -
   q_h[\bA^{\dec_1}_{\bfz},\Pi_{\bx_0,\bx_1}](\phihX2) \Big| \le
   C \Big( a^{[2]}_{h}\sqrt{\mu^{[2]}_{h}} + \big(a^{[2]}_{h}\big)^2\Big) \|\phihX2\|^2,
\end{equation}
with 
\begin{equation}\label{E:Qe2}
   \mu^{[2]}_{h} = \QR_h[\bA^{\dec_1}_{\bfz},\Pi_{\bx_0,\bx_{1}}](\phihX2)
   \quad\mbox{and}\quad
   a^{[2]}_{h}=\frac{\|(\bA^{\dec_1}-\bA^{\dec_1}_{\bfz})\phihX2\|} {\|\phihX2\|}.
\end{equation}
By Lemmas \ref{lem.TaylorA}--\ref{L:d2A}, case \eqref{It:Ld2Abis}, and since $\phihX2$ is supported in the ball $\cB(\bfz,r^{[2]}_h)$, we have
\begin{equation}\label{E:Qf2}
   a^{[2]}_h\lesssim \frac{\big(r^{[2]}_h\big)^2}{d^{[1]}_h} .
\end{equation}
Using \eqref{E:Qc2}--\eqref{E:Qf2}, we find, if $r^{[2]}_h/d^{[1]}_h$ is small enough,
\begin{equation}\label{E:Qg2}
   \left|\QR_h[\bA^{\bx_0}_{\bfz},\Pi_{\bx_0}](\phihX1) -
   \QR_h[\bA^{\dec_1}_{\bfz},\Pi_{\bx_0,\bx_{1}}](\phihX2)\right|
    \ \lesssim \  \mu^{[2]}_{h}\frac{r^{[2]}_h}{d^{[1]}_h} + 
    \frac{\big(r^{[2]}_h\big)^2}{d^{[1]}_h} \sqrt{\mu^{[2]}_{h}}+ 
    \frac{\big(r^{[2]}_h\big)^4}{\big(d^{[1]}_h\big)^2} \,.
\end{equation}

\subsubsection*{\bf[Conclusion2]}
If $\nu=2$, we set, as already mentioned, $\phihX{2} = \chi_{h}\Psi^\dx_h$. 
Note that $\bA^{\dec_1}_{\bfz}$ coincides with $\bA_\dx$.
We have now to distinguish two cases, according as $\bx_0$ is or not a conical point. 

(a) {\it Soft sliding}. If $\bx_{0}$ is not a conical point, {\it i.e.}, $\bx_0\not\in\gVc$, the local map $\diffeo^{\dec_1}$ is the translation $\bx\mapsto\bx-\dec_1$. To tune the cut-off $\chi_h$, we choose the exponent $\delta$ as $\de0$ and the shift $d^{[1]}_h$ as $h^{\de0}$. We choose the radius $R$ for the cut-off $\chi_h$ \eqref{eq.chih} so that the support of $\underline\chi_{R}$ is contained in a map neighborhood $\cV_{\dec_1}$ of $\bfz$ in $\overline\Pi_{\bx_0,\bx_1}$, {\it i.e.}, a neighborhood such that:
\[
   \diffeo^{\dec_1}(\cU_{\dec_1} \cap \Pi_{\bx_0}) = \cV_{\dec_1} \cap \Pi_{\bx_0,\bx_1},
\]
where $\diffeo^{\dec_1}(\bx)=\bx-\dec_1$ and $\cU_{\dec_1}=\cV_{\dec_1}+\dec_1$. Then the quantities $r^{[1]}_h$ and $r^{[2]}_h$ are both $\cO(h^{\de0})$ and we can combine \eqref{E:isoenergieSS} with \eqref{E:Qg1} and the cut-off estimate \eqref{eq:psinu}. Moreover for $h$ small enough, the quantities $\mu^{[1]}_h$ is $\cO(h)$, and we deduce the estimate \eqref{eq.phihX0chgvar} as in the case $\nu=1$, which leads, like in the sitting case, to the upper bound \eqref{eq:above1} with $h^{5/4}$. The latter step ends in particular the handling of the polyhedral case since we can always reduce to chains of length $\nu\le2$ in polyhedral domains, cf.\ Proposition \ref{prop.3Dpolychaine}.

(b) {\it Hard sliding}. If $\bx_{0}$ is a conical point, to tune the cut-off $\chi_h$, we choose the exponent $\delta$ as $\de0+\de1$ and the shift $d^{[1]}_h$ as $h^{\de0}$, with $\de0,\de1>0$ such that $\de0+\de1<\frac12$. We choose the radius $R$ equal to $1$. Therefore $r^{[2]}_h=\cO(h^{\de0+\de1})$ and $r^{[1]}_h=\cO(h^{\de0})$. By \eqref{eq:psinu} $\mu^{[2]}_{h}=\cO(h)$ and, since for $h$ small enough, $r^{[2]}_h/d^{[1]}_h$ is arbitrarily small, we also deduce with the help of \eqref{E:Qg2} that $\mu^{[1]}_{h}=\cO(h)$. 
Putting this together with \eqref{E:Qg1} and \eqref{E:Qg2}, and using \eqref{eq:psinu} once more, we deduce the estimate
\begin{multline}\label{eq.phihX0e2}
   \QR_h[\bA,\Omega](\phihX0)
   \leq h\Lx + C\,\big( h^{1+\de0}+h^{\frac12+2\de0}+h^{4\de0} \big) \\
   +C \,\big(h^{2-2\de0-2\de1}+h^{1+\de1}+h^{\frac12+\de0+2\de1}+h^{2\de0+4\de1}\big).
\end{multline}
The exponents that appear here are the same as for the lower bound \eqref{eq.miniglo}. Thus taking $\de0= 3/{10}$ and $\de1= 3/{20}$, we optimize remainders and deduce
\[
\lambda_{h}(\bB,\Omega) \leq
   h\sE(\bB,\Omega)+C_\Omega \big(1+\|\bA\|_{W^{2,\infty}(\Omega)}^2 \big)h^{11/10} \ . 
\]

\section{Third level of construction and doubly sliding quasimodes\label{SS.phihX2}}\index{Quasimode!Doubly sliding|textbf}
It remains to deal the case $\nu=3$. In that case, the chain $\dx=(\bx_{0},\bx_{1},\bx_{2})$ is such that
\begin{itemize}
\item $\bx_{0}$ is a conical point, 
\item $\bx_{1}$ is a vertex of $\Omega_{\bx_{0}}$, $\dir_1$ coincides with $\bx_1$, the corresponding edge of $\Pi_{\bx_{0}}$ is generated by $\bx_1$, and $\Pi_{\bx_{0},\bx_{1}}$ is a wedge,
\item $\bx_{2}$ is an end of the interval $\Omega_{\bx_{0},\bx_{1}}$, it corresponds to a point $\dir_{2}$ on a face of
$\Pi_{\bx_{0},\bx_{1}}$, defined as  in \eqref{def:tau}. Finally $\Pi_{\bx_{1},\bx_{1},\bx_{2}}=\Pi_{\dx}$ is a half-space. 
\end{itemize}
Set $\dec_{2} = d^{[2]}_h\dir_{2}$ where $d^{[2]}_h$ is a positive quantity intended to converge to $0$ with $h$. 
Let $\diffeo^{\dec_2}$ be the translation that sends 
(a neighborhood of) $\dec_2$ in $\Pi_{\bx_0,\bx_{1}}$ to (a neighborhood of) $\bfz$ in $\Pi_{\dx}=\Pi_{\bx_0,\bx_1,\bx_{2}}$. 

\subsubsection*{\em[a3]} 
By the change of variable $\diffeo^{\dec_2}$, since $\rJ^{\bv_{2}}=\Id_{3}$, the potential $\bA^{\bv_1}_{\bfz}-\bA^{\bv_1}_{\bfz}(\dec_2)$ becomes 
\begin{equation*}
    \bA^{\dec_2} = 
   \big( \bA^{\bv_1}_{\bfz}-\bA^{\bv_1}_{\bfz}({\dec_{2}})\big)
    \circ (\diffeo^{\dec_2})^{-1},
\end{equation*}
and it coincides with its linear part $\bA^{\dec_2}_\bfz$. Let $\zeta^{\dec_2}_h(\bx)=\re^{\ri\langle\bA^{\dec_1}_\bfz(\dec_2),\,{\bx}/{h}\rangle}$, for $\bx\in\Pi_{\bx_0,\bx_1}$. We define
\begin{equation}
   \phihX2 = Z_{h}^{\dec_{2}} \circ\diffeo_{*}^{\dec_2} (\phihX3).
\end{equation}
  Since $\rG^{\bv_{2}}=\Id_{3}$, we have
 \begin{equation}\label{E:Qc3}
   \QR_h[\bA_{\bfz}^{\dec_{1}},\Pi_{\bx_{0},\bx_{1}}](\phihX2) = 
   \QR_h[\bA^{\bv_2}_{\bfz},\Pi_{\dx}](\phihX3). 
\end{equation}

\subsubsection*{\bf[Conclusion3]} We set, as already mentioned $\phihX{3}=\chi_{h}\Psi_{h}^\dx$. We have $\bA_{\bfz}^{\bv_{2}}=\bA_{\dx}$. 
We choose the exponent $\delta$ as $\de0+\de1$, the shifts $d_{h}^{[2]}$ as $h^{\de0+\de1}$ and $d^{[1]}_h$ as $h^{\de0}$, with $\de0,\de1>0$ such that $\de0+\de1<\frac12$. 
We conclude as the conical case at level 2 and obtain again \eqref{eq.phihX0e2}.  
We deduce
\[
\lambda_{h}(\bB,\Omega) \leq
   h\sE(\bB,\Omega)+C_\Omega \big(1+\|\bA\|_{W^{2,\infty}(\Omega)}^2 \big)h^{11/10} \ . 
\]

\section{Conclusion}
The outcome of the last four sections is the achievement of the proof of Theorem \ref{T:generalUB}. We may notice that there is only one configuration where we cannot prove the convergence rate $h^{5/4}$: This is the case when all points with minimal local energy $\bx_0$ satisfy all the following conditions
\begin{enumerate}
\item $\bx_0$ is a conical point ($\bx_0\in\gVc$),
\item The model operator $\OP(\bA_{\bx_0},\Pi_{\bx_0})$ has no eigenvalue below its essential spectrum,
\item The geometry around $\bx_0$ is not trivial {\it i.e.}, the derivative $\rK^{\bx_0}(\bfz)$ of the Jacobian  is not zero.
\end{enumerate}

\part{Improved upper bounds}
\label{part:4}

\chapter{Stability of Admissible Generalized Eigenvectors}
\label{sec:Stab}

In order to confirm our claim for the improved upper bounds \eqref{eq:convimp}, we need to revisit AGE's (Admissible Generalized Eigenvectors)\index{Admissible Generalized Eigenvector} of model problems\index{Model operator} $\OP(\bA,\Pi)$. In particular we want to know what are their stability properties under perturbation of the constant magnetic field $\bB = \curl\bA$. 

\section{Structure of AGE's}
\label{s:age}
In this section we recall from Chapter \ref{sec:tax} the model reference configurations $(\bB,\Pi)$ owning an AGE and give a comprehensive overview of their structure in a table.

Let $\bB$ be a constant magnetic field and $\Pi$ a cone in $\gP_{3}$. Remind that $d=d(\Pi)$ is the reduced dimension of $\Pi$, cf.\ Definition \ref{def:redcone}. Let us assume that $\En(\bB,\Pi)<\seE(\bB,\Pi)$. Therefore by Theorem \ref{th:dicho} there exists an AGE $\Psi$ that has the form \eqref{eq:age1}. We recall the discriminant parameter $k\in \{1,2,3 \}$ that is the number of directions in which the generalized eigenvector has an exponential decay\index{Exponential decay}. For further use we call (G1), (G2), and (G3) the situation where $k=1$, $2$, and $3$, respectively.
 As a consequence of Lemma \ref{lem:refop}, it is enough to concentrate on \emph{reference configurations} for the magnetic field $\bB$, its potential $\bA$ and the cone $\Pi$. 
 In such a reference configuration the AGE writes as
\[
    \Psi (\by,\bz) = \re^{\ri\ee\vartheta(\by,\bz)}\,\Phi(\bz)\
    \qquad \forall\by\in \R^{3-k},\ \ \forall\bz\in\Upsilon.
\]
\vspace{-1em}

\begin{table}[ht!]
{\footnotesize
\renewcommand{\arraystretch}{1.5}
\begin{tabular}{| p{14mm}|p{27mm}|p{31mm}|p{12mm}|p{20mm}|p{33mm}|}
  \hline
\centering $(k,d)$ $(\by,\bz)$ 
& \centering Reference field $\bB$ and cone $\Pi$ & Reference potential $\bA$ 
&  $\Upsilon$ & Explicit $\Psi$  & $\Phi$ eigenvector of  \\
  \hline
  \hline
\centering $(1,1)$ $(y_1,y_2,z)$ 
& \centering $(0,1,0)$ \quad \ \ \ \ \ \ \ $\Pi=\R^2\times \R_{+}$ & $(z,0,0)$ &  $\R_{+}=\Gamma$ & $\re^{-i\sqrt{\Theta_{0}}y_{1}} \Phi(z)$ & $-\partial_{z}^2+(z-\sqrt{\Theta_{0}})^2$
\\
\hline
\centering $(2,0)$  $(y,z_1,z_2)$ 
& \centering $(1,0,0)$ \quad \ \ \ \ \ \ \ $\Pi=\R^3$ & $(0,-\frac12 z_2, \frac12 z_1)$ & $\R^2$ & 
$\re^{-|\bz|^2/4}$ & $-\Delta_\bz+i\bz\times\!\nabla_\bz+\tfrac{|\bz|^2}{4}$
\\
\hline
\centering $(2,1)$ $(y,z_1,z_2)$  &
 \centering $(0,b_{1},b_{2})$, $b_{2}\neq 0$ \ \ 
 $\Pi=\R^2\times \R_{+}$ & $(b_{1}z_{2}-b_{2}z_{1},0,0)$ &  $\R\times \R_{+}$  & $\Phi(\bz)$ & $-\Delta_{\bz}+(b_1z_2-b_2z_1)^2$
\\
\hline
\centering $(2,2)$ $(y,z_1,z_2)$ 
&  \centering $(b_{0},b_{1},b_{2})$ \ \ \  \ \ \  $\Pi=\R\times \cS_{\alpha}$ & $(b_1z_2-b_2z_1,0,b_0z_1)$ &  $\cS_{\alpha}=\Gamma$ & $\re^{i\tau^{*}y}\Phi(\bz)$ &
  $ \widehat \OP_{\tau}(\uA \ee,\cW_\alpha) $, cf.\ \eqref{D:Hhatsector}
\\
\hline
\centering $(3,3)$  
& \hrulefill & \hrulefill  &  $\Pi=\Gamma$ & $\Phi(\bz)$ & $\OP(\bA,\Pi)$ 
 \\
\hline
\end{tabular}}
\mbox{}\\[1ex]
\caption{AGE of $\OP(\bA,\Pi)$ when $\En(\bB,\Pi) < \seE(\bB,\Pi)$, written in variables $(\by,\bz)$.}
\label{T:age} 
\end{table}

In Table \ref{T:age} we gather all possible situations for the couple of dimensions $(k,d)$. We provide the explicit form of an admissible generalized eigenfunction $\Psi$ of $H(\bA,\Pi)$ in variables $(\by,\bz)\in \R^{3-k}\times \Upsilon$ where $\bA$ is a reference linear potential associated with $\bB$. Note that the cone $\Upsilon$ on which $\Psi$ has exponential decay does not always coincide with the reduced cone $\Gamma$ of $\Pi$.

\begin{remark}
Table \ref{T:age} provides all reference situations where condition (i) of the Dichotomy\index{Dichotomy} Theorem holds. This condition guarantees the existence of an AGE. However there exist cases where this condition does not hold and, nevertheless, there exists an AGE. An example of this is the half-space $\Pi=\R_+\times\R^2$ with coordinates $(y,z_1,z_2)$, and $\bB$ the field $(1,0,0)$ normal to the boundary. We take the same reference potential as in the case $\Pi=\R^3$ and we find,
as described in \cite[Lemma 4.3]{LuPan00}, that the same function $\Psi:(y,\bz)\mapsto e^{-|\bz|^2/4}$  displayed in Row 2 of Table~\ref{T:age} is also an AGE for $H(\bA,\R_+\times\R^2)$, since it satisfies the Neumann boundary conditions at the boundary $y=0$.
\end{remark}

\section{Stability under perturbation} 
Here we describe stability properties of AGE's under perturbations of the magnetic field $\bB$.

Assume that we are in case (i) of the dichotomy (Theorem \ref{th:dicho}). We recall that the notations (G1), (G2) and (G3) refer to the number $k=1,2,3$, of independent decaying directions for the AGE, cf.\ Section \ref{s:age}. 
We first note that we do not need any stability analysis in situation (G3) since the points $\bx$ in $\Omega\in \gD(\R^3)$ for which $d(\Pi_{\bx})=3$ are but corners, so they are isolated. By contrast, points in situation (G1) or (G2) are not isolated, in general.
A perturbation of the magnetic field has distinct effects in each case. The geometrical situation leading to (G1) is clearly not stable. However, we prove in the following lemma the local stability of case (i) of the dichotomy, together with local uniform estimates for exponential decay in situation (G2).
 
\begin{lemma} \label{L:Elip}
Let $\bB_{0}$ be a nonzero constant magnetic field and $\Pi$ be a cone in $\gP_{3}$ with reduced dimension $d\leq 2$. Assume that
$\En(\bB_{0},\Pi)<\seE(\bB_{0},\Pi)$. 
\begin{enumerate}[\textup{(a)}]
\item There exists a positive $\varepsilon_0$ such that in the ball $\cB(\bB_{0},\varepsilon_0)$, the function  $\bB\mapsto \En(\bB,\Pi)$ is Lipschitz-continuous 
and
$$
\En(\bB,\Pi)<\seE(\bB,\Pi) \quad \forall\bB\in\cB(\bB_{0},\varepsilon_0).
$$
\item[\textup{(b)}]
We suppose moreover that $(\bB_{0},\Pi)$ is in situation \textup{(G2)}. 
For $\bB\in\cB(\bB_{0},\varepsilon_0)$, we denote by $\Psi^{\bB}$ an AGE\index{Admissible Generalized Eigenvector!AGE} given by Theorem \Ref{th:dicho}. 
Then there exists $\varepsilon_1\in(0,\varepsilon_0]$ such that $(\bB,\Pi)$ is still in situation \textup{(G2)} if $\bB\in \cB(\bB_{0},\varepsilon)$ and $\Psi^{\bB}$ has the form\index{Exponential decay}
$$\Psi^{\bB}(\bx)=\re^{i\varphi^{\bB}(\by,\bz)}\Phi^{\bB}(\bz)\quad\mbox{for }\quad 
\udiffeo^{\bB}\bx=(\by,\bz)\in\R\times\Upsilon,$$
with $\udiffeo^{\bB}$ a suitable rotation, and there exist constants $c_{\rm e}>0$ and $C_{\rm e}>0$ such that there hold the uniform exponential decay estimates
\begin{equation}
\label{E:agmonuniform}
   \forall \bB\in \cB(\bB_{0},\varepsilon_1), \quad 
   \|\Phi^{\bB}\re^{c_{\rm e}|\bz|}\|_{L^2(\Upsilon)}\leq C_{\rm e} \|\Phi^{\bB} \|_{L^2(\Upsilon)}  \, .
\end{equation}
\end{enumerate}
\end{lemma}

\begin{proof}
Let us distinguish the three possible situations according to the value of $d$:
\begin{itemize}
\item[$d=0\,$:] When $\Pi=\R^3$, we have $\En(\bB,\Pi)=|\bB|$ and $\seE(\bB,\Pi)=+\infty$. 
Combining Row 2 of Table \ref{T:age} and Lemma \ref{lem:refop}, the admissible generalized eigenvector $\Psi^{\bB}$ is explicit.
Thus (a) and (b) are established in this case. 

\item[$d=1\,$:] When $\Pi$ is a half-space, we denote by $\theta(\bB)$ the unoriented angle in $[0,\frac\pi2]$ between $\bB$ and the boundary. Then $\En(\bB,\Pi)=|\bB|\,\sigma(\theta(\bB))$. The function $\bB\mapsto\theta(\bB)$ is Lipschitz outside $\{\bfz\}$ and, moreover, the function $\sigma$ is $\sC^1$ on $[0,\pi/2]$ (see Lemma \ref{P:continuitesigma}).
We deduce that the function $\bB\mapsto\sigma(\theta(\bB))$ is Lipschitz outside $\{\bfz\}$. Thus point (a) is proved. 
Assuming furthermore that $(\Pi,\bB_{0})$ is in situation (G2), we have $\theta(\bB_{0})\in (0,\frac{\pi}{2})$ and there exist $\varepsilon>0$, $\theta_{\min}$ and $\theta_{\max}$ such that 
$$
   \forall \bB \in \cB(\bB_{0},\varepsilon), \quad 
   \theta(\bB)\in [\theta_{\min},\theta_{\max}]
   \subset (0,\tfrac{\pi}{2}) \, .
$$
The admissible generalized eigenvector is constructed above.  
The uniform exponential estimate is proved in \cite[\S2]{BoDauPopRay12}.

\item[$d=2\,$:]  When $\Pi$ is a wedge, point (a) comes from \cite[Proposition 4.6]{Pop13}. 
Due to the continuity of $\bB\mapsto \En(\bB,\Pi)$ there exist $c>0$ and $\varepsilon>0$ such that
$$\forall\bB\in \cB(\bB_{0},\varepsilon),\quad \seE(\bB,\Pi)-\En(\bB,\Pi)>c.$$ 
Point (b) is then a direct consequence of \cite[Proposition 4.2]{Pop13}.
\end{itemize}
The proof of Lemma~\ref{L:Elip} is complete.
\end{proof}

\begin{remark}
\label{rem:10.3}
The latter lemma  can be generalized in several directions.
\begin{enumerate}[a)]
\item Lemma~\ref{L:Elip} (a) is still valid when $d=3$. This can be proved by arguments similar to those employed in \cite[Section 4]{Pop13} for wedges.
\item When $d=2$, it is proved in \cite{Pop13} that the ground state energy is also Lipschitz with respect to the aperture angle of the wedge in case (i) of the Dichotomy Theorem, whereas one can prove only $\frac{1}{3}$-H\"older regularity under perturbations in the general case ({\it i.e.}, without the condition $\En(\bB_{0},\Pi)<\seE(\bB_{0},\Pi)$). 
\end{enumerate}
\end{remark}

\begin{remark} A constant magnetic field enters the family of long range magnetic fields. 
So Lemma \ref{L:Elip} can be related to some spectral analyses of Schr\"odinger operators in $\R^n$ under long range magnetic perturbations. 
Such perturbations do not pertain to the usual Kato theory. When the spectrum has a band structure, the question of the stability of, e.g., its lower bound with respect to the strength of the perturbation has been addressed by many authors, see for example \cite{AvHeSi82,AvSi85} for the continuity, then \cite{Nen85,BriCor02} for H\"older properties, and \cite{Cor10} for Lipschitz continuity in the case of constant magnetic fields.
\end{remark}

As a consequence of the local uniform estimate \eqref{E:agmonuniform}, we obtain the following local uniform version of Lemma \ref{lem:rhoh} for situation (G2).

\begin{lemma}
\label{lem:rhohb}
Let $\bB_{0}$ be a nonzero constant magnetic field and $\Pi$ a cone in $\gP_{3}$.
Assume that $\En(\bB_{0},\Pi)<\seE(\bB_{0},\Pi)$ and that $k=2$.
With $\varepsilon_1$ given in Lemma \Ref{L:Elip} \textup{(b)}, for any $\bB\in \cB(\bB_{0},\varepsilon_1)$
let $\Psi^{\bB}$ be an AGE for $(\bB,\Pi)$.
Let $\delta_0<\frac12$ be a positive number. 
Let $\Psi^\bB_{h}$  be the rescaled function given by \eqref{eq.AGEsc} and let $\chi_{h}$ be the cut-off function defined by \eqref{D:chi}--\eqref{eq.chih} involving parameters $R>0$ and $\delta\in[0,\delta_0]$.
Let $R_0>0$.
Then there exist constants $h_0>0$, $C_1>0$ depending only on $R_0$, $\delta_0$ and on the constants $c_{\rm e}$, $C_{\rm e}$ in \eqref{E:agmonuniform} such that\index{Cut-off}
$$
   \rho_{h}=\frac{\|\,|\nabla\chi_{h}|\, \Psi^\bB_{h}\|^2}{\|\chi_{h}\Psi^\bB_{h}\|^2}\leq
    C_1\, h^{-2\delta}   \qquad \forall R\ge R_0,\ \forall h\le h_0,\ \forall\delta\in[0,\delta_0]\,.
$$
\end{lemma}

\begin{proof}
We obtain an upper bound of $\|\,|\nabla\chi_{h}|\,\Psi^\bB_{h}\|^2$ as in the proof of Lemma~\ref{lem:rhoh}. Let us now deal with the lower-bound of $\|\chi_{h}\Psi^\bB_{h}\|^2$. With $T=R h^\delta$ and $k=2$, we have
\begin{eqnarray}
\|\chi_{h}\Psi^\bB_{h}\|^2 
   &\geq & C T^{3-k} h^{k/2} \, \int_{\Upsilon\, \cap\, \big\{2|\bz|\leq Th^{-\frac12}\big\}} 
   \left|\Phi^\bB(\bz)\right|^2\rd \bz 
\nonumber\\
   &\geq& C  T^{3-k} h^{k/2} \,  \left(1-C_{\rm e}\re^{-c_{\rm e} R h^{\delta-1/2}}\right) 
   \|\Phi^\bB\|^2_{L^2(\Upsilon)}.\label{eq.minpsih}
\end{eqnarray}
Since $0\leq\delta\leq\delta_{0}<\frac 12$, there holds $C_{\rm e}\re^{-c_{\rm e} R h^{\delta-1/2}}<\frac12$ for $h$ small enough or $R$ large enough. Thus we deduce the lemma.
\end{proof}

\chapter{Improvement of upper bounds for more regular magnetic fields}
\label{sec:Improv}


For our improvement of remainders, in comparison with Theorem \ref{T:generalUB} our sole additional assumption is a supplementary regularity on the magnetic potential (or equivalently on the magnetic field). Our result is also general, in the sense that it addresses general corner domains. 

\begin{theorem}
\label{T:sUB}
Let $\Omega\in \gD(\R^3)$ be a general corner domain\index{Corner domain},  $\bA\in \sC^{3}(\overline{\Omega})$ be a magnetic potential such that the associated magnetic field does not vanish.
\begin{enumerate}
\item[(i)] 
 Then there exist $C(\Omega)>0$ and $h_{0}>0$ such that
\begin{equation}
\label{eq:above3sup}
   \forall h\in (0,h_{0}), \quad \lambda_{h}(\bB,\Omega) \leq
   h\sE(\bB,\Omega)+C(\Omega) \big(1+\|\bA\|_{W^{3,\infty}(\Omega)}^2 \big)h^{9/8}\ .  
\end{equation}
\item[(ii)]  If $\Omega$ is a polyhedral domain, this upper bound is improved: 
\begin{equation}
\label{eq:above1sup}
   \forall h\in (0,h_{0}), \quad \lambda_{h}(\bB,\Omega) \leq
   h\sE(\bB,\Omega)+C(\Omega) \big(1+\|\bA\|_{W^{3,\infty}(\Omega)}^2 \big)h^{4/3} \ . 
\end{equation}
\end{enumerate}
\end{theorem}

The strategy is to optimize the construction of adapted sitting\index{Quasimode!Sitting} or sliding\index{Quasimode!Sliding} quasimodes by taking actually advantage of the decaying properties of AGE's $\Psi^\dx$ associated with the minimal energy $\sE(\bB,\Omega)$. 
In fact, our proof of the $h^{11/10}$ or $h^{5/4}$ upper bounds as done in Chapter \ref{sec:up} weakly uses the exponential decay\index{Exponential decay} of generalized eigenfunctions in some directions. It would also work with purely oscillating generalized eigenfunctions. 
Now the proof of the $h^{9/8}$ or $h^{4/3}$ upper bound makes a more extensive use of fine properties of the model problems: First, the decay properties of admissible generalized eigenvectors, and second, the Lipschitz 
regularity of the ground state energy depending on the magnetic field, cf.\ Lemma \ref{L:Elip}. 

The method depends on the number $k$ of directions in which $\Psi^\dx$ has exponential decay, namely whether we are in situation (G1), (G2) or (G3). Indeed, situation (G3) is already handled in Theorem \ref{T:generalUB} \emph{(d)} and we have already obtained a better estimate in this case. So it remains situations (G1) and (G2) which are considered in Section \ref{SS:G1} and \ref{SS:G2}, respectively. 

Like for Theorem \ref{T:generalUB} we start from suitable AGE's $\Psi^\dx$ and construct sitting or sliding quasimodes adapted to the geometry. In comparison with the proof of Theorem \ref{T:generalUB}, the strategy is to improve step [c] that consists in the linearization of the magnetic potential, see Section \ref{SS:QM} and Figure~\ref{F:QM}: We take more precisely advantage of the decaying property of the AGE $\Psi^\dx$, choosing coordinates in which $\Psi^\dx$ takes the form of reference, as listed in Table \ref{T:age}. Then we adopt different strategies depending on whether we are in situation (G1) or (G2): The improvement relies on a Feynman-Hellmann formula\index{Feynman-Hellmann formula} for (G1), and a refined Taylor expansion of the potential for (G2) 

We recall that $\bx_{0}\in \overline{\Omega}$ is a point such that 
$\En(\bB_{\bx_{0}}\ee,\Pi_{\bx_{0}})= \sE(\bB \ee,\Omega)$.
Theorem \ref{th:dicho} and Remark~\ref{rem:chaine} provide the existence  
of a singular chain $\dx$ that satisfies
$$
\sE(\bB \ee,\Omega) = \En(\bB_{\bx_{0}},\Pi_{\bx_{0}})=
\En(\bB_{\dx},\Pi_{\dx})< \seE(\bB_{\dx},\Pi_{\dx}).
$$ 
We now split our analysis according to the two geometric configurations (G1) and (G2):
 \begin{itemize}
\item[(G1)] $\Pi_{\dx}$ is a half-space and $\bB_{\bx_{0}}$ is tangent to the boundary,
{\em cf.} Row 1 of Table \ref{T:age}.
\item[(G2)] We are in one of the following situations:
\begin{itemize}
\item $\Pi_{\dx}=\R^3$,
{\em cf.} Row 2 of Table \ref{T:age},
\item $\Pi_{\dx}$ is a half-space, $\bB_{\bx_{0}}$ is neither tangent nor normal to $\partial\Pi_{\dx}$,
{\em cf.} Row 3 of Table \ref{T:age},
\item $\Pi_{\dx}$ is a wedge,
{\em cf.} Row 4 of Table \ref{T:age}.
\end{itemize}
\end{itemize}

In each configuration, the estimates concerning the constructed quasimodes depend on the length $\nu$ of the chain $\dx$ and on whether $\bx_0$ is a conical point or not. The relevant categories of quasimodes are qualified as \emph{sitting} ($\nu=1$), \emph{hard sliding}\index{Quasimode!Sitting} ($\nu=2$, $\bx_0$ conical point), \emph{soft sliding} ($\nu=2$, $\bx_0$ not a conical point)\index{Quasimode!Sliding}, and \emph{doubly sliding}\index{Quasimode!Doubly sliding} ($\nu=3$), see Section \ref{SS:QM}.

\section{(G1) One direction of exponential decay}  
\label{SS:G1}
In situation (G1) the generalized eigenfunction has exponential decay in one variable $z$. The upper bounds \eqref{E:Qd1} and \eqref{E:Qd2} are obtained by a Cauchy-Schwarz inequality. We are going to improve them, going back to the identity \eqref{eq:diffAA'} and using a Feynman-Hellmann formula to simplify the cross term in \eqref{eq:diffAA'}.

In situation (G1) $\Pi_{\dx}$ is a half-space and $\bB_{\dx}$ is tangent to its boundary. 
Denote by $(\by,z)=(y_1,y_2,z)\in \R^2\times \R_{+}$ a system of coordinates of $\Pi_{\dx}$ such that $\bB_{\dx}$ is tangent to the $y_{2}$-axis. In these coordinates, the magnetic field $\bB_{\dx}$ writes $(0,b,0)$.

In the rest of this proof, we will assume without restriction that $b=1$. Indeed, once quasimodes are constructed for $b=1$, Lemmas \ref{lem.dilatation} and \ref{lem:sense} allow to convert them into quasimodes for any $b$. Thus we have $\Lx=\Theta_0$, {\em cf.} Row 1 of Table \ref{T:age}.

The principle of the quasimode construction is to replace the last relation \eqref{eq:psih} $\phihX{\nu} = \chi_{h}\Psi^\dx_h$ with the new relation
\begin{equation}
\label{eq:psihG1}
   \phihX{\nu} = \udiffeo_*\circ\diffeoZ^F_h(\chiboxh\uPsi_h)
\end{equation}
where $\udiffeo$ is the rotation $\bx\mapsto\bx^{\natural}:=(\by,z)$ that maps $\Pi_{\dx}$ onto the reference half-space $\R^2\times \R_{+}$, 
the function $\chiboxh$ is the cut-off in tensor product form (here for simplicity we denote $\underline\chi_{R}$ by $\chi$) defined as
\begin{equation}
\label{eq:chiboxh}
   \chiboxh(\by,z) = \chi\Big(\frac{|\by|}{h^\delta}\Big)\,\chi\Big(\frac{z}{h^\delta}\Big)
\end{equation}
$\diffeoZ^F_h$ is a change of gauge\index{Gauge transform} and $\uPsi_h$ a canonical generalized eigenvector defined as follows.

The canonical reference potential (see Row 1 of Table \ref{T:age})
\begin{equation}\label{eq.uA}
\uA(\by,z)=(z,0,0),
\end{equation}
is such that $\curl \uA=(0,1,0)$. 
We know (see Section \ref{s:age}) that the function
\begin{equation}\label{eq.uPsi}
 \uPsi_{h}(\by,z):=
 \re^{-i\sqrt{\Theta_{0}}\,y_{1}/\sqrt{h}} \,  \Phi\Big(\frac{z}{\sqrt{h}}\Big)
 \end{equation}
is a generalized eigenvector of $H_{h}(\uA,\R^2\times\R_{+})$ for the value $h\Theta_0$. Here $\Phi$ is a normalized eigenvector associated with the first eigenvalue of the de Gennes operator\index{de Gennes operator} $-\partial_{z}^2+(z-\sqrt{\Theta_{0}})^2$. By identity \eqref{eq:rhoh} and Lemma \ref{lem:rhoh} we obtain the cut-off estimate\index{Cut-off}
\begin{equation}
\label{eq:psinuG1}
   \QR_{h}[\uA,\R^2\times\R_{+}](\chiboxh\uPsi_h) =
   h\Theta_0 + \cO(h^{2-2\delta}) = h\Lx + \cO(h^{2-2\delta}).
\end{equation}
Let $\rJ$ be the matrix associated with $\udiffeo$. In variables $\bx^\natural$, the tangent potential $\bA_{\dx}$ is transformed into the potential $\bA^\natural_\bfz$
\begin{equation}
\label{eq:Anat0}
   \bA^\natural_\bfz(\bx^\natural) = \rJ^\top(\bA_{\dx}(\bx)),
\end{equation}
that satisfies
\[
   \curl\bA^\natural_\bfz = \curl\uA \,.
\]
Since $\uA$ and $\bA^\natural_\bfz$ are both linear,  
there exists a homogenous polynomial function of degree two $F^\natural$ such that 
\begin{equation}\label{eq.uAnat}
\bA^\natural_\bfz-\nabla_{\natural} F^\natural= \uA.
\end{equation} 
Therefore, $\re^{-iF^\natural/h}\uPsi_{h}$ is an admissible generalized eigenvector for 
$\OP_{h}(\bA^\natural_\bfz,\R^2\times \R_{+})$ associated with the value $h\Lx$. 

\subsection{Sitting quasimodes}\index{Quasimode!Sitting}
This is the case when $\nu=1$ and $\dx=(\bx_0)$. Thus $\Pi_{\bx_0}$ coincides with $\Pi_\dx$. 
We keep relation \eqref{eq:phihX0} linking $\phihX{0}$ to $\phihX{1}$ and $\phihX{1}$ is now defined by the formula
\begin{equation}
\label{D:spmG11}
\phihX{1}(\bx)=\re^{-iF^\natural(\bx^\natural)/h} \chiboxh(\by,z)\uPsi_h(\by,z)
= \re^{-iF^\natural(\bx^\natural)/h} \upsi_h(\by,z),
\qquad \forall \bx\in\Pi_{\dx}\,,
\end{equation}
Here we set for shortness
\[
   \upsi_h:=\chiboxh\uPsi_h \quad\mbox{and}\quad \cVboxh := \supp(\chiboxh).
\]
Let $\rJ$ be the matrix associated with $\udiffeo$. Let $\bA^\natural$ be the magnetic potential associated with $\bA^{\bx_0}$ in variables $\bx^\natural$: 
\begin{equation}
\label{E:Anat1}
   \bA^\natural(\bx^\natural) = \rJ^\top\big(\bA^{\bx_{0}}(\bx)\big)
   \quad \forall\bx\in\cV_{\bx_0}\,.
\end{equation}
Then $\bA^\natural_{\bfz}$ \eqref{eq:Anat0} is its linear part at $\bfz$.

We have
\begin{align}
\label{eq:G1qm}
  \QR_{h}[\bA^{\bx_0},\Pi_{\bx_{0}}](\phihX{1}) 
  &=   \QR_{h}[\bA^\natural,\R^2\times \R_{+}](\re^{-iF^\natural/h}\upsi_h) \\
  &=   \QR_{h}[\bA^\natural-\nabla F^\natural,\R^2\times \R_{+}](\upsi_h). 
  \nonumber 
\end{align}
Now we  apply \eqref{eq:diffAA'} with $\bA=\bA^\natural-\nabla F^\natural$ and $\bA'=\uA$.  Using \eqref{eq.uAnat} we find $\bA-\bA'=\bA^\natural-\bA^\natural_{\bfz}$, and write, instead of \eqref{E:Qd1}
\begin{align} 
\label{eq:G1refest}
  q_{h}[\bA^\natural-\nabla F^\natural,\R^2\times \R_{+}](\upsi_{h}) 
  &=   q_{h}[\uA,\R^2\times \R_{+}](\upsi_{h})\\
  &\hskip-3em  +2  \Re\int_{\R^2\times \R_{+}} (-ih\nabla+\uA)\upsi_{h}(\bx^\natural)\cdot 
  (\bA^\natural-\bA^\natural_{\bfz})(\bx^\natural)\,\overline{\upsi_{h}(\bx^\natural})\,
  \rd\bx^\natural \label{E:termcrois}\\
  &\hskip-3em  +\|(\bA^\natural-\bA^\natural_{\bfz})\upsi_{h}\|^2.
  \label{eq.RayleighG1}
\end{align}
As in Section \ref{SS.phihX0} [e1], we bound from above the term \eqref{eq.RayleighG1}  using Lemma~\ref{lem.TaylorA}
\begin{equation}
\label{eq:G1carr}
   \|(\bA^\natural-\bA_{\bfz}^\natural)\upsi_{h}\|^2 \le
   C(\Omega)\|\bA^\natural\|^2_{W^{2,\infty}( \cVboxh)} \,
   h^{4\delta}\, \|\upsi_{h}\|^2 .
\end{equation} 

Let us now deal with the term \eqref{E:termcrois}. We calculate $(-ih\nabla+\uA)\upsi_{h}$ 
using \eqref{eq.uPsi}:
\begin{multline*}
   (-ih\nabla+\uA)\upsi_{h}(\bx^\natural)
   = \re^{-i\sqrt{\Theta_{0}}\,y_{1}/\sqrt{h}} \,\times \\
   \left\{\chi\big(\tfrac{|\by|}{h^{\delta}}\big)\ 
   \chi\big(\tfrac{z}{h^{\delta}}\big)
   \begin{bmatrix}
   (z-\sqrt{h\Theta_{0}}\ee)\,\Phi\big(\tfrac{z}{\sqrt{h}}\big) \\[5pt]
   0\\[5pt]
   -i\sqrt h \,\Phi'\big(\tfrac{z}{\sqrt{h}}\big) 
   \end{bmatrix}
   -ih^{1-\delta}  
\begin{bmatrix}
   \frac{y_{1}}{|\by|}\chi'(\tfrac{|\by|}{h^{\delta}})\ 
\chi(\tfrac{z}{h^{\delta}})\\[5pt]
   \frac{y_{2}}{|\by|} \chi'(\tfrac{|\by|}{h^{\delta}})\ 
\chi(\tfrac{z}{h^{\delta}})\\[5pt]
   \chi(\tfrac{|\by|}{h^{\delta}})\ \chi'(\tfrac{z}{h^{\delta}})
\end{bmatrix}
   \Phi\big(\tfrac{z}{\sqrt{h}}\big) \right\}.
\end{multline*}
Since $\Phi$ and $\chi$ are real valued functions, the term \eqref{E:termcrois} reduces to a single term: 
\begin{align}\label{eq.termecrois}
\Re\int_{\R^2\times \R_{+}} (-ih\nabla&+\uA)\upsi_{h}(\bx^\natural)\cdot (\bA^\natural-\bA^\natural_{\bfz})(\bx^\natural)\overline{\upsi_{h}(\bx^\natural})\,\rd\bx^\natural \\
   &=\int_{\R^2\times \R_{+}} (z-\sqrt{h\Theta_{0}}\ee)\ 
   |\upsi_{h}(\bx^\natural)|^2 A^{(\mathrm{rem},2)}_{1}(\bx^\natural) \, \rd\bx^\natural
   \nonumber\\
   &=\int_{\R^2\times \R_{+}}(z-\sqrt{h\Theta_{0}}\ee)\,
   \left|\Phi\big(\tfrac{z}{\sqrt{h}}\big)\right|^2
   \ \left|\chi\big(\tfrac{|\by|}{h^{\delta}}\big)\right|^2\ 
   \left|\chi\big(\tfrac{z}{h^{\delta}}\big)\right|^2
   A^{(\mathrm{rem},2)}_{1}(\bx^\natural)\,\rd\bx^\natural,\nonumber
\end{align}
where $A^{(\mathrm{rem},2)}_{1}$ denotes the first component of $\bA^\natural-\bA_{\bfz}^\natural$.
We write 
\begin{equation}
\label{eq:Arem2}
   A^{(\mathrm{rem},2)}_{1}(\bx^\natural)=
   P_{1}^{(2)}(\by)+R_{1}^{(2)}(\bx^\natural)+A^{(\mathrm{rem},3)}_{1}(\bx^\natural),
\end{equation}
where $A^{(\mathrm{rem},3)}_{1}$ is the Taylor remainder of degree $3$ of the first component of $\bA^\natural$ at $\bfz$, whereas $P_{1}^{(2)}(\by)+R_{1}^{(2)}(\bx^\natural)$ is a representation of its quadratic part in the form
$$
   P_{1}^{(2)}(\by)=a_{1}y_{1}^2+a_{2}y_{2}^2+a_{3}y_{1}y_{2}\quad\mbox{and}\quad
   R_{1}^{(2)}(\bx^\natural)=b_{1} z^2+b_{2}z y_{1}+b_{3} zy_{2}.
$$
As in \eqref{eq:rem3}, we have:
$$
   \|A^{(\mathrm{rem},3)}_{1}\|_{L^\infty(\cVboxh)}\leq 
    C\|\bA^{\natural}\|_{W^{3,\infty}( \cVboxh)} \,h^{3\delta},
$$
leading to, with the help of the variable change $Z=z/\sqrt h$ and the exponential decay of $\Phi$:
\begin{equation}
\label{eq:G1A3}
   \left|\int_{\R^2\times \R_{+}} (z-\sqrt{h\Theta_{0}}\ee)\,|\upsi_{h}(\bx^\natural)|^2\, 
    A^{(\mathrm{rem},3)}_{1}(\bx^\natural)\,\rd\bx^\natural\right|
   \leq C\|\bA^{\natural}\|_{W^{3,\infty}( \cVboxh)} h^{\frac12+3\delta} \|\upsi_{h}\|^2 \ .
\end{equation}
Likewise, combining the exponential decay\index{Exponential decay} of $\Phi$, the change of variable $Z=z/\sqrt h$ and the localization of the support\index{Cut-off} in balls of size $C h^\delta$, we deduce
\begin{equation}
\label{eq.G1majo1}
   \left| \int_{\R^2\times \R_{+}} (z-\sqrt{h\Theta_{0}}\ee)\,|\upsi_{h}(\bx^\natural)|^2\, 
    R_{1}^{(2)}(\bx^\natural)\,\rd\bx^\natural\right|
   \leq C \|\bA^{\natural}\|_{W^{2,\infty}( \cVboxh)}
   h^{\min(\frac32,1+\delta)} \|\upsi_{h}\|^2 \ .
\end{equation}
Let us now deal with the term involving $\by\mapsto P_{1}^{(2)}(\by)$. 
Due to a Feynman-Hellmann formula\index{Feynman-Hellmann formula} applied to the de Gennes operator\index{de Gennes operator} $\DG\tau$ at $\tau= -\sqrt{\Theta_0}$ (cf.\ \cite[Lemma A.1]{HeMo01}) we find by the scaling $z\mapsto z/\sqrt{h}$ the identity
$$
   \int_{\R_{+}}(z-\sqrt{h\Theta_{0}}\ee)\,\left|\Phi\big(\tfrac{z}{\sqrt{h}}\big)\right|^2
   \, \rd z = 0 \,. 
$$
Thus we can write 
\begin{align*}
   \int_{\R^2\times \R_{+}} (z-&\sqrt{h\Theta_{0}}\ee)\, |\upsi_{h}(\bx^\natural)|^2\, 
   P_{1}^{(2)}(\by) \,\rd\bx^\natural
\\
   &= \int_{\R^2} P_{1}^{(2)}(\by)\, 
   \left|\chi\big(\tfrac{|\by|}{h^{\delta}}\big)\right|^2 \!\rd\by \ \, 
   \int_{z\in\R_{+}} (z-\sqrt{h\Theta_{0}}\ee)\, 
   \left|\Phi\big(\tfrac{z}{\sqrt{h}}\big)\right|^2 \chi\big(\tfrac{z}{h^{\delta}}\big)^2
   \,\rd z
\\
   &= \int_{\R^2} P_{1}^{(2)}(\by)\, 
   \left|\chi\big(\tfrac{|\by|}{h^{\delta}}\big)\right|^2 \!\rd\by \ \, 
   \int_{z\in\R_{+}} (z-\sqrt{h\Theta_{0}}\ee)\, 
   \left|\Phi\big(\tfrac{z}{\sqrt{h}}\big)\right|^2 
   \left(\chi\big(\tfrac{z}{h^{\delta}}\big)^2 - 1\right)
   \,\rd z.
\end{align*}
The support of the integral in $z$ is contained in $z\ge Rh^\delta$ with $\delta<\frac12$.
Therefore, using once more the changes of variables $\mathbf Y=\by/h^\delta$ and $Z=z/\sqrt{h}$, we find:
$$
   \left| \int_{\R^2\times\R_+} (z-\sqrt{h\Theta_{0}}\ee)\, \ |\upsi_{h}(\bx^\natural)|^2 
   P_{1}^{(2)}(\by) \,\rd\bx^\natural \right|
  \leq C \|\bA^\natural\|_{W^{2,\infty}( \cVboxh)} 
  h^{\frac12+4\delta}\re^{-ch^{\delta-1/2}}.
$$
Since $\|\upsi_{h}\|^2\geq C h^{\frac12+2\delta}$ (see \eqref{eq.minpsih}), 
this leads to: 
\begin{equation}\label{eq.G1majo2}
   \left| \int_{\R^2\times\R_+} (z-\sqrt{h\Theta_{0}}\ee)\, \ |\upsi_{h}(\bx^\natural)|^2 
   P_{1}^{(2)}(\by) \,\rd\bx^\natural \right|
  \leq C \|\bA^\natural\|_{W^{2,\infty}( \cVboxh)} 
  \re^{-ch^{\delta-1/2}}\,\|\upsi_{h}\|^2\ . 
\end{equation}
Collecting \eqref{eq:G1A3}, \eqref{eq.G1majo1}, and \eqref{eq.G1majo2} in \eqref{E:termcrois}, we find the upper bound 
\begin{multline}
\label{eq:G1crois}
   \left|\Re\int_{\R^2\times \R_{+}} (-ih\nabla +\bA_{\bfz}^\natural)\upsi_{h}(\bx^\natural) 
   \cdot (\bA^\natural-\bA_{\bfz}^\natural)\overline{\upsi_{h}(\bx^\natural)}
   \,\rd\bx^\natural\right|\\
   \leq C \Big(
   \|\bA^\natural\|_{W^{3,\infty}( \cVboxh)} \,h^{\frac 12+3\delta} +
   \|\bA^\natural\|_{W^{2,\infty}( \cVboxh)} \,h^{1+\delta} 
   \Big)
   \|\upsi_{h}\|^2 \ . 
\end{multline}
Returning to \eqref{eq:G1qm} via \eqref{eq:G1refest} and combining \eqref{eq:G1crois} with \eqref{eq:G1carr}, we deduce
\begin{multline*}
  \QR_{h}[\bA^{\bx_0},\Pi_{\bx_{0}}](\phihX{1})
  \leq \QR_{h}[\uA,\R^2\times \R_{+}](\upsi_{h}) \\ +
  C\Big( \|\bA^\natural\|_{W^{3,\infty}( \cVboxh)}\, h^{\frac 12+3\delta} + 
  \|\bA^\natural\|_{W^{2,\infty}( \cVboxh)} \,h^{1+\delta}  +
  \|\bA^\natural\|_{W^{2,\infty}( \cVboxh)}^2 \,h^{4\delta} \Big).
\end{multline*}
Inserting the cut-off error  \eqref{eq:psinuG1} for $q_{h}[\uA,\R^2\times\R_+](\upsi_{h})$ we obtain
\begin{multline}
\label{eq:G1sitt}
  \QR_{h}[\bA^{\bx_0},\Pi_{\bx_{0}}](\phihX{1})
  \leq h\Lx + 
  C\,h^{2-2\delta}\\ + 
  C\Big( \|\bA^\natural\|_{W^{3,\infty}( \cVboxh)}\, h^{\frac 12+3\delta} + 
  \|\bA^\natural\|_{W^{2,\infty}( \cVboxh)} \,h^{1+\delta}  +
  \|\bA^\natural\|_{W^{2,\infty}( \cVboxh)}^2 \,h^{4\delta} \Big).
\end{multline}
Using Lemma~\ref{L:d2A} for case (i) we deduce the uniform bound for the derivatives of the potential
\[
   \|\bA^\natural\|_{W^{3,\infty}( \cVboxh)} \leq
   C\|\bA^{\bx_0}\|_{W^{3,\infty}(\cV_{\bx_0})} \leq C'\|\bA\|_{W^{3,\infty}(\Omega)}.
\]
Thus, we deduce from \eqref{eq:G1sitt}
\begin{equation*}
  \QR_{h}[\bA^{\bx_0},\Pi_{\bx_{0}}](\phihX{1})
  \leq h\Lx + 
  C(\Omega)(1+\| \bA \|_{W^{3,\infty}(\Omega)}^2)
  (h^{2-2\delta}+h^{1+\delta} +h^{\frac 12+3\delta} + h^{4\delta}).
\end{equation*}
The quasimode $\phihX{0}$ on $\Omega$ being still defined by \eqref{eq:phihX0}, we deduce from \eqref{E:Qc1} with $r^{[1]}_h=\cO(h^\delta)$ the final estimate
\begin{equation}\label{eq:Ray45}
  \QR_{h}[\bA,\Omega](\phihX{0})
   \leq h\Lx+C(\Omega)(1+\| \bA \|_{W^{3,\infty}(\Omega)}^2)
  (h^{2-2\delta}+h^{1+\delta}+h^{\frac12+3\delta}+h^{4\delta}) \, .
\end{equation}
Choosing $\delta=\frac{1}{3}$ we optimize remainders and deduce the upper bound \eqref{eq:above1sup} in situation (G1)--sitting.

\subsection{Hard sliding}\index{Quasimode!Sliding} This is the case when $\nu=2$ and $\bx_0\in\gVc$ ({\it i.e.}, $\bx_0$ is a conical point). So $\dx=(\bx_0,\bx_{1})$ and $\Pi_{\bx_0,\bx_1}$ coincides with $\Pi_\dx$. We keep relations \eqref{eq:phihX0} and \eqref{eq.c2} linking $\phihX{0}$ to $\phihX{1}$ and $\phihX{1}$ to $\phihX{2}$, respectively, and $\phihX{2}$ is now defined by the formula
\begin{equation}
\label{D:spmG12}
\phihX{2}(\bx)=\re^{-iF^\natural(\bx^\natural)/h} \chiboxh(\by,z)\uPsi_h(\by,z)
= \re^{-iF^\natural(\bx^\natural)/h} \upsi_h(\by,z),
\qquad \forall \bx\in\Pi_{\dx}\,,
\end{equation}
and $\bA^\natural$ is the magnetic potential associated with $\bA^{\dec_1}$ (step [a2]) in variables $\bx^\natural$,
\begin{equation}
\label{E:Anat2}
   \bA^\natural(\bx^\natural) = \rJ^\top\big(\bA^{\dec_1}(\bx)\big)
   \quad \forall\bx\in\cV_{\dec_1}\,.
\end{equation}
 We recall that $\Pi_{\dec_{1}}=\Pi_{\dx}$. We have, instead of \eqref{eq:G1qm}:
\begin{equation}
\label{eq:G1qm2}
  \QR_{h}[\bA^{\dec_{1}},\Pi_{\dec_{1}}](\phihX{2}) 
  =   \QR_{h}[\bA^\natural-\nabla F^\natural,\R^2\times \R_{+}](\upsi_h), 
\end{equation}
and \eqref{E:Qd2} is replaced by the analysis of \eqref{eq:G1refest}--\eqref{eq.RayleighG1} which goes along the same lines as before, ending up at, instead of \eqref{eq:G1sitt}
\begin{multline}
\label{eq:G1sitt2}
  \QR_{h}[\bA^{\dec_1},\Pi_{\dec_1}](\phihX{2})
  \leq h\Lx + 
  C\,h^{2-2\delta}\\ + 
  C\Big( \|\bA^\natural\|_{W^{3,\infty}( \cVboxh)}\, h^{\frac 12+3\delta} + 
  \|\bA^\natural\|_{W^{2,\infty}( \cVboxh)} \,h^{1+\delta}  +
  \|\bA^\natural\|_{W^{2,\infty}( \cVboxh)}^2 \,h^{4\delta} \Big).
\end{multline}
But now we have to use Lemma~\ref{L:d2A} for case (ii) after specifying the different scales: As in Section \ref{SS.phihX1} step [e2] (b) we take $|\dec_1|=d^{[1]}_h=\cO(h^{\de0})$ and $\delta=\de0+\de1$, so the support of $\upsi_{h}$ is contained in a ball of radius $r^{[2]}_h=\cO(h^{\de0+\de1})$. The radius $r^{[1]}_h$ is a $\cO(h^{\de0})$. By using Remark \ref{R:dKu0},  we can see that \eqref{E:D2Au0} generalizes   to higher derivative of $\bA^{\bv_{1}}$, and thus we may estimate the  derivatives of the potential after change of variables: 
\begin{equation}\label{E:AWl}
   \|\bA^\natural\|_{W^{\ell,\infty}( \cVboxh)} \leq
   C\|\bA^{\dec_1}\|_{W^{\ell,\infty}(\cB(\bfz,r^{[2]}_h))} \leq 
   C' h^{-(\ell-1)\de0}\|\bA\|_{W^{\ell,\infty}(\Omega)},\quad \ell=2,3,
\end{equation}
and \eqref{eq:G1sitt2} provides
\begin{multline*}
  \QR_{h}[\bA^{\dec_1},\Pi_{\dec_1}](\phihX{2})
  \leq h\Lx \\
   + C( 1+\|\bA\|_{W^{3,\infty}(\Omega)}^2) \,
   \Big( h^{2-2\de0-2\de1} + h^{-2\de0} h^{\frac 12+3\de0+3\de1} +
   h^{-\de0} h^{1+\de0+\de1} +
  h^{-2\de0} h^{4\de0+4\de1} \Big).
\end{multline*}
Combining the above inequality with \eqref{E:Qg1} that bounds $\QR_h[\bA,\Omega](\phihX0) -\QR_h[\bA^{\bx_0}_{\bfz},\Pi_{\bx_0}](\phihX1)$ and \eqref{E:Qc2} that bounds $\QR_h[\bA^{\bx_0}_{\bfz},\Pi_{\bx_0}](\phihX1) - \QR_h[\bA^{\dec_1},\Pi_{\bx_0,\bx_1}](\phihX2)$ we find
\begin{multline}
\label{eq:G1sli}
  \QR_{h}[\bA,\Omega](\phihX{0})
  \leq h\Lx    + C( 1+\|\bA\|_{W^{2,\infty}(\Omega)}^2) \,
   \Big(    h^{1+\de0} + h^{\frac 12+2\de0} +
   h^{4\de0} + h^{1+\de1}\Big) \\
   + C( 1+\|\bA\|_{W^{3,\infty}(\Omega)}^2) \,
   \Big( h^{2-2\de0-2\de1} + h^{\frac 12+\de0+3\de1} +
   h^{1+\de1} +
   h^{2\de0+4\de1} \Big).
\end{multline}
Choosing $\de0=\frac{5}{16}$ and $\de1=\frac{1}{8}$, we deduce the upper bound \eqref{eq:above3sup} in situation (G1)--hard sliding.

\subsection{Soft sliding}\label{SS:SS} This is the case when $\nu=2$ and $\bx_0$ is \emph{not a conical point}. We keep relations \eqref{eq:phihX0} and \eqref{eq.c2} linking $\phihX{0}$ to $\phihX{1}$ and $\phihX{1}$ to $\phihX{2}$, respectively, and $\phihX{2}$ is defined by formula \eqref{D:spmG12} as in the hard sliding case. But now the analysis is different because we can take advantage of the fact that the change of variables $\diffeo^{\dec_1}$ is the translation $\bx\mapsto\bx-\dec_1$. Concatenating formulas \eqref{eq.c2} and \eqref{D:spmG12}, we obtain (recall that $\udiffeo$ is the rotation $\bx\mapsto\bx^\natural$)
\begin{equation}
\label{D:spmG13}
   \phihX1 = \diffeoZ^{\dec_{1}}_h \circ \diffeo_*^{\dec_1}
   \circ \udiffeo_* \Big(\re^{-iF^\natural/h} \upsi_h\Big).
\end{equation}
Our aim is a direct evaluation of $\QR_{h}[\bA^{\bx_0},\Pi_{\bx_{0}}](\phihX{1})$, based on the above representation. Here we take the potential $\bA^\natural$ in the canonical half-space $\R^2\times\R_+$ as \eqref{E:Anat1}. Let us set $\dec_1^\natural:=\udiffeo\dec_1$. Then there holds the following sequence of identities, cf.\ \eqref{eq:G1qm} for the last one,
\[
\begin{aligned}
   \QR_{h}[\bA^{\bx_0},\Pi_{\bx_{0}}](\phihX{1}) &= 
   \QR_{h}[\bA^{\bx_0}-\bA^{\bx_0}_\bfz(\dec_1),\Pi_{\dx}]\Big(\diffeo_*^{\dec_1}
   \circ \udiffeo_* \big(\re^{-iF^\natural/h} \upsi_h\big)\Big)\\
&=
   \QR_{h}[\bA^{\bx_0}(\cdot+\dec_1)-\bA^{\bx_0}_{\bfz}(\dec_1),\Pi_{\dx}]\Big(
   \udiffeo_* \big(\re^{-iF^\natural/h} \upsi_h\big)\Big)
\\ &=
   \QR_{h}[\bA^{\natural}(\cdot+\dec_1^{\natural})-\bA_{\bfz}^{\natural}(\dec_1^{\natural}),
   \R^2\times\R_+]\big(\re^{-iF^\natural/h} \upsi_h\big)
\\[0.6ex] &=
   \QR_{h}[\bA^{\natural}(\cdot+\dec_1^{\natural})-\bA_{\bfz}^{\natural}(\dec_1^{\natural}) 
   - \nabla F^\natural,
   \R^2\times\R_+](\upsi_{h}).
\end{aligned}
\]
For the calculation of the potential, we check that
\[
\begin{aligned}
   \bA^{\natural}(\cdot+\dec_1^{\natural})-\bA_{\bfz}^{\natural}(\dec_1^{\natural}) 
   - \nabla F^\natural &= 
   \bA^{\natural}(\cdot+\dec_1^{\natural}) - \bA_{\bfz}^{\natural}(\cdot+\dec_1^{\natural})
   + \bA_{\bfz}^{\natural}(\cdot+\dec_1^{\natural})
   -\bA_{\bfz}^{\natural}(\dec_1^{\natural}) 
   - \nabla F^\natural \\ &=
   \bA^{\natural}(\cdot+\dec_1^{\natural}) - \bA_{\bfz}^{\natural}(\cdot+\dec_1^{\natural})
   + \bA_{\bfz}^{\natural} - \nabla F^\natural \\ &=
   \bA^{\natural}(\cdot+\dec_1^{\natural}) - \bA_{\bfz}^{\natural}(\cdot+\dec_1^{\natural})
   + \uA \,.
\end{aligned}
\]
Then, instead of \eqref{eq:G1refest}-\eqref{eq.RayleighG1} we obtain that $q_{h}[\bA^{\natural}(\cdot+\dec_1^{\natural})-\bA_{\bfz}^{\natural}(\dec_1^{\natural})
  -\nabla F^\natural,\R^2\times\R_+](\upsi_{h})$ is now the sum of the three following terms:
\begin{align*}
  & q_{h}[\uA,\R^2\times\R_+](\upsi_{h})\\
  & \quad +2  \Re\int_{\R^2\times\R_+} (-ih\nabla+\uA)\upsi_{h}(\bx^\natural)\cdot 
  \big(\bA^{\natural}(\bx^\natural+\dec_1^{\natural}) - 
  \bA_{\bfz}^{\natural}(\bx^\natural+\dec_1^{\natural})\big)\,\overline{\upsi_{h}(\bx^\natural})\,
  \rd\bx^\natural \\
  & \quad +\|\big(\bA^{\natural}(\cdot+\dec_1^{\natural}) - 
  \bA_{\bfz}^{\natural}(\cdot+\dec_1^{\natural})\big)\upsi_{h}\|^2.
\end{align*}

Since $|\dec_1|=h^\delta$, the estimate \eqref{eq:G1carr} obviously becomes
\[
   \|(\bA^{\natural}(\cdot+\dec_1^{\natural}) - 
  \bA_{\bfz}^{\natural}(\cdot+\dec_1^{\natural}))\upsi_{h}\|^2 \le
   C(\Omega)\|\bA^\natural\|^2_{W^{2,\infty}(\dec_1^{\natural}+ \cVboxh)} \,
   h^{4\delta}\, \|\upsi_{h}\|^2 .
\]

As for estimates \eqref{eq.termecrois}-\eqref{eq:G1crois} of the crossed term, we may use the fact that the vector $\dir_{1}$ introduced in \eqref{def:tau} belongs to a face of $\Pi_{\bx_{0}}$ (see the prologue of Section~\ref{SS.phihX1}). It is the same for $\dec_1=h^\delta\dir_1$. Therefore $\dec_{1}^\natural$ is tangent to the boundary of $\R^2\times\R_{+}$, it has no component in the $z$ direction and can be written $\dec_{1}^\natural=h^\delta\dir_1^\natural=(h^\delta\bp,0)$ in coordinates $\bx^{\natural}$. We use the same splitting \eqref{eq:Arem2} of the potential, at the point $\bx^\natural+\dec_{1}^\natural$
\[
   A^{(\mathrm{rem},2)}_{1}(\bx^\natural+\dec_{1}^\natural)=
   P_{1}(\by+h^\delta\bp)+R_{1}(\bx^\natural+h^\delta\dir_{1}^\natural)+
   A^{(\mathrm{rem},3)}_{1}(\bx^\natural+h^\delta\dir_{1}^\natural).
\]
Then all estimates \eqref{eq.termecrois}-\eqref{eq:G1crois} of the crossed term are still valid now, replacing the norm in $W^{\ell,\infty}(\supp(\upsi_h))$ by the norm in $W^{\ell,\infty}(\dec^\natural_{1}+\supp(\upsi_h))$ (for $\ell=2,3$). As before we arrive to the upper bound \eqref{eq:Ray45} for the Rayleigh quotient of our quasimode and conclude as in the sitting case.

\subsection{Double sliding}\index{Quasimode!Doubly sliding} This is the case when $\nu=3$. So $\bx_0$ is a conical point. We keep relations \eqref{eq:phihX0} and \eqref{eq.c2} linking $\phihX{0}$ to $\phihX{1}$ and $\phihX{1}$ to $\phihX{2}$, respectively, and $\phihX{2}$ is now defined by the formula
\begin{equation}
\label{D:spmG14}
\phihX{2}(\bx)=
\diffeoZ^{\dec_{2}}_h \circ \diffeo_*^{\dec_2}
   \circ \udiffeo_* \Big(\re^{-iF^\natural/h} \upsi_h\Big).
\end{equation}
and $\bA^\natural$ is the magnetic potential \eqref{E:Anat2} associated with $\bA^{\dec_1}$ (step [a2]) in variables $\bx^\natural$. A reasoning similar to the soft sliding case yields the same conclusion \eqref{eq:G1sli} like in the hard sliding case.

\bigskip
The proof of Theorem \ref{T:sUB} is over in situation (G1).

\section{(G2) Two directions of exponential decay}  \index{Exponential decay}
\label{SS:G2}
In situation (G2) the generalized eigenfunction $\Psi^{\dx}$ has two directions of decay, $z_1$ and $z_2$, leaving one direction $y$ with a purely oscillating character. In this case, we are going to improve the linearization error, namely estimates \eqref{E:Qf1} and \eqref{E:Qf2}: Until now we have used that $\bA^{\bx_{0}}(\bx)-\bA^{\bx_{0}}_{\bfz}(\bx)$ is a $O(|\bx|^2)$. Here, by a suitable phase shift\index{Phase shift} (which corresponds to a change of gauge), we can eliminate from this error the term in $O(|y|^2)$, replacing it by a $O(|y|^3)$. The other terms containing at least one power of $|\bz|$, we can take advantage of the decay of $\Psi^{\dx}$. This phase shift is done by a change of gauge on the last level of construction, that is on the function $\phihX{\nu}$, as in the (G1)-case. The sitting modes will be constructed following exactly this strategy, whereas concerning sliding modes, we have to linearize the potential at a moving point $\dec:=h^{\delta}\dir$, instead of $\bfz$ as previously. Let us develop details now. 
 The quasimode $\phihX{0}$ is still defined on $\Omega$ by formula \eqref{eq:phihX0} $\phihX0 = \diffeoZ^{\bx_0}_h \circ \diffeo_*^{\bx_0}(\phihX1),$ and relations~\eqref{E:Qc1orig0}--\eqref{E:Qc1} are still valid. 

\subsection{Sitting quasimodes}  \index{Quasimode!Sitting}
Here we make an improvement of step [c1], see Figure \ref{F:QM}. Let $\udiffeo$ be the rotation $\bx\mapsto\bx^{\natural}:=(\by,z)$ that maps $\Pi_{\bx_{0}}$ onto the model domain $\R\times\Upsilon$ which equals $\R\times \cS_{\alpha}$, $\R^2\times\R_{+}$ or $\R^3$.
Let $\bA^\natural$ be the magnetic potential associated with $\bA^{\bx_{0}}$ in variables $\bx^\natural$ given by \eqref{E:Anat1} and $\bA^\natural_\bfz$, $\bA^{\bx_{0}}_\bfz(=\bA_{\dx})$ be their linear parts at $\bfz$. 
Applying Lemma \ref{L:d2ell0} in variables $(u_1,u_2,u_3)=(y,z_1,z_2)$ with $\ell=1$ gives us a function $F$ such that $\partial^2_{y}(\bA^\natural-\nabla F)(\bfz) = 0$ leading to the estimates 
\begin{equation}
\label{eq:d2y0est}
   \big|\big(\bA^\natural - \bA^\natural_\bfz - 
   \nabla F\big)(\bx^\natural)\big| 
   \le C(\cV_{\bx_0})\, \big(\|\bA^{\bx_{0}}\|_{W^{2,\infty}(\cV_{\bx_{0}})}\big(|y||\bz| + |\bz|^2\big) +\|\bA^{\bx_{0}}\|_{W^{3,\infty}(\cV_{\bx_{0}})}|y|^3 \big)\,.
\end{equation}
We define our new quasimode by
\begin{equation}
\label{eq:qmX}
   \phihX{1}  = 
   \udiffeo (\re^{-iF/h} \upsi_{h}), 
\quad \mbox{in} \quad \Pi_{\bx_0}, 
\end{equation}
with $\upsi_{h}$ a given function in $\R\times\Upsilon$. 
Using \eqref{E:chgG} and \eqref{eq:diffAA'}, we have
\begin{equation}\label{E:QR1G2}
\QR_{h}[\bA^{\bx_{0}},\Pi_{\dx}](\phihX{1}) 
= \QR_{h} [\bA^\natural-\nabla F,\R\times\!\Upsilon](\upsi_{h})
\leq \mu_{h}^{[1]}+ 2 \hat a_{h}^{[1]}\ \sqrt{\mu_{h}^{[1]}} + (\hat a_{h}^{[1]})^2,
\end{equation}
where we have set, by analogy with \eqref{E:Qe1},
\begin{equation}\label{eq:hata}
   \mu^{[1]}_{h} = \QR_h[\bA^{\natural}_{\bfz},\R\times\!\Upsilon](\upsi_{h})
   \quad\mbox{and}\quad
   \hat a^{[1]}_{h}=\frac{\|(\bA^\natural -\bA^\natural_{\bfz} - \nabla F)\upsi_{h}\|} {\|\upsi_{h}\|}. 
\end{equation}
We set  $\upsi_{h}=\chi_h\, \uPsi_h$ where $\uPsi$ is the admissible generalized eigenvector of $\OP(\bA^\natural_{\bfz},\R\times\Upsilon)$ in natural variables as introduced in \eqref{eq:age1} and $\uPsi_{h}$ its scaled version. 

The following Lemma provides an improvement when compared to Lemmas~\ref{lem.TaylorA}--\ref{L:d2A}, due to estimates \eqref{eq:d2y0est} which replace \eqref{E:taylorA}. 

\begin{lemma}\label{lem:ahsuper}
With the previous notation, there exist constants $C(\Omega)>0$ and $h_{0}>0$ such that for all $h\in (0,h_{0})$
\begin{equation}\label{eq:ah}
   \hat a_{h}^{[1]} = 
  \frac{\|(\bA^\natural -\bA^\natural_{\bfz} - \nabla F)\upsi_{h}\|}
   {\|\upsi_{h}\|}
   \leq C(\Omega)
   (\|\bA^\natural\|_{W^{2,\infty}(\cVboxh)}(h+h^{\frac12+\de0})+\|\bA^\natural\|_{W^{3,\infty}(\cVboxh)} h^{3\de0}) .
\end{equation}
\end{lemma}

\begin{proof}
 Using the form of the admissible generalized eigenvector $\uPsi$:
$$
   \uPsi(\bx^\natural)=
   \re^{\ri\vartheta(\bx^\natural)}\Phi(\bz)
   \quad\mbox{with}\quad
   \bx^\natural=(y,\bz)\, ,
$$ 
we obtain by definition of $\upsi_h$
\[
   |\upsi_h(\bx^\natural)| = 
   \underline\chi_{R}\left(\frac{|\bx^\natural|}{h^{\de0}}\right)  
   \Big|\Phi\left(\frac{\bz}{h^{1/2}}\right)\Big|\,.
\] 
Using
the changes of variables $\bZ=\bz h^{-1/2}$ and $Y=yh^{-\de0}$, we find the bounds 
\begin{eqnarray*}
   \Big\||y|^3 \ \underline\chi_{R}\left(\frac{|\bx^\natural|}{h^{\de0}}\right)  
   \Phi\left(\frac{\bz}{h^{1/2}}\right)\Big\| &\le& h^{3\de0} \,
   \|\upsi_h\|\\
   \Big\||y|\,|\bz| \ \underline\chi_{R}\left(\frac{|\bx^\natural|}{h^{\de0}}\right)  
   \Phi\left(\frac{\bz}{h^{1/2}}\right)\Big\| &\le& h^{\de0+\frac12} \,
   \|\upsi_h\| \\
   \Big\||\bz|^2 \ \underline\chi_{R}\left(\frac{|\bx^\natural|}{h^{\de0}}\right)  
   \Phi\left(\frac{\bz}{h^{1/2}}\right)\Big\| &\le& h  \,
   \|\upsi_h\|.
\end{eqnarray*}
Summing up the latter three estimates and using \eqref{eq:d2y0est} lead to the lemma.
\end{proof}

Now, since Remark \ref{R:dKu0} allows to generalize Lemma~\ref{L:d2A} to higher derivatives of the potential as in \eqref{E:AWl}, we use \eqref{eq:psinu} and Lemmas~\ref{lem:ahsuper}, \ref{lem.TaylorA}  and \ref{L:d2A} for case (i) in \eqref{E:QR1G2} and combine this with \eqref{E:Qc1} to deduce
\begin{equation}\label{eq:Ray4G2nu1}
  \QR_{h}[\bA,\Omega](\phihX{0}) \leq 
  h\Lambda_{\dx}+C(\Omega)(1+\| \bA \|_{W^{3,\infty}(\Omega)}^2)
  (h^{2-2\de0} + h^{\frac 32} + h^{1+\de0} + h^{\frac12+3\de0} + h^{6\de0}) \, .
  \end{equation}
We optimize this upper bound by taking $\de0=\frac13$. The min-max principle provides Theorem \ref{T:sUB} with a remainder in $\mathcal O(h^{4/3})$ in the case (G2) with $\dx=(\bx_0)$.

\subsection{Sliding quasimodes}\index{Quasimode!Sliding}
We assume now $\nu\geq 2$, so $\dx=(\bx_{0},\bx_{1})$ or $(\bx_{0},\bx_{1},\bx_{2})$. We use the notation of Section~\ref{SS.phihX1}. 
The main difference with Section~\ref{SS.phihX1} is that we deal with the linear part of $\bA^{\bx_{0}}$ at $\dec_{1}$ instead of $\bfz$, that is:
$$\bA^{\bx_{0}}_{\dec_{1}}(\bx):=\nabla \bA^{\bx_{0}}(\dec_{1}) \cdot\bx, \ \ \bx\in \Pi_{\bx_{0}} \, .$$
By the change of variable $\diffeo^{\dec_1}$, the potential $\bA^{\bx_0}_{\dec_{1}}$ becomes $\widehat\bA^{\dec_1}$ (cf.\ \eqref{E:ABtilde})
\begin{equation*}
    \widehat\bA^{\dec_1} = 
   (\rJ^{\dec_1})^\top \Big(\big( \bA^{\bx_0}_{\dec_{1}}-\bA^{\bx_0}_{\dec_{1}}({\dec_{1}})\big)
    \circ (\diffeo^{\dec_1})^{-1}\Big)
   \quad\mbox{with}\quad
   \rJ^{\dec_1} = \rd (\diffeo^{\dec_1})^{-1}.
\end{equation*}
Let $\widehat\zeta^{\dec_1}_h(\bx)=\re^{\ri\langle\bA^{\bx_0}_{\dec_{1}}(\dec_1),\,{\bx}/{h}\rangle}$, for $\bx\in\Pi_{\bx_0}$ and $\widehat \diffeoZ^{\dec_{1}}_{h}$ be the operator of multiplication by $\overline{\widehat\zeta^{\dec_{1}}_{h}}$.
By analogy with \eqref{eq.c2}, we introduce the relation
\begin{equation}\label{eq.c2G2}
\phihX1 = \widehat\diffeoZ^{\dec_{1}}_h \circ \diffeo_*^{\dec_1}(\phihX2).
\end{equation}
Let us assume for the end of this section that $\nu=2$.  
Let $\widehat\bA^{\dec_1}_\bfz$ be the linear part of $\widehat\bA^{\dec_1}$ at $\bfz\in\Pi_{\bx_0,\bx_1}$.
We have $\curl \widehat\bA^{\dec_{1}}_{\bfz}=\bB^{\bx_{0}}_{\dec_{1}}$ where the constant $\bB^{\bx_{0}}_{\dec_{1}}$ is the magnetic field $\bB^{\bx_{0}}$ frozen at $\dec_{1}$. 

We have $\En(\bB_{\bx_{0}},\Pi_{\dx})< \seE(\bB_{\bx_{0}},\Pi_{\dx})$. Due to Lemma~\ref{L:Elip}, we have 
\begin{equation}
\label{E:remarkimportante}
 \exists \varepsilon>0, \ \ \forall \dec_{1}\in \cB(0,\varepsilon)\cap \overline\Pi_{\bx_{0}}, \quad \En(\bB^{\bx_{0}}_{\dec_{1}},\Pi_{\dx})<\seE(\bB^{\bx_{0}}_{\dec_{1}},\Pi_{\dx}) \,,
 \end{equation}
and $(\bB^{\bx_{0}}_{\dec_{1}},\Pi_{\dx})$ is still in situation (G2). 
Let $\udiffeo^{\dec_{1}}$ ($\rJ$ the associated matrix) be the rotation $\bx\mapsto\bx^{\natural}:=(\by,z)$ that maps $\Pi_{\dx}$ onto the model domain $\R\times\Upsilon$.
Let $\bA^{\natural,\dec_{1}}$ be the magnetic potential associated with $\widehat\bA^{\dec_{1}}$ in variables $\bx^\natural$ and $\bA^{\natural,\dec_{1}}_{\bfz}$ be its linear part at $\bfz$. Due to \eqref{E:remarkimportante}, we are still in case (i) of the Dichotomy Theorem~\ref{th:dicho}. We use now the admissible generalized eigenvector $\uPsi^{\dec_{1}}$ of $\OP(\bA^{\natural,\dec_{1}}_{\bfz},\R\times\Upsilon)$ in natural variables as introduced in \eqref{eq:age1} and its scaled version $\uPsi_h^{\dec_{1}}$. The associated ground state energy is denoted by 
\begin{equation}\label{E:Lambdav1}
\Lambda_{\dec_{1}}=\En(\bB^{\bx_{0}}_{\dec_{1}},\Pi_{\dx}).
\end{equation} 
An important point is that, choosing $\varepsilon>0$ small enough, we may assume that, in virtue of Lemma \ref{L:Elip} (b), the functions $\uPsi^{\dec_{1}}$ are uniformly exponentially decreasing
\begin{equation}
\label{E:agmuniform}
\exists c>0,\ C>0,\quad
\forall \dec_{1}\in \cB(\bfz,\varepsilon), \quad 
\|\uPsi^{\dec_{1}}\re^{c|\bz|}\|_{L^2(\Upsilon)}\leq C \|\uPsi^{\dec_{1}} \|_{L^2(\Upsilon)}  \, .
\end{equation}
We are arrived at point where the situation is similar as in the sitting case, with the new feature that the generalized eigenvectors $\uPsi_{h}^{\dec_{1}}$ depend (in some smooth way) on the parameter $\dec_{1}$. 
We define the new function on $\Pi_{\dx}$ by
\begin{equation}
\label{eq:spmp}
     \phihX2 = \udiffeo^{\dec_{1}} (\re^{-iF^{\dec_{1}}/h}\upsi_{h}^{\dec_{1}}),
\end{equation}
where $\upsi_{h}^{\dec_{1}}=\chi_{h}\uPsi_{h}^{\dec_{1}}$ has a support of size $r_{h}^{[2]}=\cO(h^{\de0+\de1})$ and the phase shift $F^{\bv_{1}}$ will be chosen later. As always we denote by $\mu^{[2]}_{h}=\QR_{h}[\bA^{\natural,\dec_{1}}_{\bfz},\Pi_{\dx}](\upsi_{h}^{\dec_{1}})$.

The function $\dec\mapsto\Lambda_{\dec}$ is Lipschitz-continuous by Lemma~\ref{L:Elip} (a) and thus $|\Lambda_{\dec_{1}}-\Lambda_{\bfz}|\leq C|\dec_{1}|$. 
Combining this with Lemma \ref{lem:rhohb}, we have
$$\mu^{[2]}_{h} \leq h\Lambda_{\dec_{1}} + C h^{2-2\delta_{0}} 
\leq h\Lambda_{\dx} +C (h^{1+\delta_{0}}+h^{2-2\delta_{0}}).$$
Now we distinguish whether our quasimode is soft or hard sliding ($\bx_{0}$ is not, or is, a conical point).

\subsubsection*{Soft sliding}
If $\bx_{0}$ is not a conical point, we recall as mentioned in Section~\ref{SS.phihX1} that $\diffeo^{\dec_{1}}$ is a translation. As in Section \ref{SS:SS} we have
\begin{align*}
\QR_{h}[\bA^{\bx_0},\Pi_{\bx_{0}}](\phihX{1}) &= 
\QR_{h}[\bA^{\natural}(\cdot+\dec_{1}^{\natural})-\bA^{\natural}_{\dec_{1}^{\natural}}(\dec_{1}^{\natural})-\nabla F^{\bv_{1}},\R\times\Upsilon](\upsi_{h}^{\dec_{1}})
\\
& \leq \mu^{[2]}_{h}+2\sqrt{\mu^{[2]}_{h}} \|\big(\bA^{\natural}(\cdot+\dec_1^{\natural}) - 
  \bA_{\dec_{1}^{\natural}}^{\natural}(\cdot+\dec_1^{\natural})-\nabla F^{\dec_{1}}\big)\upsi_{h}^{\dec_{1}}\|
  \\ & + \|\big(\bA^{\natural}(\cdot+\dec_1^{\natural})- \bA_{\dec_{1}^{\natural}}^{\natural}(\cdot+\dec_1^{\natural})-\nabla F^{\dec_{1}}\big)\upsi_{h}^{\dec_{1}}\|^2,
   \end{align*}
where we have used the relation 
   $\bA^{\natural}(\cdot+\dec_{1}^{\natural})-\bA_{\dec_{1}^{\natural}}^{\natural}(\dec_{1}^{\natural})-\bA^{\natural,\dec_{1}}_{\bfz}=\bA^{\natural}(\cdot+\dec_{1}^{\natural})-\bA_{\dec_{1}^{\natural}}^{\natural}(\cdot+\dec_{1}^{\natural})$. We now use Lemma \ref{L:d2ell0} to choose $F^{\dec_{1}}$ such that $\bA^{\natural}(\cdot+\dec_1^{\natural})- \bA_{\dec_{1}^{\natural}}^{\natural}(\cdot+\dec_1^{\natural})-\nabla F^{\dec_{1}}$ is still controlled by the r.h.s. of \eqref{eq:d2y0est}. The proof of Lemma \ref{lem:ahsuper} is still valid due to the uniform control \eqref{E:agmuniform}, and provides: 
\begin{multline*}
    \|\big(\bA^{\natural}(\cdot+\dec_1^{\natural})- \bA_{\dec_{1}^{\natural}}^{\natural}(\cdot+\dec_1^{\natural})-\nabla F^{\dec_{1}}\big)\upsi_{h}^{\dec_{1}}\|    \\
   \leq C(\Omega)
   (\|\bA^\natural\|_{W^{2,\infty}(\cVboxh)}(h+h^{\frac12+\de0})+\|\bA^\natural\|_{W^{3,\infty}(\cVboxh)} h^{3\de0})\    {\|\upsi_{h}^{\dec_{1}}\|} . 
   \end{multline*}
The proof goes along as in the sitting case and we deduce the same estimate \eqref{eq:Ray4G2nu1} with a remainder in $\cO(h^{4/3})$.

\subsubsection*{Hard sliding}
 If $\bx_{0}$ is a conical point, using formulas \eqref{E:chgG} and \eqref{eq:diffAA'}, we have
\begin{equation}\label{E:QRG2Ax02}
 \QR_{h}[\widehat\bA^{\dec_1},\Pi_{\dx}](\phihX1)=\QR_{h} [\bA^{\natural,\dec_{1}}-\nabla F^{\dec_{1}},\R\times\!\Upsilon](\upsi_{h}^{\dec_{1}})
  \leq \mu_{h}^{[2]}+ 2 \hat a_{h}^{[2]}\ \sqrt{\mu_{h}^{[2]}} + (\hat a_{h}^{[2]})^2,
\end{equation}
where we have set
\begin{equation}\label{eq:hata2}
   \hat a^{[2]}_{h}=\frac{\|(\bA^{\natural,\dec_{1}} -\bA^{\natural,\dec_{1}}_{\bfz} - \nabla F^{\dec_{1}})\upsi_{h}^{\dec_{1}}\|}
   {\|\upsi_{h}^{\dec_{1}}\|}. 
\end{equation}
Like previously, Lemma~\ref{L:d2ell0} gives a function $F^{\dec_{1}}$ satisfying  
\begin{equation}
\label{eq:d2y0estsliding}
   \big|\big(\bA^{\natural,\dec_{1}} - \bA^{\natural,\dec_{1}}_\bfz - 
   \nabla F^{\dec_{1}}\big)(\bx^\natural)\big| 
   \le C(\cV_{\bx_0})\,
   \big( \|\bA^{\natural,\dec_{1}}\|_{W^{2,\infty}}(|y||\bz| + |\bz|^2) + \|\bA^{\natural,\dec_{1}}\|_{W^{3,\infty}}|y|^3 \big)\,.
\end{equation}
Due to the uniform estimate \eqref{E:agmuniform}, the proof of Lemma~\ref{lem:ahsuper} still applied. Combine this with \eqref{E:AWl} gives
\begin{eqnarray*}
\hat a^{[2]}_{h}&\leq& C  (\|\bA^{\natural,\dec_{1}}\|_{W^{2,\infty}(\supp(\upsi_{h}^{\dec_{1}}))}\ (h+h^{\frac 12+\de0+\de1})+\|\bA^{\natural,\dec_{1}}\|_{W^{3,\infty}(\supp(\upsi_{h}^{\dec_{1}}))}\ h^{3\de0+3\de1})\\
&\leq& C(\|\bA\|_{W^{2,\infty}(\Omega)}\  (h^{1-\de0}+h^{\frac 12+\de1})+\|\bA\|_{W^{3,\infty}(\Omega)}\ h^{\de0+3\de1}).
\end{eqnarray*}
Then Relation \eqref{eq.phihX0e2} becomes
\begin{multline}\label{eq.phihX0e2G2}
   \QR_h[\bA,\Omega](\phihX0)
   \leq h\Lx + C\, (h^{2-2\de0}+h^{1+\de0}) + C(h^{2-2\de0-2\de1}+ h^{1+\de0}+h^{1+\de1})\\
   +C\ \big(h^{\frac 32-\de0}+h^{1+\de1}+h^{\frac 12+\de0+3\de1}+h^{2\de0+6\de1}\big).
\end{multline}

Choosing $\de0=\frac{5}{16}$ and $\de1=\frac18$ gives the upper-bound \eqref{eq:above3sup} in situation (G2) for hard sliding quasimodes.

\subsection{Doubly sliding quasimode}\index{Quasimode!Doubly sliding}
In that case, as mentioned in Section~\ref{SS.phihX2}, $\nu=3$, $\dx=(\bx_{0},\bx_{1},\bx_{2})$, $\bx_{0}$ is a conical point and $\diffeo^{\bv_{2}}$ is a translation. 
We define 
\begin{equation}
   \phihX2 = \widehat Z_{h}^{\dec_{2}} \circ\diffeo_{*}^{\dec_2} (\phihX3), 
\end{equation}
where 
$\widehat Z_{h}^{\dec_{2}}$ is the operator of multiplication by $\overline{\widehat\zeta^{\dec_2}}$ with $\widehat\zeta^{\dec_2}_h(\bx)=\re^{\ri\langle\widehat\bA^{\dec_1}_{\dec_{2}}(\dec_2),\,{\bx}/{h}\rangle}$ and 
$$    \widehat\bA^{\dec_2} = 
   \big( \widehat\bA^{\dec_1}_{\dec_{2}}-\widehat\bA^{\dec_1}_{\dec_{2}}(\dec_{2})\big)
    \circ (\diffeo^{\dec_2})^{-1},$$
with coincides with its linear part $\widehat\bA^{\dec_2}_\bfz$. 
Since $\rG^{\bv_{2}}=\Id_{3}$, we have
 \begin{equation}\label{E:Qc3G2}
   \QR_h[\widehat\bA_{\bfz}^{\dec_{1}},\Pi_{\bx_{0},\bx_{1}}](\phihX2) = 
   \QR_h[\widehat\bA^{\bv_2}_{\bfz},\Pi_{\dx}](\phihX3). 
\end{equation}
We set in the same spirit as above, $\phihX{3}=\udiffeo^{\dec_{2}}(\re^{-iF^{\dec_{2}}/h}\chi_{h}\Psi_{h}^{\dec_{2}})$. The constant magnetic field $\bB^{\dec_{1},\dec_{2}}_{\bfz}=\curl \widehat\bA^{\dec_2}_\bfz$ is the magnetic field $\bB^{\bx_{0}}$ frozen at $\dec_{1}$, transformed by $\diffeo^{\bv_{1}}$ and then frozen at $\bv_{2}$. Once again, $(\bB^{\dec_{1},\dec_{2}}_{\bfz},\Pi_{\dx})$ is still in situation (G2) for $h$ small enough and we may use Lipschitz estimates for the associated ground state energy and uniform decay estimates for the associated AGE. As in the soft sliding case described above, we take advantage of the translation $\diffeo^{\dec_{2}}$ and get a better estimate for the last linearization (that is step [c2], see Figure \ref{F:QM}) by a suitable choice of $F^{\dec_{2}}$. We can conclude as the conical case at level 2 and obtain again \eqref{eq.phihX0e2G2}.  
We deduce
\[
\lambda_{h}(\bB,\Omega) \leq
   h\sE(\bB,\Omega)+C(\Omega) \big(1+\|\bA\|_{W^{3,\infty}(\Omega)}^2 \big)h^{9/8} \ . 
\]

The proof of Theorem~\ref{T:sUB} is now complete in case (G2).

\chapter{Conclusion: Improvements and extensions}
\label{sec:conc}
In this work we have shown how a recursive structure of corner domains allows to analyze the Neumann magnetic Laplacian and its ground state energy $\lambda_{h}(\bB,\Omega)$. To conclude, we discuss some standard consequences in the situation of corner concentration\index{Corner concentration}. We also address the issues of generalizing our results to any dimension. We finally mention the adaptation of our methods to different boundary value problems, namely the Dirichlet magnetic Laplacian and the Robin Laplacian in the attractive limit.

\section{Corner concentration and standard consequences}
\label{SS:CornerC}
Let $\Omega$ be a 3D corner domain and $\bB$ be a magnetic field.
For each corner $\bv\in\gV$ of $\Omega$, let us denote by $K_\bv$ the number of eigenvalues of the tangent model operator $\OP(\bA_\bv \ee,\Pi_\bv)$ which are below the minimal local energy\index{Local ground energy} outside the corners $\inf_{\bx\in\overline\Omega\setminus\gV} \,\En(\bB_\bx \ee,\Pi_\bx) \,.$
If no such eigenvalue exists, we set $K_\bv=0$. If they do exist, we denote them by $\lambda^{(k)}(\bB_{\bv}\ee,\Pi_{\bv})$, $k=1,\ldots,K_\bv$, so that
$$
   \forall\bv\in\gV,\quad  \forall 1\leq k\leq K_{\bv},\quad 
   \lambda^{(k)}(\bB_{\bv}\ee,\Pi_{\bv})<
   \inf_{\bx\in\overline\Omega\setminus\gV} \,\En(\bB_\bx \ee,\Pi_\bx).
$$
Setting
$
   K(\bB,\Omega) = \sum_{\bv\in\gV} K_\bv
$, we assume that we are in the case of corner concentration, {\it i.e.},
\[
   K(\bB,\Omega)>0\,.
\]
Then several standard consequences hold for the eigenvalue asymptotics of the first $K(\bB,\Omega)$ eigenvalues $\lambda^{(k)}_h(\bB,\Omega)$ of the magnetic Laplacian $\OP_{h}(\bA \ee,\Omega)$. Indeed, for $1\leq k\leq K(\bB,\Omega)$, we denote by $\sE^{(k)}(\bB\ee,\Omega)$ the $k$-th element (repeated with multiplicity) of the collection of eigenvalues $\lambda^{(j)}(\bA_\bv \ee,\Pi_{\bv})$ of the model operators, for $\bv\in\gV$ and $1\leq j\leq K_{\bv}$. Then we have
\begin{equation}\label{eq.lk}
   \big| \lambda^{(k)}_h(\bB,\Omega) - h \sE^{(k)}(\bB\ee,\Omega) \big| \leq C h^{3/2},
   \quad\forall 1\leq k\leq K(\bB,\Omega).
\end{equation}
In fact, we can prove like in \cite[Section 7]{BonDau06} a complete asymptotics expansion in power of $h^{1/2}$ for the eigenvalues $\lambda^{(k)}_h(\bB,\Omega)$, $1\leq k\leq K(\bB,\Omega)$ and \eqref{eq.lk} is a consequence. 
Furthermore, we have corner localization\index{Exponential decay} of the eigenvectors. 
Another consequence of the complete expansion of the low-lying eigenvalues is the monotonicity of the ground state energy $B \mapsto \lambda(\breve\bB,\Omega)$ \eqref{lienBh} in the point of view of large magnetic field. This can be seen as a strong diamagnetic inequality and relies on the same arguments as in \cite[Section 2.1]{BoFou07}.

\section{The necessity of a taxonomy}\label{ss:tax}\index{Taxonomy}
Let us emphasize the role of the taxonomy of model problems discussed earlier. The proof of upper bounds with remainder for $\lambda_{h}$ strongly relies on the existence of generalized eigenfunctions for model operators associated with the minimum of local energies. Our Dichotomy\index{Dichotomy} Theorem provides a positive answer and is based on an exhaustive description of the ground state of model operators depending on the dimension $d\in \{0,\ldots,3 \}$ of reduced cones, {\it i.e.}, on spaces, half-spaces, wedges and 3D cones, respectively. In cases $d\le2$, the analysis is made through a  fibration ({\it i.e.}, a partial Fourier transform), leading to a new operator that is not a standard magnetic Laplacian. As consequence, the analysis of the key quantity $\seE$\index{Lowest local energy} seems to be specific to each dimension.

Besides, in higher dimensions, a magnetic field $\bB$ can be identified in each point $\bx\in\overline\Omega$ with a $n\times n$ antisymmetric matrix, thus determines $\frac{n}{2}$ or $\frac{n-1}{2}$ two-dimensional invariant subspaces $P^j_\bx$ when $n$ is even or odd, respectively (for instance, in dimension $n=3$, the space $P^1_\bx$ is the orthogonal space to the vector $\bB_\bx$). Given a cone $\R^{\nu}\times \Gamma$ with $\nu >0$, its interaction with the planes $P^j_\bx$ can be highly non-trivial and there is no reason that there exists a magnetic potential which depends on less variables than $n$. Thus the fibration process we have used does not seem available in general in the $n$ dimensional case.
At this stage, a recursive analysis of the ground state of the magnetic Laplacian does not seem possible without a deeper analysis of tangent model operators\index{Model operator}, namely a complete taxonomy valid for all dimension.

\section{Continuity of local energies}
A standard procedure to investigate the stability of the ground state energy of a self-adjoint operator consists in constructing quasimodes\index{Quasimode} issued from the spectrum of the unperturbed problem, using them for the perturbed operator, and concluding with the min-max principle\index{Min-max principle}. This procedure applied to the ground state energy of model problems associated with $\OP(\bA \ee,\Omega)$ would provide upper semicontinuity  under perturbation and, therefore, upper semicontinuity for the local energy $\bx\mapsto \index{Semicontinuity}\En(\bB_{\bx},\Pi_{\bx})$ on each stratum $\bt$ of $\overline{\Omega}$.

In the case of Neumann boundary conditions, we have proved the continuity on each stratum by using once more the taxonomy of model problems. In particular Lemma \ref{lem:contwedge} uses intensively the structure of the magnetic Laplacian on wedges and is linked to our Dichotomy Theorem, see \cite{Pop13}.  The lower semicontinuity of the local energy between strata is a consequence of Theorem \ref{th:scichain}, and relies on the continuity on each stratum. By contrast with Dirichlet conditions, Neumann boundary conditions imply a decrease of the local ground energy on strata of higher codimensions,
including possible discontinuities between strata. 

In the general $n$ dimensional case, the sole known result is the continuity of the local energy on the interior stratum, {\it i.e.}, $\Omega$ itself. Indeed, for any $\bx\in\Omega$, we have $\En(\bB_{\bx},\Pi_{\bx})=b(\bx)$ with $b(\bx)$ defined in \eqref{eq.b(x)}. The generic regularity is in fact H\"older of exponent $\frac{1}{2n}$ as mentioned in \cite[Lemma 5.4]{HeRo84}).

\section{Dirichlet boundary conditions}
\label{SS:Dir}
If one considers now the magnetic Laplacian with Dirichlet boundary conditions, the situation of the local energies denoted now $\En^{\rm D}(\bB_{\bx},\Pi_{\bx})$ is far simpler than in the Neumann case. For any interior point $\bx\in\Omega$, $\En^{\rm D}(\bB_{\bx},\Pi_{\bx})=\En(\bB_{\bx},\R^n)$ is equal to the intensity $b_\bx$ of $\bB_\bx$ (with $b_{\bx}=b(\bx)$ defined in \eqref{eq.b(x)}). If $\bx$ lies in the boundary of $\Omega$, by Dirichlet monotonicity, $\En^{\rm D}(\bB_{\bx},\Pi_{\bx})\ge\En(\bB_{\bx},\R^n)$, and the converse inequality is the consequence of a standard argument involving Persson's Lemma\index{Persson's Lemma}, cf.\ Theorem \ref{th:cone-ess}. Thus, like in the case without boundary, the sole ingredient in local energies is the intensity of the magnetic field in each point $\bx\in\partial\Omega$.
At this point, we could generalize the estimates of \cite{HeMo96}
\[
   - C^{-} h^{5/4} \le  \lambda_h(\bB \ee,\Omega) - h\ee \sE(\bB \ee,\Omega) \le   
    C^{+} h^{4/3}
\]
to any domain $\Omega$ with Lipschitz boundary\index{Lipschitz domain} and $\sC^{3}(\overline\Omega)$ magnetic potential with nonvanishing magnetic field $\bB$, including the case when the minimum is attained on the boundary. The key arguments are the following:
\begin{description}
\item[\sc Lower bound] One uses a IMS\index{IMS formula} partition technique in order to {\em linearize} the potential on each piece of the partition, but {\em without local maps}. Then, when a local support crosses the boundary of $\Omega$, one simply uses the lower bound $\lambda_h(\bB_{\bx_0} \ee,\Omega)\ge \lambda_h(\bB_{\bx_0} \ee,\R^n)$ for the ``central point'' $\bx_0$ of this local support.

\item[\sc Upper bound] For $\bx_0\in\partial\Omega$, one constructs interior sliding	\index{Quasimode!Sliding} quasimodes with support in a cone interior to $\Omega$ and with vertex $\bx_0$. In order to obtain the refined convergence rate $h^{4/3}$ instead of $h^{5/4}$, one has to use a gauge transform\index{Gauge transform} similar to that in \cite[p. 54-55]{HeMo96}.
\end{description}

\section{Robin boundary conditions with a large parameter for the Laplacian}
The spectral behavior of the Neumann magnetic Laplacian has some analogy with the following Robin boundary eigenvalue problem that consists in solving 
\begin{equation}
\label{eq:robin}
   \begin{cases}
   -\Delta \psi=\lambda\psi\ \ &\mbox{in}\ \ \Omega, \\
  \nabla \psi \cdot {\bf n}-\beta \psi=0\ \ &\mbox{on}\ \ \partial\Omega,
   \end{cases}
\end{equation}
where $\beta \in \R$ is a parameter.  
This problem also arises from a linearization of the Ginzburg-Landau equation, in the zero field regime (\cite{GiSm07}). The asymptotics of the ground state energy $\lambda_{\beta}^{\rm R}(\Omega)$ in the attractive limit $\beta\to+\infty$ has been studied in \cite{LevPar08,HePank15,PankPof16} and presents several similarities with the semiclassical Neumann magnetic Laplacian. It is still relevant to define the local energies $\En(\Pi_{\bx})$ as the ground state energies of tangent operators (with $\beta=1$). These energies satisfy $E(\Pi_{\bx}) \leq -1$ for any $\bx\in \partial\Omega$.
It is proved in \cite{LevPar08} that for any domain with corner $\Omega$ satisfying the uniform interior cone condition, we have
$$\lim_{\beta\to+\infty} \frac{\lambda_{\beta}^{\rm R}(\Omega)}{\beta^2} = \sE(\Omega),$$
where $\sE(\Omega)$ is defined as $\inf_{\bx\in \overline{\Omega}} \En(\Pi_{\bx})$ like in the magnetic case. But the finiteness of $\sE(\Omega)$ is not guaranteed in this framework. In \cite{BruPof16},  general $n$-dimensional corner domains belonging to the class $\gD(\R^n)$ are considered, and the bottom of the spectrum is analyzed using the technique developed in the present work. In comparison with the magnetic Laplacian, a more favorable feature is a convenient separation of variables on any tangent cone written in reduced form as $\R^{n-d}\times \Gamma$: The associated tangent operator becomes $\Id_{n-d} \otimes H^{\rm R}(\Gamma) + (-\Delta|_{\R^{n-d}})\otimes \Id_d$ where $H^{\rm R}(\Gamma)$ is the Robin Laplacian on $\Gamma$ for $\beta=1$. Thus the difficulties linked to the taxonomy\index{Taxonomy} mentioned in \ref{ss:tax} disappear in this case, and the analysis can be performed in any dimension. In a first step, the lower semicontinuity\index{Semicontinuity} of the local energies\index{Local ground energy} is proved by recursion over the dimension, giving the existence of a minimizer for the local energies, hence the finiteness of $\sE(\Omega)$. For large $\beta$, the estimate 
$$| \lambda_{\beta}^{\rm R}(\Omega) -\sE(\Omega) \beta^2| \leq C \beta^{2-\frac{2}{2\nu+3}}$$
is proved for the same integer $\nu$ depending on the domain as introduced in Section \ref{ss:genlow}.
The upper bound relies on a recursive multi-scale construction of quasimodes\index{Quasimode}, whereas the lower bound is based on a $\nu+1$-scale partition of the unity adapted to admissible atlases.

\part{Appendices}
\label{part:5}

\appendix
\chapter{Magnetic identities}
\label{sec:techniq}

\section{Gauge transform}

\begin{lemma}\label{lem:gauge}\index{Gauge transform}
Let $\cO \subset \R^n$ be a domain and let $\vartheta$ be a regular function on $\overline\cO$. Let $\bA$ be a regular potential. Then
\begin{equation*}
\forall \psi\in \dom(q_{h}[\bA,\cO]),\quad q_{h}[\bA+\nabla\vartheta,\cO](\re^{-i\vartheta/h}\psi) =  q_{h}[\bA,\cO](\psi).
\end{equation*}
\end{lemma}

This well-known result is a consequence of the commutation formula 
\[(-ih\nabla+\bA+\nabla{\vartheta})\left(\re^{-i\vartheta/h}\psi\right)
= \re^{-i\vartheta/h} (-ih\nabla+\bA)\psi\,.
\]

\begin{lemma} 
\label{L:d2ell0}
Let $\cO$ be a bounded domain such that $\bfz\in\overline\cO$. Let $\bu=(u_1,u_2,u_3)$ denote Cartesian coordinates in $\cO$. 
Let $\bA\in \sC^{3,\infty}(\overline{\cO})$ be a magnetic potential such that $\bA(\bfz)=0$. Let $\bA_{\bfz}$ denote the linear
part of $\bA$ at $\bfz$. Let $\ell$ be an index in $\{1,2,3\}$. 

(a) There exists a change of gauge $\nabla F$ where $F$ is a polynomial function of degree $3$, so that
\begin{enumerate}
\item The linear part of $\bA-\nabla F$ at $\bfz$ is still $\bA_\bfz$,
\item The second derivative of $\bA-\nabla F$ with respect to $u_\ell$ cancels at $\bfz$:
\begin{equation*}
   \partial^2_{u_\ell}(\bA-\nabla F)(\bfz) = 0.
\end{equation*}
\item The coefficients of $F$ are bounded by $\|\bA\|_{W^{2,\infty}(\cO)}$.
\end{enumerate}
(b) Let us choose $\ell=1$ for instance. We have the estimate
\begin{multline}
\label{eq:d2ell0est}
   |\bA(\bu) - \bA_\bfz(\bu) - \nabla F(\bu)| \\ \le C(\cO)\,
   \left(  \|\bA\|_{W^{2,\infty}(\cO)}
   \big(|u_1u_2| + |u_1u_3| + |u_2|^2 + |u_3|^2\big) +
   \|\bA\|_{W^{3,\infty}(\cO)} |u_{1}|^3\right)\,,
\end{multline}
where the constant $C(\cO)$ depends only on the outer diameter of $\cO$.
\end{lemma}

\begin{proof}
The Taylor expansion of $\bA$ at $\bfz$ takes the form
\[
   \bA = \bA_\bfz + \bA^{(2)} + \bA^{(\mathrm{rem},3)},
\]
where $\bA^{(2)}$ is a homogeneous polynomial of degree $2$ with $3$ components and $\bA^{(\mathrm{rem},3)}$ is a remainder:
\begin{equation}
\label{eq:rem3}
   |\bA^{(\mathrm{rem},3)}(\bu)| \le 
   \|\bA\|_{W^{3,\infty}(\cO)} |\bu|^3
   \quad\mbox{for}\quad\bu\in\cO.
\end{equation}
Let us write the $m$-th component $A^{(2)}_{m}$ of $\bA^{(2)}$ as
\[
   A^{(2)}_{m}(\bu) = 
   \sum_{|\alpha|=2} a_{m,\alpha} u_1^{\alpha_1}u_2^{\alpha_2}u_3^{\alpha_3}
   \quad\mbox{for}\quad\bu=(u_1,u_2,u_3)\in\cO .
\]

(a) Now, the polynomial $F$ can be explicitly determined. It suffices to take
\[
   F(\bu) =
   u_\ell^2 \big(a_{1,\alpha^*}u_1 +  a_{2,\alpha^*}u_2 + a_{3,\alpha^*}u_3 
   - \tfrac{2}{3} a_{\ell,\alpha^*}u_\ell \big),
\]
where $\alpha^*$ is such that $\alpha^*_\ell=2$ (and the other components are $0$). Then
\[
   \nabla F(\bu) =  
   u_\ell^2
   \begin{pmatrix}
   a_{1,\alpha^*} \\ a_{2,\alpha^*} \\ a_{3,\alpha^*} 
   \end{pmatrix}
\]
and point (a) of the lemma is proved.

(b) Choosing $\ell=1$, we see that the $m$-th components of $\bA^{(2)}-\nabla F$ is
\begin{multline*}
   A^{(2)}_{m}(\bu) - (\nabla F)_m(\bu) \\= 
   a_{m,(1,1,0)} u_1u_2 + a_{m,(1,0,1)} u_1u_3 + a_{m,(0,1,1)} u_2u_3 + 
   a_{m,(0,2,0)} u^2_2 + a_{m,(0,0,2)} u^2_3 \,.
\end{multline*}
Hence $\bA^{(2)}-\nabla F$ satisfies the estimate
\[
   |(\bA^{(2)}(\bu)-\nabla F(\bu)| \le \|\bA\|_{W^{2,\infty}(\cO)}
   \big(|u_1u_2| + |u_1u_3| + |u_2|^2 + |u_3|^2\big).
\]
But
\[
   \bA - \bA_\bfz - \nabla F = \bA^{(2)} - \nabla F + \bA^{(\mathrm{rem},3)}.
\]
Therefore, with \eqref{eq:rem3}
\[
   |\bA(\bu) - \bA_\bfz(\bu) - \nabla F(\bu)| \le 
   \|\bA\|_{W^{2,\infty}(\cO)}
   \big(|u_1u_2| + |u_1u_3| + |u_2|^2 + |u_3|^2\big) +
   \|\bA\|_{W^{3,\infty}(\cO)} |\bu|^3.
\]
Using finally that $|\bu|^3\le 12(|u_1|^3 + |u_2|^3 + |u_3|^3) \le  C(\cO) (|u_1|^3 + |u_2|^2 + |u_3|^2)$, we conclude the proof of estimate \eqref{eq:d2ell0est}.
\end{proof}

\section{Change of variables}\label{SA:CV}

Let $\rG$ be a metric of $\R^3$, that is a $3\times3$ positive symmetric matrix with regular coefficients. For a smooth magnetic potential, the quadratic form of the associated magnetic Laplacian with the metric $\rG$ is denoted by $q_{h}[\bA,\cO,\rG]$ and is defined in \eqref{D:fqG}.
The following lemma describes how this quadratic form is involved when using a change of variables: 

\begin{lemma}\label{L:chgvar}
Let $\diffeo:\cO\to\cO'$, $\bu\mapsto\bv$ be a diffeomorphism with $\cO$, $\cO'$ domains in $\R^3$. 
We denote by 
$\rJ:=\rd (\diffeo^{-1})$ the jacobian matrix of the inverse of $\diffeo$. 
Let $\bA$ be a magnetic potential and $\bB=\curl\bA$ the associated magnetic field. 
Let $f$ be a function of $\dom(q_{h}[\bA,\cO])$ and $\psi:=f\circ \diffeo^{-1}$ defined in $\cO'$. 
For any $h>0$ we have 
\begin{equation}
\label{E:chgG}
q_{h}[\bA,\cO](f)=q_{h}[\pot,\cO',\rG](\psi) \quad \mbox{and} \quad \| f\|_{L^2(\cO)}=\| \psi\|_{L^2_{\rG}(\cO')}
\end{equation}
where the new magnetic potential and the metric are respectively given by 
\begin{equation}
\label{E:Atilde}
   \pot:=\rJ^{\top} \big(\bA\circ \diffeo^{-1} \big) \quad 
   \mbox{and}\quad \rG:=\rJ^{-1}(\rJ^{-1})^{\top} \ . 
\end{equation}
The magnetic field $\tbB=\curl\pot$ in the new variables is given by 
\begin{equation}
\label{D:tbB}
\tbB:=|\det \rJ|\,\rJ^{-1} \big(\bB\circ \diffeo^{-1} \big).
\end{equation}
\end{lemma}

Let $\rho>0$, using the previous lemma with the scaling $\diffeo^{\rho}:=\bx \mapsto \sqrt{\rho}\,\bx$ we get 

\begin{lemma}\label{lem.dilatation}
Let $\cO$ be a domain in $\R^n$ and set $r\cO:=\{\bx\in\R^n,\ \bx=r\bx'\ \mbox{with}\ \bx'\in\cO\}$ for a chosen positive $r$. Let $\bB$ be a constant magnetic field and $\bA$ be an associated linear potential.
For any $\psi\in\dom(q[\bA,\cO])$ normalized in $L^2(\cO)$, we define for any positive $\rho$
$$\psi_{\rho}(\bx):=\rho^{-n/4}\psi\Big(\frac\bx{\sqrt\rho}\Big),\qquad \forall\bx\in\cO.$$
Then $\psi_{\rho}$ belongs to $\dom(q_{\rho}[\bA,\sqrt\rho\,\cO])$, is normalized in $L^2(\sqrt\rho\,\cO)$ and we have
\begin{enumerate}
\item 
$q[\bA,\cO](\psi)
=\rho\ q[\rho^{-1}\bA,\sqrt\rho\, \cO](\psi_{\rho})
=\rho^{-1} q_{\rho}[\bA,\sqrt\rho\, \cO](\psi_{\rho}).
$
\item 
$\En(\bB \ee,\cO) = \rho\,\En\big(\rho^{-1}\bB \ee,\sqrt\rho\,\cO\big)$.
\end{enumerate}
\end{lemma}

\section{Comparison formula}
 Let $\cO$ be a domain and let $\bA$ and $\bA'$ be two magnetic potentials. Then, for any function $\psi$ of $\dom(q_{h}[\bA,{\cO}])\cap \dom(q_{h}[\bA',{\cO}])$, we have:
\begin{equation}
\label{eq:diffAA'}
  q_{h}[\bA,{\cO}](\psi) =   q_{h}[\bA',{\cO}](\psi)
   + 2  \Re\big\langle (-ih\nabla+\bA')\psi,(\bA-\bA')\psi\big\rangle_{\cO}
  + \|(\bA-\bA')\psi\|^2\, .
\end{equation}

\section{Cut-off effect}\index{Cut-off}
In this section we recall standard IMS  
formulas. This kind of formulas appear for Schr\"odinger operators in \cite{CyFrKiSi87}, but they can also be found in older works like \cite{Mel71}. In this section $\bA$ denotes a regular magnetic potential and notations are those introduced in Section \ref{ss:not}.

The first formula describes the effect of a partition of the unity on the energy of a function which is in the form domain, see for example \cite[Lemma 3.1]{Si82}:
\begin{lemma}[IMS formula]\label{lem:IMS}\index{IMS formula|textbf}
Assume that $\chi_1,\ldots,\chi_L\in\sC^\infty(\overline\cO)$ are such that
$\sum_{\ell=1}^L \chi_\ell^2 \equiv 1$ on $\cO$.
Then, for any $\psi\in \dom(q_{h}[\bA,\cO])$ 
\[
   q_{h}[\bA,\cO](\psi) = \sum_{\ell=1}^L q_{h}[\bA,\cO](\chi_\ell \psi)
   - h^2 \sum_{\ell=1}^L \|\psi\nabla\chi_\ell\|_{L^2(\cO)}^2
\]
\end{lemma}

The second formula describes the energy of a function satisfying locally the Neumann boundary conditions when applying a cut-off function, see for example \cite[(6.11)]{HeMo01}: 
\begin{lemma}
\label{lem:tronc}
Let  $\chi\in\sC^\infty_0(\overline{\cO})$ a real smooth function. Then for any $\psi\in\dom_{\,\loc} (\OP_{h}(\bA,\cO))$
\begin{equation*}
  q_{h}[\bA,\cO](\chi \psi) = 
  \Re \big\langle\chi^2 \OP_{h}(\bA,\cO) \psi,\psi\big\rangle_{\cO}
  + h^2\| \, |\nabla\chi|\,\psi \|^2_{L^2(\cO)} \, .
\end{equation*}
\end{lemma}

\subsubsection*{Orientation of the magnetic field}  
Let $\bB$ be a magnetic field. It is known that changing $\bB$ into $-\bB$ does not affect the spectrum of the associated magnetic Laplacian. More precisely we have:
\begin{lemma}
\label{lem:sense}
Let $\cO\subset \R^3$ be a domain, $\bB$ be a magnetic field and $\bA$ an associated potential. Then $\OP_{h}(-\bA,\cO)$ and $\OP_{h}(\bA,\cO)$ are unitarily equivalent. We have
$$\forall \psi \in \dom(q_{h}[\bA,\cO]), \quad q_{h}[-\bA,\cO](\overline{\psi})=q_{h}[\bA,\cO](\psi)$$
and $\psi$ is an eigenfunction of $\OP_{h}(\bA,\cO)$ if and only if $\overline\psi$ is an eigenfunction of $\OP_{h}(-\bA,\cO)$.
\end{lemma}

\chapter{Partition of unity suitable for IMS type formulas}\index{Partition of unity}

Our partitions of unity on general corner domains have to be compatible with an admissible atlas (Definition \ref{def.atlas}).

\begin{lemma}
\label{lem:IMScov}
Let $n\ge1$ be the space dimension. $M$ denotes $\R^n$ or $\dS^n$. 
Let $\Omega\in\gD(M)$ be a corner domain with an admissible atlas $(\cU_{\bx},\diffeo^{\bx})_{\bx\in \overline{\Omega}}$. Let $K>1$ be a coefficient. Then there exist a positive integer $L$ and two positive constants $\rho_{\max}$ and $\kappa\le1$ (depending on $\Omega$ and $K$) such that for all $\rho\in(0,\rho_{\max}]$, there exists a (finite) set $\sZ\subset \overline\Omega \times [\kappa\rho,\rho]$ satisfying the following three properties
\begin{enumerate}
\item We have the inclusion $\overline\Omega \subset \cup_{(\bx,r)\in\sZ}\,\overline\cB(\bx,r)$
\item For any $(\bx,r)\in\sZ$, the ball $\cB(\bx,Kr)$ is contained in the map-neighborhood $\cU_{\bx}$, 
\item Each point $\bx_0$ of $\overline\Omega$ belongs to at most $L$ different balls $\cB(\bx,Kr)$.
\end{enumerate}
\end{lemma}

Before performing the proof of this lemma, let us draw some easy consequence on the existence of suitable IMS type partitions of unity in corner domains.

\begin{lemma}
\label{lem:IMSpart}
Let $\Omega\in\gD(\R^n)$ and choose $K=2$.  Let $(L,\rho_{\max},\kappa)$ be the parameters provided by Lemma \Ref{lem:IMScov}. For any $\rho\in(0,\rho_{\max}]$ let $\sZ\subset \overline\Omega \times [\kappa\rho,\rho]$ be an associate set of pairs (center, radius). Then there exists a collection of smooth functions $(\troncg)_{(\bx,r)\in\sZ}$ with $\troncg\in \sC^{\infty}_0(\cB(\bx,2r))$ satisfying the identity (partition of unity) 
$$
   \sum_{(\bx,r)\in\sZ}\troncg^2=1  \quad\mbox{on}\quad\overline\Omega
$$
and the uniform estimate of gradients 
\begin{equation*}
\exists C>0,\quad \forall (\bx,r)\in\sZ, \quad 
\|\nabla\troncg\|_{L^{\infty}(\Omega)} \leq C \rho^{-1} \, ,
\end{equation*}
where $C$ only depends on $\Omega$. By construction any ball $\cB(\bx,2r)$ is a map-neighborhood of $\bx$ included the maps of an admissible atlas.
\end{lemma}

\begin{proof}
Let $\troncp\in \sC^{\infty}_0(\cB(\bx,2r))$, with the property that $\troncp\equiv1$ in $\cB(\bx,r)$, and satisfying the gradient bound
$
   \|\nabla\troncp\|_{L^{\infty}(\R^3)} \leq C r^{-1}
$ 
where $C$ is a universal constant. Then we set for each $(\bx_0,r_0)\in\sZ$
\[
   \chi_{(\bx_{0},r_{0})} = \frac{\troncz}{(\sum_{(\bx,r)\in\sZ}\troncp^2)^{1/2}}\ .
\]
Due to property (1) in Lemma \ref{lem:IMScov}, $\sum_{(\bx,r)\in\sZ}\troncp^2\ge1$ and due to property (3), 
\[
   \|\sum_{(\bx,r)\in\sZ}\nabla\troncp^2\|_{L^{\infty}(\R^3)} \le CL_\Omega\,.
\]
We deduce the lemma.
\end{proof}

Here are preparatory notations and lemmas for the proof of Lemma \ref{lem:IMScov}.

Let $\Omega\in\gD(M)$ and $K>1$. If the assertions of Lemma \ref{lem:IMScov} are true for this $\Omega$ and this $K$, we say that Property $\sP(\Omega,K)$ holds. We may also specify that the assertion by the sentence
\[
   \mbox{Property $\sP(\Omega,K)$ holds with parameters $(L,\rho_{\max},\kappa)$.}
\]
Let $\cU^*\subset\subset\cU$ be two nested open sets. We say that the property $\sP(\Omega,K;\cU^*,\cU)$ holds\footnote{This is the localized version of property $\sP(\Omega,K)$.} if the assertions of Lemma \ref{lem:IMScov} are true for this $\Omega$ and this $K$, with discrete sets $\sZ\subset (\cU^*\cap\overline\Omega) \times [\kappa_\Omega\rho,\rho]$ and with (1)-(3) replaced by
\begin{enumerate}
\item We have the inclusion $\cU^*\cap\overline\Omega \subset \cup_{(\bx,r)\in\sZ}\,\overline\cB(\bx,r)$
\item For any $(\bx,r)\in\sZ$, the ball $\cB(\bx,Kr)$ is included in $\cU$ and is a map-neighborhood of $\bx$ for $\Omega$ 
\item Each point $\bx_0$ of $\cU\cap\overline\Omega$ belongs to at most $L$ different balls $\cB(\bx,Kr)$.
\end{enumerate}
Like above the specification is
\[
   \mbox{Property $\sP(\Omega,K;\cU^*,\cU)$ holds with parameters $(L,\rho_{\max},\kappa)$.}
\]
In the process of proof, we will construct coverings which are not exactly balls, but domains uniformly comparable to balls. Let us introduce the local version of this new assertion. For $0<a\le a'$ we say that
\[
   \mbox{Property $\sP[a,a'](\Omega,K;\cU^*,\cU)$ holds with parameters $(L,\rho_{\max},\kappa)$}
\]
if for all $\rho\in(0,\rho_{\max}]$, there exists a finite set $\sZ\subset (\cU^*\cap\overline\Omega )\times [\kappa_\Omega\rho,\rho]$ and open sets $\cD(\bx,r)$ satisfying the following four properties
\begin{enumerate}
\item We have the inclusion $\cU^*\cap\overline\Omega \subset \cup_{(\bx,r)\in\sZ}\,\overline\cD(\bx,r)$
\item For any $(\bx,r)\in\sZ$, the set\footnote{Here $\cD(\bx,Kr)$ is the set of $\by$ such that $\bx+(\by-\bx)/K\in\cD(\bx,r)$.} $\cD(\bx,Kr)$ is included in $\cU$ and is a map-neighborhood of $\bx$ for $\Omega$ 
\item Each point $\bx_0$ of $\cU\cap\overline\Omega$ belongs to at most $L$ different sets $\cD(\bx,Kr)$
\item For any $(\bx,r)\in\sZ$, we have the inclusions
$\cB(\bx,ar) \subset\cD(\bx,r) \subset\cB(\bx,a'r)$.
\end{enumerate}
Note that $\sP[1,1](\Omega,K;\cU^*,\cU)=\sP(\Omega,K;\cU^*,\cU)$.

\begin{lemma}
\label{lem:aa'}
If Property $\sP[a,a'](\Omega,K;\cU^*,\cU)$ holds with parameters $(L,\rho_{\max},\kappa)$, then
\[
   \mbox{Property $\sP(\Omega,\frac{a}{a'}K;\cU^*,\cU)$ 
   holds with parameters $(L,a'\rho_{\max},\kappa)$.}
\]
\end{lemma}

\begin{proof}
Starting from the covering of $\cU^*\cap\overline\Omega$ by the sets $\overline\cD(\bx,r)$ and using condition (4), we can consider the covering of $\cU^*\cap\overline\Omega$ by the balls $\cB(\bx,a'r)$. Then $r':=a'r\in[\kappa a'\rho,a'\rho]=[\kappa\rho',\rho']$ with $\rho'<a'\rho_{\max}$. 

Concerning conditions (2) and (3), it suffices to note the inclusions
\[
   \cB(\bx,\frac{a}{a'}Kr') \subset \cD(\bx,\frac{1}{a'}r'K) = \cD(\bx,rK) \,.
\]
The lemma is proved.
\end{proof}

\begin{proof} {\em \!\!\!of Lemma \ref{lem:IMScov}.\ }
The principle of the proof is a recursion on the dimension $n$.

{\em Step 1.\ } Explicit construction when $n=1$.\\ 
The domain $\Omega$ and the localizing open sets $\cU^*$ and $\cU$ are then open intervals. Let us assume for example that $\cU^*=(-\ell,\ell)$, $\cU=(-\ell-\delta,\ell+\delta)$ and $\Omega=(0,\ell+\delta')$ with $\ell,\delta>0$ and $\delta'>\delta$. Let $K\ge1$. We can take
\[
   \rho_{\max} = \min\Big\{\frac{\ell}{K},\delta\Big\}
\]
and for any $\rho\le\rho_{\max}$ the following set of couples $(\bx_j,r_j)$, $j=0,1,\ldots,J$
\[
   \bx_0 = 0, \ r_0=\rho \quad\mbox{and}\quad
   \bx_j=\rho+\frac{2j-1}{K}\rho, \ r_j=\frac{\rho}{K}\quad\mbox{for}\quad j=1,\ldots,J
\]
with $J$ such that $\bx_J<\ell$ and $\rho+\frac{2J+1}{K}\rho\ge\ell$. If $\bx_J<\ell-\frac{\rho}{K}$, we add the point $\bx_{J+1}= \rho+\frac{2J}{K}\rho$. The covering condition (1) is obvious.

Concerning condition (2), we note that the bound $\rho_{\max}\le \frac{\ell}{K}$ implies that $[0,Kr_0)=[0,K\rho)$ is a map-neighborhood for the boundary of $\Omega$, and the bound $\rho_{\max}\le \delta$ implies that when $j\ge1$, the ``balls'' $(\bx_j-Kr_j,\bx_j+Kr_j)=(\bx_j-\rho,\bx_j+\rho)$ are map-neighborhoods for the interior of $\Omega$.

Concerning condition (3), we can check that $L=K+2$ is suitable.

{\em Step 2.\ } Localization.\\ 
Let $\Omega\in\gD(\R^n)$ or $\Omega\in\gD(\dS^n)$. For any $\bx\in\overline\Omega$, there exists a ball $\cB(\bx,r_x)$ with positive radius $r_\bx$ that is a map-neighborhood for $\Omega$. We extract a finite covering of $\overline\Omega$ by open sets $\cB(\bx^{(\ell)},\frac12 r^{(\ell)})$. We set
\[
   \cU^*_\ell = \cB(\bx^{(\ell)},\frac12 r^{(\ell)})\quad\mbox{and}\quad
   \cU_\ell = \cB(\bx^{(\ell)},r^{(\ell)}).
\]
The map $\diffeo^\ell:=\diffeo^{\bx^{(\ell)}}$ transforms $\cU^*_\ell$ and $\cU_\ell$ into neighborhoods $\cV^*_\ell$ and $\cV_\ell$ of $0$ in the tangent cone $\Pi_\ell:=\Pi_{\bx^{(\ell)}}$. Thus we are reduced to prove the local property $\sP(\Pi_\ell,K;\cV^*_\ell,\cV_\ell)$ for any $\ell$. Indeed
\begin{itemize}
\item The local diffeomorphism $\diffeo^\ell$ allows to deduce Property $\sP(\Omega,K;\cU^*_\ell,\cU_\ell)$ from Property $\sP(\Pi_\ell,K';\cV^*_\ell,\cV_\ell)$ for a ratio $K'/K$ that only depends on $\diffeo^\ell$ (this relies on Lemma \ref{lem:aa'}).
\item Properties $\sP(\Omega,K;\cU^*_\ell,\cU_\ell)$ imply Property $\sP(\Omega,K;\cup_\ell\,\cU^*_\ell,\cup_\ell\,\cU_\ell) = \sP(\Omega,K)$ (it suffices to merge the (finite) union of the sets $\sZ$ corresponding to each $\cU_\ell$).
\end{itemize}

{\em Step 3.\ } Core recursive argument: If $\Omega_{0}$ is the section of the cone $\Pi$, Property $\sP(\Omega_{0},K)$ implies Property $\sP(\Pi,K';\cB(0,1),\cB(0,2))$ for a suitable ratio $K'/K$. We are going to prove this separately in several lemmas (\ref{lem:tens} to \ref{lem:cone}). Then the proof Lemma \ref{lem:IMScov} will be complete.
\end{proof}

\begin{lemma}
\label{lem:tens}
Let $\Gamma$ be a cone in $\gP^{n-1}$. For $\ell=1,2$, let $\cB_\ell$ and $\cI_\ell$ be the ball $\cB(0,\ell)$ of $\R^{n-1}$ and the interval $(-\ell,\ell)$, respectively. We assume that Property $\sP(\Gamma,K;\cB_1,\cB_2)$ holds (with parameters $(L,\rho_{\max},\kappa)$). Then Property $\sP[1,\sqrt2](\Gamma\times\R,K;\cB_1\times\cI_1,\cB_2\times\cI_2)$ holds.
\end{lemma}

\begin{proof}
Let us denote by $\by$ and $z$ coordinates in $\Gamma$ and $\R$, respectively. For $\rho\le\rho_{\max}$, let $\sZ_\Gamma$ be an associate set of couples $(\by,r_\by)$. For each $\by$ we consider the unique set of equidistant points $\sZ_\by=\{z_j\in[-1,1],\ j=1,\ldots,J_\by\}$ such that
\[
   z_j-z_{j-1}=2r_\by \quad\mbox{and}\quad z_1+1=1-z_{J_\by}<r_\by\,.
\]
Then we define
\begin{equation}
\label{eq:Zprod}
   \sZ^{(\rho)} = \big\{(\bx,r_\bx),\quad\mbox{for}\ \ 
   \bx=(\by,z) \ \mbox{with}\   
   (\by,r_\by)\in\sZ_\Gamma,\ z\in\sZ_\by \ \mbox{and}\ r_\bx=r_\by\big\}.
\end{equation}
The associate open set $\cD(\bx,r_\bx)$ is the product
\[
   \cD(\bx,r_\bx) = \cB(\by,r_\by)\times (z-r_\by,z+r_\by)\,.
\]
We have the inclusions $\cB(\bx,r_\bx)\subset\cD(\bx,r_\bx)\subset\cB(\bx,\sqrt2\, r_\bx)$ and it is easy to check that Property $\sP[1,\sqrt2](\Gamma\times\R,K;\cB(0,1)\times\cI_1,\cB(0,2)\times\cI_2)$ holds with parameters $(L',\rho_{\max},\kappa)$ with $L'=LK$.
\end{proof}

\begin{lemma}
\label{lem:annulus}
Let $\Omega$ be a section in $\gD(\dS^{n-1})$, let $\Pi$ be the corresponding cone, and let $\cI_\ell$ be the interval $(2^{-\ell},2^\ell)$ for $\ell=1,2$. We define the annuli 
\[
   \cA_\ell = \Big\{\bx\in\Pi, \quad 
   |\bx|\in\cI_\ell \ \ \mbox{and}\ \ \frac{\bx}{|\bx|}\in\Omega\Big\}.
\]
We assume that Property $\sP(\Omega,K)$ holds (with parameters $(L,\rho_{\max},\kappa)$). Then,
for suitable constants $a$ and $a'$ (independent of $\Omega$ and $K$),
 Property $\sP[a,a'](\Pi,K;\cA_1,\cA_2)$ holds. 
\end{lemma}

\begin{proof}
Let us consider the diffeomorphism
\begin{equation}
\label{eq:diffeoT}
   \begin{aligned}
   \diffeoT : \ \ &\Omega\times(-2,2) &\longrightarrow\ \ \ &\cA_2\\
   &\bx = (\by,z) &\longmapsto \ \ \ &\breve\bx = 2^z\by
   \end{aligned}
\end{equation}
in view of proving Property $\sP[a,a'](\Pi,K;\cA_1,\cA_2)$, for a given $\rho\le\rho_{\max}$, we define a suitable set $\breve\sZ^{(\rho)}$ using the set $\sZ^{(\rho)}$ introduced in \eqref{eq:Zprod}
\begin{equation}
\label{eq:Zann}
   \breve\sZ^{(\rho)} = \big\{(\breve\bx,r_\bx),\quad\mbox{for}\ \ 
   \breve\bx=\diffeoT\bx  \ \ \mbox{with}\ \ 
   (\bx,r_\bx)\in\sZ^{(\rho)} \big\},
\end{equation} 
and the associated open sets
\[
   \breve\cD(\breve\bx,r_\bx) = \diffeoT \big(\cD(\bx,r_\bx)\big) .
\] 
We can check that
\[
   \cB(\breve\bx,ar_\bx) \subset \breve\cD(\breve\bx,r_\bx) \subset
   \cB(\breve\bx,a'r_\bx) 
\]
with $a = \frac18\log2$ and $a'=8\sqrt2\log2$ and that Property $\sP[a,a'](\Pi,K;\cA_1,\cA_2)$ holds with parameters $(L',\rho_{\max},\kappa)$ for $L'=NLK$ with an integer $N$ independent of $L$ and $K$.
\end{proof}

\begin{lemma}
\label{lem:cone}
Let $\Omega$ be a section in $\gD(\dS^{n-1})$, let $\Pi$ be the corresponding cone, and let $\cB_\ell$ be the balls $\cB(0,\ell)$ of $\R^{n}$ for $\ell=1,2$. We assume that Property $\sP(\Omega,K)$ holds with parameters $(L,\rho_{\max},\kappa)$ for a $\rho_{\max}\le1$. Then Property $\sP[a,a'](\Pi,K;\cB_1,\cB_2)$ holds for suitable constants $a$ and $a'$ (independent of $\Omega$ and $K$) and with parameters $(L',1,\kappa\rho_{\max})$. 
\end{lemma}

\begin{proof}
Let $\rho\le1$ and let $M$ be the natural number such that
\[
   2^{-M-1} <\rho\le 2^{-M}.
\]
On the model of \eqref{eq:diffeoT}-\eqref{eq:Zann}, we set
\[
   \breve\sZ^m = \big\{(2^{-m}\diffeoT\bx,2^{-m}r_\bx),\ \ \mbox{with}\ \
   (\bx,r_\bx)\in\sZ^{(2^m\rho_{\max}\rho)} \big\}, \quad
   m=0,\ldots,M,
\]
and the associated open sets are
\begin{equation}
\label{eq:Zm}
   2^{-m}\diffeoT \big(\cD(\bx,r_\bx)\big)\ \ \mbox{with}\ \
   (\bx,r_\bx)\in\sZ^{(2^m\rho_{\max}\rho)} .
\end{equation}
The set $\breve\sZ$ associated with the cone $\Pi$ in the ball $\cB_1$ is
\[
   \{(0,\rho)\} \cup \bigcup_{m=0}^M \breve\sZ^m
\]
and the associated open sets are the reunion of the sets \eqref{eq:Zm} for $m=0,\ldots,M$ and of the ball $\cB(0,\rho)$. As the radii $r_\bx$ belong to $[\kappa 2^m\rho_{\max}\rho,2^m\rho_{\max}\rho]$, we have $2^{-m}r_\bx\in[\kappa \rho_{\max}\rho,\rho_{\max}\rho]$. Since $\rho$ itself belongs to the full collection of radii $r$, we finally find $r\in[\kappa \rho_{\max}\rho,\rho]$. The finite covering holds with $L'=3NLK+1$ for the same integer $N$ appearing at the end of the proof of Lemma \ref{lem:annulus}.
\end{proof}

\backmatter

\providecommand{\bysame}{\leavevmode ---\ }
\providecommand{\og}{``}
\providecommand{\fg}{''}
\providecommand{\smfandname}{\&}
\providecommand{\smfedsname}{\'eds.}
\providecommand{\smfedname}{\'ed.}
\providecommand{\smfmastersthesisname}{M\'emoire}
\providecommand{\smfphdthesisname}{Th\`ese}

\printindex

\end{document}